\documentclass[reqno,9pt]{amsart}
\usepackage[margin=1in]{geometry}

\usepackage{graphicx}
\usepackage{amsmath,amssymb}

\usepackage[makeroom]{cancel}

\newtheorem{theorem}{Theorem}

\newtheorem{lemma}{Lemma}
\newtheorem{proposition}[theorem]{Proposition}

\usepackage{enumerate}
\usepackage{ amssymb }

\makeatletter
\renewcommand*\env@matrix[1][*\c@MaxMatrixCols c]{%
  \hskip -\arraycolsep
  \let\@ifnextchar\new@ifnextchar
  \array{#1}}
\makeatother

\let\e=\varepsilon

\let\p=\partial

\let\O=\Omega

\let\o=\omega

\numberwithin{equation}{section}

\let\hide\iffalse
\let\unhide\fi

\newcommand{\R}{\mathbb{R}}

\renewcommand{\S}{\mathbb{S}}

\newcommand{\be}{\begin{equation}}
\newcommand{\bm}{\begin{multline}}
\newcommand{\ee}{\end{equation}}
\newcommand{\dd}{\mathrm{d}}

\newcommand{\xb}{x_{\mathbf{b}}}
\newcommand{\xbp}{x_{\mathbf{b},\parallel}}
\newcommand{\tb}{t_{\mathbf{b}}}
\newcommand{\vb}{v_{\mathbf{b}}}

\newcommand{\xf}{x_{\mathbf{f}}}
\newcommand{\tf}{t_{\mathbf{f}}}

\newcommand{\xba}{x_{\mathbf{b},1}}
\newcommand{\xbb}{x_{\mathbf{b},2}}

\newcommand{\Bes}{\begin{eqnarray*}}
\newcommand{\Ees}{\end{eqnarray*}}
\newcommand{\Be}{\begin{equation} }
\newcommand{\Ee}{\end{equation}}

\def\p{\partial}

\def\O{\Omega}
\def\R{\mathbb{R}}

\def\B{\begin{equation}}
\def\E{\end{equation}}
\def\BN{\begin{eqnarray*}}
\def\EN{\end{eqnarray*}}

\begin{document}

\title{Regularity of Boltzmann equation with External Fields in Convex Domains of Diffuse Reflection}
\author{Yunbai Cao}

 \address{Department of Mathematics, University of Wisconsin, Madison, WI 53706 USA}
\email{ycao35@wisc.edu}
\begin{abstract}
We consider the Boltzmann equation with external fields in strictly convex domains with diffuse reflection boundary condition. As long as the normal derivative of external fields satisfy some sign condition on the boundary (\ref{signEonbdry}) we construct classical $C^1$ solutions away from the grazing set. As a consequence we construct solutions of Vlasov-Poisson-Boltzmann system having bounded derivatives away from the grazing set (weighted $W^{1,\infty}$ estimate). In particular this improves the recent regularity estimate of such system in weighted $W^{1,p}$ space for $p<6$ in \cite{VPBKim}. 
\end{abstract}

\maketitle

\tableofcontents


\section{Introduction}
The object of kinetic theory is the modeling of particles by a distribution function in the phase space: $F(t,x,v)$ for $(t,x,v) \in [0, \infty) \times  {\O} \times \R^{3}$ where $\O$ is an open bounded subset of $\R^{3}$. Dynamics and collision processes of dilute charged particles with a field $E$ can be modeled by the Boltzmann equation
\Be\label{Boltzmann_E}
\partial_{t} F + v\cdot \nabla_{x} F + E\cdot \nabla_{v} F = Q(F,F).
\Ee
The collision operator measures ``the change rate'' in binary collisions and takes the form of
\Be\begin{split}\label{Q}
Q(F_{1},F_{2}) (v)&: = Q_\mathrm{gain}(F_1,F_2)-Q_\mathrm{loss}(F_1,F_2)\\
&: =\int_{\R^3} \int_{\S^2} 
B(v-u) \cdot \omega) [F_1 (u^\prime) F_2 (v^\prime) - F_1 (u) F_2 (v)]
 \dd \omega \dd u,
\end{split}\Ee   
where $u^\prime = u - [(u-v) \cdot \omega] \omega$ and $v^\prime = v + [(u-v) \cdot \omega] \omega$. Here, $B(v-u,\omega) = |v - u|^\kappa q_0( \frac{v-u}{|v -u |} \cdot \omega ) $ and $0 \le \kappa \le 1 $ (hard potential) and $0 \le q_0( \frac{v-u}{|v -u |} \cdot \omega ) \le C |\frac{v-u}{ |v -u | } \cdot \omega | $ (angular cutoff).

The collision operator enjoys collision invariance: for any measurable function $G$,  
\Be\label{collison_invariance}
\int_{\R^{3}} \begin{bmatrix}1 & v & \frac{|v|^{2}-3}{2}\end{bmatrix} Q(G,G) \dd v = \begin{bmatrix}0 & 0 & 0 \end{bmatrix} .
\Ee 
It is well-known that a global Maxwellian $\mu$ 
satisfies $Q(\mu,\mu)=0$ where
\Be\label{Maxwellian}
\mu(v):= \frac{1}{(2\pi)^{3/2}} \exp\bigg(
 - \frac{|v |^{2}}{2 }
 \bigg).
\Ee

Throughout this paper we assume that $\Omega$ is a bounded open subset of $\mathbb R^3$ and there exists a $C^3$ function $\xi: \mathbb R^3 \to \mathbb R$ such that $\Omega = \{ x \in \mathbb R^3: \xi(x) < 0 \}$, and $\partial \Omega = \{ x\in \mathbb R^3 : \xi(x) = 0 \}$. Moreover we assume the domain is \textit{strictly convex}:
\[
\sum_{i,j} \partial_{ij} \xi(x) \zeta_i \zeta_j \ge C_\xi |\zeta|^2 \, \text{ for all } \, \zeta \in \mathbb R^3 \text{ and for all } x\in \bar \Omega = \Omega \cup \partial \Omega.
\]
We assume that 
\Be \label{gradientxinot0}
\nabla \xi(x) \neq 0  \text{ when } |\xi(x) | \ll 1,
\Ee
and we define the outward normal as $n(x) = \frac{ \nabla \xi(x) }{ | \nabla \xi (x) |}$ at the boundary.
%
The boundary of the phase space $
\gamma := \{ (x,v) \in \partial \Omega \times \mathbb R^3 \}$ can be decomposed as 
\begin{equation} \begin{split}
\gamma_- = \{ (x,v) \in \partial \Omega \times \mathbb R^3 : n(x) \cdot v < 0 \}, &\quad (\text{the incoming set}),
\\ \gamma_+ = \{ (x,v) \in \partial \Omega \times \mathbb R^3 : n(x) \cdot v > 0 \}, &\quad (\text{the outcoming set}),
\\ \gamma_0 = \{ (x,v) \in \partial \Omega \times \mathbb R^3 : n(x) \cdot v = 0 \}, &\quad (\text{the grazing set}).
\end{split} \end{equation}
In general the boundary condition is imposed only for the incoming set $\gamma_-$ for general kinetic PDEs. In this paper we consider a so-called diffuse boundary condition
\Be \label{diffuseF}
F(t,x,v) = c_\mu \mu(v) \int_{n(x) \cdot u > 0 } F(t,x,u) \{ n(x) \cdot u \} du, \text{ on } (x,v) \in \gamma_-,
\Ee
with $ c_\mu \int_{n(x) \cdot u > 0 } \mu(u) \{ n(x) \cdot u \} du = 1$. For other important boundary condition, such as the specular reflection boundary condition, we refer \cite{Guo10,KL1,KL2} and the references therein.

Due to its importance of the Boltzmann equation in the mathematical theory and application,
there have been explosive research activities in analytic study of the equation. Notably the nonlinear energy method has led to solutions of many open problems including global strong solution of Boltzmann equation coupled with either the Poisson equation or the Maxwell system for electromagnetism when the initial data are close to the Maxwellian $\mu$ in periodic box (no boundary). See \cite{Guo_M} and the references therein. 
%
%
%
%
In many important physical applications, e.g. semiconductor and tokamak, the charged dilute
gas is confined within a container, and its interaction with the boundary plays a crucial role both in physics and mathematics. 

However, in general, higher regularity may not be expected for solutions of the Boltzmann equation in physical bounded domains. Such a drastic difference of solutions with boundaries had been
demonstrated as the formation and propagation of discontinuity in non-convex domains \cite{Kim11, EGKM}, and a non-existence of some second order derivative at the boundary in convex domains \cite{GKTT1}. Evidently the nonlinear energy method is not generally available to the boundary problems.
In order to overcome such critical difficulty, Guo developed a $L^2$-$L^\infty$ framework in \cite{Guo10} to study global solutions of the Boltzmann equation with various boundary conditions. The core of the method lays in a direct approach (without taking derivatives) to achieve a pointwise bound using trajectory of the transport operator, which leads substantial development in various directions including \cite{EGKM2, EGKM, GKTT1, GKTT2, KBOX}.
In \cite{GKTT1}, with the acid of some distance function towards the grazing set, they construct weighted classical $C^1$ solutions of Boltzmann equation ($E\equiv0$ in (\ref{Boltzmann_E})) with various boundary conditions away from the grazing set. They also construct $W^{1,p}$ solution for $1<p<2$ and weighted $W^{1,p}$ solutions for $2 \le p < \infty$ as well.
%
%

In the first part of the paper, we extend a result of \cite{GKTT1} to the Boltzmann equation (\ref{Boltzmann_E}) with an external field $(E\neq 0 )$ satisfying a crucial sign condition on the boundary:
\Be \label{signEonbdry}
E(t,x) \cdot n(x) > C_E > 0 \quad \text{ for all } t \text{ and all } x \in \partial \Omega.
\Ee
One of the major difficulties is that trajectories are curved and behave in a very complicated way when they hit the boundary. 

 We denote $\| \cdot \|_p$ the $L^p(\Omega \times \mathbb R^3 )$ norm, while $| \cdot |_{\gamma,p } = |\cdot |_p$ is the $L^p(\partial \Omega \times \mathbb R^3; d\gamma)$ norm, $| \cdot |_{\gamma_{\pm},p} = | \cdot \textbf{1}_{\gamma_{\pm}} |_{\gamma,p}$, $d \gamma = |n(x) \cdot v | d S_x dv$ with the surface measure $dS_x$ on $\partial \Omega$.

Our main results are $W^{1,p} (1<p<2)$ estimate, weighted $W^{1,p} (2\le p<\infty)$ estimate and weighted $C^1$ estimate for the solution of (\ref{Boltzmann_E}) with diffuse boundary condition (\ref{diffuseF}) in a short time. For the $W^{1,p}$ estimate with $1< p<2$, the result is

\begin{theorem}[$W^{1,p}$ Estimate for $1 < p< 2$] \label{W1ppless2thm} 
Suppose $E$ satisfies (\ref{signEonbdry}), and  $\| E \|_\infty< \infty$. Assume the compatibility condition of $F_0 = \sqrt \mu f_0 $ on $(x,v) \in \gamma_-$,
\begin{equation} \label{compatibility}
f_0(x,v) = c_\mu \sqrt{\mu(v) } \int_{n(x) \cdot u > 0 } f_0 (x,u) \sqrt{\mu(u) } (n(x) \cdot u ) du.
\end{equation}
If $\| e^{\theta |v|^2 } f_0 \|_\infty + \|  \nabla_{x,v} f_0 \|_p < \infty$ for some $0< \theta <1/4$ and any fixed $1 < p <2$, then there exists a unique solution $F(t) = \sqrt \mu f(t)$ for $t \in [0, T]$ with $0 < T \ll 1 $ to the system (\ref{Boltzmann_E}), (\ref{diffuseF}) that satisfies
, for all $0 \le t \le T$, 
\Be
\| e^{-\varpi \langle v \rangle t } \nabla_{x,v} f(t) \|_p^p + \int_0^t | e^{-\varpi \langle v \rangle s } \nabla_{x,v} f(s) |_{\gamma, p}^p ds + \| e^{\theta' |v|^2 }f(t) \|_\infty \lesssim_t \|\nabla_{x,v} f_0 \|_p^p + P(\| e^{\theta |v|^2} f_0 \|_\infty), 
\Ee
for some polynomial $P$, $0 < \theta' < \theta$, and $\varpi \gg 1$.
\end{theorem}

In order to have weighted $W^{1,p}$ estimate for $p \ge 2$ and the weighted $C^1$ estimate, we introduce a distance function $\alpha(t,x,v)$ towards the grazing set $\gamma_0$:
\Be \label{alphatilde}
 \alpha(t,x,v) \sim \bigg[ |v \cdot \nabla \xi (x)| ^2 + \xi (x)^2 - 2 (v \cdot \nabla^2 \xi(x) \cdot v ) \xi(x) - 2(E(t,\overline x ) \cdot \nabla \xi (\overline x ) )\xi(x) \bigg]^{1/2}
\Ee
for $x \in \Omega$ close to boundary, where $\overline x := \{ \bar x \in \p \Omega :  d(x,\bar x ) = d(x, \partial \Omega) \}$ is uniquely defined. The precise definition of $\alpha$ can be found in (\ref{alphadef}). Note that $\alpha \vert_{\gamma_-} \sim | n(x) \cdot v |$, and similar distance function towards $\gamma_0$ was used in \cite{GKTT1,Guo_V,Hwang}. 

One of the crucial property $\alpha$ enjoys, under the assumption of the sign condition \eqref{signEonbdry}, is the velocity lemma (Lemma \ref{velocitylemma}):
\begin{equation} \label{vlemma}
e^{ - C \int_s ^ t \langle V(\tau') \rangle d \tau'} \alpha ( s,X(s),V(s) ) \le \alpha (t,x,v) \le  e^{C \int_s ^ t  \langle V(\tau') \rangle d \tau'} \alpha (s,X(s),V(s)).
\end{equation}
This can be seen by directly taking derivatives along the trajectory:
\Be \label{transderivbeta1}
 |\{  \p_t + v\cdot \nabla_x + E \cdot \nabla_v \} { \alpha}^2(t,x,v) | \sim |v| { \alpha}^2 +C |v| \xi (x),
\Ee
for some $C \lesssim_{\xi, E } 1 $. Now under \eqref{signEonbdry}, we get an extra stronger control for $\xi(x)$ from the last term of $ \alpha^2$, and therefore the second term on the right-hand side of \eqref{transderivbeta1} can be bounded by:
\Be \label{boundalphaxi}
C |v| \xi (x) \le \frac{C}{ \inf_{y \in \p \O} E(t,y) \cdot \nabla \xi (y) } |v| (E(t,\overline x ) \cdot \nabla \xi (\overline x ) ) \xi (x)  \le \frac{C}{ C_E}  \alpha^2(t,x,v).
\Ee
Thus combing \eqref{transderivbeta1} and \eqref{boundalphaxi} we obtain \eqref{vlemma} from Gronwall. \eqref{vlemma} tells that $\alpha$ is almost invariant along the characteristics, especially for small $t \ll 1$, which is crucially used for establishing the following theorems.


\begin{theorem}[Weighted $W^{1,p}$ Estimate for $2 \le p < \infty$]  \label{weightedw1ppge2theorem} Suppose $E$ satisfies the sign condition (\ref{signEonbdry}), and
\Be \label{c1bddforthepotentail}
\| E(t,x) \|_\infty + \| \nabla_x E(t,x) \|_\infty +  \| \partial_t E(t,x) \|_\infty < \infty.
\Ee
Assume the compatibility condition (\ref{compatibility}). For any fixed $2 \le p < \infty$ and $\frac{p-2}{p} < \beta < \frac{p-1}{p}$, if $\| \alpha^\beta \nabla_{x,v} f_0 \|_p + \| e^{\theta |v|^2} f_0 \|_\infty < \infty$ for some $0 < \theta < \frac{1}{4}$, then there exists a unique solution $F(t) = \sqrt \mu f(t)$ for $t \in [0,T]$ with $0 < T \ll 1$ to the system (\ref{Boltzmann_E}), (\ref{diffuseF}) that satisfies, for all $0 \le t \le T$, 
%
\Be
\| e^{-\varpi \langle v \rangle t }  \alpha ^\beta \nabla_{x,v} f (t) \|_p^p + \int_0^t |e^{-\varpi \langle v \rangle s }  \alpha ^\beta \nabla_{x,v} f(s)  |_{\gamma, p}^p ds + \| e^{\theta' |v|^2 } f(t) \|_\infty \lesssim_t  \| \alpha^\beta \nabla_{x,v} f_0 \|_p + P(\| e^{\theta |v|^2} f_0 \|_\infty),
\Ee
for some polynomial $P$, $0 < \theta' < \theta$, and $\varpi \gg 1$.
\end{theorem}

\begin{theorem}[Weighted $C^1$ Estimate]  \label{C1linearthm} 
Suppose $E$ satisfies (\ref{signEonbdry}), and (\ref{c1bddforthepotentail}).
Assume the compatibility condition (\ref{compatibility}).
If $\| \alpha \nabla_{x,v} f \|_\infty + \| e^{\theta |v|^2} f_0 \|_\infty < \infty$ for some $0 < \theta < \frac{1}{4}$, then there exists a unique solution $F(t) = \sqrt \mu f(t)$ for $t \in [0,T]$ with $0 < T \ll 1$ to the system (\ref{Boltzmann_E}), (\ref{diffuseF}) that satisfies 
%
for all $0 \le t \le T$,
\Be \label{weightedW1inftyforexternalp}
\| e^{-\varpi \langle v \rangle t }  \alpha  \nabla_{x,v} f (t) \|_\infty + \| e^{\theta' |v|^2 } f(t) \|_\infty  \lesssim_t  \| \alpha \nabla_{x,v} f_0 \|_\infty + P(\| e^{\theta |v|^2} f_0 \|_\infty), \text{ for all } 0\le t \le T,
\Ee
for some polynomial $P$, $0 < \theta' < \theta$, and $\varpi \gg 1$. If $\alpha \nabla f_0 \in C^0( \bar \Omega \times \mathbb R^3 ) $
is valid for $\gamma_-$, then $f \in C^1$ away from the gazing set $\gamma_0$.
\end{theorem}

For the second part of this paper we consider a so-called Vlasov-Poisson-Boltzmann system (VPB) where the potential consists of a self-generated electrostatic potential and an external potential: $E = \nabla \phi$, where
%
\begin{equation} \label{VPB2}
\phi(t,x) = \phi_F(t,x) + \phi_E(t,x), \text{ with } \frac{\p \phi_E}{\p n } > C_E > 0 \text{ on } \p \Omega,
\end{equation}
\begin{equation} \label{VPB3}
-\Delta_x \phi_F(t,x) = \int_{\mathbb R^3 } F (t,x,v)dv - \rho_0 \text{ in } \Omega,  \,\,  \frac{\partial \phi_F}{\partial n } = 0 \text{  on } \partial \Omega,
\end{equation}
with the same diffuse boundary condition (\ref{diffuseF}). The coupled system (\ref{Boltzmann_E}), (\ref{VPB2}), (\ref{VPB3}) describes the dynamics of collisional electrons in the presence of a external field. With the help of the external field $\phi_E$ and its sign condition on the boundary (\ref{signEonbdry}), we could construct a short time weighted $W^{1,\infty}$ solution to the VPB system, which improves the recent regularity estimate of such system in weighted $W^{1,p}$ space for $p< 6$ in \cite{VPBKim, CaoKRM}. It is important to note that $\alpha$ in (\ref{alphadef}) only depends on $E|_{\p \Omega}$, therefore $\nabla \phi_E$, but not $\phi_F$.
Our main result is
\begin{theorem}[Weighted $W^{1,\infty}$ estimate for the VPB system]  \label{WlinftyVPBthm} 
 Let $\phi_E (t,x)$ be a given external potential with $\nabla_x \phi_E$ satisfying (\ref{signEonbdry}), and 
\Be
\| \nabla_x \phi_E(t,x) \|_\infty + \| \nabla_x^2 \phi_E(t,x) \|_\infty +  \| \partial_t \nabla_x \phi_E(t,x) \|_\infty < \infty.
\Ee
 Assume that
\Be \label{VPBf0assumption}
\|  e^{\theta |v|^2 } \alpha \nabla_{x,v} f_0 \|_\infty + \| e^{\theta |v|^2 } f_0 \|_\infty + \| e^{\theta |v|^2 } \nabla_v f_0 \|_{L^3_{x,v}} < \infty,
\Ee
for some $0< \theta < \frac{1}{4}$.Then there exists a unique solution $F(t,x,v) = \sqrt \mu f(t,x,v) $ to (\ref{Boltzmann_E}), (\ref{VPB2}), (\ref{VPB3}) for $t \in [0,T]$ with $0 < T \ll 1$, such that for some $0< \theta' < \theta $, $\varpi \gg 1$,
\begin{equation}  \label{linfinitybddsolution}
\sup_{0\le t \le T} \| e^{\theta' |v|^2 } f(t) \|_\infty < \infty.
\end{equation}
Moreover
\begin{equation}\label{weightedC1bddsolution}
\sup_{0 \le t \le T}  \| e^{\theta' |v| ^2 } e^{-\varpi  \langle v \rangle t } \alpha  \nabla_{x,v}  f^{} (t,x,v) \|_\infty < \infty,
\end{equation}
and
\begin{equation} \label{L3L1plusbddsolution}
 \sup_{0 \le t \le T}\| e^{-\varpi \langle v \rangle t } \nabla_v f(t) \|_{L^3_x(\Omega)L_v^{1+\delta}(\mathbb R^3 ) } < \infty \text{ for } 0< \delta \ll 1.
\end{equation}
\end{theorem}

We now illustrate the main ideas in the proof of the theorems. The intrinsic difficulty of regularity estimates stems from the singularity of the spatial normal derivative of $F$ at the boundary. From the equation (\ref{Boltzmann_E}), formally we have
\Be \label{pnF}
\frac{\p F}{\p n } \sim \frac{1}{n \cdot v } \Big\{ Q(F,F) - E \cdot \nabla_v F - \p_t F - \sum_{i=1}^2 \tau_i \p_{\tau_i} F \Big\} \text{ on } \p \Omega,
\Ee
where $\tau_1(x)$ and $\tau_2(x)$ are unit tangential vectors to $\p \Omega$ satisfying 
\begin{equation} \label{tangentialderivativedef}
\tau_1(x) \cdot n(x) = 0 = \tau_2 (x) \cdot n(x) \text{ and } \tau_1(x) \times \tau_2(x) = n(x).
\end{equation}
We note that the non-local term $Q(F,F)$ prevents the right hand side of (\ref{pnF}) from vanishing and hence this singularity persists in general.

The proofs of Theorem 1-3 devote a \textit{nontrivial} extension of the argument of \cite{GKTT1} in the presence of external fields with the crucial sign condition (\ref{signEonbdry}). For Theorem \ref{W1ppless2thm}, we establish the green's identity for transport equation with external field and apply it to the derivatives $\nabla_{x,v} f$.
Clearly, the $v$ derivatives behave nicely for the diffuse boundary condition. For the $x$ derivatives on the boundary, one can decompose $\nabla_x$ as the tangential derivatives $\partial_\tau$ and normal derivative $\partial_n$. As in \cite{GKTT1}, we use the Boltzmann equation and the diffuse boundary condition to find a formula of $\partial_n f$ on $\gamma_-$:
\Be \label{pnf}
\p_n f \sim \frac{1}{n\cdot v } \int_{n \cdot u > 0 } \Big\{ - u \cdot \nabla_x f +\sum_{i=1}^2 \p_{\tau_i} f + \nabla_v f + \text{ lower order terms } \Big\}  \left( n(x) \cdot u \right) du.
\Ee
Due to the crucial factor $|n(x) \cdot u |$ in the integral of (\ref{pnf}), the boundary integral of $L^p$ in the green's identity has integrand with singularity as order
%
\[
\frac{1}{(n\cdot v)^{p-1}} \in L^1_{\text{loc}}(v) \ \ \text{for} \ \ 1<p<2. 
\]

The distance function $\alpha$ plays a crucial role in the proofs of Theorem \ref{weightedw1ppge2theorem}, Theorem \ref{C1linearthm}, and Theorem \ref{WlinftyVPBthm}, which can be controlled along the characteristics via the geometric velocity lemma (Lemma \ref{velocitylemma}). Note that in the presence of external fields and \eqref{transderivbeta1}, \eqref{boundalphaxi}, we can prove the velocity lemma \textit{only} when the sign condition (\ref{signEonbdry}) holds. Because of the non-local nature of the Boltzmann collision operator, which mixes up different velocities $u \in \mathbb R^3$, we establish a delicate estimate for the interaction of $ \alpha^\beta(t,x,v)$ with the collision kernel in (\ref{uv}), where by the way $\alpha$ is defined, we can control
\[
\int_{|u|<1 } \frac{1}{\{ \alpha(s,x,u) \}^{\frac{\beta p }{ p -1 } }}  du \lesssim \int_{|u| < 1}  \frac{1}{|n(x) \cdot u|^{\frac{\beta p }{ p -1 } }} du < \infty \ \ \text{for} \  \ \beta < \frac{p-1}{p}.
\]
 On the other hand, the appearance of $|n(x) \cdot v |^{ \beta p -p + 1 } $ in the boundary estimate will need
 an additional requirement $\beta > \frac{p-2}{p}$ to control the boundary singularity in (\ref{gamma-pge2}). These estimates are sufficient to treat the case for $\beta < 1$, but unfortunately fail for the use $\beta =1$, which accounts for the important $C^1$ estimate.

In order to establish the $C^1$ estiamte, we employ the Lagrangian view point, estimating along the trajectory. Even though one can not hope to control the regularity near $\gamma_0$ due to non-local nature of the collision operator, one can control its singular behavior (i.e. with weight $\alpha$) with an important dynamical non-local to local estimate (Lemma \ref{keylemma}). The crucial gain of $\alpha$, which only can be obtained for expected singular behavior with negative power of $\alpha$, is due to a combination of two facts: the gain of power $1$ is due to a velocity average, and gain of the local behavior of $\alpha$ is due to time integration and convexity.

The proof of such non-local to local estimates is a combination of analytical and geometrical arguments. The first part (Lemma \ref{int1overalphabetainu}) is a precise estimate of the velocity integration which is bounded by $  |\xi (X(s)) | ^{-\frac{\beta -1 }{2}},$ here one may roughly regard $\xi(X(s)) \sim \text{dist}(X(s),\p \Omega )$.
 In this part of the proof we make use of a series of change of variables to obtain the precise power $\frac{\beta -1}{2}$. The second part is to relate the time integration back to $\frac{1}{\alpha}$. For this part of proof, we first have the velocity lemma (Lemma \ref{velocitylemma}) and the boundedness of the external field to ensure the monotonicity of $|\xi(X(s))|$ near the boundary, where we can use the change of variable
\[
dt \simeq \frac{d\xi}{|v\cdot \nabla \xi |},
\]
and recover a power of $\alpha$ as in the bound of $\xi$-integration through the velocity lemma (Lemma \ref{velocitylemma}). On the other hand, we use the sign condition (\ref{signEonbdry}) crucially to establish a lower bound for $|\xi(X(s))|$ when it's away from the boundary, which helps to recover a power of $\alpha$ as wanted.

In Theorem \ref{WlinftyVPBthm}, we apply the idea of weighted $C^1$ estimate, essentially the non-local to local estimate (Lemma \ref{keylemma}), to the VPB system. 
Here the argument is more delicate as the potential is no longer fixed as in previous case. Thus in the bulk we have to control the quadratic nonlinear term
\[
  \p \nabla \phi \cdot \nabla_v f  .
\]
In order to handle this term we need a bound for $\phi_F(t)$ in $C^2_x$. Unfortunately such estimate is a boarder line case of the well-known Schauder elliptic regularity theory in (\ref{VPB3}) when $F$ is merely continuous or bounded. A key observation is that
\[
 \left \| \int_{\mathbb R^3} \nabla_x f(t) \sqrt \mu dv \right \|_{L^p(\Omega ) } \lesssim  \left\| e^{-\varpi \langle v \rangle t } \alpha \nabla_{x} f (t) \right\|_\infty \left\| \int_{\mathbb R^3 } e^{ \varpi \langle v \rangle t } \sqrt \mu \frac{1}{\alpha} dv \right\|_{L^p(\Omega ) },
\]
which leads $C^{2,0+}$ bound of $\phi_F$ by the Morrey inequality for $p>3$ as we can bound $\left\| \int_{\mathbb R^3 } e^{ \varpi \langle v \rangle t } \sqrt \mu \frac{1}{\alpha} dv \right\|_{L^p(\Omega ) }< \infty$ in (\ref{int1alphaLpbdd}).

For constructing a solution and proving its uniqueness, we need some \textit{stability} estimate of the difference of the solutions $f-g$. The difficulty comes from again the term of $\nabla_x \phi_F \cdot \nabla_ v f$. To prove $L^{q}$-stability for $q=1+\delta$ with $0< \delta \ll 1$ we have, by Sobolev embedding $\nabla_x \phi_{f-g} \in W^{1, q } (\O) \subset L(\O)^{ \frac{3q}{3-q}}$, 
\Be\notag
\iint |\nabla_x \phi_{f-g} \cdot \nabla_v f | |f-g|^{q-1}
\lesssim  \|\nabla_x \phi_{f-g}\|_{L_x^{ \frac{3q}{3-q}}}  \left\| \| \nabla_v f \|_{L^q_v} \right\|_{L^3_x} \left\| |f-g|^{q-1}\right\|_{L^{\frac{q}{q-1}}_{x,v}  }. 
\Ee
Note that $\nabla_v f$ is bounded from the boundary condition (\ref{diffuseF}). However the equation of $\nabla_v f$ has $\nabla_x f$ as a forcing term. Therefore the key term to bound $\left\| \| \nabla_v f \|_{L^{q}_v} \right\|_{L^3_x}$ for $q=1+ \delta $ is
\Be\notag\begin{split}
\left\|\left\|\int^t_0 \nabla_x f(s,X(s;t,x,v), V(s;t,x,v)) \dd s \right\|_{L^{1+\delta}_v}\right\|_{L^{3}_x} 
\lesssim  \sup_t
\Big\| \frac{e^{-\frac{\theta'}{2} |v|^2 }}{ \alpha} \Big\|_{L^3_x L^{1+\delta}_v}
\| e^{\theta' |v|^2} e^{-\varpi \langle v \rangle t } \alpha \nabla_x f \|_{\infty} < \infty,
\end{split}\Ee
as $\sup_t \| \frac{e^{-\frac{\theta'}{2} |v|^2 }}{ \alpha} \|_{L^3_x L^{1+\delta}_v} < \infty$.


\section{Traces and in-flow problems with external fields}
Now let $F(t,x,v) = \sqrt{\mu} f(t,x,v)$. Then the corresponding problem to (\ref{Boltzmann_E}), (\ref{diffuseF}) is
\begin{equation} \label{Bextfield1}
(\partial_t + v \cdot \nabla_x + E \cdot \nabla_v  - \frac{v}{2} \cdot E+ \nu( \sqrt \mu f) ) f = \Gamma_{\text{gain}} (f,f),
\end{equation}
\Be \label{diffusebdycondtionf}
f(t,x,v) = c_\mu \sqrt \mu(v) \int_{n(x) \cdot u > 0 } f(t,x,u) \sqrt{\mu(u)} \{ n(x) \cdot u \} du, \text{ on } (x,v) \in \gamma_-.
\Ee
Here
 \[ \begin{split}
 \nu (\sqrt \mu f )(v) : & = \frac{1}{\sqrt \mu (v) } Q_{\text{loss}} (\sqrt \mu f, \sqrt \mu f )(v)
 \\ & = \int_{\mathbb R^3 } \int_{\mathbb S^2} |v - u|^{\kappa} q_0 (\frac{ v -u}{ |v -u| } \cdot w ) \sqrt{\mu(u) } f(u) d\omega du,
\end{split} \]
and
\[ \begin{split}
 \Gamma_{\text{gain}} (f_1,f_2) (v): & = \frac{1}{\sqrt \mu (v) } Q_{\text{gain}} (\sqrt \mu f_1, \sqrt \mu f_2 )(v)
 \\ & = \int_{\mathbb R^3 } \int_{\mathbb S^2} |v - u|^{\kappa} q_0 (\frac{ v -u}{ |v -u| } \cdot w ) \sqrt{\mu(u) } f_1(u')f_2(v') d\omega du.
\end{split} \]
%
%
Throughout this paper we extend $f$ for a \textit{negative time}. Let 
\Be\label{negative_t_extension}
f(s,x,v) := e^{s} f_0(x,v) \ \ \text{for} \ - \infty<s<0.
\Ee
Note that this allows $\phi_F$ to solve (\ref{VPB3}) for a negative time. 

For $(t,x,v) \in (-\infty,T]  \times  \O \times \R^3$, let $(X(s;t,x,v),V(s;t,x,v))$ denotes the characteristics
\Be\label{hamilton_ODE}
\frac{d}{ds} \left[ \begin{matrix}X(s;t,x,v)\\ V(s;t,x,v)\end{matrix} \right] = \left[ \begin{matrix}V(s;t,x,v)\\ 
E (s, X(s;t,x,v))\end{matrix} \right]  \ \ \text{ for }   -\infty <  s ,  t \le T  ,
\Ee
with $(X(t;t,x,v), V(t;t,x,v)) =  (x,v)$.

We define \textit{the backward exit time} $\tb(t,x,v)$ as   
\Be\label{tb}
\tb (t,x,v) := \sup \{s \geq 0 : X(\tau;t,x,v) \in \O \ \ \text{for all } \tau \in (t-s,t) \}.
\Ee
Furthermore, we define $\xb (t,x,v) := X(t-\tb(t,x,v);t,x,v)$, and $\vb(t,x,v) := V(t-\tb(t,x,v);t,x,v)$. We also define the  \textit{the forward exit time} $\tf(t,x,v)$ as   
$
\tf (t,x,v) := \sup \{s \geq 0 : X(\tau;t,x,v) \in \O \ \ \text{for all } \tau \in (t,t+s) \}.
$

For the rest part of the section we prove some estimates for the initial-boundary problems of the transport equation with a given time dependent potential $E(t,x)$ which is defined for all $t \in \mathbb R$.
\begin{equation} \label{traninflowfixE}
\partial_t f + v \cdot \nabla_x f+ E \cdot \nabla_v f + \nu f = H,
\end{equation}
where $H=H(t,x,v)$ and $\nu = \nu (t,x,v)$ are given. 

\begin{lemma} \label{lbtb}
Let $D = \sup\{|x -y | : x, y \in \bar \Omega \}$ be the diameter of the domain $\Omega$. Suppose $\| E \|_\infty < \infty$, and let $0 < T<1$ be fixed. Then for any $(t,x,v) \in [0,T] \times \bar \Omega \times \mathbb R^3$ we have
\begin{equation} \label{intVbdd}
\int_{\max\{0, t -\tb\}} ^t |V(s) | ds < 5t(\| E\|_\infty + D ) + 4D.
\end{equation}

\end{lemma}

\begin{proof}
Let
\begin{equation} \label{Mv}
M_v = 4(  \| E\|_\infty + D  ),
\end{equation}
If $| v | > M_v$, then
 \begin{equation} \label{Vsgev/2}
|V(s;t,x,v) |  \le |v| + T \| E \|_\infty < |v| + \frac{ |v|}{ 4 } <2|v|,
\end{equation}
and
\begin{equation} \label{Vslowerbddbyv/2}
V(s) \cdot \frac{v}{|v|}  = v \cdot \frac{v}{|v|} - \int_s^t \left( E(\tau,X(\tau)) \cdot \frac{v}{|v|} \right) d\tau \ge |v| - t  \| E \|_\infty > \frac{|v|}{2}.
\end{equation}
Thus from (\ref{Vslowerbddbyv/2}) we have
\begin{equation} \label{intVsbddbyD}
D > \int_{\max\{0,t-\tb\}}^t V(s) \cdot \frac{v}{|v| } ds >  \int_{\max\{0,t-\tb\}}^t \frac{|v|}{2} ds \ge \frac{ t |v|}{2 }.
\end{equation}
Therefore (\ref{Vsgev/2}), (\ref{intVsbddbyD}) implies
\[
 \int_{\max\{0,t-\tb\}}^t | V(s) | ds <  \int_{\max\{0,t-\tb\}}^t 2 |v| ds < 2t |v| < 4D.
\]

On the other hand if $|v| \le M_v$, 
\[
 \int_{\max\{0,t-\tb\}}^t | V(s) | ds \le  \int_{\max\{0,t-\tb\}}^t ( |v| + t \| E \|_\infty ) ds < t M_v + t^2 \| E \|_\infty < 5t( \| E \|_\infty + D ),
\]
as wanted.
\end{proof}

\begin{lemma} \label{covlemma1}
For fixed $s$ with $t -\tb(t,x,v) < s <t$, the map
\begin{equation} \label{covfirstmap}
(t,x,v) \in (s,T] \times \gamma_+ \mapsto (X(s;t,x,v), V(s;t,x,v) ) \in \Omega \times \mathbb R^3
\end{equation}
is injective with determinant
\begin{equation} \label{detcofv1}
\det \left( \frac{\p ( X( s;t, x, v),  V( s;t, x, v) )}{\p (t, \bar{x}, v)} \right) = |n(x) \cdot v |.
\end{equation}
\end{lemma}

\begin{proof}
First from (\ref{gradientxinot0}), we have that locally for any ${p} \in \partial{\Omega}$, there exists sufficiently small $\delta_{1}>0, \delta_{2}>0$, and an one-to-one and onto $C^{2}$-map
	\begin{equation}\label{eta}
	\begin{split}
	\eta_{{p}}:  \{ x_{{ \parallel}} \in \mathbb{R}^{2}: |x_{ \parallel}| < \delta_1  \}  \ &\rightarrow  \ \p\Omega \cap B({p}, \delta_{2}),\\
	x_{{ \parallel}}=(x_{ \parallel,1},x_{\parallel,2} )	 \ &\mapsto \    \eta_{{p}}  (x_{ \parallel,1},x_{\parallel,2} ).
	\end{split}
	\end{equation} 

Now the map (\ref{covfirstmap}) is injective as the characteristics are deterministic. 
From (\ref{eta}), we can compute the determinant of this change of variable:

\Be
\begin{split}\label{XV_txv}
&\frac{\p ( X( s;t, \eta(x_\parallel), v),  V( s;t, \eta(x_\parallel), v) )}{\p (t, x_\parallel, v)}\\
= & \ \begin{bmatrix}
\p_t X( s;t, \eta(x_\parallel), v) & \nabla_{x_\parallel} X( s;t, \eta(x_\parallel), v) & \nabla_v X( s;t, \eta(x_\parallel), v)\\
\p_t V( s;t, \eta(x_\parallel), v)& \nabla_{x_\parallel} V( s;t, \eta(x_\parallel), v) & \nabla_v V( s;t, \eta(x_\parallel), v)
\end{bmatrix}\\
= & \ 
\begin{bmatrix}
\p_t X( s;t, \eta(x_\parallel), v)& \nabla_{x_\parallel} \eta (x_\parallel) \cdot \nabla_x X( s;t, \eta(x_\parallel), v)  & \nabla_v X( s;t, \eta(x_\parallel), v)\\
\p_t V( s;t, \eta(x_\parallel), v) &  \nabla_{x_\parallel} \eta (x_\parallel) \cdot \nabla_x V( s;t, \eta(x_\parallel), v)  & \nabla_v V( s;t, \eta(x_\parallel), v)
\end{bmatrix}.
\end{split}
\Ee
 Note that 
 \Be\notag
 \begin{split}
X(s; t+ \Delta, X(t+ \Delta;t, \eta(x_\parallel),v), V(t+ \Delta;t, \eta(x_\parallel),v))& = X(s;t, \eta(x_\parallel), v),\\
V(s; t+ \Delta, X(t+ \Delta;t, \eta(x_\parallel),v), V(t+ \Delta;t, \eta(x_\parallel),v))& = V(s;t, \eta(x_\parallel), v).\\
 \end{split}
 \Ee
Therefore 
\Be \begin{split}\notag
[\p_t + v\cdot \nabla_x -\nabla_x \phi  (t, \eta(x_\parallel)) \cdot \nabla_v]X(s;t, \eta(x_\parallel), v) & = 0,\\
[\p_t + v\cdot \nabla_x -\nabla_x \phi  (t, \eta(x_\parallel)) \cdot \nabla_v]V(s;t, \eta(x_\parallel), v) & = 0.
\end{split}
\Ee
Equivalently     
\Be\label{pt_XV}
 \begin{split}
 \begin{bmatrix}
\p_t X( s;t, \eta(x_\parallel), v)\\
\p_t V( s;t, \eta(x_\parallel), v)
\end{bmatrix} 
=
\begin{bmatrix}
 \nabla_x X( s;t, \eta(x_\parallel), v)  & \nabla_v X( s;t, \eta(x_\parallel), v)\\
 \nabla_x V( s;t, \eta(x_\parallel), v)  & \nabla_v V(s;t, \eta(x_\parallel), v)
\end{bmatrix}
 \begin{bmatrix}
-v  \\
\nabla \phi (t, \eta(x_\parallel)) 
\end{bmatrix}.
\end{split}
\Ee
From (\ref{XV_txv}) and (\ref{pt_XV}) we conclude that 
\Be
\begin{split}\label{XV_txv_final}
&\frac{\p ( X( s;t, \eta(x_\parallel), v),  V( s;t, \eta(x_\parallel), v) )}{\p (t, x_\parallel, v)}
\\ = &  \begin{bmatrix}
 \nabla_x X( s;t, \eta(x_\parallel), v)  & \nabla_v X( s;t, \eta(x_\parallel), v)\\
 \nabla_x V( s;t, \eta(x_\parallel), v)  & \nabla_v V(s;t, \eta(x_\parallel), v)
\end{bmatrix}  
\begin{bmatrix}
-v & 
 \p_{\bar{x}} \eta   & 0_{3 \times 3} \\
\nabla \phi(t, \eta(x_\parallel))
 & 
  0_{3 \times 2} & \mathrm{Id}_{3\times 3}
\end{bmatrix}.  
\end{split}
\Ee
Since $$\det \begin{bmatrix}
 \nabla_x X( s;t, \eta(x_\parallel), v)  & \nabla_v X( s;t, \eta(x_\parallel), v)\\
 \nabla_x V( s;t, \eta(x_\parallel), v)  & \nabla_v V(s;t, \eta(x_\parallel), v)
\end{bmatrix} =1,$$
we conclude that
\Be \label{cov2}
\begin{split} 
\det\left(\frac{\p ( X( s;t, \eta(x_\parallel), v),  V( s;t, \eta(x_\parallel), v) )}{\p (t, x_\parallel, v)}\right) 
=   \ \det \begin{bmatrix}
-v & 
 \p_{\bar{x}} \eta   & 0_{3 \times 3} \\
\nabla \phi(t, \eta(x_\parallel))
 & 
  0_{3 \times 2} & \mathrm{Id}_{3\times 3}
\end{bmatrix}
= \ - v \cdot (\p_1 \eta(x_\parallel) \times \p_2 \eta(x_\parallel) ).
\end{split}\Ee
From (\ref{eta}) the surface measure of $\p \Omega$ equals $dS_x = | \partial_1 \eta(x_\parallel ) \times \partial_2 \eta(x_\parallel) | d x_\parallel$, thus we conclude (\ref{detcofv1}).
\end{proof}

\begin{lemma} \label{covlemma2}
For any $t \ge \tb(t,x,v)$, the map
\begin{equation} \label{maptxvtotgamma-}
(t,x,v)\in [0,T] \times \gamma_+ \mapsto (t-\tb(t,x,v),\xb(t,x,v),\vb(t,x,v)) \in [0,T) \times \gamma_-
\end{equation}
is injective and has determinant
\begin{equation} \label{covgamma+togamma-}
\det\left( \frac{\partial (t-\tb,\xb,\vb) }{\partial (t,x,v)} \right) = \frac{|n(x) \cdot v| }{|n (\xb) \cdot \vb |}.
\end{equation}
For any $T \ge \tb(T,x,v)$, the map
\begin{equation} \label{mapxvtotgamma-}
 (x,v) \in \Omega \times \mathbb R^3 \mapsto  (T - \tb(T,x,v) ,\xb ,\vb ) \in  [0, T) \times \gamma_-
\end{equation}
is injective and has determinant
\begin{equation} \label{covxvtotgamma-}
\det \left( \frac{\partial (T-\tb,\xb,\vb) }{\partial (x,v)} \right) = \frac{1}{|n (\xb) \cdot \vb |}.
\end{equation}
\end{lemma}

\begin{proof}
The map (\ref{maptxvtotgamma-}) is clearly injective as the characteristics is deterministic. We first claim (\ref{covgamma+togamma-}).
Since $x, \xb \in \partial \Omega$, so from (\ref{eta}) locally we have two functions $\eta,\eta_b$ such that $x = \eta(x_\parallel) =\eta (x_{\parallel,1}, x_{\parallel,2})$ and $\xb = \eta_b(  \xbp) = \eta_b( \xba, \xbb$). 

We now compute the Jacobian matrix $J$ of the map (\ref{maptxvtotgamma-}):
\Be
\begin{split}\label{XVtb_txv}
J = &\frac{\p (  t -\tb, {\eta_b} ^{-1} (X(t -\tb; t ,\eta(x_\parallel ) ,v)), V(t-\tb; t, \eta (x_\parallel), v)  )}{\p (t, x_\parallel, v)}\\
= & \ \begin{bmatrix}
1 & - \partial_{x_\parallel} \tb  & - \nabla_v \tb \\
0_{2 \times 1 } & \nabla_x (\eta_b ^{-1} ) \cdot ( \nabla_x X \cdot \partial_{x_\parallel} \eta - \partial_s X \cdot \partial_{x_\parallel } \tb )   & \nabla_x (\eta_b ^{-1} ) \cdot ( \nabla_v X  - \partial_s X \cdot \nabla_v \tb ) \\
0_{3 \times 1} & \nabla_x V \cdot \partial_{x_\parallel } \eta - \partial_s V \cdot \partial_{x_\parallel } \tb & \nabla_v V - \partial_s V \cdot \nabla_v \tb 
\end{bmatrix}. \\
\end{split}
\Ee
Let
\begin{equation} \label{covgamma-toXV}
M =
\begin{bmatrix}
-\partial_s X(t-\tb;t,\eta(x_\parallel),v) & \partial_{\xbp } \eta_b(\xbp) & 0_{3 \times 3} \\
-\partial_s V (t-\tb;t,\eta(x_\parallel),v)&  0_{3 \times 2} & \mathrm{Id}_{3 \times 3}
\end{bmatrix}.
\end{equation}
Then we have
\small
\[ \begin{split}
& M \cdot  J \\  =  & 
\begin{bmatrix}
-\partial_s X & \partial_{\xbp} \eta_b & 0_{3 \times 3} \\
-\partial_s V &  0_{3 \times 2} & \mathrm{Id}_{3 \times 3}
\end{bmatrix}
\cdot
\begin{bmatrix}
1 & - \partial_{x_\parallel} \tb  & - \nabla_v \tb \\
0_{2 \times 1 } & \nabla_x (\eta_b ^{-1} ) \cdot ( \nabla_x X \cdot \partial_{x_\parallel} \eta - \partial_s X \cdot \partial_{x_\parallel } \tb )   & \nabla_x (\eta_b ^{-1} ) \cdot ( \nabla_v X  - \partial_s X \cdot \nabla_v \tb ) \\
0_{3 \times 1} & \nabla_x V \cdot \partial_{x_\parallel } \eta - \partial_s V \cdot \partial_{x_\parallel } \tb & \nabla_v V - \partial_s V \cdot \nabla_v \tb
\end{bmatrix}
\\ 
 = & \begin{bmatrix}
 - \partial_s X & \partial_s X \cdot \partial_{x_\parallel} \tb + \partial_{\xbp } \eta_b \cdot  \nabla_x (\eta_b ^{-1} ) \cdot ( \nabla_x X \cdot \partial_{x_\parallel} \eta - \partial_s X \cdot \partial_{x_\parallel } \tb )   & \partial_s X \cdot \nabla_v \tb + \partial_{\xbp} \eta_b \cdot \nabla_x (\eta_b ^{-1} ) \cdot ( \nabla_v X  - \partial_s X \cdot \nabla_v \tb )   \\
 -\partial_s V & \partial_s V \cdot \partial_{x_\parallel} \tb + \nabla_x V \cdot \partial_{x_\parallel } \eta - \partial_s V \cdot \partial_{x_\parallel} \tb & \partial_s V \cdot \nabla_v \tb + \nabla_v V - \partial_s V \cdot \nabla_v \tb
\end{bmatrix}
\\
= & \begin{bmatrix}
- \partial_s X & \nabla_x X \cdot \partial_{x_\parallel } \eta & \nabla_v X \\
-\partial_s V & \nabla_x V \cdot \partial_{ x_\parallel } \eta & \nabla_v V 
\end{bmatrix},
\end{split} \]
\normalsize
since
\begin{equation} \label{detofgamma-toXV}
\partial_{ \xbp }  \eta_b \cdot \nabla_x (\eta_b ^{-1} ) = \nabla_x ( \eta_b \circ \eta_b ^{-1} )  = \mathrm{Id}_{3\times 3 }.
\end{equation}
Now from (\ref{cov2}) we have
\[ \begin{split}
 \det  (M \cdot  J )   = & \det
 \begin{bmatrix}
- \partial_s X & \nabla_x X \cdot \partial_{x_\parallel} \eta & \nabla_v X \\
-\partial_s V & \nabla_x V \cdot \partial_{ x_\parallel} \eta & \nabla_v V 
\end{bmatrix}
 =  \det \begin{bmatrix}
-v & 
 \p_{x_\parallel} \eta   & 0_{3 \times 3} \\
-E
 & 
  0_{3 \times 2} & \mathrm{Id}_{3\times 3}
\end{bmatrix}
 =  - v \cdot (\partial_1 \eta (x_\parallel ) \times \partial_2 \eta (x_\parallel )).
\end{split} \]
Since
\[ \begin{split}
  \det  ( M )   =  \det
 \begin{bmatrix}
-\partial_s X & \partial_{ \xbp } \eta_b & 0_{3 \times 3} \\
-\partial_s V &  0_{3 \times 2} & \mathrm{Id}_{3 \times 3}
\end{bmatrix}
=  - \vb \cdot (\partial_1 \eta_b (\xbp ) \times \partial_2 \eta_b (\xbp ) ),
\end{split} \]
therefore
\[
\det(J ) = \frac{v \cdot (\partial_1 \eta (x_\parallel ) \times \partial_2 \eta (x_\parallel )) }{\vb \cdot (\partial_1 \eta_b (\xbp) \times \partial_2 \eta_b (\xbp ))},
\]
and we conclude (\ref{covgamma+togamma-}).

The map (\ref{mapxvtotgamma-}) is also injective as the characteristics is deterministic. We then claim (\ref{covxvtotgamma-}). Let $J'$ be the Jacobian matrix of (\ref{mapxvtotgamma-}), then

\Be
\begin{split}\label{XVtb_txv}
J' = &\frac{\p (  t -\tb, {\eta_b} ^{-1} (X(t -\tb; t ,x ,v)), V(t-\tb; t, x, v)  )}{\p (x, v)}\\
= & \ \begin{bmatrix}
  - \partial_{ x} \tb  & - \nabla_v \tb \\
 \nabla_x (\eta_b ^{-1} ) \cdot ( \nabla_x X - \partial_s X \cdot \partial_{ x } \tb )   & \nabla_x (\eta_b ^{-1} ) \cdot ( \nabla_v X  - \partial_s X \cdot \nabla_v \tb ) \\
  \nabla_x V  - \partial_s V \cdot \partial_{ x } \tb & \nabla_v V - \partial_s V \cdot \nabla_v \tb 
\end{bmatrix}.\\
\end{split}
\Ee
Let
\begin{equation} \label{covgamma-toXV}
M' =
\begin{bmatrix}
-\partial_s X(t-\tb;t,\eta(x),v) & \partial_{ \xbp } \eta_b(\xbp) & 0_{3 \times 3} \\
-\partial_s V (t-\tb;t,\eta( x),v)&  0_{3 \times 2} & \mathrm{Id}_{3 \times 3}
\end{bmatrix}.
\end{equation}
Then
\begin{equation}
M' \cdot J' =
\begin{bmatrix}
 \nabla_x X &  \nabla_v  X \\
\nabla_x V & \nabla_v V
\end{bmatrix}.
\end{equation}
Since $\det (A' \cdot M' ) = 1$, and $\det(M') = - \vb \cdot (\partial_1 \eta_b( \xbp) \times \partial_2 \eta_b( \xbp) )$, therefore
\[
\det(J') = \frac{1}{ - \vb \cdot (\partial_1 \eta_b( \xbp) \times \partial_2 \eta_b( \xbp) )},
\]
and we conclude (\ref{covxvtotgamma-}).

\end{proof}

\begin{lemma} \label{int_id}
Suppose $h(t,x,v) \in L^1 ( [0,T] \times \Omega \times \mathbb R^3) $ then:
\begin{equation} \label{covalongchar}
\begin{split}
\int_0 ^T & \iint _{\Omega \times \mathbb R^3} h (t,x,v) dv dx dt \\ = & \iint_{\Omega \times \mathbb R^3 } \int_{- \min ( T,\tb(T,x,v) )} ^0 h (T +s, X(T+s; T,x,v), V(T+s;T,x,v) ) ds dv dx \\ & + \int_0 ^T \int_{\gamma_+ } \int_{-\min ( t, \tb(t,x,v)) } ^0 h(t +s , X(t+s;t,x,v), V(t +s;t,x,v) )ds d\gamma dt.
\end{split}
\end{equation}
\end{lemma}

\begin{proof}
The region $\{ (t,x,v) \in [0,T] \times \Omega \times \mathbb R^3 \}$ is the disjoint union of 
\[
A:= \{ (t,x,v) \in [0,T] \times \Omega \times \mathbb R^3: \tf (t,x,v) + t\le T \},
\] and 
\[ B:= \{ (t,x,v) \in [0,T] \times \Omega \times \mathbb R^3: \tf (t,x,v)+ t > T \} .\]
Now let:
\[
A': =  \{ (t,s,x,v) \in [0,T]^2 \times \gamma_+ : s < \tb(t,x,v), s \le t \},
\] and
\[
B':= \{ (s,x,v ) \in [0,T] \times \Omega \times \mathbb R^3 : s < \tb(T,x,v) \}.
\]
Consider the map $\mathcal A: A' \to A$ with
\[
\mathcal A(t,s,x,v) = (t-s, X(t-s; t,x,v ) , V(t-s;t,x,v) ).
\]
Since $\tf(t-s, X(t-s ; t,x,v) , V(t-s;t,x,v) ) + (t-s) = s + (t-s) = t \le T $, $\mathcal A$ is well-defined. And since the characteristic flow is deterministic, $\alpha$ is injective. And for any $(t,x,v) \in A$, since $\tf \le t + \tf$ and $\tb ( t + \tf, X(t + \tf;t,x,v ) , V(t + \tf;t,x,v ) ) > \tf $ as $x \in \Omega$ is in the interior, we have
\[
 (t + \tf(t,x,v), \tf(t,x,v), X(t + \tf(t,x,v);t,x,v), V(t + \tf(t,x,v);t,x,v) ) \in A'.
\] 
Moreover 
\[
\mathcal A (t + \tf, \tf, X(t + \tf;t,x,v), V(t + \tf;t,x,v) ) = (t,x,v) ,
\] 
so $\mathcal A$ is surjective. Therefore $\mathcal A$ is bijective with inverse $\mathcal A^{-1} (t,x,v ) =  (t + \tf, \tf, X(t + \tf;t,x,v), V(t + \tf;t,x,v) )$.

Suppose locally at $x \in \p \Omega$ we have $x = \eta(x_\parallel ) $ as in (\ref{eta}), and let 
\[
J_{\mathcal A} = \frac{ (t-s, X(t-s; t,x,v ) , V(t-s;t,x,v) )}{\partial(t,s,x_\parallel, v )}
\]
be the Jacobian matrix of $\mathcal A$. 

Then we have:
\small
\[ \begin{split}
 J_{\mathcal A } = 
 \begin{bmatrix}
1 & -1 & 0_{1\times 2 } &  0_{1 \times 3 } \\
\partial_s X(t-s; t,x,v ) +  \partial_t X(t-s; t,x,v ) & -\partial_s X(t-s; t,x,v ) & \partial_{x_\parallel } X(t-s; t,x,v ) & \partial_v X(t-s; t,x,v ) \\
\partial_s V(t-s; t,x,v ) +  \partial_t V(t-s; t,x,v ) & -\partial_s V(t-s; t,x,v ) & \partial_{x_\parallel } V(t-s; t,x,v ) & \partial_v V(t-s; t,x,v ) \\
\end{bmatrix}.
\end{split} \]
\normalsize

Let $J_{\mathcal A} '$ be the matrix obtained by adding the first column of $J_{\mathcal A}$ to its second column, so from (\ref{cov2}) and  (\ref{detcofv1}) we have
\small
\begin{equation} \label{77}
\begin{split}
 \det & (J_{\mathcal A }) = \det(J_{\mathcal A }') \\
= & \det \left \{ \begin{bmatrix}
1 & 0 & 0_{1\times 2 } &  0_{1 \times 3 } \\
\partial_s X(t-s; t,x,v ) +  \partial_t X(t-s; t,x,v ) & \partial_t X(t-s; t,x,v ) & \partial_{x_\parallel } X(t-s; t,x,v ) & \partial_v X(t-s; t,x,v ) \\
\partial_s V(t-s; t,x,v ) +  \partial_t V(t-s; t,x,v ) & \partial_t V(t-s; t,x,v ) & \partial_{x_\parallel } V(t-s; t,x,v ) & \partial_v V(t-s; t,x,v ) \\
\end{bmatrix} \right \} \\
= & \det \left \{ \begin{bmatrix}
 \partial_t X(t-s; t,x,v ) & \partial_{x_\parallel} X(t-s; t,x,v ) & \partial_v X(t-s; t,x,v ) \\
\partial_t V(t-s; t,x,v ) & \partial_{x_\parallel } V(t-s; t,x,v ) & \partial_v V(t-s; t,x,v ) \\
\end{bmatrix} \right \} 
\\ = &\det \left \{ \begin{bmatrix}
 \nabla_x X( t-s;t, \eta(x_\parallel), v)  & \nabla_v X( t-s;t, \eta(x_\parallel), v)\\
 \nabla_x V( t-s;t, \eta(x_\parallel), v)  & \nabla_v V(t-s;t, \eta(x_\parallel), v)
\end{bmatrix}  
\begin{bmatrix}
-v & 
 \p_{x_\parallel} \eta   & 0_{3 \times 3} \\
\nabla \phi(t, \eta(x_\parallel))
 & 
  0_{3 \times 2} & \mathrm{Id}_{3\times 3}
\end{bmatrix} \right \} \\
= &  \ \det \begin{bmatrix}
-v & 
 \p_{x_\parallel} \eta   & 0_{3 \times 3} \\
\nabla \phi(t, \eta(x_\parallel))
 & 
  0_{3 \times 2} & \mathrm{Id}_{3\times 3}
\end{bmatrix}\\
=& \ - v \cdot (\p_1 \eta(x_\parallel) \times \p_2 \eta(x_\parallel) ).
\end{split} \end{equation}
\normalsize
Therefore
\small
\[ \begin{split}
& \iiint_A  h(t,x,v ) dt dxdv 
\\ = & \int_0^T \int_{\gamma_+ } \int_0 ^{\min( \tb(t,x,v) , t ) } h(t-s, X(t-s;t,x,v ) ,V(t-s; t,x,v )    ds d \gamma dt.
\end{split} \]
\normalsize
Now consider the map $\mathcal B: B' \to B$ with
\[
\mathcal B(s,x,v) = (T -s , X(T-s, T,x,v ), V(T-s, T,x,v ) ).
\]
Since $ \tf(T-s, X(T-s, T,x,v ),V(T-s, T,x,v ) ) + (T-s) > s + (T-s) = T$, $\mathcal B$ is well-defined. And since the characteristic flow is deterministic, $\beta$ is injective. And for any $(t,x,v) \in B$, since $\tb(T, X(T; t,x,v ) , V(T; t,x,v ) ) > T -t$ as $x \in \Omega$ is in the interior, we have
\[
  (T-t, X(T; t, x,v ) , V(T;t,x,v) ) \in B'.
\]
 Moreover 
\[
\mathcal B (T-t, X(T; t, x,v ) , V(T;t,x,v) ) = (t,x,v),
\]
so $\mathcal B$ is surjective. Therefore $\mathcal B$ is bijective with inverse $\mathcal B ^{-1} (t,x,v) =  (T-t, X(T; t, x,v ) , V(T;t,x,v) ).$
And since $\mathcal B$ is a measure preserving change of variable we have:
\[
\iiint_B h(t,x,v ) dtdxdv =  \iint_{\Omega \times \mathbb R^3 } \int_0^{ \min (T,\tb(T,x,v) ) }  h(T-s, X(T-s;T,x,v ) ,V(T-s; T,x,v ) ds dx dv.
\]
Thus:
\[
\begin{split}
\int_0 ^T  \iint _{\Omega \times \mathbb R^3} h (t,x,v) dv dx dt  & = \iiint_A h(t,x,v ) dt dxdv + \iiint_B h(t,x,v ) dt dxdv 
\\ = & \iint_{\Omega \times \mathbb R^3 } \int_{- \min ( T,\tb(T, x,v) )} ^0 h (T +s, X(T+s; T,x,v), V(T+s;T,x,v) ) ds dv dx 
\\ & + \int_0 ^T \int_{\gamma_+ } \int_{-\min ( t, \tb(t,x,v)) } ^0 h(t +s , X(t+s;t,x,v), V(t +s;t,x,v) )ds d\gamma dt,
\end{split}
\]
so we conclude (\ref{covalongchar}).
\end{proof}

\begin{lemma} \label{green'sidentity}[Green's identity]
For $p \in [1, \infty)$ assume $f$, $\partial_t f + v \cdot \nabla_x f + E \cdot \nabla_v f \in L^p ([0,T]; L^p (\Omega \times \mathbb R^3 ) )$ and $f_{\gamma_- } \in L^p ( [0,T]; L^p (\gamma ) )$. Then $f \in C^0( [0, T] ; L^p (\Omega \times \mathbb R^3 ) )$ and $f_{\gamma_+ } \in L^p ([0,T]; L^p (\gamma) )$ and for almost every $T' \in [0, T]$:
\Be \label{greens}
\begin{split}
\| f(T')\|_p ^p + \int_0^{T'} |f|_{ \gamma_+, p } ^p = \| f(0) \|_p^p + \int_0^{T'} |f|_{ \gamma_-, p }^p + \int_0^{T'} \iint _{\Omega \times \mathbb R^3 }  p \{ \partial_t + v \cdot \nabla_x f + E \cdot \nabla_v f \} |f|^{p-2} f.
\end{split} \Ee

\end{lemma}

\begin{proof}

For almost every $T' \in [0,T]$, By Holder's inequality we have 
\[
 \| (\partial_t f + v \cdot \nabla_x f + E \cdot \nabla_v f) |f|^{p-2} f  \|_{L^1 ( [0,T] \times \Omega \times \mathbb R^3 )}  \le \|(\partial_t f + v \cdot \nabla_x f + E \cdot \nabla_v f) \|_{L^p ([0,T] \times \Omega \times \mathbb R^3 ) } \| |f|^{p-1}\|_{L^{p/(p-1)} ([0,T] \times \Omega \times \mathbb R^3 )} < \infty.
 \]
Thus by Lemma (\ref{int_id}) we have:
\small
\[ \begin{split}
& \int_0^{T'}  \int_{\Omega \times \mathbb R^3 }  p(\partial_t f + v \cdot \nabla_x f + E \cdot \nabla_v f) |f|^{p-2} f  dxdvdt  
 \\ = & \iint_{\Omega \times \mathbb R^3 } \int_{- \min ( T',\tb(T',x,v) )} ^0 p(\partial_t f + v \cdot \nabla_x f + E \cdot \nabla_v f) |f|^{p-2} f (T' +s, X(T'+s; T',x,v), V(T'+s;T',x,v) ) ds dv dx 
 \\ & + \int_0 ^{T'} \int_{\gamma_+ } \int_{-\min ( t, \tb(t,x,v)) } ^0 p(\partial_t f + v \cdot \nabla_x f + E \cdot \nabla_v f) |f|^{p-2} f(t +s , X(t+s;t,x,v), V(t +s;t,x,v) )ds d\gamma dt.
\end{split} \]
\normalsize
Since
\[ \begin{split}
\frac{d}{ds } |f|^p & ( T'+s, X(T'+s; T',x,v ), V(T'+s; T',x,v )) 
\\ =  & p(\partial_t f + v \cdot \nabla_x f + E \cdot \nabla_v f) |f|^{p-2} f (T' +s, X(T'+s; T',x,v), V(T'+s;T',x,v) ),
\end{split} \]
and
\[ \begin{split}
\frac{d}{ds } |f|^p & ( t+s, X(t+s; t,x,v ), V(t+s; t,x,v )) 
\\ =  & p(\partial_t f + v \cdot \nabla_x f + E \cdot \nabla_v f) |f|^{p-2} f (t +s, X(t+s; t,x,v), V(t+s;t,x,v) ).
\end{split} \]
We have
\begin{equation} \label{green} \begin{split}
 & \int_0^{T'}   \int_{\Omega \times \mathbb R^3 }  p(\partial_t f + v \cdot \nabla_x f + E \cdot \nabla_v f) |f|^{p-2} f dxdvdt  
 \\ = & \iint_{\Omega \times \mathbb R^3 } \int_{- \min ( T',\tb(T',x,v) )} ^0 \frac{d}{ds } |f|^p ( T'+s, X(T'+s; T',x,v ), V(T'+s; T',x,v ))ds dv dx 
 \\ & + \int_0 ^{T'} \int_{\gamma_+ } \int_{-\min ( t, \tb(t,x,v)) } ^0 \frac{d}{ds } |f|^p ( t+s, X(t+s; t,x,v ), V(t+s; t,x,v )) ds d\gamma dt
 \\ = & \iint_{\Omega \times \mathbb R^3 } |f|^p (T',x,v ) dxdv - \iint_{\Omega \times \mathbb R^3 } \textbf{1}_{\{ T' \ge \tb(T',x,v) \}} |f|^p (T'-\tb,\xb,\vb ) dx dv 
 \\&  - \iint_{\Omega \times \mathbb R^3 } \textbf{1}_{\{ T' < \tb(T',x,v) \}} |f|^p (0,X(0;T',x,v) ,V(0;T',x,v ) ) dx dv
 \\ & + \int_0^{T'} \int_{\gamma_+ } |f|^p (t,x,v ) d\gamma dt - \int_0^{T'} \int_{\gamma_+ } \textbf{1}_{\{ t \ge \tb (t,x,v ) \} }|f|^p (t -\tb, \xb,\vb ) d\gamma dt
 \\& - \int_0^{T'} \int_{\gamma_+ }  \textbf{1}_{\{ t < \tb (t,x,v ) \} }|f|^p (0, X(0;t,x,v ),V(0;t,x,v) ) d\gamma dt.
\end{split} \end{equation}

First consider the map 
\[ \begin{split}
 \mathcal A_1 : \{ (x,v) \in \Omega \times \mathbb R^3 : T' < \tb(T',x,v) \} & \to  \{ (x,v) \in \Omega \times \mathbb R^3 : \tf(0,x,v) >  T' \},
 \\ (x,v) & \mapsto  (X(0; T',x,v), V(0; T',x,v) ).
\end{split} \]
This map is well defined as $\tf(0, X(0; T',x,v), V(0; T',x,v)) > T'$ since $x \in \Omega$ is in the interior. $\mathcal A_1$ is injective as the characteristic flow is unique. And for any $ (x,v) \in \Omega \times \mathbb R^3$ such that $\tf(0,x,v) > T'$, we have $X(T';0,x,v ) \in \Omega$ and $\mathcal A_1 ( X(T';0,x,v ), V(T';0,x,v ) ) = (x,v )$, so $\mathcal A_1$ is surjective. Therefore $\mathcal A_1$ is a bijection.
And since the trajectory of this change of variable is measure preserving, we have
\begin{equation} \label{green1}
\iint_{\Omega \times \mathbb R^3 } |f_0|^p \textbf{1}_{\{ \tf(0,x,v) > T' \} } dxdv = \iint_{\Omega \times \mathbb R^3 } |f_0 | ^p (X(0; T',x,v), V(0;T',x,v)) \textbf{1}_{\{ T' < \tb(T',x,v) \}}  dx dv.
\end{equation}
Next, we consider the map 
\[ \begin{split}
 \mathcal A_2 :  \{ (t,x,v) \in �(0, T'] \times \gamma_+ : t < \tb(t,x,v) \} & \to \{ (x,v) \in \Omega \times \mathbb R^3 : \tf(0,x,v) \le T' \},
 \\ (t,x,v ) & \mapsto (X(0;t,x,v), V(0;t,x,v)).
\end{split} \]
This map is well defined as $\tf( 0, X(0;t,x,v), V(0;t,x,v) ) = t \le T'$. $\mathcal A_2$ is injective as the characteristic flow is unique. And for any $(x,v) \in \Omega \times \mathbb R^3$ such that $\tf(0,x,v) \le T'$, we have $(\tf, X(\tf; 0,x,v) , V(\tf; 0,x,v) ) \in (0, T'] \times \gamma_+$ and $\tb(\tf, X(\tf; 0,x,v) , V(\tf; 0,x,v)) > \tf$ as $x \in \Omega$ is in the interior; moreover, $\mathcal A_2 ( \tf, X(\tf; 0,x,v) , V(\tf; 0,x,v) ) = (x,v)$, so $\mathcal A_2$ is surjective. Therefore $\mathcal A_2$ is a bijection. 
So by our change of variable computation (\ref{cov2}) we have:
\begin{equation} \label{green2}
\iint_{\Omega \times \mathbb R^3 } |f_0 |^p \textbf{1}_{ \{ \tf(0,x,v) \le T' \} }   dx dv = \int_0^{T' } \int_{ \gamma_+ } |f_0|^p (X(0;t,x,v), V(0;t,x,v) ) \textbf{1}_{ \{ t < \tb(t,x,v) \} }     d\gamma d t.
\end{equation}
Therefore we have
\[ \begin{split}
\iint_{\Omega \times \mathbb R^3 }  |f_0|^p dx dv =  & \iint_{\Omega \times \mathbb R^3 } |f_0 |^p \textbf{1}_{ \{ \tf (0,x,v) > T' \} } dx dv + \iint_{\Omega \times \mathbb R^3 } |f_0 |^p \textbf{1}_{ \{ \tf(0,x,v) \le T' \} } dx dv
\\ = &  \iint_{\Omega \times \mathbb R^3 } |f_0 | ^p (X(0;T',x,v), V(0;T',x,v)) \textbf{1}_{\{ T' < \tb(T',x,v) \}}  dx dv 
 \\ & + \int_0^{T' } \int_{ \gamma_+ } |f_0|^p (X(0;t,x,v), V(0;t,x,v) ) \textbf{1}_{ \{ t < \tb(t,x,v) \} }     d\gamma d t.
\end{split} \]
Then consider the map 
\[ \begin{split}
\mathcal A_3 : \{ (t,x,v ) \in [0, T' ] \times \gamma_+ : t \ge \tb(t,x,v) \} & \to \{ (s,x,v) \in [0, T' ) \times \gamma_- : T'  \ge s + \tf(s,x,v) \},
\\ (t,x,v ) & \mapsto (t -\tb(t,x,v) , \xb, \vb ).
\end{split} \]
This map is well defined as $\tf( t -\tb ,\xb ,\vb ) + (t - \tb) = \tb + t - \tb = t \le T'$. $\mathcal A_3$ is injective as the characteristic flow is unique. And for any $(s,x,v ) \in [0, T') \times \gamma_-$ such that $s + \tf(s,x,v) \le T'$, we have $(s + \tf, X(s + \tf; s,x,v ) , V(s + \tf; x,v ) ) \in [ 0, T'] \times \gamma_+$ and $\tb ( s + \tf, X(s + \tf; s,x,v ) , V(s + \tf; s,x,v ) )  = \tf \le s + \tf$; moreover, $\mathcal A_3 (s + \tf, X(s + \tf; s,x,v ) , V(s + \tf; s,x,v )) = (s,x,v ) $, so $\mathcal A_3$ is surjective. Therefore $\mathcal A_3$ is a bijection.
With the determinant of this change of variable computed in (\ref{covgamma+togamma-}) we conclude
\begin{equation} \label{green3}
\int_0^{T' } \int_{\gamma_- } |f|^p (t,x,v) \textbf{1}_{ \{T'   \ge s +  \tf(s,x,v) \} } d\gamma ds = \int_0^{T' } \int_{\gamma_+ } |f|^p ( t- \tb(t,x,v), \xb,\vb) \textbf{1}_{ \{ t  \ge \tb(t,x,v) \} } d\gamma dt.
\end{equation}
Finally, consider the map 
\[ \begin{split}
\mathcal A_4: \{ (x,v) \in \Omega \times \mathbb R^3 : T' \ge \tb(T',x,v) \} & \to \{ (s,x,v ) \in [0, T') \times \gamma_- : T'  < s + \tf(s,x,v) \},
\\  ( x,v ) &  \mapsto (T' - \tb(T',x,v) ,\xb ,\vb ).
\end{split} \]
This map is well defined as $\tf(T' - \tb ,\xb ,\vb ) + (T' - \tb )  > \tb +(T' - \tb) = T$ as $x \in \Omega$ is in the interior. $\mathcal A_4$ is injective as the characteristic flow is unique. And for any $\{ (s,x,v ) \in [0, T') \times \gamma_-$ such that $T'  < s + \tf(s,x,v) $, we have $ (X( T'; s, x,v ) , V(T';s,x,v ) ) \in \Omega \times \mathbb R^3$ and $ \tb ( T', X( T'; s, x,v ) , V(T';s,x,v )) = T'-s \le T'$; moreover, $\mathcal A_4 ( X( T'; s, x,v ) , V(T';s,x,v ) ) = (s,x,v )$, so $\mathcal A_4$ is surjective. Therefore $\mathcal A_4$ is a bijection.  

Therefore by the computation of the change of variable (\ref{covxvtotgamma-}) we have:
\begin{equation} \label{green4}
\int_0^{T' } \int_{\gamma_- } |f|^p (t,x,v) \textbf{1}_{ \{T'   < s + \tf(s,x,v) \} }   d\gamma dt =   \iint_{ \Omega \times \mathbb R^3 } |f |^p (T' - \tb(T',x,v) , \xb, \vb ) \textbf{1}_{ \{ T' \ge \tb(T',x,v) \}} dxdv.
\end{equation}

Now substitute all these identities (\ref{green1}), (\ref{green2}), (\ref{green3}),(\ref{green4}) into equation (\ref{green}) we finally get:

\[ \begin{split}
 \int_0^{T'} &  \int_{\Omega \times \mathbb R^3 }  p(\partial_t f + v \cdot \nabla_x f + E \cdot \nabla_v f) |f|^{p-2} f dxdvdt  
 \\ = & \iint_{\Omega \times \mathbb R^3 } |f|^p (T',x,v ) dxdv - \iint_{\Omega \times \mathbb R^3 } \textbf{1}_{\{ T' \ge \tb(T',x,v) \}} |f|^p (T'-\tb,\xb,\vb ) dx dv 
 \\&  - \iint_{\Omega \times \mathbb R^3 } \textbf{1}_{\{ T' < \tb(T',x,v) \}} |f|^p (0,X(0;T',x,v) ,V(0;T',x,v ) ) dx dv
 \\ & + \int_0^{T'} \int_{\gamma _+ } |f|^p (t,x,v ) d\gamma dt - \int_0^{T'} \int_{\gamma_+ } \textbf{1}_{\{ t \ge \tb (t,x,v ) \} }|f|^p (t -\tb, \xb,\vb ) d\gamma dt
 \\& - \int_0^{T'} \int_{\gamma_+ }  \textbf{1}_{\{ t < \tb (t,x,v ) \} }|f|^p (0, X(0;t,x,v ),V(0;t,x,v) ) d\gamma dt
  \\ = & \iint_{\Omega \times \mathbb R^3 } |f|^p (T',x,v ) dxdv -\int_0^{T' } \int_{\gamma_- } |f|^p (t,x,v) \textbf{1}_{ \{T'   < s + \tf(s,x,v) \} }   d\gamma dt 
  \\&  - \iint_{\Omega \times \mathbb R^3 } |f_0|^p \textbf{1}_{\{ \tf(0,x,v) > T' \} } dxdv
 \\ & + \int_0^{T'} \int_{\gamma_+ } |f|^p (t,x,v ) d\gamma dt - \int_0^{T' } \int_{\gamma_- } |f|^p (t,x,v) \textbf{1}_{ \{T'   \ge s +  \tf(s,x,v) \} } d\gamma ds
  \\& - \iint_{\Omega \times \mathbb R^3 } |f_0 |^p \textbf{1}_{ \{ \tf(0,x,v) \le T' \} }   dx dv 
\\ = & \iint_{\Omega \times \mathbb R^3 } |f|^p (T',x,v ) dxdv + \int_0^{T'} \int_{\gamma_+ } |f|^p (t,x,v ) d\gamma dt
\\ & - \iint_{\Omega \times \mathbb R^3 } |f_0 |^p dxdv - \int_0^{T' } \int_{\gamma_- } |f|^p ( t,x,v) d\gamma dt,
\end{split} \]
so we conclude (\ref{greens}).

Note since the left hand side of the above equality is finite, and by our assumption all the terms on the right hand side except $\int_0^{T'} \int_{\gamma _+ } |f|^p (t,x,v ) d\gamma dt$ is finite, thus $f \in L^p ( [0,T] ; L^p (\gamma_+ ) ) $.

\end{proof}

We now define $\gamma_+^\epsilon$ to be the set of almost grazing velocities or large velocities
\begin{equation} \label{defgammaepsilon}
\gamma_+^\epsilon = \{ (x,v) \in \gamma_+:  n(x) \cdot v  < \epsilon \text{ or } |v| >  1/\epsilon \}.
\end{equation}

\begin{lemma}[trace theorem for bounded potential]\label{tracebddpotential}
Let $0<T<1$ be fixed. Assume that $ | \nu (t,x,v) |  \lesssim \langle v \rangle $ and $\| E \|_\infty  < \infty$. Then for any $0< \epsilon \ll 1$, there exists a $C_{\Omega } > 0$ depending only on $\Omega$ such that 
for all $0 \le t \le T$,
\begin{equation} \label{traceestimate}
\begin{split}
\int_0^t & \int_{\gamma_+ \setminus \gamma_+^\epsilon } |h | d\gamma ds \\ & \le C_\Omega e^{ T \| E \|_\infty} \frac{ 1 + \epsilon^2 \| E\|_\infty^2}{\epsilon^3} \left[ \| h_0 \|_1 + \int_0^t \left( \| h(s) \| _1 +  \| [\partial_t + v \cdot \nabla_x  + E \cdot \nabla_v  + \nu ]h(s) \|_1\right) ds  \right].
\end{split}
\end{equation}

\end{lemma}

\begin{proof}
%
%
%
%
%

For $(t,x,v) \in [0 ,T] \times \gamma_+ \setminus \gamma_+^\epsilon $, we claim
\begin{equation} \label{lowerbdfortb}
\inf_{ (t,x,v) \in [0,T] \times \gamma_+\setminus \gamma_+^\epsilon } \tb(t,x,v) \gtrsim_\Omega \frac{\epsilon^3}{1 + \epsilon^2 \| E\|_\infty^2}.
\end{equation}
Since 
\[
 \nabla \xi (x) \cdot v = | \nabla \xi (x) | n(x) \cdot v > | \nabla \xi (x) | \epsilon  \gtrsim_\Omega \epsilon, \, \nabla \xi (\xb ) \cdot \vb < 0,
 \]
 and
\[ \begin{split}
\frac{d}{ds}  \left( \nabla \xi (X(s) ) \cdot V(s)\right) =& V(s) \cdot \nabla^2 \xi (X(s) ) \cdot V(s) + \nabla \xi (X(s)) \cdot E(s,X(s) ) 
\\ \lesssim_{\Omega} & (|V(s) | ^2 +\| E \|_\infty ) \lesssim_{\Omega} (|v|^2 + \| E \|_\infty^2 + 1  ) \lesssim_{\Omega} (\frac{1}{\epsilon^2} + \| E \|_\infty^2 + 1  ),
\end{split} \]
for all $ t -\tb \le s \le  t$. Thus
\[
\tb(t,x,v) \ge   \frac{    \epsilon }{C_\Omega (\frac{1}{ \epsilon^2} + \| E \|_\infty^2 + 1)} \ge \frac{\epsilon^3}{C_\Omega (1 + \epsilon^2 \| E\|_\infty^2)}.
\]
This proves (\ref{lowerbdfortb}). Let 
\[
\epsilon_1 = \frac{\epsilon^3}{C_\Omega (1 + \epsilon^2 \| E\|_\infty^2)}.
\]
Now if $h$ solves (\ref{traninflowfixE}), then for $(t,x,v) \in [0,T] \times \gamma_+ $ and $-  \min \{ t ,\tb (t,x,v) \} \le s \le 0 $, we have
\begin{equation} \label{h_expand}
\begin{split}
 h(t,x,v) =  &h(t+s, X(t+s),V(t+s) ) e^{- \int_s^{0} \nu (V(t + \tau' )) d\tau'} \\ &+  \int_{s}^0 e^{-\int_\tau^0 \nu(V(t + \tau' )) d\tau' } H(\tau, X(t +\tau), V(t + \tau) ) d \tau,
\end{split}
\end{equation}
where $X(t + \tau ) = X(t + \tau; t ,x,v ) $, and $V(t + \tau ) = V( t + \tau; t,x,v  )$.

Then by (\ref{lbtb})
\[ \begin{split}
\min & \{ t, \tb(t,x,v) \} |h (t,x,v)| =  \int_{ - \min \{ t, \tb(t,x,v) \} } ^0 |h(t,x,v)| ds
\\   \le &  C_\Omega' e^{ T \| E \|_\infty} \left( \int_{ - \min \{ t, \tb(t,x,v) \} } ^0 |h(t + s, X(t+s),V(t+s) ) | ds 
 +  \int_{ - \min \{ t, \tb(t,x,v) \} } ^0 \int_{s}^0 | H(\tau, X(t +\tau), V(t + \tau)) |  d \tau ds \right)
 \\ \le & C_\Omega' e^{ T \| E \|_\infty} \left( \int_{ - \min \{ t, \tb(t,x,v) \} } ^0 |h(t+s, X(t+s),V(t+s) ) | ds 
+  T \int_{ - \min \{ t, \tb(t,x,v) \} } ^0  |  H(\tau, X(t +\tau), V(t + \tau) ) | d \tau \right).
\end{split} \]

We then integrate (\ref{h_expand}) over $\int_{\epsilon_1} ^T \int_{\gamma_+ \setminus \gamma_+^\epsilon} \int_{-\min \{ t, \tb (t,x,v) \} } ^0 $ to get 
\begin{equation} \label{traceestimateepsilon1toT}
\begin{split}
 \epsilon_1 \times  &\int_{\epsilon_1} ^T \int_{\gamma_+ \setminus \gamma_+^\epsilon} | h(t,x,v) | d\gamma dt 
\\ \le & \displaystyle {\min_{ [\epsilon_1 , T ] \times [ \gamma_+ \setminus \gamma_+^\epsilon ]} } \{ t, \tb (t,x,v) \} \times \int_{\epsilon_1} ^T \int_{\gamma_+ \setminus \gamma_+^\epsilon} | h(t,x,v) | d\gamma dt 
\\  \le & C_\Omega' e^{ T \| E \|_\infty}  \int_0 ^T \int_{\gamma_+ \setminus \gamma_+^\epsilon } \int_{-\min\{t, \tb (t,x,v) \} } ^0 |h(t+s, X(t+s) ,V(t+s)) | ds d\gamma dt 
\\  & + C_\Omega' e^{ T \| E \|_\infty}   T \int_0^T \int_{\gamma_+ \setminus \gamma_+ ^\epsilon } \int_{- \min \{ t, \tb (t,x,v) \}  }^0 |H(\tau, X(t + \tau), V(t +\tau)) | d\tau d\gamma dt
\\  \le & C_\Omega' e^{ T \| E \|_\infty} \left(\int_0^T \| h(t) \|_1 dt + \int_0^T \| [\partial_t + v \cdot \nabla_x + E \cdot \nabla_v + \phi ] h(t) \| _1 dt\right),
\end{split}
\end{equation}
where in the last inequality we have used the identity (\ref{covalongchar}). 

On the other hand, because of our choice $\epsilon$ and $\epsilon_1$, by (\ref{lowerbdfortb}) we have $\tb(t,x,v) > t $ for all $(t,x,v) \in [0, \epsilon_1] \times \gamma_+ \setminus \gamma_+^\epsilon $. Then 
\[
\begin{split}
 |h(t,x,v)| \le | h_0( X(0),V(0) )| +  \int_{-t}^0 | H ( t + \tau , X(t + \tau ) , V(t + \tau )) |d \tau.
\end{split}
\]
Integrating over $\int_0^{\epsilon_1} \int_{\gamma_+\setminus \gamma_+^\epsilon} $ we get
\begin{equation} \label{traceestimate0toepsilon1}
\int_0^{\epsilon_1} \int_{\gamma_+\setminus \gamma_+^\epsilon}  |h(t,x,v)| d\gamma dt \le\int_0^{\epsilon_1} \int_{\gamma_+\setminus \gamma_+^\epsilon}  | h_0( X(0),V(0) )| d\gamma dt+ \int_0^{\epsilon_1} \int_{\gamma_+\setminus \gamma_+^\epsilon}  \int_{-t}^0 | H ( t + \tau , X(t + \tau ) , V(t + \tau )) |d \tau d\gamma dt.
\end{equation}

where the second term is bounded, again from (\ref{covalongchar}), by 
\[
\begin{split}
\int_0^{\epsilon_1}  & \int_{\gamma_+ \setminus \gamma_+^\epsilon} \int_{-t}^0   | H ( t + \tau , X(t + \tau ) , V(t + \tau )) |d \tau
\le \int_0^{\epsilon_1} \| [\partial_t + v \cdot \nabla_x + E \cdot \nabla_v + \phi ] h(t) \| _1 dt.
\end{split}
\]

And by (\ref{green2}) the first term is bounded by
\[
\begin{split}
\int_0^{\epsilon_1}  \int_{\gamma_+ \setminus \gamma_+^\epsilon} |h_0(X(0;t,x,v), V(0;t,x,v)) | d\gamma dt 
 \le & \int_0^{T } \int_{ \gamma_+ } |h_0| (X(0;t,x,v), V(0;t,x,v) ) \textbf{1}_{ \{ t < \tb(t,x,v) \} }     d\gamma d t
\\ = & \iint_{\Omega \times \mathbb R^3 } |h_0 | \textbf{1}_{ \{ \tf(0,x,v) \le T' \} }   dx dv  
 \le \| h_0 \|_1.
\end{split}
\]
Combine (\ref{traceestimateepsilon1toT}) and (\ref{traceestimate0toepsilon1}) we conclude (\ref{traceestimate}).

\end{proof}

\bigskip
We need a cutoff function for our weight function:

%

For any $\epsilon >0$, let $\chi_\epsilon: [0,\infty) \to [0,\infty) $ be a smooth function satisfying:
%

\begin{equation} \label{chicondition} \begin{split}
& \chi_\epsilon(x) =x \,\, \text{for} \,\, 0 \le x \le \frac{\epsilon}{4}, 
\\ &\chi_\epsilon(x) = C_\epsilon  \,\, \text{for}\,\, x \ge \frac{\epsilon}{2},
\\& \chi_\epsilon(x) \text{ is increasing}  \,\, \text{for} \,\, \frac{\epsilon}{4} < x < \frac{\epsilon}{2},
\\ & \chi_\epsilon'(x) \le 1.
\end{split} \end{equation}

Let $d(x,\partial \Omega) := \inf_{y \in \partial \Omega} \| x - y \| $. And for any $\delta > 0$, let 
\[
\Omega ^\delta : = \{ x \in \Omega : d(x, \partial \Omega ) < \delta \}.
\]
Since $\partial \Omega$ is $C^2$, we claim that if $\delta \ll 1$ is small enough we have:
\begin{equation} \label{distfunctionunique}
\text{for any} \, x \in \Omega ^\delta \, \text{ there exists a unique} \, \bar x \in \partial \Omega \, \text{such that} \, d(x,\bar x ) = d(x, \partial \Omega), \, \text{ moreover } \sup_{x\in \Omega^\delta } | \nabla_x \bar x | < \infty.
\end{equation}
To prove the claim, we have by (\ref{eta}) WLOG locally we can assume $\eta$ takes the form $\eta ( x_\parallel)  = (x_{\parallel,1} , x_{\parallel,2} , \bar \eta(x_{\parallel,1}, x_{\parallel,2}))$, and $\bar x = \eta(\bar x_\parallel) =(\bar x_{\parallel,1} , \bar x_{\parallel,2} , \bar \eta(\bar x_{\parallel,1}, \bar x_{\parallel,2}))$. Denote $\partial_i \bar \eta = \frac {\p }{ \partial x_{\parallel, i } } \bar \eta (x_{\parallel,1} , x_{\parallel,2} ) $, and $\partial_{i,j} \bar \eta = \frac {\p^2 }{ \partial x_{\parallel, i } \partial x_{\parallel, j } } \bar \eta (x_{\parallel,1} , x_{\parallel,2} )$.

Now since $|\eta(\bar x_\parallel) - x|^2 = \inf_{y \in \partial \Omega } | y -x |^2$, $\bar x_\parallel$ satisfies
\[ 
\omega (x_1,x_2,x_3, \bar x_{\parallel,1},\bar x_{\parallel,2} )= 
\begin{bmatrix}
(\bar x_{\parallel,1} - x_1) + (\bar \eta(\bar x_{\parallel,1}, \bar x_{\parallel,2}) - x_3) \p_1 \bar \eta ( \bar x_{\parallel,1}, \bar x_{\parallel,2})  \\
(\bar x_{\parallel,2} - x_2) + (\bar \eta(\bar x_{\parallel,1}, \bar x_{\parallel,2}) - x_3) \p_2 \bar \eta ( \bar x_{\parallel,1}, \bar x_{\parallel,2})
\end{bmatrix} 
=0.
\]
Since
\[ \begin{split}
\det (\frac{\partial \omega}{ \partial x_\parallel} )= & \det
\begin{bmatrix}
1 + (\p_1 \bar \eta)^2 + (\bar \eta - x_3 ) \p_{1,1} \bar \eta_{} & \p_2 \bar \eta_{}\p_1\bar \eta_{} + (\bar \eta - x_3 ) \p_{1,2} \bar \eta \\
\p_1 \bar \eta_{}\p_2 \bar \eta_{} +(\bar \eta -x_3) \p_{1,2} \bar \eta_{} & 1 + (\p_2 \bar \eta_{}) ^2 + (\bar \eta -x_3 ) \p_{1,2} \bar \eta_{}
\end{bmatrix}
\\ & = ( 1 + (\p_1 \bar \eta_{})^2 ) ( 1 + (\p_2 \bar \eta_{}) ^2 ) - (\p_1 \bar \eta_{} \p_2 \bar \eta_{ } ) ^2 +O(|\bar \eta - x_3 |)
\\ & = 1 + (\p_1 \bar \eta_{})^2 + (\p_2 \bar \eta_{}) ^2+O(|\bar \eta - x_3 |) > 0,
\end{split} \]
if $|\bar \eta (x_\parallel) - x_3 | $ is small enough. By the implicit function theorem $(\bar x_{\parallel,1}, \bar x_{\parallel,2 } )$ are functions of $x_1, x_2, x_3$ if $x$ is close enough to $\partial \Omega$.

Moreover,
\[ \begin{split}
\frac{\partial \bar x_{\parallel} }{ \partial x_j }&  = - ( \frac{\partial \omega}{ \partial \bar x_{\parallel} } ) ^{-1}  \cdot \frac{\partial \omega}{ \partial x_j }
\\ & = \frac{1}{ \det ( \frac{\partial \omega}{ \partial \bar x_{\parallel} }) }
\begin{bmatrix}
1 +  (\p_2 \bar \eta) ^2 + (\bar \eta -x_3 )  \partial_{1,2} \bar \eta & - \p_2  \bar \eta \p_1 \bar \eta - (\bar \eta - x_3 ) \p_{1,2} \bar \eta \\
- \p_1  \bar \eta \p_2 \bar \eta -(\bar \eta -x_3) \p_{1,2} \bar \eta & 1 + (\p_1 \bar \eta)^2 + (\bar \eta - x_3 ) \p_{1,1} \bar \eta 
\end{bmatrix}
\cdot \frac{\partial \omega }{ \partial x_j }
\end{split}\]
is bounded as $\frac{\partial \omega }{\partial x_j } $ is bounded and $\det ( \frac{\partial \omega}{ \partial \overline x })$ is bounded from below if $x$ is close enough to the boundary. Therefore $| \nabla_x \bar x |$ is bounded. This proves (\ref{distfunctionunique}).

Now define
\[
\beta(t,x,v) = \bigg[ |v \cdot \nabla \xi (x)| ^2 + \xi (x)^2 - 2 (v \cdot \nabla^2 \xi(x) \cdot v ) \xi(x) - 2(E(t,\overline x ) \cdot \nabla \xi (\overline x ) )\xi(x) \bigg]^{1/2},
\]
for all $(x, v ) \in  \Omega ^{\delta} \times \mathbb R^3 $. Let $\delta': = \min \{| \xi (x)| : x\in \Omega, d(x, \partial \Omega) = \delta \} $, and let $\chi_{\delta'}$ be a smooth cutoff function satisfies (\ref{chicondition}), then define
\begin{equation} \label{alphadef}
\alpha(t,x,v) : = 
\begin{cases}
 (\chi_{\delta'}  ( \beta(t,x,v) ) )^{} & x \in  \Omega^\delta, \\
 C_{\delta'} ^{} & x \in \Omega \setminus  \Omega^\delta.
\end{cases} \end{equation}

\begin{lemma}[Velocity lemma near boundary] \label{velocitylemma} 
Suppose $E(t,x)$ satisfies \eqref{c1bddforthepotentail} and the sign condition (\ref{signEonbdry}).
Then $\alpha$ is continuous, and for $\delta \ll 1 $ small enough, we have for any $0 \le s<t $ and trajectory $X(\tau), V(\tau)$ solving (\ref{hamilton_ODE}), if $X(\tau ) \in \Omega $ for all $s \le \tau \le t $, then $\alpha$ satisfies
\begin{equation}  \label{velocitylemmaintform}
e^{ - C \int_s ^ t ( |V(\tau')| + 1 ) d \tau'} \alpha ( s,X(s),V(s) ) \le \alpha (t,X(t),V(t)) \le  e^{C \int_s ^ t ( |V(\tau')| + 1 ) d \tau'} \alpha (s,X(s),V(s)),
\end{equation}
for any $C \ge \frac{C_{\xi}(\| E \|_\infty + \| \nabla E \|_\infty +\|\p_t E \|_\infty + 1 )}{C_E} $, where $C_\xi > 0$ is a large constant depending only on $\xi$.
\end{lemma}
Similar estimates have been used in \cite{Guo_V} and then in \cite{Hwang,GKTT1}.

\begin{proof}

Since $\beta(t,x,v) \ge |\xi(x) | $ for all $x \in \partial \Omega$, $\beta (t,x,v) \ge \frac{\delta'}{2}$ on an open neighborhood $U$ of $\{ x \in \Omega: d(x, \partial \Omega ) = \delta \}$. So by (\ref{chicondition}), $\alpha  = C_{\delta'} ^{} $ on $U$, and therefore $\alpha$ is continuous.


Now let's first claim that if $X(\tau) \in \Omega^\delta$ for all $\tau$, then $ \beta^2 $ satisfies:
\begin{equation} \label{velocitylemmabeta}
- C( |V(\tau)|+1) \beta^2 (\tau, X(\tau),V(\tau)) \le \frac{d}{d\tau}\beta^2 (\tau, X(\tau),V(\tau))  \le C (|V(\tau)|+1) \beta^2(\tau, X(\tau),V(\tau)),
\end{equation}
for any $C \ge \frac{C_{\xi}(\| E \|_\infty + \| \nabla E \|_\infty +\|\p_t E \|_\infty + 1 )}{C_E} $.

By direct computation
\Be \label{transderivbeta} \begin{split}
\{ & \p_t + v\cdot \nabla_x + E \cdot \nabla_v \} \beta^2(t,x,v)
\\ = & 2 (v\cdot \nabla \xi(x)) ( E(t,x) \cdot \nabla \xi (x)) + \cancel{2 (v\cdot \nabla \xi (x) ) ( v \cdot \nabla^2 \xi(x) \cdot v )} + 2 \xi (x) ( v\cdot \nabla \xi (x))
\\ & - 2 (E(t,x) \cdot  (\nabla^2 \xi(x) + \nabla^2 \xi (x)^t )  \cdot v) \xi (x) -  \cancel{2 (v\cdot \nabla \xi (x) ) ( v \cdot \nabla^2 \xi(x) \cdot v )} - 2 v \cdot ( v \cdot \nabla^3 \xi (x) \cdot v ) \xi(x)
\\ & - 2 ( E(t,\overline x ) \cdot  \xi (\overline x ) ) ( v\cdot \nabla \xi (x)) - 2 (\nabla_x \overline x ) \left[ v \cdot \nabla_x E(t,\overline x ) \cdot \nabla \xi (\overline x ) + v \cdot \nabla^2 \xi (\overline x ) \cdot E(t,\overline x ) \right] \xi (x) - 2 (\p_t E(t,\overline x ) \cdot \nabla \xi (\overline x ))\xi(x) .
\end{split} \Ee

Since 
\Be
 ( E(t,x) \cdot \nabla \xi (x)) =  E(t,\overline x ) \cdot \nabla \xi (\overline x ) + \nabla_x (E\cdot \nabla \xi ) (x') \cdot ( x - \overline x ) = E(t,\overline x ) \cdot \nabla \xi (\overline x ) + \left[  \nabla_x (E\cdot \nabla \xi )  (x') \cdot  \frac{( x - \overline x )}{\xi (x) }  \right] \xi (x).
\Ee

We claim that $\frac{ x - \overline x }{ \xi (x) } $ is bounded for all $x \in \Omega$. This is obvious when $x$ is away from the boundary $\partial \Omega$. When $x$ is close to $\partial \Omega$ since 
\Be \label{expxtobdy}
\xi (x) = \xi ( \overline x ) + \nabla \xi (x'' )\cdot ( x - \overline x ) = \nabla \xi (x'' )\cdot ( x - \overline x ) = |\nabla \xi (x'')  | | x - \overline x | \cos ( \theta ).
\Ee
So
\[
\left| \frac{x - \overline x } { \xi (x ) } \right| = \frac{ 1 }{ | \nabla \xi (x'' )|| \cos (\theta ) |},
\]
where $x''$ is a point on the line segment linking $x$ and $\overline x$ and $\theta$ is the angle between the two vectors $- \nabla \xi (\overline x) $ and $\nabla \xi (x'')$ by our choice of $\overline x$.

Now since we have $|\nabla \xi(x) | > c > 0$ when $x$ is close to $\partial \Omega$, we can choose $\delta$ so small that if $d(x, \partial \Omega ) = d (x, x^*) <\delta$, the angle between $\nabla \xi (x) $ and $\nabla \xi (x^*)$ will be small enough such that $|\cos(\theta) | > 1/2$.

Therefore
\begin{equation} \label{fracxxibdd}
\left| \frac{x - \overline x } { \xi (x ) } \right| \lesssim \frac{1}{c},
\end{equation}
for all $x \in \Omega$ as claimed.
From \eqref{transderivbeta}, \eqref{expxtobdy}, and \eqref{fracxxibdd} we have
\Be \label{transderivbeta2} \begin{split}
\{ & \p_t + v\cdot \nabla_x + E \cdot \nabla_v \} \beta^2(t,x,v)
\\ = & \cancel{2 (v\cdot \nabla \xi(x)) ( E(t,\overline x) \cdot \nabla \xi (\overline x))}  + C_{\frac{1}{c}, \| \nabla E \|_\infty, \| \xi \|_{C^2 } }(v\cdot \nabla \xi(x)) \xi (x) + 2 \xi (x) ( v\cdot \nabla \xi (x))
\\ & - 2 (E(t,x) \cdot  (\nabla^2 \xi(x) + \nabla^2 \xi (x)^t )  \cdot v) \xi (x)  - 2 v \cdot ( v \cdot \nabla^3 \xi (x) \cdot v ) \xi(x)
\\ & -\cancel{ 2 ( E(t,\overline x ) \cdot  \xi (\overline x ) ) ( v\cdot \nabla \xi (x))} - 2 (\nabla_x \overline x ) \left[ v \cdot \nabla_x E(t,\overline x ) \cdot \nabla \xi (\overline x ) + v \cdot \nabla^2 \xi (\overline x ) \cdot E(t,\overline x ) \right] \xi (x) - 2 (\p_t E(t,\overline x ) \cdot \nabla \xi (\overline x ))\xi(x) .
\end{split} \Ee
From \eqref{signEonbdry} and \eqref{distfunctionunique},
\[ \begin{split}
 | \{ \p_t + v\cdot \nabla_x + E \cdot \nabla_v \} \beta^2(t,x,v) | \le & C_\xi(\|E \|_\infty + \|\nabla E \|_\infty +\|\p_t E \|_\infty + 1 ) \left(  |v| + |v|^3 \right)|\xi (x) | 
 \\ \le & \frac{ C_\xi(\|E \|_\infty + \|\nabla E \|_\infty +\|\p_t E \|_\infty + 1 )}{C_E } |v| \beta^2 (t,x,v).
\end{split} \]
 Since 
 \[
 \frac{d}{d\tau} \beta^2 (\tau, X(\tau) ,V(\tau ) ) = \{ \p_t + v\cdot \nabla_x + E\cdot \nabla_v \} \beta^2 (\tau, X(\tau) ,V(\tau )),
\]
we conclude \eqref{velocitylemmabeta}.

Next we show that $\alpha ^2 (\tau, X(\tau), V(\tau) ) $ satisfies
\[
- C( |V(\tau)|+1) \alpha^2 (\tau, X(\tau),V(\tau)) \le \frac{d}{d\tau}\alpha^2 (\tau, X(\tau),V(\tau))  \le C (|V(\tau)|+1) \alpha^2(\tau, X(\tau),V(\tau)).
\]
This is clearly true if $X(\tau ) \in \Omega \setminus \Omega^\delta$ as $\alpha $ is constant there.
For $X(\tau) \in \Omega ^\delta$ we have if $\beta(\tau, X(\tau), V(\tau )) \ge \frac{\delta'}{2}$,
\[
\frac{d}{d\tau}\alpha^2 (\tau, X(\tau),V(\tau))   =  \frac{ d }{d\tau} \chi_{\delta'} (  \beta^2( \tau, X(\tau), V(\tau ) )) = \chi_{\delta'}'( \beta^2(\tau, X(\tau), V(\tau) ) \frac{d}{d \tau } \beta^2( \tau, X(\tau), V(\tau) ) = 0,
\]
so the inequalities are automatically true.
If $\beta(\tau, X(\tau), V(\tau ) ) < \frac{\delta'}{2}$, we have by (\ref{chicondition}) $\beta(\tau, X(\tau), V(\tau) ) < 2 \chi_{\delta' } (\beta(\tau, X(\tau), V(\tau )))$.
Therefore by (\ref{velocitylemmabeta}) and $\chi_{\delta'}' \le 1$ we have:
\begin{equation} \label{velocitylemmaderivative}
- 2C( |V(\tau)|+1) \alpha^2 (\tau, X(\tau),V(\tau)) \le \frac{d}{d\tau}\alpha^2 (\tau, X(\tau),V(\tau))  \le 2C (|V(\tau)|+1) \alpha^2(\tau, X(\tau),V(\tau)).
\end{equation}
Finally, by the gronwall inequality we have
\[  e^{ - 2C \int_s ^ t ( |V(\tau')| + 1 ) d \tau'} \alpha^2(s, X(s),V(s) ) \le \alpha^2 (t,X(t),V(t)) \le  e^{2C \int_s ^ t ( |V(\tau')| + 1 ) d \tau'} \alpha^2 (s,X(s),V(s)) .\]
Taking square root we get the desired inequality.

\end{proof}

\begin{lemma}If $E(t,x) \in C([0,T]; C^1(\mathbb R^3 ) )$, and $n(\xb(t,x,v)) \cdot \vb(t,x,v) \neq 0$, then $(\tb,\xb,\vb)$ is differentiable and 
	\Be\begin{split}\label{computation_tb_x}
		\frac{\p\tb}{\p x_i}  =& \  \frac{1}{n(\xb) \cdot \vb}n(\xb) \cdot  \left[
		e_i + \int^{t-\tb}_t \int^s_t \Big(\frac{\p X(\tau )}{\p x_i} \cdot \nabla\Big) E(\tau, X(\tau )) \dd \tau \dd s 
		\right]  ,\\
		\frac{ \p\xb}{\p x_i} = & \  e_i - \frac{\p \tb}{\p x_i} \vb + \int^{t-\tb}_{t} \int^s_t    \Big(\frac{\p X(\tau )}{\p x_i} \cdot \nabla\Big) E(\tau, X(\tau ))  \dd \tau \dd s,\\
		\frac{\p \vb}{\p x_i} = & \ - \frac{\p \tb}{\p x_i} E(t-\tb, \xb) + \int^{t-\tb}_t   \Big(\frac{\p X(\tau )}{\p x_i} \cdot \nabla\Big) E(\tau, X(\tau )) \dd \tau,\\
		\frac{\p\tb}{\p v_i}  =& \  \frac{1}{n(\xb) \cdot \vb}n(\xb) \cdot  \left[
		e_i + \int^{t-\tb}_t \int^s_t \Big(\frac{\p X(\tau )}{\p v_i} \cdot \nabla\Big) E(\tau, X(\tau )) \dd \tau \dd s 
		\right]  ,\\ 
		\frac{\p \xb}{\p v_i} = & \ - \tb e_i - \frac{
			\p \tb}{\p v_i} \vb + \int^{t-\tb}_{t} \int^s_t    \Big(\frac{\p X(\tau )}{\p v_i} \cdot \nabla\Big) E(\tau, X(\tau ))  \dd \tau \dd s ,\\
		\frac{\p \vb}{\p v_i} = & \ e_i - \frac{\p \tb}{\p v_i} E(t-\tb, \xb) + \int^{t-\tb}_t
		\Big(\frac{\p X(\tau )}{\p v_i} \cdot \nabla\Big) E(\tau, X(\tau ))  \dd \tau.
	\end{split} \Ee
\end{lemma}
\begin{proof}
	The equalities are derived from direct computations and an implicit function theorem. For details see \cite{GKTT1}.
\end{proof}




Denote
\[ \begin{split}
& \iint  \partial_{x,v} E = \int _{ t -\tb}^t \int_s^t \partial_{x,v} E(X(\tau ) ) d\tau ds =  \int _{ t -\tb}^t \int_s^t \nabla_x E(X(\tau ) ) \cdot \nabla_{x,v} X(\tau) d\tau ds, 
\\ & \int \partial_{x,v} E = \int_{t-\tb}^t \partial_{x,v} E(X(s))ds =\int_{t-\tb}^t \nabla_x E(X(s)) \cdot \nabla_{x,v} X(s) ds .
\end{split} \]

Let $\tau_1(x)$ and $\tau_2(x)$ bet unit tangential vector to $\partial \Omega$ satisfying (\ref{tangentialderivativedef}).
And let $\partial_{\tau_i} g$ be the tangential derivative at direction $\tau_i$ for $g$ defined on $\partial \Omega$.
Define 
\Be \label{defofgradientxg}
 \begin{split}
\nabla_x g =  \sum_{i =1 }^2  \tau_i \partial_{\tau_i } g  - \frac{ n }{ n \cdot \vb } \Big \{ \partial_t g + \sum_{i =1}^2 ( \vb \cdot \tau_i ) \partial_{\tau_i } g  + \nu g - H +  E \cdot \nabla_v g  \Big \}.
 \end{split} 
 \Ee


\begin{proposition} \label{inflowexistence1}
Assume the compatibility condition
\[ f_0(x,v) = g(0,x,v)\quad  \text{for} \quad (x,v) \in \gamma_- .\]
Let $p \in [1, \infty )$ and $0 < \theta < 1/4$. $|\psi(t,x,v) | \lesssim \langle v \rangle $. $\| E \|_\infty + \| \nabla_x E \|_\infty < \infty$.

Assume
\[ \begin{split}
  \nabla_x f_0 , \nabla_v f_0 , 
  \in L^p (\Omega \times \mathbb R^3 ),
\\ \nabla_v g, \partial_{\tau_i } g \in L^p ( [0, T] \times \gamma_- ) ,
\\ \frac{ n(x) }{ n(x) \cdot v } \Big \{ \partial_t g + \sum_{i =1}^2 ( v \cdot \tau_i ) \partial_{\tau_i } g  + \nu g - H +  E \cdot \nabla_v g \Big \} \in L^p ([0, T ] \times \gamma_- ),
\\ 
 \frac{ n(x) \cdot \iint \partial_x E }{ n(x) \cdot v } \Big \{ \partial_t g + \sum_{i =1}^2 ( v \cdot \tau_i) \partial_{ \tau_i } g   - \nu g + H \Big \} \in L^p ([0, T ] \times \gamma_- ),
\\ \nabla_x H ,\nabla_v H \in L^p  ([0, T ] \times \Omega \times \mathbb R^3 ),
\\ e^{- \theta |v|^2 } \nabla_x \nu, e^{-\theta |v|^2 } \nabla_v \nu \in L^p ([ 0, T ] \times \Omega \times \mathbb R^3 ),
\\ e^{\theta |v|^2 } f_0 \in L^\infty ( \Omega \times \mathbb R^3 ) , e^{\theta |v|^2 } g \in L^\infty ( [0, T] \times \gamma_- ),
\\ e^{\theta |v|^2 } H \in L^\infty ([0, T] \times \Omega \times \mathbb R^3 ).
\end{split} \]
Then for any $T > 0$, there exists a unique solution $f$ to (\ref{traninflowfixE}), such that $f, \partial_t, \nabla_x f ,\nabla_v f \in C^0( [ 0, T ] ; L^p (\Omega \times \mathbb R^3) ) $ and their traces satisfiey
\Be \label{traceofinflow} \begin{split}
\nabla_v f |_{\gamma_-} = \nabla_v g, \nabla_x f|_{\gamma_-} = \nabla_x g, \quad \text{on} \quad \gamma_- ,
\\ \nabla_x f(0,x,v) = \nabla_x f_0, \nabla_v f(0,x,v) = \nabla_v f_0, \quad \text{in} \quad \Omega \times \mathbb R^3,  
\\ \partial_t f(0,x,v) = \partial_t f_0, \quad \text{in} \quad \Omega \times \mathbb R^3. 
\end{split} \Ee
where $\nabla_x g $ is given by (\ref{defofgradientxg}).

\end{proposition}

\begin{proof}


Consider the case $t \le \tb$ and $t > \tb$ separately and integrate along the trajectory $X(s),V(s)$ we have for $t<\tb$:
\[ \begin{split}
f(t,x,v) = & f_0(X(0) , V(0) ) e^{-\int_0^t \nu }- \int_0^t \frac{d}{ds} \left( f(t-s, X(t-s) , V(t-s) ) e^{-\int_0^s \nu } \right) ds
\\ =  & f_0 (X(0), V(0) ) e^{- \int_0^t \nu } + \int_0^t e^{- \int _0^s \nu} H (t-s, X(t-s), V(t-s ) ) ds,
\end{split} \]
where $H = \{ \partial_t + v \cdot \nabla_x  + E \cdot \nabla_v  + \nu \} f $, $\nu = \nu ( t - \tau, X(t- \tau), V(t- \tau) )$.
And for $t > \tb$:
\[ \begin{split}
f(t,x,v) =  e^{-\int_0^ {\tb}  \nu} g(t -\tb, \xb,\vb ) + \int_0 ^ {\tb} e^{-\int_0^s \nu } H(t-s ,X(t-s), V(t-s) ) ds.
\end{split} \]
We can rewrite it as:
\[ \begin{split}
f(t,x,v) = & \textbf{1}_{ \{ t \le \tb \} } e^{-\int_0^t \nu } f_0 (X(0) ,V(0)) + \textbf{1}_{ \{ t > \tb \} } e^{-\int_0^ {\tb}  \nu} g(t -\tb, \xb,\vb ) \\ & + \int_0^{\min \{ \tb, t\} } e^{-\int_0^s \nu } H(t-s ,X(t-s), V(t-s) ) ds.
\end{split} .\]
By direct computation we have
\[ \begin{split}
\nabla_x & f(t,x,v ) \textbf{1}_{ \{ t\neq \tb \}} 
\\ = & \textbf{1}_{\{ t < \tb \}} e^{-\int_0^t \nu } \left[  \nabla_x f_0 \cdot \nabla_x X(0) + \nabla_v f_0 \cdot \nabla_x V(0) - f_0 \int_0^t ( \nabla_x \nu \cdot \nabla_x X + \nabla_v \nu \cdot \nabla_x V)(t-\tau)   \right]
\\ & + \textbf{1}_{\{ t > \tb \} } e^{- \int_0^{\tb} \nu } \left\{ -\nabla_x\tb \nu(t-\tb)   g(t-\tb) + \nabla_x \tb H(t-\tb) - g(t-\tb) \int_0^{\tb} (\nabla_x \nu \cdot \nabla_x X + \nabla_v \nu \cdot \nabla_x V)(t-\tau)     \right\} 
\\ & + \textbf{1}_{\{ t > \tb \} } e^{- \int_0^{\tb} \nu }  \partial_x ( g  (t -\tb, \xb,\vb) ) 
\\ & + \int_0^{ \min \{ t, \tb \}} e^{- \int_0 ^s \nu }   \bigg [  \nabla_x H(t-s) \cdot \nabla_x X(t-s)  + \nabla_v H(t-s) \cdot \nabla_x V(t-s)
-H(t-s)\int_0^s ( \nabla_x \nu \cdot \nabla_x X +  \nabla_v \nu \cdot \nabla_x V)(t-\tau)   \bigg ]   ds.
\end{split} \]

\[ \begin{split}
 \nabla_v & f(t,x,v ) \textbf{1}_{\{ t\neq \tb \}}
\\ = & \textbf{1}_{\{ t < \tb \} } e^{-\int_0^t \nu } \left[ \nabla_x f_0 \cdot \nabla_v X(0) + \nabla_v f_0 \cdot \nabla_v V(0) - f_0\int_0^t (\nabla_x \nu \cdot \nabla_v X + \nabla_v \nu \cdot \nabla_v V) (t-\tau) \right] 
\\ & + \textbf{1}_{\{ t > \tb \}} e^{-\int_0^{\tb } \nu } \left\{ - \nabla_v \tb \nu(t-\tb)  g(t-\tb)+ \nabla_v \tb H(t-\tb) - g(t-\tb) \int_0^{\tb} (\nabla_x \nu \cdot \nabla_v X + \nabla_v \nu \cdot \nabla_v V) (t-\tau)  \right\}
\\ & + \textbf{1}_{\{ t > \tb \}} e^{-\int_0^{\tb } \nu } \partial_v ( g(t -\tb, \xb,\vb) )
\\ & + \int_0^{\min \{ t, \tb \}} e^{-\int_0^s \nu} \bigg [ \nabla_x H(t-s) \cdot \nabla_v X(t-s) + \nabla_v H(t-s) \cdot \nabla_v V(t-s)
- H(t-s) \int_0^s (\nabla_x \nu \cdot \nabla_v X + \nabla_v \nu \cdot \nabla_v V)(t-\tau) \bigg ] ds.
\end{split} \]
Regarding $g(t-\tb,\xb(t,x,v),v)$ as function on $[0,T]  \times \bar \Omega \times \mathbb R^3$, we obtain from (\ref{tangentialderivativedef}) that 
\[
\partial_x  [ g(t -\tb, \xb, \vb ) ]  = \nabla_{\tau} g \cdot \nabla_x \xb = (\tau_1 \partial_{\tau_1} g  + \tau_2 \partial_{\tau_2} g) \cdot  \nabla_x \xb.
\]
Thus 
from (\ref{computation_tb_x}) we have:
\[ \begin{split}
\partial_x  & [ g(t -\tb, \xb, \vb ) ]  
\\= &  - \nabla_x \tb \partial_t g + \nabla_\tau g  \nabla_x \xb + \nabla_v g \nabla_x \vb
\\  = & -\frac{n(\xb) }{ n(\xb) \cdot \vb } \partial_t g - \frac{ n(\xb) \cdot \iint \partial_x E }{ n(\xb ) \cdot \vb } \partial_t g 
\\ & + \tau_1 \partial_{\tau_1} g + \tau_2 \partial_{\tau_2} g - \frac{ n(\xb ) } { n(\xb ) \cdot \vb } \left( \vb \cdot \tau_1 \partial_{\tau_1} g + \vb \cdot \tau_2 \partial_{\tau_2} g \right) - \frac{ n(\xb) \cdot \iint \partial_x E } {n (\xb) \cdot \vb } \left( \vb \cdot \tau_1 \partial_{\tau_1} g + \vb \cdot \tau_2 \partial_{\tau_2 } g \right)
\\ & - \frac{n(\xb)}{n (\xb) \cdot \vb } ( E \cdot \nabla_v g ) - \nabla_v g \cdot \int \partial_x E,
\end{split} \]
\[ \begin{split}
\partial_v  & [ g(t -\tb, \xb, \vb ) ]  
\\= &  - \nabla_v \tb \partial_t g + \nabla_v \xb \nabla_{\tau } g + \nabla_v g \nabla_v \vb
\\  = & -\frac{\tb n(\xb) }{ n(\xb) \cdot \vb } \partial_t g - \frac{ n(\xb) \cdot \iint \partial_v E }{ n(\xb ) \cdot \vb } \partial_t g 
\\ & - \tb ( \tau_1 \partial_{\tau_1} g +  \tau_2 \partial_{\tau_2} g) - \tb \frac{ n(\xb ) } { n(\xb ) \cdot \vb } \left( \vb \cdot \tau_1 \partial_{\tau_1} g + \vb \cdot \tau_2 \partial_{\tau_2} g \right) - \frac{ n \cdot \iint \partial_v E } {n \cdot \vb } \left( \vb \cdot \tau_1 \partial_{\tau_1} g + \vb \cdot \tau_2 \partial_{\tau_2 } g \right)
\\ & - \frac{\tb n(\xb)}{n (\xb) \cdot \vb } ( E \cdot \nabla_v g ) - \nabla_v g \cdot \int \partial_v E + \nabla_v g.
\end{split} \]
Plug into the previous equation we eventually have:
\Be \label{gradientfexpalongtraj} \begin{split}
\nabla_x & f(t,x,v ) \textbf{1}_{ \{ t\neq \tb \}} 
\\ = & \textbf{1}_{\{ t < \tb \}} e^{-\int_0^t \nu } \left[  \nabla_x f_0 \cdot \nabla_x X(0) + \nabla_v f_0 \cdot \nabla_x V(0) - f_0\int_0^t ( \nabla_x \nu \cdot \nabla_x X + \nabla_v \nu \cdot \nabla_x V)(t-\tau) \right]  
\\ & + \textbf{1}_{ \{ t > \tb \} } e^{- \int_0 ^{\tb } \nu } \Bigg \{ \sum_{i =1 }^2  \tau_i \partial_{\tau_i } g  - \nabla_v g \cdot \int \partial_x E- g \int_0^{\tb} (\nabla_x \nu \cdot \nabla_x X + \nabla_v \nu \cdot \nabla_x V)(t -\tau)   
\\ & \qquad \qquad \qquad \qquad- \frac{ n }{ n \cdot \vb } \Big \{ \partial_t g + \sum_{i =1}^2 ( \vb \cdot \tau_i ) \partial_{\tau_i } g  + \nu g - H +  E \cdot \nabla_v g \Big \}
\\  & \qquad\qquad\qquad\qquad
- \frac{ n \cdot \iint \partial_x E }{ n \cdot \vb } \Big \{ \partial_t g + \sum_{i =1}^2 ( \vb \cdot \tau_i) \partial_{ \tau_i } g  - \nu g + H \Big \}
\Bigg \} ( t -\tb, \xb, \vb )
\\ & + \int_0^{ \min \{ t, \tb \}} e^{- \int_0 ^s \nu }   \bigg [  \nabla_x H(t-s) \cdot \nabla_x X(t-s)  + \nabla_v H(t-s) \cdot \nabla_x V(t-s)
\\  & \qquad \qquad \qquad \qquad \qquad
-H(t-s) \int_0^s ( \nabla_x \nu \cdot \nabla_x X +  \nabla_v \nu \cdot \nabla_x V)(t -\tau)   \bigg ]   ds,
%
\\ 
\nabla_v & f(t,x,v ) \textbf{1}_{\{ t\neq \tb \}}
\\ = & \textbf{1}_{\{ t < \tb \} } e^{-\int_0^t \nu } \left[ \nabla_x f_0 \cdot \nabla_v X(0) + \nabla_v f_0 \cdot \nabla_v V(0) - f_0 \int_0^t (\nabla_x \nu \cdot \nabla_v X + \nabla_v \nu \cdot \nabla_v V)(t-\tau)  \right]
\\ & - \textbf{1}_{ \{t > \tb  \}} \tb e^{-\int _0 ^ {\tb} \nu } \Bigg \{ \sum_{i =1}^2 \tau_i \partial_{\tau_i } g 
- \frac{ n}{ n \cdot \vb } \Big \{ \partial_t g + \sum_{i=1}^2 ( \vb \cdot \tau_i ) \partial_{\tau_i } g + \nu g - H   + E \cdot \nabla_v g \Big \}  \Bigg \}  ( t -\tb, \xb, \vb )
\\ & + \textbf{1}_{ \{ t > \tb \} } e^{-\int_0^{\tb}  \nu } \Bigg \{ \nabla_v g - g \int_0 ^ {\tb}  ( \nabla_x \nu \cdot \nabla_v X + \nabla_v \nu \cdot \nabla_v V )(t-\tau)  - \nabla_v g \cdot \int \partial_v E 
\\ & \qquad\qquad\qquad\qquad + \frac{ n \cdot \iint \partial_v E }{ n \cdot \vb } \Big \{ \partial_t g + \sum_{i =1}^2 ( \vb \cdot \tau_i) \partial_{ \tau_i } g  - \nu g + H \Big \} \Bigg \} ( t -\tb, \xb, \vb )
\\ & + \int_0^{\min \{ t, \tb \}} e^{-\int_0^s \nu} \bigg [ \nabla_x H(t-s) \cdot \nabla_v X(t-s) + \nabla_v H(t-s) \cdot \nabla_v V(t-s)
        \\& \qquad \qquad \qquad \qquad \qquad
-  H(t-s) \int_0^s (\nabla_x \nu \cdot \nabla_v X + \nabla_v \nu \cdot \nabla_v V) (t-\tau) \bigg ]ds.
\end{split} \Ee
From (\ref{hamilton_ODE}) with replacing $- \nabla_x \phi_f$ by $E$,  
	\[
	\frac{d}{ds} \left[ \begin{matrix}\nabla_{x,v} X(s;t,x,v) \\ \nabla_{x,v} V(s;t,x,v)\end{matrix} \right]
	= \left[\begin{matrix} 
	0_{3\times 3} & \text{Id}_{3\times 3}\\
	\nabla_x E(s,X(s;t,x,v)) & 0_{3\times 3} 
	\end{matrix}\right]
	\left[ \begin{matrix}\nabla_{x,v} X(s;t,x,v) \\ \nabla_{x,v} V(s;t,x,v)\end{matrix} \right].
	\] 
	Then by Gronwall's inequality, easily we have
	\Be
	|\nabla_{x,v} X(s;t,x,v)|+ 
	|\nabla_{x,v} V(s;t,x,v)| \lesssim e^{ (1+\| \nabla_x E \|_\infty) |t-s|}.\notag
	\Ee
Therefore by the change of variables from lemma (\ref{covlemma1}) and lemma (\ref{covlemma2}), and (\ref{intVbdd}) we have:
\[
\| f(t) \textbf{1}_{ \{t \neq \tb \} } \| _p
\lesssim e^{t(\| E \|_\infty+1)}  \left( \| f_0 \|_p + \left [ \int_0^t \int_{\gamma_- } | g|^p d\gamma ds \right ] ^{1/p} + \left [ \int_0^t \| H \|_p^p ds \right ]^{1/p} \right),
\]
%
\[ \begin{split}
\| \nabla_x & f(t) \textbf{1}_{ \{ t \neq \tb \} } \|_p
\\  \lesssim & e^{t(\| E \|_\infty+1)}   \Bigg(  \| \nabla_x f_0 \|_p + \| \nabla_v f_0 \|_p + \left[ \int_0^t \| \nabla_x H \|_p^p +  \| \nabla_v H \|_p^p \right]^{1/p } + \left[ \int_0^t \int_{\gamma _- } | \nabla_v g |^p d\gamma ds \right] ^{1/p}
\\ & + \{ \| e^ {\theta |v|^2 } f_0 \|_\infty + \| e^{\theta |v|^2 } H\|_\infty + | e^{\theta |v| ^2} g |_\infty \} \left[ \int_0^t \| e^{-\theta |v|^2 } \partial_t \nu \|_p^p  + \| e^{-\theta |v|^2 } \nabla_v \nu \|_p^p \right]^{1/p}
\\ & + \Bigg [ \int_0^t \int_{\gamma _- } d \gamma ds | \{ \sum_{i =1 }^2  \tau_i \partial_{\tau_i } g - \frac{ n }{ n \cdot \vb } \Big \{ \partial_t g + \sum_{i =1}^2 ( \vb \cdot \tau_i ) \partial_{\tau_i } g  + \nu g - H +  E \cdot \nabla_v g \Big \}
\\  & \qquad\qquad\qquad\qquad- \frac{ n \cdot \iint \partial_x E }{ n \cdot \vb } \Big \{ \partial_t g + \sum_{i =1}^2 ( \vb \cdot \tau_i) \partial_{ \tau_i } g  - \nu g + H \Big \} | ^p \Bigg ] ^{1/p} \Bigg).
\end{split} \]
and
\[ \begin{split}
\| \nabla_v & f(t) \textbf{1}_{ \{ t \neq \tb \} } \|_p
\\   \lesssim & e^{t(\| E \|_\infty+1)}   \Bigg( \| \nabla_x f_0 \|_p + \| \nabla_v f_0 \|_p + \left[ \int_0^t \| \nabla_x H \|_p^p +  \| \nabla_v H \|_p^p \right]^{1/p } + \left[ \int_0^t \int_{\gamma _- } | \nabla_v g |^p d\gamma ds \right] ^{1/p}
\\ & + \{ \| e^ {\theta |v|^2 } f_0 \|_\infty + \| e^{\theta |v|^2 } H\|_\infty + | e^{\theta |v| ^2} g |_\infty \} \left[ \int_0^t \| e^{-\theta |v|^2 } \partial_t \nu \|_p^p  + \| e^{-\theta |v|^2 } \nabla_v \nu \|_p^p \right]^{1/p}
\\ & + \Bigg [ \int_0^t \int_{\gamma _- } d \gamma ds | \{ \sum_{i =1 }^2  \tau_i \partial_{\tau_i } g - \frac{ n }{ n \cdot \vb } \Big \{ \partial_t g + \sum_{i =1}^2 ( \vb \cdot \tau_i ) \partial_{\tau_i } g  + \nu g - H +  E \cdot \nabla_v g \Big \}
\\  & \qquad\qquad\qquad\qquad- \frac{ n \cdot \iint \partial_v E }{ n \cdot \vb } \Big \{ \partial_t g + \sum_{i =1}^2 ( \vb \cdot \tau_i) \partial_{ \tau_i } g  - \nu g + H \Big \} | ^p \Bigg ] ^{1/p} \Bigg).
\end{split} \]
From our hypothesis, these terms on the RHS are bounded, therefore
\[
\partial f \textbf{1}_{ \{ t \neq \tb \} } \equiv [ \partial_t f \textbf{1}_{ \{ t \neq \tb \} } , \nabla _x f \textbf{1}_{ \{ t \neq \tb \} }, \nabla_v f \textbf{1}_{ \{ t \neq \tb \} } ] \in L^\infty ([0,T]; L^p(\Omega \times \mathbb R^3 )).
\]
On the other hand, thanks to the compatibility condition, we need to show $f$ has the same trace on the set
\[
\mathcal M \equiv \{ t = \tb (x,v) \} \equiv \{ (\tb (x,v) , x, v  ) \in [0, T] \times \Omega \times \mathbb R^3 \}.
\]
We claim the following fact: Let $\phi (t,x,v) \in C_c ^\infty ((0,T) \times \Omega \times \mathbb R^3 )$, then we have
\[
\int_0^T \iint_{\Omega \times \mathbb R^3 } f \partial \phi = - \int_0 ^T \iint_{ \Omega \times \mathbb R^3 } \partial f \textbf{1}_{ \{ t \neq \tb \} } \phi,
\]
so that $f \in W^{1,p} $ with weak derivatives given by $\partial f \textbf{1}_{ \{ t \neq \tb \} } $.

\textit{Proof of claim.} We first fix the test function $\phi (t,x,v ) $. There exists $\delta = \delta_{\phi } > 0$ such that $\phi \equiv 0$ for $t \ge \frac{1}{\delta}$, or dist$(x,\partial \Omega) < \delta$, or $|v| \ge \frac{1}{\delta}$. Let $\phi (t,x,v) \neq 0 $ and $(t,x,v) \in \mathcal M$, so $t = \tb(t,x,v)$. We have $n(\xb(t,x,v)) \cdot \vb(t,x,v) \le 0$.

Recall the velocity lemma. Since 
\[
\alpha(t- \tb(t,x,v), \xb(t,x,v), \vb(t,x,v) ) \le |n(\xb(t,x,v)) \cdot \vb(t,x,v)|
\]
from the definition of $\alpha$. And by (\ref{velocitylemmaintform}) $\alpha$ satisfies
\[
0<\alpha (t,x,v) \le  e^{C \int_0 ^ t ( |V(\tau')| + 1 ) d \tau'} \alpha (t- \tb(t,x,v) , \xb(t,x,v) , \vb(t,x,v)) \le e^{C \int_0 ^ t ( |V(\tau')| + 1 ) d \tau'}  |n(\xb(t,x,v)) \cdot \vb(t,x,v)|.
\]
So we have $n(\xb(t,x,v)) \cdot \vb(t,x,v) \neq 0$. Therefore
\[
n(\xb(t,x,v)) \cdot \vb(t,x,v) < 0.
\]


Now since $\{ \phi \neq 0 \}$ is compact, $n(\xb(t,x,v)) \cdot \vb(t,x,v)$ reaches a maximum. Therefore $| n(\xb(t,x,v)) \cdot \vb(t,x,v)| > \delta' > 0$ so $\{ \phi \neq 0 \} \cap \mathcal M $ is a smooth $6$D hypersurface. 

We next take a $C^1$ approximation of $f^l_0$, $H^l$, and $g^l$ (by partition of unity and localization) such that
\[
\| f^l_0 - f_0 \|_{W^{1,p}} \to 0, \, \| g^l - g \|_{W^{1,p} ( [0,T] \times \gamma_- \setminus \gamma_-^{\delta'} )} \to 0, \, \| H^l - H \|_{W^{1,p} ( [0,T] \times \Omega \times \mathbb R^3 ) } \to 0,
\]
where $W^{1,p} ( [0,T] \times \gamma_- \setminus \gamma_-^{\delta'} )$ is the standard Sobolev space in $ [0,T] \times \gamma_- \setminus \gamma_-^{\delta'} $.
This implies, from the trace theorem, that
\[
f^l_0(x,v ) \to f_0(x,v) \, \text{and} \, g^l(0,x,v) \to g(0,x,v) \, \text{in } \, L^p(\gamma_- \setminus \gamma_-^{\delta'} ).
\]
We define accordingly, for $(t,x,v) \in [0,T] \times \Omega \times \mathbb R^3$,
\[ \begin{split}
f^l(t,x,v) = & \textbf{1}_{ \{ t \le \tb \} } e^{-\int_0^t \nu } f^l_0 (X(0) ,V(0)) + \textbf{1}_{ \{ t > \tb \} } e^{-\int_0^ {\tb}  \nu} g^l(t -\tb, \xb,\vb ) \\ & + \int_0^{\min \{ \tb, t\} } e^{-\int_0^s \nu } H^l(t-s ,X(t-s), V(t-s) ) ds,
\end{split} \]
and
\[
f^l_-(t,x,v ) =\textbf{1}_{ \{ t \le \tb \} } e^{-\int_0^t \nu } f^l_0 (X(0) ,V(0)) + \int_0^{\min \{ \tb, t\} } e^{-\int_0^s \nu } H^l(t-s ,X(t-s), V(t-s) ) ds,
\] 
\[
f^l_+(t,x,v ) = \textbf{1}_{ \{ t \ge \tb \} } e^{-\int_0^ {\tb}  \nu} g^l(t -\tb, \xb,\vb ) + \int_0^{\min \{ \tb, t\} } e^{-\int_0^s \nu } H^l(t-s ,X(t-s), V(t-s) ) ds.
\] 
Therefore for all $(x,v) \in \gamma_-$,
\[
f^l_+(s, X(s;0,x,v ) , V(s;0,x,v) ) - f^l_-(s,   X(s;0,x,v ) , V(s;0,x,v) ) = e^{- \int_0 ^s \nu } \left[ g^l(0,x,v ) - f^l_0(x,v) \right].
\]
Since $\{ \phi \neq 0 \} \cap \mathcal M$ is a smooth hypersurface, we apply the Gauss theorem to $f^l$ to obtain
\[ \begin{split}
 \iiint \partial_{\bf e} \phi f^l dxdvdt = & \iint [ f^l_+ - f^l_- ] \phi {\bf e} \cdot {\bf n}_{\mathcal M} d \mathcal M
\\ & - \left\{ \iiint_{t > \tb} \phi \partial_{\bf e } f^l_+ dxdvdt + \iiint_{t < \tb} \phi \partial_{\bf e} f^l_- dxdvdt  \right\},
\end{split} \]
where $\partial_{\bf e} = [ \partial_t, \nabla_x,\nabla_v ] = [ \partial_t, \partial_{x_1},\partial_{x_2},\partial_{x_3},\partial_{v_1},\partial_{v_2},\partial_{v_3}] $ and 
\[
{\bf n}_{ \mathcal M } = \frac{ 1 }{ \sqrt{ ( 1 - \partial_t \tb ) ^2 + |\nabla_x \tb|^2 + |\nabla_v \tb|^2}} ( 1 - \partial_t \tb, -\nabla_x \tb, - \nabla_v \tb) \in \mathbb R^7.
\]
Using $(s, X(s;0,x,v) , V(s;0,x,v) ) $ and $(x,v) \in \gamma_-$ as our parametrization for the manifold $\{ \phi \neq 0 \} \cap \mathcal M$, and from (\ref{eta}), letting $x = \eta(x_\parallel) =  \eta(x_{\parallel,1},x_{\parallel,2})$ for $x \in \partial \Omega$, we have the Jacobian matrix
\[ J =
\begin{bmatrix}[ccc|c]
1 & 0 & 0 &  \\
\partial_s X & \nabla_{x_\parallel } X  & \nabla_v X & ({\bf n}_{ \mathcal M })^T \\
\partial_s V & \nabla_{x_\parallel} V & \nabla_v V & 
\end{bmatrix} .
\]
Then since $|v \cdot n(x) |> \delta'$, the surface measure of $\mathcal M$ is $|\det(J)| dx_\parallel dv ds$ which is bounded from above, thus
\[ \begin{split}
\iint & [ f^l_+ - f^l_- ] \phi {\bf e} \cdot {\bf n} _{\mathcal M} d \mathcal M  \\
& \le \int_0^T \int_{n(x) \cdot v \ge \delta'} |f^l_+(s, X(s;0,x,v ) , V(s;0,x,v) ) - f^l_-(s,   X(s;0,x,v ) , V(s;0,x,v) ) | |\det (J) | d  x_\parallel dv ds
\\ & \lesssim_{T,\phi,\delta} \int_{n(x) \cdot v \ge \delta'}  | g^l(0,x,v ) - f^l_0(x,v) | \frac{1}{|v \cdot n(x)|  \| \partial_1 \eta \times \partial_2 \eta \| } |v \cdot n(x)|  \| \partial_1 \eta \times \partial_2 \eta \| d x_\parallel dv 
\\ & \lesssim_{T,\phi,\delta} \int_{n(x) \cdot v \ge \delta'}  | g^l(0,x,v ) - f^l_0(x,v) | | d \gamma  \to 0, \, \text{as } l \to \infty,
\end{split} \]
due to the compatibility condition $f_0(x,v) = g(0,x,v)$ for $(x,v) \in \gamma_-$.

Clearly, taking difference of $f^l -f $ and using the strong $L^p$ estimate we deduce that $f^l \to f$ strongly in $L^p ( \{ \phi \neq 0 \})$. Furthermore, due to the same estimate for $\nabla_x f $ and $\nabla_v f$ we have a uniform-in-$l$ bound of $f^l_{\pm}$ in $W^{1,p}(\{  t \neq \tb, \phi \neq 0 \}) $. Therefore we have up to a subsequence, $ \partial_{\textbf{e}} f^l_{\pm} $ converges weakly. And since the weak limits coincides with the pointwise limit we have
\[
\partial_{\textbf{e}} f^l_+ \rightharpoonup \partial_{\textbf{e}}  f \textbf{1}_{ t > \tb}, \, \partial_{\textbf{e}} f^l_- \rightharpoonup \partial_{\textbf{e}}  f \textbf{1}_{ t < \tb}.
\]
Finally we conclude the claim by letting $l \to \infty$. 

Now since we assume all the data are compactly supported in the velocity space, $f$ itself is compactly supported in the velocity space, so $e^{\theta |v|^2 }f \in L^\infty$ as $f_0,g, H \in L^\infty$. From this and the $L^p$ bounds above, we conclude:
\[
\{ \partial_t + v \cdot \nabla_x + E \cdot \nabla_v + \nu \} \partial f = \partial H - \partial v \cdot \nabla_x f - \partial E \cdot \nabla_v f - \partial \nu f  \in L^p.
\]

By the trace theorem, the traces of $\partial_t f, \nabla_x f, \nabla_v f$ exist. To evaluate these traces, we use the fact that for almost every $(t,x,v)$, $\partial f$ is absolutely continuous along the trajectory $(t-s, X(t-s;t,x,v), V(t-s;t,x,v))$.

First consider $t > \tb(t,x,v)> s$, as $s \to \tb(t,x,v)$, $\tb(t-s,X(t-s),V(t-s)) = \tb(t,x,v) -s  \to 0$. Thus by our formulas for $\partial f$ we have $\partial f(t-s, X(t-s), V(t-s)) \to \partial g ( t-\tb, \xb, \vb)$ as $s \to \tb(t,x,v)$. Therefore $\partial f |_{\gamma_- } = \partial g $.

If $ \tb(t,x,v) > t > s$. Again using the explicit formula for $\partial f$ and the fact that $(\partial_{x,v} X) (0;t-s,x,v) = ( id, 0) $ and $(\partial_{x,v} V) (0;t-s,x,v) = ( 0, id)$ as $s \to t$, we have that $\partial f(t-s, X(t-s), V(t-s)) \to \partial f(0, X(0), V(0))$ as $s \to t$. Therefore $\partial f(0,x,v) = \partial f_0$. This proves (\ref{traceofinflow}).

In order to remove the compact support assumption we employ a cut-off function $\chi$. Define $f^m = \chi(|v|/ m)f$ then $f^m$ satisfies
\Be \label{cutofffunctionvspace} \begin{split}
\{ \partial_t + v \cdot \nabla_x + E \cdot \nabla_v +(\chi ' (|v| /m ) - E \cdot \nabla_v \chi(|v|/m) ) \} f^m = \chi(|v|/m)H, 
\\ f^m(0,x,v) = \chi (|v|/m)f_0, f^m_{\gamma_-} = \chi (|v|/m) g. 
\end{split} \Ee

Now by previous argument we have the traces of $\partial f^m$ exists and $\partial f^m(0,x,v) = \partial (\chi (|v|/m) f_0)$, $\partial f^m|_{\gamma_-} = \partial (\chi(|v| /m )g)$. And $\partial (\chi (|v|/m) f_0, g) = \chi(|v|/m ) \partial f_0, \partial g + \partial \chi(|v|/m ) f_0,g \to \partial f_0, \partial g$ in $L^p$ as $m \to \infty$. On the other hand we have $ \partial f^m = \chi (|v|/m ) \partial f + \partial \chi ( |v| /m ) f$, so the traces of $\partial f^m$ goes to the traces of $\partial f$ almost everywhere as $m \to \infty$. Therefore we conclude $\partial f (0,x,v) = \partial f_0$ and $\partial f |_{\gamma_- } = \partial g |_{\gamma_-}$ as desired.

\end{proof}

 \begin{proposition} \label{inflowproposition2}
 Let $f$ be a solution of (\ref{traninflowfixE}). Assume $f_0 (x,v) = g(0,x,v)$ for all $(x,v) \in \gamma_- $.   
 
 For any fixed $p \in [ 2, \infty]$, $0 < \theta < 1/4$, $\beta > 0$, and $\varpi \gg 1$ assume
 \[ \begin{split}
 \alpha^\beta \nabla_x f_0, \alpha^\beta \nabla_v f_0 & \in L^p (\Omega \times \mathbb R^3 ),
 \\  e^{ -\varpi \langle v \rangle t } \alpha^\beta \nabla_v g, e^{-\varpi \langle v \rangle t } \alpha^\beta \partial_{\tau_i } g &  \in L^p ( [0,T] \times \gamma_- ),
 \\ \frac{ e^{ -\varpi \langle v \rangle t } \alpha ^\beta } { n(x) \cdot v } \Big \{ \partial_t g + \sum ( v \cdot \tau_i ) \partial_{\tau_i } g + \nu g - H + E\cdot \nabla_v g \Big \}  
 \\   +  \frac{ e^{-\varpi \langle v \rangle t } \alpha ^ \beta n(x) \cdot \iint \partial_x E  }{ n(x) \cdot v } \Big \{ \partial_t g + \sum ( \vb \cdot \tau_i ) \partial_ {\tau_i } g - \nu g + H \Big \} & \in L^p ([0, T] \times \gamma_- ),
 \\  e^{ -\varpi \langle v \rangle t } \alpha^\beta \nabla_v H, e^{-\varpi \langle v \rangle t } \alpha^\beta \nabla_x H &  \in L^p ( [0,T] \times \Omega \times \mathbb R^3 ),
 \\   e^{-\theta |v| ^2 }  e^{ -\varpi \langle v \rangle t } \alpha^\beta \nabla_v \nu, e^{-\theta |v|^2} e^{-\varpi \langle v \rangle t } \alpha^\beta \nabla_x \nu &  \in L^p ( [0,T] \times \Omega \times \mathbb R^3 ),
 \\ e^{\theta |v|^2} f_0 \in L^\infty( \Omega \times \mathbb R^3 ),  e^{\theta |v|^2 } g \in L^\infty ( [0, T ] \times \gamma_-), e^{\theta |v| ^2 } H  & \in L^\infty ([ 0, T] \times \Omega \times \mathbb R^3 ).
\end{split} \]
Then for $\partial \in \{ \nabla_x, \nabla_v \}$, we have $e^{-\varpi \langle v \rangle t} \alpha^\beta \partial f(t,x,v) \in L^\infty([0,T]; L^p(\Omega \times \mathbb R^3 ) )$, and 
\[
e^{-\varpi \langle v \rangle t} \alpha^\beta \partial f |_{t=0} = e^{-\varpi \langle v \rangle t} \alpha ^\beta \partial f_0, e^{-\varpi \langle v \rangle t} \alpha^\beta \partial f |_{\gamma_-} = e^{-\varpi \langle v \rangle t} \alpha^\beta \partial g,
\]
where $\partial g$ is given in (\ref{defofgradientxg}).
 \end{proposition}
 
 \begin{proof}
 First we assume $f_0$, $g$ and $H$ have compact supports in $\{ v \in \mathbb R^3: |v| < m \}$.
 By (\ref{velocitylemmaintform}) we have for $\varpi \gtrsim \frac{(\| E \|_\infty + \| \nabla E \|_\infty )}{C_E}$, and for any $0 \le s_1, s_2 \le t $ and any $(x,v) \in \Omega \times \mathbb R^3$ that
\[
e^{ - \varpi \int _{s_1} ^{s_2 }\langle V(\tau) \rangle d \tau} \alpha (s_1,X(s_1),V(s_1) ) \le \alpha (s_2,X(s_2),V(s_2) ) \le e^{ \varpi  \int_{s_1}^{s_2} \langle V(\tau) \rangle d\tau} \alpha (s_1,X(s_1),V(s_1)).
\]
And since $ \left | \int_{\max\{0,t-\tb\}}^t \langle V(s;t,x,v) \rangle ds - \langle v \rangle t \right | \le \| E \|_\infty t^2$, we have for any $ \beta >0 $
\begin{equation} \label{exp}
 \sup_{ t \le \tb } \frac{ e^{- \varpi \langle v \rangle t } \alpha^\beta (t,x,v ) }{\alpha ^\beta (0,X(0),V(0)) } \le e^{ \beta \varpi \| E \|_\infty t^2  }, \,
  \sup_{ t \ge \tb } \frac{ e^{- \varpi \langle v \rangle t } \alpha^\beta (t,x,v ) }{e^{- \varpi \langle \vb \rangle (t-\tb)} \alpha ^\beta (t-\tb, \xb,\vb) } \le e^{2 \beta \varpi \| E \|_\infty t^2  } ,\end{equation}
 \[
    \sup_{ \max \{ t-\tb, 0 \} \le s \le t } \frac{ e^{- \varpi \langle v \rangle t } \alpha^\beta (t,x,v ) }{e^{- \varpi  \langle V(t-s) \rangle  (t-s)} \alpha ^\beta (t-s, X(t-s),V(t-s)) } \le e^{ 2 \beta \varpi \| E \|_\infty t^2  }.
 \]
 Multiplying $e^{ - \varpi \langle v \rangle  t } \alpha ^\beta (t,x,v)$ to (\ref{gradientfexpalongtraj}), and then using the change of variables from  (\ref{covlemma1}) and lemma (\ref{covlemma2}), and the bound from (\ref{exp}), we get
 
 \[ \begin{split}
 \| & e^{ - \varpi \langle v \rangle t } \alpha ^\beta \nabla_x f(t) \|_{L^p } \lesssim_{ \beta }  e^{\varpi t^2 (\| E \|_\infty^2+\| \nabla E \|_\infty^2 +1)}  \Bigg( \| \alpha ^\beta \nabla_x f_0 \|_p + \| \alpha^\beta \nabla_v f_0 \|_p
 \\ & + \Bigg [ \int_0 ^t \bigg | \frac{ e^{ -\varpi \langle v \rangle s  } \alpha ^\beta } { n(x) \cdot v } \Big \{ \partial_t g + \sum ( v \cdot \tau_i ) \partial_{\tau_i } g + \nu g - H + E\cdot \nabla_v g \Big \} 
 \\ & \qquad \qquad \qquad \qquad  +  \frac{ e^{-\varpi   \langle v \rangle  s } \alpha ^ \beta n \cdot \iint \partial_x E  }{ n(x) \cdot v } \Big \{ \partial_t g + \sum ( v \cdot \tau_i ) \partial_ {\tau_i } g - \nu g + H \Big \} \bigg | _{\gamma , p } ^ p  d\tb \Bigg ]^{1/p }
 \\ & + \Bigg [ \int_0 ^t \sum_{ i =1}^2 \Big | e^{-\varpi \langle v \rangle s  } \alpha ^\beta \partial_{\tau_i } g (s) \Big | _{\gamma, p } ^p + \Big | e^{-\varpi \langle v \rangle s } \alpha ^ \beta \nabla_v g (s)\Big | _{ \gamma, p } ^p 
 \\ & \qquad \qquad \qquad \qquad + \| e^{ -\varpi \langle v \rangle s } \alpha ^ \beta \nabla_x H(s) \|_p ^ p + \| e^{ -\varpi \langle v \rangle s  } \alpha ^ \beta \nabla_v H(s) \|_p ^p ds \Bigg ] ^{ 1/ p}
 \\ & + C' \Bigg [ \int_0 ^t \| e^{ -\theta |v| ^ 2} e^{ -\varpi \langle v \rangle s  } \alpha ^\beta \nabla_x v \|_p ^ p +  \| e^{ -\theta |v| ^ 2} e^{ -\varpi \langle v \rangle s } \alpha ^\beta \nabla_v v \|_p ^ p ds \Bigg ] ^{1/p } \Bigg),
 \end{split} \]
 
  \[ \begin{split}
 \| & e^{ - \varpi \langle v \rangle t } \alpha ^\beta \nabla_v f(t) \|_{ L^p  } \lesssim_{\beta }e^{\varpi t^2(\| E \|_\infty^2+\| \nabla E \|_\infty^2 +1)} \Bigg( \| \alpha ^\beta \nabla_x f_0 \|_p + \| \alpha^\beta \nabla_v f_0 \|_p
 \\ & + \Bigg [ \int_0 ^t \bigg | \frac{ e^{  - \varpi \langle v \rangle s  } \alpha ^\beta } { n(x) \cdot v } \Big \{ \partial_t g + \sum ( v \cdot \tau_i ) \partial_{\tau_i } g + \nu g - H + E\cdot \nabla_v g \Big \} 
 \\ & \qquad \qquad \qquad \qquad  +  \frac{ e^{ - \varpi \langle v \rangle s  } \alpha ^ \beta n(x) \cdot \iint \partial_v E  }{ n(x) \cdot v } \Big \{ \partial_t g + \sum ( v \cdot \tau_i ) \partial_ {\tau_i } g - \nu g + H \Big \} \bigg | _{\gamma , p } ^ p  ds \Bigg ]^{1/p }
 \\ & + \Bigg [ \int_0 ^t \sum_{ i =1}^2 \Big | e^{ - \varpi \langle v \rangle s } \alpha ^\beta \partial_{\tau_i } g (s) \Big | _{\gamma, p } ^p + \Big | e^{ - \varpi \langle v \rangle s  } \alpha ^ \beta \nabla_v g (s)\Big | _{ \gamma, p } ^p 
 \\ & \qquad \qquad \qquad \qquad + \| e^{  - \varpi \langle v \rangle s} \alpha ^ \beta \nabla_x H(s) \|_p ^ p + \| e^{  - \varpi \langle v \rangle s } \alpha ^ \beta \nabla_v H(s) \|_p ^p ds \Bigg ] ^{ 1/ p}
 \\ & + C' \Bigg [ \int_0 ^ t \| e^{ -\theta |v| ^ 2} e^{  - \varpi \langle v \rangle s } \alpha ^\beta \nabla_x v \|_p ^ p +  \| e^{ -\theta |v| ^ 2} e^{  - \varpi \langle v \rangle s  } \alpha ^\beta \nabla_v v \|_p ^ p ds \Bigg ] ^{1/p } \Bigg),
 \end{split} \]
 where $C' = \| e^{\theta |v|^2 } f_0 \|_{\infty} + \| e^{\theta |v|^2 } H \|_{\infty} + | e^{\theta |v|^2 } g |_{\infty} .$
 By the hypotheses of the proposition, the right hand sides are bounded and hence $e^{ - \varpi \langle v \rangle t} \alpha^\beta \partial f \in L^\infty ( [ 0, T]; L^p ( \Omega \times \mathbb R^3 ) )$.
 
 Since $f_0, g$ and $H$ are compactly supported inside $\{ v \in \mathbb R^3: |v| \le m \}$ we have by direct computation that if we let
 \[
\bar \nu : = \nu  + \varpi \langle v \rangle + \varpi \frac{v}{\langle v \rangle} \cdot E t     -  \beta \alpha^{-1} ( \partial_t \alpha + v \cdot \nabla_x \alpha + E \cdot \nabla_v \alpha ),
\]
then 
 \[ \begin{split}
 \{ \partial_t + & v \cdot \nabla_x + E \cdot \nabla_v + \overline \nu \} (e^{-\varpi  \langle v \rangle t} \alpha^\beta \partial f ) 
 \\ &= e^{- \varpi \langle v \rangle t } \alpha^\beta   [\partial_t + v \cdot \nabla_x + E \cdot \nabla_v + \nu ] (\partial f )
 \\& = e^{- \varpi \langle v \rangle t } \alpha^\beta  [ \partial H - \partial \nu \cdot \nabla_x f - \partial E \cdot \nabla_v f - \partial \nu f ] \in L^p.
 \end{split} \]
Therefore by the trace theorem the traces of $e^{- \varpi \langle v \rangle t } \alpha^\beta \partial f $ exist and by choosing a test function multiplied by $e^{- \varpi \langle v \rangle t } \alpha^\beta$, we deduce $e^{- \varpi \langle v \rangle t } \alpha^\beta \partial f$ has the same trace as $e^{- \varpi \langle v \rangle t } \alpha^\beta [\partial f |_\gamma]$.
 
 Finally we use (\ref{cutofffunctionvspace}) to remove the compact support condition and pass to the limit to conclude the proof.

 \end{proof}
 
 \section{$W^{1,p}$ estimate}
 The goal of this section is to prove the $W^{1,p}$ $(1<p<2)$ estimate, and the weighted $W^{1,p}$ $(2\le p < \infty)$ estimate for the system (\ref{Bextfield1}),
 (\ref{diffusebdycondtionf}), with $E$ satisying (\ref{signEonbdry}).

 
 Let $f^0 = \sqrt \mu $. We apply Proposition \ref{inflowexistence1} for $m =0,1,2,...$ to get
\begin{equation} \label{seqfnoselfgeneratedpotential}
(\partial_t + v \cdot \nabla_x + E \cdot \nabla_v -\frac{v}{2} \cdot E + \nu(\sqrt \mu f^m) ) f^{m+1} = \Gamma_{\text{gain}} (f^m,f^m),
\end{equation}
with the initial data $f^m(0,x,v) = f_0(x,v)$, and boundary condition for all $(x,v) \in \gamma_-$ be
\Be \label{diffusebdryconditionsequencenoselfgeneratedpotential}
\begin{split}
f^1(t,x,v) & =  c_\mu \sqrt{\mu(v)} \int_{n\cdot u >0 } f_0(x,u) \sqrt{\mu(u)} (n(x) \cdot u ) du,
\\ f^{m+1}(t,x,v) & = c_\mu \sqrt{\mu(v)} \int_{n\cdot u >0 } f^m(t,x,u) \sqrt{\mu(u)} (n(x) \cdot u ) du, \, m \ge 1.
\end{split}
\Ee
We first need a local existence result which is standard. 
%
%

\begin{lemma} \label{localexslemma}[Local Existence]
Suppose $\| E \|_\infty < \infty$, and $\| e^{\theta |v|^2} f_0 \|_\infty < \infty$, $0< \theta < \frac{1}{4}$. And $f_0$ satisfy the compatibility condition for diffuse boundary condition. Then there exists $0 < T \ll 1$ small enough such that $f \in L^\infty ([0,T) \times \Omega \times \mathbb R^3)$ solves the system (\ref{Bextfield1}) with diffuse boundary condition (\ref{diffusebdycondtionf}).
%
\end{lemma}

\begin{proof}
We first claim:
\Be \label{uniformlinfinitybddnoselfgeneratedpotential}
\sup_m \sup_{0 \le t \le T} \| e^{\theta' |v|^2 } f^m(t) \|_\infty \lesssim \| e^{\theta |v|^2} f_0 \|_\infty < \infty,
\Ee
where $\theta' = \theta - T$. The proof of (\ref{uniformlinfinitybddnoselfgeneratedpotential}) is essentially the same (and easier) to the proof of the same bound in the case with self-generated potential. See the proof of (\ref{uniformmlinftybdd}).

From (\ref{uniformlinfinitybddnoselfgeneratedpotential}) we have up to a subsequence we have the weak-$\ast$ convergence:
\Be \label{weakstarcovlinear}
 e^{\theta' |v|^2} f^m(t,x,v) \overset{\ast}{\rightharpoonup} e^{\theta' |v|^2} f(t,x,v)
 \Ee
 in $L^\infty ([0,T) \times \Omega \times \mathbb R^3) \cap L^\infty ([0,T) \times \gamma) $ for some $f$.

Apply the same argument of (\ref{uniformlinfinitybddnoselfgeneratedpotential}) to the sequence $
 e^{(\theta - t ) |v|^2 } (f^{m+1 } - f^m ) $ we get that the sequence $ e^{\theta ' |v|^2 } f^m (t,x,v)  \in L^\infty ([0,T) \times \Omega \times \mathbb R^3) \cap L^\infty ([0,T) \times \gamma) $ is a Cauchy sequence and therefore
\Be \label{stronglinftycovlinear}
\| e^{\theta ' |v|^2 } f^m (t,x,v) - e^{\theta ' |v|^2 } f(t,x,v) \|_\infty \to 0, \, \text{as } \, m \to \infty.
\Ee
Now for any $\phi \in C_c^\infty ( [0,T) \times \Omega \times \mathbb R^3) $ we have from (\ref{seqfnoselfgeneratedpotential}) that
\Be \label{phifmeqvlinear}
\int_0^T \iint_{\Omega \times \mathbb R^3 } f^{m+1} \left[ \partial_t + v \cdot \nabla_x + E\cdot \nabla_v - \frac{v}{2} \cdot E  +\nu(\sqrt{\mu} f^{m} ) \right] \phi = \int_0^T \iint_{\Omega \times \mathbb R^3 } - \Gamma_{\text{gain}} (f^m,f^m) \phi.
\Ee
Then from (\ref{weakstarcovlinear}) and (\ref{stronglinftycovlinear}), by the standard argument we can pass the limit $m \to \infty$ in (\ref{phifmeqvlinear}) to conclude that
\[
\int_0^T \iint_{\Omega \times \mathbb R^3 } f \left[ \partial_t + v \cdot \nabla_x + E\cdot \nabla_v - \frac{v}{2} \cdot E   +\nu(\sqrt{\mu} f ) \right] \phi = \int_0^T \iint_{\Omega \times \mathbb R^3 }  - \Gamma_{\text{gain}} (f,f) \phi.
\]
This proves the lemma.

\end{proof}

Now we are ready to prove Theorem \ref{W1ppless2thm}

\begin{proof} [proof of Theorem \ref{W1ppless2thm}]

Let $\partial \in \{ \nabla_x, \nabla_v \}$. Taking $ \partial [ (\ref{seqfnoselfgeneratedpotential}) ] $ we have

\begin{equation} \begin{split} \label{seqf3}
& (\partial_t + v \cdot \nabla_x + E \cdot \nabla_v - \frac{v}{2} \cdot E + \nu(\sqrt \mu f^m) ) \partial f^{m+1} 
\\ = & \partial \Gamma_{gain} (f^m,f^m) - \partial v \cdot \nabla_x f^{m+1} - \partial E \cdot \nabla_v f^{m+1} - \partial ( \frac{v}{2} \cdot E ) f^{m+1} - \partial (\nu (\sqrt \mu f^m ) ) f^{m+1}  \\ : = & \mathcal G^m.
\end{split} \end{equation}
By direct computation we have from (\ref{seqf3}):
\begin{equation} \begin{split} \label{seqfw1pp<2}
 (\partial_t + v \cdot \nabla_x + E \cdot \nabla_v - \frac{v}{2} \cdot E + \varpi \langle v \rangle + t \varpi \frac{v}{\langle v \rangle } \cdot E + \nu(\sqrt \mu f^m) ) e^{-\varpi \langle v \rangle t } \partial f^{m+1} = e^{-\varpi \langle v \rangle t } \mathcal G^m.
\end{split} \end{equation}
And for $\varpi > 4(\| E \|_\infty + 1) $ and $T < \frac{1}{4(\| E\|_\infty + 1 ) }$, we have
\[
\nu^m_{\varpi} := \frac{v}{2} \cdot E + \varpi \langle v \rangle + t \varpi \frac{v}{\langle v \rangle }\cdot E  + \nu(\sqrt \mu f^m) \ge \frac{\varpi}{2} \langle v \rangle.
\]
%
From (\ref{uniformlinfinitybddnoselfgeneratedpotential}) we have
\[
| \mathcal G ^m | \lesssim |\partial f^{m+1} | + e^{-\frac{\theta}{2} |v|^2 } \| e^{\theta |v|^2 } f_0 \|_{\infty}^2 + P(\| e^{\theta |v|^2} f_0 \|_\infty) \times \int_{\mathbb R^3 } \frac{e^{-C_\theta |v-u|^2 }}{|v -u |^{2 - \kappa}  } | \partial f^m (u) | du,
\]
where $P$ is a polynomial.

We need some estimates for the derivatives on the boudnary. We claim that for $(x,v) \in \gamma_-$:
\begin{equation} \label{seqbdrynoalpha}
|\partial f^{m+1} (t,x,v) | \lesssim \sqrt {\mu (v) } \langle v \rangle ( 1 + \frac{1}{ |n(x) \cdot v | }) \int_{n(x) \cdot u > 0} | \partial f^m(t,x,u) |\mu ^{1/4}  (n(x) \cdot u ) du + \frac{ e^{-\frac{\theta}{2} } |v| ^2 }{|n(x) \cdot v | } P(\| e^{\theta |v|^2 } f_0 \|_\infty).
\end{equation}
Let $\tau _{1}(x)$ and $\tau _{2}(x)$ be unit tangential vectors to $\partial\Omega$ satisfying (\ref{tangentialderivativedef}), then from (\ref{seqfnoselfgeneratedpotential}),
		\Be\label{fn}
		\begin{split}
			&\partial _{n}f^{m+1}(t,x,v)\\
			=&\frac{-1}{n(x)\cdot v}\bigg\{ \partial
			_{t}f^{m+1}+\sum_{i=1}^{2}(v\cdot \tau _{i})\partial _{\tau _{i}}f^{m+1} 
			+ E \cdot \nabla_v f^{m+1}
		      -  \frac{v}{2} \cdot E f^{m+1} + \nu(\sqrt{\mu} f^m ) f^{m+1}
			-\Gamma_{\text{gain}}
			(f^m,f^m)  . \end{split} \Ee
Define the orthonormal transformation from $\{n,\tau
		_{1},\tau _{2}\}$ to the standard bases $\{\mathbf{e}_{1},\mathbf{e}_{2},%
		\mathbf{e}_{3}\}$, i.e. $\mathcal{T}(x)n(x)=\mathbf{e}_{1},\ \mathcal{T}%
		(x)\tau _{1}(x)=\mathbf{e}_{2},\ \mathcal{T}(x)\tau _{2}(x)=\mathbf{e}_{3},$
		and $\mathcal{T}^{-1}=\mathcal{T}^{T}.$ Upon a change of variable: $%
		u^{  \prime }=\mathcal{T}(x)u,$ we have%
		\begin{equation*}
		n(x)\cdot u=n(x)\cdot \mathcal{T}^{t}(x)u^{\prime }=n(x)^{t}%
		\mathcal{T}^{t}(x)u^{ \prime }=[\mathcal{T}(x)n(x)]^{t}u^{
			\prime }=\mathbf{e}_{1}\cdot u^{  \prime }=u_{1}^{  \prime },
		\end{equation*}%
		then the RHS of the diffuse BC (\ref{diffusebdryconditionsequencenoselfgeneratedpotential}) equals
		\begin{equation*}
		c_{\mu }%
		\sqrt{\mu (v)}\int_{u_{1}^{  \prime }>0}f^m(t,x,\mathcal{T}%
		^{t}(x)u^{  \prime })\sqrt{\mu (u^{  \prime })}\{u_{1}^{
			\prime }\}\mathrm{d}u^{  \prime }.
		\end{equation*}%
Then we can further take tangential derivatives $\partial _{\tau _{i}}$
		as, for $(x,v)\in \gamma _{-},$
		\begin{equation}\label{boundary_tau}
		\begin{split}
		\partial _{\tau _{i}}f^{m+1}(t,x,v)
				 = & c_{\mu }\sqrt{\mu (v)}\int_{n(x)\cdot u >0}\partial _{\tau
			_{i}}f^m(t,x,u)\sqrt{\mu (u)}\{n(x)\cdot u\}%
		\mathrm{d}u\\
		& \ \ +c_{\mu }\sqrt{\mu (v)}\int_{n(x)\cdot u>0}\nabla
		_{v}f^m(t,x,u)\frac{\partial \mathcal{T}^{t}(x)}{\partial \tau _{i}}%
		\mathcal{T}(x)u\sqrt{\mu (u)}\{n(x)\cdot u\}%
		\mathrm{d}u.
		\end{split}
		\end{equation}%
		We can take velocity derivatives directly to (\ref{diffusebdryconditionsequencenoselfgeneratedpotential}) and obtain that for $%
		(x,v)\in \gamma _{-},$
		\begin{eqnarray}
		\nabla _{v}f^{m+1}(t,x,v) &=&c_{\mu }\nabla _{v}\sqrt{\mu (v)}\int_{n(x)\cdot
			u>0}f^m(t,x,u)\sqrt{\mu (u)}\{n({x})\cdot
		u\}\mathrm{d}u,\label{boundary_v}\\
		\partial _{t}f^{m+1}(t,x,v) &=&c_{\mu }\sqrt{\mu (v)}\int_{n(x)\cdot u>0}\partial _{t}f^m(t,x,u)\sqrt{\mu (u)}\{n({x})\cdot
		u\}\mathrm{d}u. \notag
		\end{eqnarray}%
		For the temporal derivative, we use (\ref{seqfnoselfgeneratedpotential}) again to deduce that 
		\Be\label{boundary_t} 
		\begin{split}
			\partial _{t}f^{m+1}(t,x,v)
			= & c_{\mu }\sqrt{\mu (v)}\int_{n(x)\cdot u>0}
			\Big\{
			- u\cdot \nabla_x f^m - E  \cdot \nabla_v f^{m} + \frac{u}{2} \cdot E f^m - \nu(\sqrt{\mu} f^{m-1} ) f^m		\\ & +  \Gamma_{\text{gain}}(f^{m-1},f^{m-1}) 
			\Big\}
			\sqrt{\mu (u)}\{n({x})\cdot
			u\}\mathrm{d}u.
		\end{split}
		\Ee 
		From (\ref{fn})-(\ref{boundary_t}), and (\ref{uniformlinfinitybddnoselfgeneratedpotential}), we conclude (\ref{seqbdrynoalpha}).

Now we claim that for $1 \le p <2 $ and for $T_*$ small enough we have the uniformly-in-$m$ bound:
\begin{equation} \label{uniinmp<2bd}
\sup_{0 \le t \le T_*} \| e^{-\varpi \langle v \rangle t } \partial f^m \|_p^p + \int_0^{T_*} |e^{-\varpi \langle v \rangle t } \partial f^m |_{\gamma,p}^p \lesssim_{\Omega, T_* } \| \partial f_0 \|_p^p + P ( \| e^{\theta |v|^2} f_0 \|_\infty ).
\end{equation}
We remark that the sequence (\ref{seqfnoselfgeneratedpotential}) is shown to be a Cauchy sequence in $L^\infty$. Due to the weak lower semi-continuity for $L^p$ in case of $p>1$, once we have (\ref{uniinmp<2bd}), then we pass a limit $\partial f^m \rightharpoonup \partial f $ weakly in $\sup_{t \in [0, T_* ] } \| \cdot \|_p^p $ and $\partial f^m |_{\gamma} \rightharpoonup \partial f|_{\gamma} $ in $\int_0^{T_*} |\cdot |_{\gamma,p}^p$ (up to a subsequence) to conclude that $\partial f $ satisfies the same estimate of (\ref{uniinmp<2bd}). Repeat the same procedure for $[T_*, 2T_*], [2T_*, 3T_*],...,$ to conclude the theorem.

Applying the Green's identity (Lemma \ref{green'sidentity}) to (\ref{seqfw1pp<2}) we have:

\begin{equation} 
\begin{split}
& \| e^{-\varpi \langle v \rangle t } \partial f^{m+1}(t) \|_p^p + p \int_0^t  |  e^{-\varpi \langle v \rangle s }  \partial f^{m+1} |_{\gamma_+,p}^p 
\\ \lesssim &  \|   \partial f_0 \|_p^p + p\int_0^t  | e^{-\varpi \langle v \rangle s } \partial f ^{m+1} | _{\gamma_-, p }^p + p\int_0^t \iint_{\Omega \times \mathbb R^3 } |\mathcal G^m | e^{-p\varpi \langle v \rangle t }|  \partial f^{m+1} | ^{p-1} 
\\  \lesssim & \| \partial f_0 \|_p^p + \int_0^t | e^{-\varpi \langle v \rangle s } \partial f ^{m+1} | _{\gamma_-, p }^p 
\\ & +P(\| e^{\theta |v|^2 } f_0 \|_\infty) \int_0^t \int_\Omega \int_{\mathbb R^3 }  e^{-p\varpi \langle v \rangle s } |\partial f^{m+1}(v) | ^{p-1} (\int_{\mathbb R^3 } \frac{e^{-C_\theta |v-u|^2 }}{|v -u |^{2 - \kappa}  } | \partial f^m (u) | du )dv dx ds.
\end{split} \end{equation}
By Holder's inequality we have
\[ \begin{split}
\int_{\mathbb R^3 } (\frac{e^{-C_\theta |v-u|^2 }}{|v -u |^{2 - \kappa}  } )^{1/p + 1/q }| \partial f^m (u) | du & \le (\int_{\mathbb R^3 } (\frac{e^{-C_\theta |v-u|^2 }}{|v -u |^{2 - \kappa}  } )| \partial f^m (u) |^p du)^{1/p} (\int_{\mathbb R^3 } (\frac{e^{-C_\theta |v-u|^2 }}{|v -u |^{2 - \kappa}  } ) du)^{1/q}
\\ & \lesssim (\int_{\mathbb R^3 } (\frac{e^{-C_\theta |v-u|^2 }}{|v -u |^{2 - \kappa}  } )| \partial f^m (u) |^p du)^{1/p}.
\end{split} \]
And since $\frac{e^{-\varpi \langle v \rangle s }}{e^{-\varpi \langle u \rangle s }} = e^{s \varpi (\langle u \rangle - \langle v \rangle )} \le e^{2 \varpi s \langle u - v \rangle} $, we have
\[ \begin{split}
\int_0^t & \int_\Omega \int_{\mathbb R^3 }  (e^{-\varpi \langle v \rangle s })^p |\partial f^{m+1}(v) | ^{p-1} (\int_{\mathbb R^3 } \frac{e^{-C_\theta |v-u|^2 }}{|v -u |^{2 - \kappa}  } | \partial f^m (u) | du )dv dx ds
\\ & \lesssim \int_0^t \int_\Omega \int_{\mathbb R^3 } (e^{-\varpi \langle v \rangle s })^p |\partial f^{m+1}(v) | ^{p-1}  (\int_{\mathbb R^3 } (\frac{e^{-C_\theta |v-u|^2 }}{|v -u |^{2 - \kappa}  } )| \partial f^m (u) |^p du)^{1/p}dv dx ds
\\ & \lesssim \int_0^t \int_\Omega \int_{\mathbb R^3 } |e^{-\varpi \langle v \rangle s } \partial f^{m+1}(v) | ^p dvdxds + \int_0^t \int_\Omega \int_{\mathbb R^3 } \int_{\mathbb R^3 } (e^{-\varpi \langle v \rangle s })^p(\frac{e^{-C_\theta |v-u|^2 }}{|v -u |^{2 - \kappa}  } )| \partial f^m (u) |^p du dv dx ds
\\ & =  \int_0^t \int_\Omega \int_{\mathbb R^3 }  |e^{-\varpi \langle v \rangle s } \partial f^{m+1}(v) | ^p dvdxds + \int_0^t \int_\Omega \int_{\mathbb R^3 } (\int_{\mathbb R^3 } \frac{(e^{-\varpi \langle v \rangle s })^p}{ (e^{-\varpi \langle u \rangle s })^p} \frac{e^{-C_\theta |v-u|^2 }}{|v -u |^{2 - \kappa}  } ) dv) |e^{-\varpi \langle v \rangle s }  \partial f^m (u) |^p dudxds
\\ & \lesssim  \int_0^t \int_\Omega \int_{\mathbb R^3 }  |e^{-\varpi \langle v \rangle s } \partial f^{m+1}(v) | ^p dvdxds + \int_0^t \int_\Omega \int_{\mathbb R^3 } (\int_{\mathbb R^3 } \frac{e^{s\varpi \langle v - u \rangle-C_\theta |v-u|^2 }}{|v -u |^{2 - \kappa}  } ) dv) e^{-\varpi \langle u \rangle s } | \partial f^m (u) |^p dudxds
\\ & \lesssim \int_0^t \int_\Omega \int_{\mathbb R^3 } |e^{-\varpi \langle v \rangle s } \partial f^{m+1}(v) | ^p dvdxds + \int_0^t \int_\Omega \int_{\mathbb R^3 }   |e^{-\varpi \langle u \rangle s }  \partial f^m (u) |^p dudxds.
\end{split} \]
Thus
\begin{equation} \label{seqfgreen} \begin{split}
& \sup_{0 \le s \le t} \| e^{-\varpi \langle v \rangle s } \partial f^{m+1}(s) \|_p^p + \int_0^t | e^{-\varpi \langle v \rangle s }\partial f^{m+1} |_{\gamma_+,p}^p 
\\ & \lesssim
\| \partial f_0 \|_p^p + \int_0^t |e^{-\varpi \langle v \rangle s } \partial f ^{m+1} | _{\gamma_-, p }^p +P(\| e^{\theta |v|^2 } f_0 \|_\infty) \times \left( \int_0^t \| e^{-\varpi \langle v \rangle s }\partial f^{m+1 } (s) \|_p^p + \int_0^t \| e^{-\varpi \langle v \rangle s } \partial f^{m } (s) \|_p^p\right).
\end{split} \end{equation}
Now we consider the boundary contributions. We use (\ref{seqbdrynoalpha}) to obtain
\begin{equation} \label{gamma-intestimate}
\begin{split}
\int_0^t & \int_{\gamma_- } | e^{-\varpi \langle v \rangle s } \partial f^{m+1}(s) |^p
\\ \lesssim & \sup_{x \in \partial \Omega} \left( \int_{n(x) \cdot v < 0 }  (e^{-\varpi \langle v \rangle s })^p \sqrt{\mu(v)}^p \langle v \rangle ^p( |n \cdot v| + \frac{1}{|n\cdot v | ^{p-1}}   ) dv \right) 
\\ & \times \int_0^t \int_{\partial \Omega} \left[ \int_{n(x) \cdot u > 0 } | e^{-\varpi \langle u \rangle s } \partial f^m(s,x,u) | e^{\varpi \langle v \rangle s } \mu^{1/4} (u) (n\cdot u ) du \right]^p dS_x ds 
\\ & + \sup_{x \in \partial \Omega} \left( \int_{n(x) \cdot v < 0}  (e^{-\varpi \langle v \rangle s })^pe^{-\frac{p \theta}{2} |v|^2} |n(x) \cdot v | ^{1-p } dv  \right) \times t P( \| e^{\theta |v|^2 } f_0 \|_\infty )
\\ \lesssim  & \int_0^t \int_{\partial \Omega} \left[ \int_{n(x) \cdot u > 0 } | e^{-\varpi \langle u \rangle s } \partial f^m(s,x,u) | \mu^{1/8} (u) (n\cdot u ) du \right]^p dS_x ds  + t P( \| e^{\theta |v|^2 } f_0 \|_\infty ).
\end{split} \end{equation}
Now we focus on $\int_0^t \int_{\partial \Omega} \left[ \int_{n(x) \cdot u > 0 } | e^{-\varpi \langle u \rangle s } \partial f^m(s,x,u) | \mu^{1/8} (u) (n\cdot u ) du \right]^p dS_x ds$. 
Recall (\ref{defgammaepsilon}), we split the $\{ u \in \mathbb R^3 : n(x) \cdot u > 0 \}$ as
\begin{equation} \begin{split}
\int_0^t & \int_{\partial \Omega} \left[ \int_{n(x) \cdot u > 0 } |  e^{-\varpi \langle u \rangle s } \partial f^m(s,x,u) | \mu^{1/8} (u) (n\cdot u ) du \right]^p dS_x ds
\\ & \lesssim \int_0^t \int_{\Omega} \left[ \int_{(x,u) \in \gamma_+ \setminus \gamma_+^\epsilon} du \right]^p + \int_0^t \int_{\Omega} \left[ \int_{(x,u) \in \gamma_+^\epsilon} du \right]^p.
\end{split} \end{equation}
By Holder's inequality we have:
\[
 \left[ \int_{(x,u) \in \gamma_+^\epsilon} du \right]^p \le   \left[ \int_{(x,u) \in \gamma_+^\epsilon} \mu^{\frac{p}{8(p-1)}} (n \cdot u ) du \right]^{p-1} \left[ \int_{(x,u) \in \gamma_+^\epsilon} |e^{-\varpi \langle u \rangle s } \partial f^m(s,x,u) | ^p (n \cdot u ) du \right] ,
\]
And the term $  \left[ \int_{(x,u) \in \gamma_+^\epsilon} \mu^{\frac{p}{8(p-1)}} (n \cdot u ) du \right]^{p-1} <\epsilon' \ll 1 $ if $\epsilon$ is small enough.

For the first term (non-grazing part), note that from (\ref{seqfw1pp<2}) we have
\Be \label{tracethmlp}
(\partial_t + v \cdot \nabla_x + E \cdot \nabla_v +\nu^{m-1}_\varpi ) |e^{-\varpi \langle v \rangle t } \partial f^{m}|^p = p|e^{-\varpi \langle v \rangle t } \partial f^{m}|^{p-2}  e^{-\varpi \langle v \rangle t } \partial f^{m} e^{-\varpi \langle v \rangle t } \mathcal G^{m-1}.
\Ee
So we can apply (\ref{traceestimate}) to (\ref{tracethmlp}) to get
\[ \begin{split}
 \int_0^t & \int_{\Omega} \left[ \int_{(x,u) \in \gamma_+ \setminus \gamma_+^\epsilon} du \right]^p
 \\ & \lesssim \| \partial f_0 \|_p^p + \int_0^t \| e^{-\varpi \langle v \rangle s }\partial f^m(s) \|_p^p ds + \int_0^t \iint_{\Omega \times \mathbb R^3 } |\mathcal G^{m-1} |e^{-\varpi \langle v \rangle s } |^p\partial f^m | ^{p-1} 
 \\ & \lesssim \| \partial f_0 \|_p^p + \int_0^t \| \partial f^m(s) \|_p^p ds + P(\| e^{\theta |v|^2 } f_0 \|_\infty) \times \left( \int_0^t \| e^{-\varpi \langle v \rangle s } \partial f^m (s) \|_p^p + \int_0^t \| e^{-\varpi \langle v \rangle s } \partial f^{m-1} (s) \|_p^p \right).
\end{split} \]
Putting together all the estimates (\ref{seqfgreen}) becomes
\[ \begin{split}
 \sup_{0 \le s \le t} & \| e^{-\varpi \langle v \rangle s } \partial f^{m+1}(s) \|_p^p + \int_0^t | e^{-\varpi \langle v \rangle s }\partial f^{m+1} |_{\gamma_+,p}^p 
\\  \lesssim & \| \partial f_0 \|_p^p + \epsilon' \int_0^t |e^{-\varpi \langle v \rangle s } \partial f^m(s) |_{\gamma_+, p }^p ds 
\\ &+ P(\| e^{\theta |v|^2} f_0 \|_\infty ) \times \left(  \int_0^t \| e^{-\varpi \langle v \rangle s }  \partial  f^{m+1} (s) \|_p^p + 2 \int_0^t  \| e^{-\varpi \langle v \rangle s } \partial f^m (s) \|_p^p + \int_0^t \| e^{-\varpi \langle v \rangle s } \partial f^{m-1} (s) \|_p^p \right).
\\ & + tP(\| e^{\theta |v|^2} f_0 \|_\infty ) 
\end{split} \]
Choose $\epsilon \ll 1$ and $0 < T^* \ll 1 $ we have:
\[ \begin{split}
 \sup_{0 \le s \le T^*} & \| e^{-\varpi \langle v \rangle s } \partial f^{m+1}(s) \|_p^p + \int_0^{T^*} | e^{-\varpi \langle v \rangle s }\partial f^{m+1} |_{\gamma_+,p}^p 
\\  \lesssim & \| \partial f_0 \|_p^p  + P(\| e^{\theta |v|^2} f_0 \|_\infty ) 
\\ &+ \frac{1}{8} \max_{i = m,m-1} \left(  \sup_{0 \le t \le T^*} \| e^{-\varpi \langle v \rangle s }  \partial  f^i (s) \|_p^p + \int_0^{T^*} | e^{-\varpi \langle v \rangle s }\partial f^i | _{\gamma_+, p }^p \right).
\end{split} \]

To conclude the proof we use the following fact: Suppose $a_i \ge 0, D \ge 0$ and $A_i = \max \{a_i,a_{i-1},...,a_{i-(k-1)} \}$ for fixed $k \in \mathbb N$. If $a_{m+1} \le \frac{1}{8} A_m +D$, then 
\begin{equation} \label{combfact}
A_m \le \frac{1}{8} A_0 + (\frac{8}{7})^2D,
\end{equation}
 for $\frac{m}{k}  \gg 1$.

Setting $k = 2$ and $a_i = \sup_{0 \le t \le T^*} \| e^{-\varpi \langle v \rangle t} \partial f^i(t) \|_p^p + \int_0^{T^*} | e^{-\varpi \langle v \rangle t } \partial f^i|_{\gamma_+,p}^p$, $D = C \left( \| \partial f_0 \|_p^p + P(\| e^{\theta |v|^2} f_0 \|_\infty) \right)$, we complete the proof of the claim.

\end{proof}
Next, we prove Theorem \ref{weightedw1ppge2theorem}
%

\begin{proof}[Proof of Theorem \ref{weightedw1ppge2theorem}]

By (\ref{seqf3}) and direct compuation, we have
 \begin{equation} \label{pge2seqf} \begin{split}
 \bigg \{ & \partial_t + v\cdot \nabla_x  + E \cdot \nabla_v 
 \\ & + \nu(\sqrt \mu f^m )+ \frac{v}{2} \cdot E  + \varpi \langle v \rangle  + t \varpi \frac{v}{\langle v \rangle }\cdot E    - \beta \alpha^{-1} ( \partial_t \alpha + v \cdot \nabla_x \alpha + E \cdot \nabla_v \alpha ) \bigg\} ( e^{-\varpi \langle v \rangle t } \alpha^\beta \partial f^{m+1})
\\ = &  e^{-\varpi \langle v \rangle t }  \alpha^\beta ( \partial_t + v \cdot \nabla_x + E \cdot \nabla_v  + \frac{v}{2} \cdot E + \nu(\sqrt \mu f^m )) \partial f^{m+1}
=  e^{-\varpi \langle v \rangle t }  \alpha^\beta \mathcal G^m .
\end{split} \end{equation}
And since $\beta \alpha^{-1} ( \partial_t \alpha + v \cdot \nabla_x \alpha + E \cdot \nabla_v \alpha )  \lesssim \frac{(\| E \|_\infty + \| \nabla E \|_\infty )}{C_E}$. Thus if we choose $\varpi \gtrsim  \frac{(\| E \|_\infty + \| \nabla E \|_\infty )}{C_E}$ large enough and $T \le \frac{1}{4(\| E \|_\infty + 1 ) } $, we have
\[
\nu(\sqrt \mu f^m )+ \frac{v}{2} \cdot E  + \varpi \langle v \rangle  + t \varpi \frac{v}{\langle v \rangle }\cdot E    - \beta \alpha^{-1} ( \partial_t \alpha + v \cdot \nabla_x \alpha + E \cdot \nabla_v \alpha )  \ge \frac{\varpi }{2} \langle v \rangle.
\]
Now fix $p \ge 2$, $\frac{p-2}{p} < \beta < \frac{p-1}{p}$.
We claim that there exists $0 < T_* \ll 1 $ such that we have the following estimates uniformly-in-$m$,
\Be \label{uniformbddweightedw1ppge2} \begin{split}
\sup_{0 \le t \le T_*} \| e^{-\varpi \langle v \rangle t }  \alpha ^\beta \partial f^m(t) \|_p^p + \int_0^{T_*} | e^{-\varpi \langle v \rangle s } \alpha^\beta \partial f^m |_{\gamma,p}^p \lesssim_{\Omega, T_*} P(\| e^{\theta |v|^2} f_0 \|_\infty ) + \| \alpha^\beta \partial f_0 \|_p^p.
\end{split} \Ee
Once we have (\ref{uniformbddweightedw1ppge2}) then we pass to limit, $e^{-\varpi \langle v \rangle t }  \alpha ^\beta \partial f^m(t) \rightharpoonup e^{-\varpi \langle v \rangle t }  \alpha ^\beta \partial f(t)$ weakly with norms $\sup_{t \in[0,T^*] } \| \cdot \|_p^p$ and $e^{-\varpi \langle v \rangle t }  \alpha ^\beta \partial f^m|_\gamma \rightharpoonup e^{-\varpi \langle v \rangle t }  \alpha ^\beta \partial f|_\gamma$ in $\int_0^{T_*} | \cdot |_{\gamma,p}^p$ and $e^{-\varpi \langle v \rangle t }  \alpha ^\beta \partial f $ satisfies (\ref{uniformbddweightedw1ppge2}). Repeat the same procedure for $[T_*,2T_*], [2T_*,3T_*], ...,$ up to the local existence time interval $[0,T]$ in Lemma \ref{localexslemma} to conclude Theorem \ref{weightedw1ppge2theorem}.

We prove the claim by induction. Apply Proposition \ref{inflowexistence1} to (\ref{pge2seqf}), $\partial f^1$ exists. Because of our choice of $\partial f^0$, by Proposition \ref{inflowproposition2} the estimate in the claim hold for $m=1$. Now assume that $\partial f^i$ exists and the estimate is valid for all $i=1,2,...,m$. From (\ref{uniformlinfinitybddnoselfgeneratedpotential}) we have the bound 
\[ \begin{split}
& e^{-\varpi \langle v \rangle t  } \alpha^\beta |\mathcal G^m| \lesssim 
 e^{-\varpi \langle v \rangle t  } \alpha^\beta \left\{  | \nabla_x f^{m+1}|  + P(\| e^{\theta |v|^2} f_0 \|_\infty ) \left[ e^{-\frac{\theta}{2} |v|^2 } + \int_{\mathbb R^3}  \frac{e^{-C_\theta |v-u|^2 }}{|v -u |^{2 - \kappa}  } | \partial f^m (u) | du \right] \right\}.
\end{split} \]
Apply the Green's identity to (\ref{pge2seqf}) we have:
\begin{equation} \label{seqfpge2}
\begin{split}
& \| e^{-\varpi \langle v \rangle t } \alpha^\beta \partial f^{m+1}(t) \|_p^p + p \int_0^t  |  e^{-\varpi \langle v \rangle s } \alpha^\beta   \partial f^{m+1} |_{\gamma_+,p}^p 
\\ & + p \int_0^t  \| \langle v \rangle ^{1/p} e^{-\varpi \langle v \rangle s } \alpha^\beta  \partial f^{m+1} \|_p^p
\\ \lesssim &  \|  \alpha^\beta \partial f_0 \|_p^p + p\int_0^t  | e^{-\varpi \langle v \rangle s } \alpha^\beta  \partial f ^{m+1} | _{\gamma_-, p }^p + p\int_0^t \iint_{\Omega \times \mathbb R^3 } [e^{-\varpi \langle v \rangle s } \alpha^\beta ]^p  |\mathcal G^m | | \partial f^{m+1} | ^{p-1} 
\\  \lesssim &  \|  \alpha^\beta \partial f_0 \|_p^p + \int_0^t  | e^{-\varpi \langle v \rangle s } \alpha^\beta  \partial f ^{m+1} | _{\gamma_-, p }^p 
\\ & + t P(\| e^{\theta |v|^2} f_0 \|_\infty ) + t \sup_{0 \le s \le t} \|  e^{-\varpi \langle v \rangle s} \alpha^\beta   \partial f^{m+1}(s) \|_p^p 
\\ & +P(\| e^{\theta |v|^2} f_0 \|_\infty )   \int_0^t \iint_{\Omega \times \mathbb R^3 } [e^{-\varpi \langle v \rangle s } \alpha^\beta ]^p | \partial f^{m+1} | ^{p-1}  \times \int_{\mathbb R^3}  \frac{e^{-C_\theta |v-u|^2 }}{|v -u |^{2 - \kappa}  } | \partial f^m (u) | du.
\end{split} \end{equation}
\textit{Step 1. Estimate for the nonlocal term:} The key estimate is the following: For $0 < \beta < \frac{p-1}{p}$, $0 < \theta < \frac{1}{4}$, and some $C_{\varpi, \beta, p} >0$,
\begin{equation} \label{uv}
\sup_{x \in \Omega} \int_{\mathbb R^3} \frac{e^{C_\theta |v -u |^2}}{|v - u |^{2 - \kappa}} \frac{[e^{-\frac{\varpi }{\beta} \langle v \rangle s} \alpha(s,x,v) ] ^{\frac{\beta p}{p-1}}}{[e^{-\frac{\varpi }{\beta} \langle u \rangle s } \alpha(s,x,u) ] ^{\frac{\beta p}{p-1}}} du \lesssim_{\Omega, \theta} 
 e^{C_{\varpi, \beta, p}  s^2}.
\end{equation} 


Recall the definition of $\alpha$ in (\ref{alphadef}), we only have to show the claim for $x \in \Omega^\delta$ as $\alpha$ is constant for $x \in \Omega \setminus \Omega^\delta$.
We decompose $u_n = u \cdot n(x) $ and $u_\tau = u - u_n n(x) $. Note that
\begin{equation}\label{expcombine}
\frac{[e^{-\frac{\varpi }{\beta} \langle v \rangle s }  ] ^{\frac{\beta p}{p-1}}}{[e^{-\frac{\varpi }{\beta} \langle u \rangle s } ] ^{\frac{\beta p}{p-1}}}  \lesssim e^{C_{\varpi } s^2} \times e^{ \frac{ C_\theta |v - u | ^2}{2}},
\end{equation}
for some $C_{\varpi}  > 0 $.
And since $\alpha \le C$ is bounded, 
therefore for $0 \le \kappa \le 1$, we have the bound
\[ \begin{split}
\sup_{x \in \Omega} & \int_{\mathbb R^3} \frac{e^{C_\theta |v -u |^2}}{|v - u |^{2 - \kappa}} \frac{[e^{-\frac{\varpi }{\beta} \langle v \rangle s } \alpha(s,x,v) ] ^{\frac{\beta p}{p-1}}}{[e^{-\frac{\varpi }{\beta} \langle u \rangle s } \alpha(s,x,u) ] ^{\frac{\beta p}{p-1}}} du
\\ & \lesssim e^{C_{\varpi, p, E } s^2} \int_{\mathbb R^3 } | v -u | ^{-2 + \kappa} e^{- C_\theta |v - u | ^2 } e^{ \frac{ C_\theta |v - u | ^2}{2}} \frac{1}{ | u \cdot \nabla \xi(x) |^{\frac{ \beta p }{p-1}} } du 
\\ & \lesssim  e^{C_{\varpi, p, E } s^2} \int_{\mathbb R^3 } | v -u | ^{-2 + \kappa} e^{ -\frac{ C_\theta |v - u | ^2}{2}} |u_n|^{\frac{ - \beta p }{p-1}}  du 
\\ & = e^{C_{\varpi, p, E } s^2}\int_{\mathbb R^2} d u_{\tau}  \int_{ \mathbb R }  | v -u | ^{-2 + \kappa} e^{ -\frac{ C_\theta |v - u | ^2}{2}} |u_n|^{\frac{ - \beta p }{p-1}}  du_n.
\end{split} \]
Now if $0 < \kappa \le 1$, we have
\[ \begin{split}
\int_{\mathbb R^2} & d u_{\tau}  \int_{ \mathbb R }  | v -u | ^{-2 + \kappa} e^{ -\frac{ C_\theta |v - u | ^2}{2}} |u_n|^{\frac{ - \beta p }{p-1}}  du_n
\\ \le & \int_{\mathbb R^2} |v_\tau - u_\tau|^{-2 + \kappa} e^{ -\frac{ C_\theta |v_\tau - u_\tau | ^2}{2}} du_\tau \int_{\mathbb R} e^{ -\frac{ C_\theta |v_n - u_n | ^2}{2}} |u_n|^{\frac{ - \beta p }{p-1}}  du_n \lesssim 1,
\end{split} \]
since we can split the last integration as $\int_{\mathbb R} e^{ -\frac{ C_\theta |v_n - u_n | ^2}{2}} |u_n|^{\frac{ - \beta p }{p-1}}  du_n = \int_{ |u_n| \le |v_n - u_n| } + \int_{ |u_n| > |v_n - u_n|}$ and both terms can be bounded together by:
\[
\int_{\mathbb R} \left( e^{ -\frac{ C_\theta | u_n | ^2}{2}} |u_n|^{\frac{ - \beta p }{p-1}} + e^{ -\frac{ C_\theta | u_n | ^2}{2}} |v_n - u_n|^{\frac{ - \beta p }{p-1}} \right) du_n.
\]

If $\kappa = 0$, first let $u' = v - u$ then using the cylindrical coordinate $u'_\tau = (r,\theta), u'_n = z$ we can compute the integration:
\[ \begin{split}
\int_{\mathbb R^2} & d u_{\tau}  \int_{ \mathbb R }  | v -u | ^{-2 } e^{ -\frac{ C_\theta |v - u | ^2}{2}} |u_n|^{\frac{ - \beta p }{p-1}}  du_n
\\ & = \int_{\mathbb R^2}  d u'_{\tau}  \int_{ \mathbb R }  | u' | ^{-2 } e^{ -\frac{ C_\theta | u' | ^2}{2}} |u'_n - v_n |^{\frac{ - \beta p }{p-1}}  du'_n
\\ & = \int_{-\infty}^\infty \int_0^\infty \frac{r}{r^2 + z^2} e^{-\frac{C_\theta (r^2 + z^2)}{2}} |z - c |^a drdz,
\end{split} \]
where we let $a = \frac{-\beta p }{p-1} > -1 $ and $c = v_n$. WLOG we assume $c \ge 0$.

Separating the integration into regions $D =\{(r,z) \in \mathbb R^2: 0\le r<1 , |z| <1 \}$ and $\mathbb R^2 \setminus D$ we have:
\[ \begin{split}
 \int_{-\infty}^\infty & \int_0^\infty \frac{r}{r^2 + z^2} e^{-\frac{C_\theta (r^2 + z^2)}{2}} |z - c |^a drdz 
 \\ = & \iint_M  \frac{r}{r^2 + z^2} e^{-\frac{C_\theta (r^2 + z^2)}{2}} |z - c |^a drdz + \iint_{\mathbb R^2 \setminus D}  \frac{r}{r^2 + z^2} e^{-\frac{C_\theta (r^2 + z^2)}{2}} |z - c |^a drdz
 \\ \le & \int_{-1}^1 \int_0^1  \frac{r}{r^2 + z^2} |z - c |^a drdz +  \int_{-\infty}^\infty  \int_0^\infty r e^{-\frac{C_\theta (r^2 + z^2)}{2}} |z - c |^a drdz
 \\  = & \frac{1}{2} \int_{-1}^1  \log(\frac{1}{z^2} +1 ) |z - c|^a dz + \frac{1}{C_\theta} \int_{-\infty}^\infty e^{-\frac{C_\theta z^2}{2}} |z - c |^a  dz.
\end{split} \]

For the second integration we can split as $\int_{|z - c | < |z|} + \int_{|z - c | \ge |z|}$, then both terms can be bounded by
\[
 \int_{-\infty}^\infty \left( e^{-\frac{C_\theta |z-c| ^2}{2}} |z - c |^a + e^{-\frac{C_\theta z^2}{2}} |z  |^a \right) dz  \lesssim 1.
 \]
For the first integration, since $\log(z^2 + z ^4) < 1$ for $|z| <1$, we have $\log(\frac{1}{z^2} +1 ) < 2 \log(\frac{1}{z^2} ) +1$. So we only have to show
\[
\int_{-1}^1 2\log(\frac{1}{z^2} ) |z - c |^a dz = -4 \int_{-1}^1 \log(|z|)|z - c |^a dz \lesssim 1.
\]
Split the integral into $\int_{|z - c | < |z|} + \int_{|z - c | \ge |z|}$, since we assume $c \ge 0$, we have
\[
- \int_{-1}^1 \log(|z|)|z - c |^a dz \le - 2\int_{0}^1 \log(z)z ^a dz + \int_0^1 |\log(|z-c|)||z-c|^a dz.
\]
Finally since $\int_{0}^1 \log(z)z ^a dz = \frac{-1}{(a + 1) ^2 } $ for $a > -1$, and since $\log(z) z^a < M$ is bounded for $z > 1$. We therefore have for all $c \in \mathbb R$,
\[
- 2\int_{0}^1 \log(z)z ^a dz + \int_0^1 |\log(|z-c|)||z-c|^a dz \le 3 \frac{1}{(a + 1) ^2 }  + M,
\]
and this proves
\[
\int_{\mathbb R^2}  d u_{\tau}  \int_{ \mathbb R }  | v -u | ^{-2 } e^{ -\frac{ C_\theta |v - u | ^2}{2}} |u_n|^{\frac{ - \beta p }{p-1}}  du_n \lesssim 1,
\]
thus we conclude the claim.
%

Therefore
\[ \begin{split}
 & e^{-\varpi \langle v \rangle s }  \alpha^\beta \int_{\mathbb R^3}  \frac{e^{-C_\theta |v-u|^2 }}{|v -u |^{2 - \kappa}  } | \partial f^m (u) | du
\\ = & \int_{\mathbb R^3} \frac{e^{C_\theta |v -u |^2}}{|v - u |^{2 - \kappa}} \frac{[e^{-\varpi \langle v \rangle s } \alpha(s,x,v) ] ^\beta }{[e^{ -\varpi \langle u \rangle s } \alpha(s,x,u) ] ^\beta}[ e^{-\varpi  \langle u \rangle s } \alpha(s,x,u) ] ^\beta | \partial f^m(u)| du
\\ \lesssim & \left( \int_{\mathbb R^3} \frac{e^{C_\theta |v -u |^2}}{|v - u |^{2 - \kappa}} \frac{[e^{-\frac{\varpi }{\beta} \langle v \rangle s } \alpha(s,x,v) ] ^{\frac{\beta p}{p-1}}}{[e^{-\frac{\varpi }{\beta} \langle u \rangle s } \alpha(s,x,u) ]^{\frac{\beta p}{p-1}}} du \right)^{1/q} 
 \times \left( \int_{\mathbb R^3} \frac{e^{C_\theta |v -u |^2}}{|v - u |^{2 - \kappa}} |e^{-\varpi  \langle u \rangle s } \alpha(s,x,u)^\beta  \partial f^m(u)| ^p du \right)^{1/p}
\\  \lesssim & e^{Cs^2} \left( \int_{\mathbb R^3} \frac{e^{C_\theta |v -u |^2}}{|v - u |^{2 - \kappa}} |e^{-\varpi  \langle u \rangle s } \alpha(s,x,u)^\beta  \partial f^m(u)| ^p du \right)^{1/p}.
\end{split} \]
Finally we use the Young's inequality to bound the last term (nonlocal term) of (\ref{seqfpge2}) by
\begin{equation} \label{nonlocalpge2} \begin{split}
Cte^{Ct^2} & P(\| e^{\theta |v|^2} f_0 \|_\infty) \sup_{0 \le s \le t} \iint_{\Omega \times \mathbb R^3} |e^{-\varpi  \langle v \rangle s } \alpha^\beta  \partial f^{m} | ^p 
\\ & + \delta  P(\| e^{\theta |v|^2} f_0 \|_\infty)  \int_0^t \iint_{\Omega \times \mathbb R^3}  |e^{-\varpi  \langle v \rangle s } \alpha^\beta  \partial f^{m+1} | ^p.
\end{split} \end{equation}

\textit{Step 2. Boundary Estimate:}  At the boundary, by (\ref{seqbdrynoalpha}), the contribution of $\gamma_-$ is:
\begin{equation} \label{gamma-pge2} \begin{split}
\int_0^t & \int_{\gamma_- } |e^{ - \varpi \langle v \rangle s } \alpha^\beta \partial f^{m+1}(s) |^p
\\ \lesssim & \int_0^t  \int_{\gamma_-} [e^{- \varpi \langle v \rangle s } \alpha^\beta]^p \sqrt{\mu(v)}^p \langle v \rangle ^p( |n \cdot v| + \frac{1}{|n\cdot v | ^{p-1}}   ) dv  
\\ & \times \left[ \int_{n(x) \cdot u > 0 } | \partial f^m(s,x,u) | \mu^{1/4} (u) (n\cdot u ) du \right]^p dS_x ds 
\\ & + P(\|e^{\theta |v|^2 } f_0 \|_\infty) \int_0^t \int_{\gamma_-} \frac{  [e^{- \varpi \langle v \rangle s } \alpha^\beta] ^p e^{-\frac{\theta p}{2}|v|^2 }  }{|n(x) \cdot v |^p }  d\gamma ds.
\end{split} \end{equation}
Since $ \alpha(s,x,v) \le |\nabla \xi(x) \cdot v | $ for $x \in \partial \Omega$, the last term is bounded by:
\[
P(\|e^{\theta |v|^2 } f_0 \|_\infty) \int_0^t \int_{\partial \Omega} \int_{\mathbb R^3}   e^{-\frac{\theta p}{2} |v| ^2} |n(x) \cdot v |^{\beta p -p +1}   d\gamma ds \lesssim_{\Omega, p ,\xi} t P(\|e^{\theta |v|^2 } f_0 \|_\infty),
\]
as long as $\beta p -p +1 > -1$, i.e. $\beta > \frac{p-2}{p}$.

For the first term in (\ref{gamma-pge2}) we split as:
\[
\left[ \int_{n(x) \cdot u > 0 } \right]^p \lesssim_p \left[ \int_{(x,u) \in \gamma_+^\epsilon } \right]^p + \left[ \int_{(x,u) \in \gamma_+ \setminus \gamma_+^\epsilon } \right]^p.
\]
By Holder's inequality in $u$, the $\gamma_+^\epsilon$ contribution (grazing part) is bounded as: 
\begin{equation} \begin{split}
\int_0^t   \int_{\gamma_-} & [e^{- \varpi \langle v \rangle s } \alpha^\beta]^p \sqrt{\mu(v)}^p \langle v \rangle ^p( |n \cdot v| + \frac{1}{|n\cdot v | ^{p-1}}   ) dv  
\\ & \times \left[ \int_{(x,u) \in \gamma_+^\epsilon }e^{- \varpi \langle u \rangle s } \alpha^\beta (s,x,u) | \partial f^m(s,x,u) | \frac{\mu^{1/4} (u) (n\cdot u )}{e^{- \varpi \langle u \rangle s } \alpha^\beta (s,x,u)} du \right]^p dS_x ds
\\ \lesssim \int_0^t &  \int_{\gamma_-} [e^{-\varpi \langle v \rangle s } \alpha^\beta]^p \sqrt{\mu(v)}^p \langle v \rangle ^p( |n \cdot v| + \frac{1}{|n\cdot v | ^{p-1}}   ) dv  
\\ & \times \left[ \int_{(x,u) \in \gamma_+^\epsilon } [ e^{- \varpi \langle u \rangle s } \alpha^\beta (s,x,u)]^p | \partial f^m(s,x,u) | ^p  (n\cdot u ) du \right]
\\ & \times  \left[ \int_{(x,u) \in \gamma_+^\epsilon } [ e^{- \varpi \langle u \rangle s } \alpha^\beta (s,x,u)]^{-q} \mu^{q/4}  (n\cdot u ) du \right]^{p/q} dS_x ds.
\end{split} \end{equation}
Again, since $\alpha(t,x,v) \le |\nabla \xi(x) \cdot v | $ for $x \in \partial \Omega$, we have
\[ \begin{split}
\int_0^t   \int_{\gamma_-}  & [e^{- \varpi \langle v \rangle s} \alpha^\beta]^p \sqrt{\mu(v)}^p \langle v \rangle ^p( |n \cdot v| + \frac{1}{|n\cdot v | ^{p-1}}   ) dv
\\ \lesssim & \int_0^t   \int_{\gamma_-} \mu^{p/2} \langle v \rangle^p  (| n \cdot v | ^{ \beta p + 1 } + |n \cdot v | ^{\beta p - (p-1) }) dv < \infty,
\end{split} \]
if $\beta p - (p-1) < -1$, i.e. $\beta > \frac{p-2}{p}$.

Also, with $\frac{p-1}{p} = \frac{1}{q}$. If $1 - \beta q >0$, i.e. $\beta < \frac{1}{q} = \frac{p-1}{p}$, 
\[ \begin{split}
\int_{(x,u) \in \gamma_+^\epsilon } &  [ e^{- \varpi \langle u \rangle s } \alpha^\beta (s,x,u)]^{-q} \mu^{q/4}  (n\cdot u ) du
\\ \lesssim & \int_{\gamma_+^\epsilon} |n \cdot u |^{- \beta q + 1 } e^{q \left( - \frac{ |u|^2 }{8} + s \varpi \langle  u \rangle  \right)  } du
\\ \lesssim & \int_{n \cdot u < \epsilon } \epsilon^{- \beta q + 1 } e^{q \left( - \frac{ |u|^2 }{8} +  s \varpi \langle  u \rangle  \right)  } du + \int_{|u | > \frac{1}{\epsilon} } | u |^{- \beta q + 1 } e^{q \left( - \frac{ |u|^2 }{8} + s \varpi \langle  u \rangle  \right)  } du
\\ \lesssim &  C_{\Omega, p , s }  \epsilon ^{1 -\beta q },
\end{split} \]
when $\epsilon \ll 1$.

Thus we have the bound for the grazing part:
\begin{equation} \label{grazingpge2} \begin{split}
\int_0^t   \int_{\gamma_-} & [e^{-\varpi \langle v \rangle s } \alpha^\beta]^p \sqrt{\mu(v)}^p \langle v \rangle ^p( |n \cdot v| + \frac{1}{|n\cdot v | ^{p-1}}   ) dv  
\\ & \times \left[ \int_{(x,u) \in \gamma_+^\epsilon }e^{-\varpi \langle u \rangle s } \alpha^\beta (s,x,u) | \partial f^m(s,x,u) | \frac{\mu^{1/4} (u) (n\cdot u )}{e^{-\varpi \langle u \rangle s } \alpha^\beta (s,x,u)} du \right]^p dS_x ds
\\ \lesssim & C \epsilon^{1 - \beta q } \int_0^t | e^{-\varpi \langle v \rangle s } \alpha^\beta  \partial f^m | ^p  _{\gamma_+, p } ds.
\end{split} \end{equation}
Therefore the contribution for the grazing part could be absorbed by the left hand side of the inequality if $\epsilon$ is small enough.

On the other hand, for the non-grazing contribution $\gamma_+ \setminus \gamma_+^\epsilon$, by similar estimate we get:
\[ \begin{split}
\int_0^t   \int_{\gamma_-} & [e^{-\varpi \langle v \rangle s } \alpha^\beta]^p \sqrt{\mu(v)}^p \langle v \rangle ^p( |n \cdot v| + \frac{1}{|n\cdot v | ^{p-1}}   ) dv  
\\ & \times \left[ \int_{(x,u) \in \gamma_+ \setminus \gamma_+^\epsilon }e^{-\varpi \langle v \rangle s } \alpha^\beta (s,x,u) | \partial f^m(s,x,u) | \frac{\mu^{1/4} (u) (n\cdot u )}{e^{-\varpi \langle v \rangle s } \alpha^\beta (s,x,u)} du \right]^p dS_x ds
\\ \lesssim & C_{\Omega,p, s} \int_0^t \int_{\gamma_+ \setminus \gamma_+^\epsilon}  | e^{-\varpi \langle v \rangle s } \alpha^\beta  \partial f^m | ^p  d \gamma ds,
\end{split} \]
where we used
\[
\int_{\gamma_+ }  [ e^{-\varpi \langle v \rangle s } \alpha^\beta (s,x,u)]^{-q} \mu^{q/4}  (n\cdot u ) du < C_{\Omega, p ,s } < \infty.
\]
Now we can apply the trace theorem so that the non-grazing part is further bounded by
\begin{equation} \label{nongrazingpge2} \begin{split}
\int_0^t &  \int_{\gamma_+ \setminus \gamma_+^\epsilon}  | e^{-\varpi \langle v \rangle s } \alpha^\beta  \partial f^m | ^p  d \gamma ds 
\\ \lesssim_\epsilon & \| \alpha^\beta(0) \partial f_0 \|_p^p + \int_0^t \| e^{-\varpi \langle v \rangle s } \alpha^\beta \partial f^m \|_p^p
+ \int_0^t \iint_{\Omega \times \mathbb R^3} |\mathcal G^{m-1}| [ e^{-\varpi \langle v \rangle s } \alpha^\beta ]^p |\partial f^m|^{p-1}
\\ \lesssim &  \| \alpha^\beta(0) \partial f_0 \|_p^p + \int_0^t \| e^{-\varpi \langle v \rangle s } \alpha^\beta \partial f^m \|_p^p
\\ & + t P(\| e^{\theta |v|^2} f_0 \|_\infty ) + t \sup_{0 \le s \le t} \|  e^{-\varpi \langle v \rangle s } \alpha^\beta   \partial f^{m}(s) \|_p^p 
\\ & +  Cte^{Ct^2}  P(\| e^{\theta |v|^2} f_0 \|_\infty) \sup_{0 \le s \le t} \iint_{\Omega \times \mathbb R^3} |e^{-\varpi \langle v \rangle s } \alpha^\beta  \partial f^{m-1} | ^p 
\\ & + \delta  P(\| e^{\theta |v|^2} f_0 \|_\infty)  \int_0^t \iint_{\Omega \times \mathbb R^3} \langle v \rangle |e^{-\varpi \langle v \rangle s } \alpha^\beta  \partial f^{m} | ^p.
\end{split} \end{equation}

Finally, collecting all the terms (\ref{seqfpge2}), (\ref{nonlocalpge2}), (\ref{grazingpge2}), (\ref{nongrazingpge2}) we have:
\[ \begin{split} 
& \sup_{0 \le t \le T} \| e^{-\varpi \langle v \rangle s } \alpha^\beta \partial f^{m+1}(t) \|_p^p +  \int_0^T  |  e^{-\varpi \langle v \rangle s } \alpha^\beta   \partial f^{m+1} |_{\gamma_+,p}^p 
+  \int_0^T \| \langle v \rangle ^{1/p} e^{-\varpi \langle v \rangle s } \alpha^\beta  \partial f^{m+1} \|_p^p
\\ \le & C_{T,\Omega, p , \epsilon} \left( \| \alpha^\beta \partial f_0 \|_p^p + P( \| e^{\theta |v|^2} f_0 \|_\infty )  \right)
 + ( C_{T,\Omega, p} \epsilon + C_{T,\Omega, p , \epsilon} \delta + C_{T,\Omega, p , \epsilon , \delta} Te^{C T^2 } ) P( \| e^{\theta |v|^2} f_0 \|_\infty )  
\\ & \times \max_{i = m,m-1} \bigg\{ \sup_{0 \le t \le T} \|e^{-\varpi \langle v \rangle s } \alpha^\beta  \partial f^i(t) \|_p^p + \int_0^T | e^{-\varpi \langle v \rangle s } \alpha^\beta \partial f^i |_{\gamma_+,p}^p  
 + \int_0^T  \| \langle v \rangle^{1/p} e^{-\varpi \langle v \rangle s } \alpha^\beta  \partial f^i \|_p^p\bigg\}.
\end{split} \]
Therefore we can first choose $\epsilon$ small enough, then choose $\delta$ small enough correspondingly, and finally let $T$ be small enough correspondingly, we have:
\[ \begin{split} 
 \sup_{0 \le t \le T}  &\| e^{-\varpi \langle v \rangle t } \alpha^\beta \partial f^{m+1}(t) \|_p^p +  \int_0^T  |  e^{-\varpi \langle v \rangle s } \alpha^\beta   \partial f^{m+1} |_{\gamma_+,p}^p 
 +  \int_0^T \| \langle v \rangle ^{1/p} e^{-\varpi \langle v \rangle s } \alpha^\beta  \partial f^{m+1} \|_p^p
\\ \le & C_{T,\Omega, p , \epsilon} \left( \| \alpha^\beta \partial f_0 \|_p^p + P( \| e^{\theta |v|^2} f_0 \|_\infty )  \right)
\\ &+ \frac{1}{8}  \max_{i = m,m-1} \bigg\{ \sup_{0 \le t \le T} \|e^{-\varpi \langle v \rangle t } \alpha^\beta  \partial f^i(t) \|_p^p + \int_0^T | e^{-\varpi \langle v \rangle s } \alpha^\beta \partial f^i |_{\gamma_+,p}^p  
 + \int_0^T  \| \langle v \rangle ^{1/p} e^{-\varpi \langle v \rangle s } \alpha^\beta  \partial f^i \|_p^p\bigg\}.
\end{split} \]
Set
\[ \begin{split}
a_i   =& \sup_{0 \le t \le T} \|e^{-\varpi \langle v \rangle t } \alpha^\beta  \partial f^i(t) \|_p^p + \int_0^T | e^{-\varpi \langle v \rangle s } \alpha^\beta \partial f^i |_{\gamma_+,p}^p  
 + \int_0^T  \| \langle v \rangle ^{1/p} e^{-\varpi \langle v \rangle s } \alpha^\beta  \partial f^i \|_p^p\bigg\}
\\ D  = & C_{T,\Omega, p , \epsilon} \left( \| \alpha^\beta \partial f_0 \|_p^p + P( \| e^{\theta |v|^2} f_0 \|_\infty )  \right),
\end{split}, \] 
from \eqref{combfact} we complete the proof.

\end{proof}

\section{Weighted $C^1$ estimate}

In this chapter we prove some key lemmas which will be used in the proof of Theorem \ref{C1linearthm} and Theorem \ref{WlinftyVPBthm}, and then we prove Theorem \ref{C1linearthm}.

\begin{lemma} \label{int1overalphabetainu}
Suppose $E$ satisfies \eqref{signEonbdry}, then
for any y $\in \bar \Omega$, $1 < \beta < 3$, $0 < \kappa \le 1$, and $\theta > 0$ we have
\Be \label{kernelbddfor1overalpha}
\int_{\mathbb R^3} \frac{ e^{-\theta |v- u|^2}}{ |v - u | ^{2 - \kappa} [\alpha(s, y, u ) ]^\beta } du \le C \left( \frac{ 1}{( |v|^2 \xi(y) + c (y) )^{\frac{\beta - 1}{2} } } +1 \right) , 
\Ee
where $c(y) =  \xi(y) ^2 - C_E  \xi (y)$.
\end{lemma}

\begin{proof}
Recall the definition of $\alpha(t,x,v)$ from (\ref{alphadef}). If $\alpha(s,y,u) = C_{\delta'}$, then
\[
\int_{\mathbb R^3} \frac{ e^{-\theta |v- u|^2}}{ |v - u | ^{2 - \kappa} [\alpha(s, y, u ) ]^\beta } du = \int_{\mathbb R^3} \frac{ e^{-\theta |v- u|^2}}{ |v - u | ^{2 - \kappa}C_{\delta'}^\beta } du < C.
\]

For the case when $\alpha(s,y,u) < C_{\delta'}$, we have $|\xi (y) | \le \delta'/2 \ll 1 $. From the assumption, we have $\nabla \xi (y) \neq 0 $ and therefore there is a uniquely determined unit vector $n(y) = \frac{ \nabla \xi (y)}{ | \nabla \xi (y) | } $. We choose two unit vector $\tau_1$ and $\tau_2$ so that $\{ \tau_1, \tau_2, n(y) \}$ is an orthonormal basis of $\mathbb R^3$.

We decompose the velocity variables $u \in \mathbb R^3$ as
\[
u = u_n n(y) + u_\tau \cdot \tau = u_n n(y) + \sum_{i=1}^2 u_{\tau,i} \tau_i.
\]

We note that $u_\tau \in \mathbb R^2$ are completely free coordinate. Therefore using the Fubini's theorem we can rearrange the order of integration freely. Then we have 
\[
\alpha^2(s , y, u) \ge \frac{\beta^2(s,y,u)}{4}  = \frac{1}{4} [\nabla \xi(y) \cdot u + 2 (u \cdot \nabla ^2 \xi ( y) \cdot u)\xi(y) + \xi(y)^2 - 2 E(s, \bar y ) \cdot \nabla \xi(\bar y)  \xi (y) ] \ge c( |u_n|^2 + |u|^2 \xi (y) + c(y))
\]
for some $c > 0$.

Now we split
\[ \begin{split}
\int_{\mathbb R^3} & \frac{ e^{-\theta |v- u|^2}}{ |v - u | ^{2 - \kappa} [\alpha(s, y, u ) ]^\beta } du
\\ \le   & C \int_{\mathbb R^2} \int_{\mathbb R}   \frac{ e^{-\theta |v -u |^2 }}{|v - u|^{2 -\kappa} [ |u_n|^2 + |\xi(y) ||u|^2 + c(y) ]^{\beta/2} } du_n du_\tau
\\ = & \int_{|u| \le \frac{ |v| }{2} } + \int_{ |u| \ge \frac{ |v| }{2} } = (\textbf{I}) + ( \textbf{II}).
\end{split} \]
If $|u | \le \frac{|v|}{2}$, then $|v - u | \ge |v| - |u| \ge \frac{|v|}{2}$, apply the change of variable $u \mapsto |v|u$ we have
\[ \begin{split}
 (\textbf{I})   = \int_{|u | \le \frac{|v|}{2}} & \frac{ e^{-\theta |v -u |^2 }}{|v - u|^{2 -\kappa} [ |u_n|^2 + |\xi(y) ||u|^2 + c(y) ]^{\beta/2} } du_n du_\tau
\\ \le & \frac{2^{2-\kappa}}{|v|^{2 - \kappa}} \int_{|v|(|u_n| + |u_\tau|) \le \frac{|v|}{2} }  \frac{ e^{-\frac{\theta}{4}|v|^2}  |v|^3 } { [|v| ^2  |u_n|^2 +|v|^2 |\xi(y) ||u_\tau|^2 + c(y) ]^{\beta/2} } du_n du_\tau
\\ \le &  \frac{ 2^{2-\kappa} e^{-\frac{\theta}{4}  |v  |^2} } {|v|^{\beta - \kappa - 1}} \int_{|u_\tau| \le \frac{1}{2} }  \int_{|u_n| \le \frac{1}{2} } \frac{1 } {  [ |u_n|^2 + |\xi(y) ||u_\tau|^2 + \frac{c(y)}{|v|^2} ]^{\beta/2} } du_n du_\tau.
\end{split} \]
Now we apply the change of variables $|u_n | = ( |\xi| |u_\tau| ^2 + \frac{c(y) }{|v|^2 } ) ^{\frac{1}{2}} \tan \theta$ for $\theta \in [0, \frac{\pi}{2} ]$ with $d u_n = ( |\xi| |u_\tau| ^2 + \frac{c(y) }{|v|^2 } ) ^{\frac{1}{2}}  \frac{1}{\cos ^2 \theta }  d\theta$ to have
\[ \begin{split}
(\textbf{I}) \le & \frac{ 2^{2-\kappa}e^{-\frac{\theta}{4}  |v  |^2} } {|v|^{\beta - \kappa - 1}}  \int_{|u_\tau| \le \frac{1}{2} } \int_0^{\frac{\pi}{2}}\frac{ ( |\xi| |u_\tau| ^2 + \frac{c(y) }{|v|^2 } ) ^{\frac{1}{2}} d\theta }{[( | \xi| |u_\tau| ^2 + \frac{c(y) }{|v|^2 } )( \tan^2 \theta + 1) ]^{\beta/2} \cos^2 \theta} d u_\tau
 \\ = & \frac{ 2^{2-\kappa}e^{-\frac{\theta}{4}   |v  |^2} } {|v|^{\beta - \kappa - 1}}  \int_{|u_\tau| \le \frac{1}{2}}  ( |\xi| |u_\tau| ^2 + \frac{c(y) }{|v|^2 } ) ^{\frac{1 - \beta }{2}} d u_\tau \int_0^{\frac{\pi}{2}} \frac{1}{\cos^{2 - \beta} } d \theta
 \\ \le & C   \frac{ e^{-\frac{\theta}{4}   |v  |^2} } {|v|^{\beta - \kappa - 1}}  \int_{|u_\tau| \le \frac{1}{2}}  ( |\xi| |u_\tau| ^2 + \frac{c(y) }{|v|^2 } ) ^{\frac{1 - \beta }{2}} d u_\tau,
\end{split} \]
as $  \int_0^{\frac{\pi}{2}} \frac{1}{\cos^{2 - \beta} } d \theta < \infty $ for $\beta > 1$.

We then use polar coordinates for $u_\tau = ( r ,\phi)$ with $d u_\tau = r dr d\phi$ to have
\[ \begin{split}
(\textbf{I}) \le & C \frac{ 2 \pi e^{-\frac{\theta}{4}   |v  |^2} } {|v|^{\beta - \kappa - 1}} \int_{0 }^{1/2}  \frac{ r }{  ( |\xi| r ^2 + \frac{c(y) }{|v|^2 } ) ^{\frac{ \beta - 1   }{2}}} dr  =  \frac{ 2 \pi e^{-\frac{\theta}{2}   |v  |^2} } {|v|^{\beta - \kappa - 1}}   \left[  \frac{ ( | \xi| r ^2 + \frac{c(y) }{|v|^2  } ) ^{ - \frac{ \beta - 1   }{2} +1 } }{ (- \frac{ \beta - 1   }{2} +1  ) 2 |\xi| } \right]_{r = 0 }^{r = 1/2}
\\ = & C \frac{ 2 \pi e^{-\frac{\theta}{4}   |v  |^2} } {(3 - \beta)|v|^{\beta - \kappa - 1}  } \left[ \frac{ ( |\xi|  + \frac{c (y) }{|v|^2} )} { ( |\xi|  + \frac{c (y) }{|v|^2} )^{\frac{\beta - 1}{2} } |\xi|}  - \frac{ \frac{c (y)}{|v|^2} } { (\frac{  c (y) }{|v|^2} )^{\frac{\beta - 1}{2} } |\xi|} \right]
\\ = & C \frac{ 2 \pi e^{-\frac{\theta}{4}  |v  |^2} } {(3 - \beta)|v|^{\beta - \kappa - 1}  } \left[ \frac{ 1} { ( |\xi|  + \frac{c (y) }{|v|^2} )^{\frac{\beta - 1}{2} } } + \frac{ \frac{c (y) }{|v|^2} } { ( |\xi|  + \frac{c (y) }{|v|^2} )^{\frac{\beta - 1}{2} } |\xi|}   - \frac{ \frac{c (y)}{|v|^2} } { (\frac{  c (y) }{|v|^2} )^{\frac{\beta - 1}{2} } |\xi|} \right]
\\ \le  & C \frac{  e^{-\frac{\theta}{4}   |v  |^2} } {|v|^{\beta - \kappa - 1}  } \left[ \frac{ 1} { ( |\xi|  + \frac{c (y) }{|v|^2} )^{\frac{\beta - 1}{2} } }  \right] 
\\ = & C\frac{  e^{-\frac{\theta}{4}   |v  |^2} } {|v|^{\beta - \kappa - 1}  } \frac{ |v|^{\beta -1 }} { ( |v|^2 |\xi|  + c (y)  )^{\frac{\beta - 1}{2} } }  
= C \frac{  e^{-\frac{\theta}{4}   |v  |^2} |v| ^{\kappa} }{( |v|^2 |\xi| + c (y) )^{\frac{\beta - 1}{2} }} \le C \frac{ 1 }{( |v|^2 |\xi| + c (y) )^{\frac{\beta - 1}{2} }},
\end{split} \] 
for $1 < \beta < 3$.

For the second term $ (\textbf{II}) $, we use the lower bound $|u| \ge \frac{|v| }{2}$ to have $[|u_n |^2 + |\xi| |u|^2 + c(y) ] ^{\beta/2} \ge  [|u_n |^2 + |\xi|\frac{ |v|^2}{4} + c(y) ] ^{\beta/2} \ge2^{-\beta} [|u_n |^2 + |\xi| |v|^2 + c(y) ] ^{\beta/2}$, and
\[ \begin{split}
(\textbf{II}) = \int_{|u | \ge \frac{|v|}{2}} & \frac{ e^{-\theta |v -u |^2 }}{|v - u|^{2 -\kappa} [ |u_n|^2 + |\xi(y) ||u|^2 + c(y) ]^{\beta/2} } du_n du_\tau
\\ \le &2^{-\beta}  \int_{\mathbb R^2 } \frac{ e^{-\theta |v_\tau -u_\tau |^2 }}{|v_\tau - u_\tau|^{2 -\kappa}}  du_\tau \int_0^\infty \frac{1}{ [|u_n |^2 + |\xi| |v|^2 + c(y) ] ^{\beta/2} } d u_n
\\ \le & C \int_0 ^\infty \frac{1}{ [|u_n |^2 + |\xi| |v|^2 + c(y) ] ^{\beta/2} } d u_n,
\end{split} \]
as $ \int_{\mathbb R^2 } \frac{ e^{-\theta |v_\tau -u_\tau |^2 }}{|v_\tau - u_\tau|^{2 -\kappa}}  du_\tau < \infty$ for $\kappa > 0$. Then apply a change of variables: $|u_n| = ( |\xi| |v|^2 + c(y) )^{1/2} \tan \theta$ for $\theta \in [ 0, \pi/2]$ with $du_n =  ( |\xi| |v|^2 + c(y) )^{1/2} \frac{1}{\cos^2(\theta) } d \theta$ to have
\begin{equation} \label{frac1alphaint}
\begin{split}
(\textbf{II} ) \le &C \int_0 ^\infty \frac{1}{ [|u_n |^2 + |\xi| |v|^2 + c(y) ] ^{\beta/2} } d u_n = C\int_0^{\frac{\pi}{2}} \frac{(|\xi| |v|^2 + c(y)) ^{1/2}  }{ (|\xi| |v|^2 + c(y))^{\beta/2} ( \tan^2 (\theta) +1 ) ^{\beta/2} \cos^2 (\theta) } d \theta
\\ = & \frac{C}{ (|\xi| |v|^2 + c(y))^{\frac{\beta -1}{2}}} \int_0^{\frac{\pi}{2}} \frac{ 1}{ \cos^{2 - \beta} (\theta) } d\theta
\le   \frac{C}{ (|\xi| |v|^2 + c(y))^{\frac{\beta -1}{2}}},
\end{split} 
\end{equation}
as $\int_0^{\frac{\pi}{2}} \frac{ 1}{ \cos^{2 - \beta} (\theta) } d\theta < \infty$ for $\beta >1$.

Thus $(\textbf{I} ) + (\textbf{II} ) \le   \frac{C}{ (|\xi| |v|^2 + c(y))^{\frac{\beta -1}{2}}} $ as wanted.

\end{proof}

\begin{lemma}\label{keylemma}
Let $(t,x,v) \in [0,T] \times  \Omega \times \mathbb R^3$, $ 1 < \beta < 3$, $0 < \kappa \le 1 $. 
Suppose $E$ satisfies (\ref{signEonbdry}) and (\ref{c1bddforthepotentail}), 
then for $\varpi \gg 1$ large enough, we have for any $0 < \delta \ll 1$ small enough,
\begin{equation} \label{genlemma1}
\begin{split}
 & \int_{\max\{ 0, t - \tb \}} ^t  \int_{\mathbb R^3} e^{ -\int_s^t\frac{\varpi}{2} \langle V(\tau;t,x,v) \rangle d\tau }  \frac{e^{-\frac{C_\theta}{2} |V(s)-u|^2 }}{|V(s) -u |^{2 - \kappa}  }  \frac{1}{(\alpha(s,X(s) ,u))^\beta} du  ds
\\  \lesssim &  e^{ 2C_\xi \frac{\| \nabla E\|_\infty  + \| E \|_\infty^2 + \| E \|_\infty}{C_E}}  \frac{  \delta ^{\frac{3 -\beta}{2} }}{C_E^ {\frac{\beta-1}{2}} (\alpha(t,x,v))^{\beta -2 }  (|v|^2 +  \| E \|_\infty^2 +  \| E \|_\infty +1)^{\frac{3 -\beta}{2}}}   
\\ & +  \frac{ (|v| +\|E \|_\infty + \| E\|_\infty^2+1)^{\beta -1} } { C_E^{\beta-1} \delta ^{\beta -1 } ( \alpha ( t,x,v) )^{\beta - 1 } } \frac{2}{ \varpi }.
\end{split} 
\end{equation}

%
\end{lemma}

\begin{proof}
We separate the proof into several cases.

In \textit{Step 1, Step 2, Step 3} we prove (\ref{genlemma1}) for the case when $x \in \partial \Omega$ and $t \le \tb $.

In \textit{Step 4} we prove (\ref{genlemma1}) for the case when $x \in \partial \Omega$ and $t > \tb $.

In \textit{Step 5} we prove (\ref{genlemma1}) for the case when $x \in  \Omega$ and $t \le \tb $.

In \textit{Step 6} we prove (\ref{genlemma1}) for the case when $x \in  \Omega$ and $t > \tb $.


\bigskip

\textit{Step 1} \qquad
Let's first start with the case $t \ge \tb$ and prove (\ref{genlemma1}), 
%
Let's shift the time variable: $s \mapsto t - \tb + s$, and let $\tilde X(s) = X(t -\tb +s )$, $\tilde V(s) = V(t -\tb +s ) $. Then $s \in [0, \tb] $ and from (\ref{kernelbddfor1overalpha}) we only need to bound the integral
\begin{equation} \label{trajint}
\int_0^{\tb} e^{ -\int_{t -\tb +s}^t\frac{\varpi}{2} \langle V(\tau;t,x,v) \rangle d\tau }  \frac{ 1}{\left[ |\tilde V(s)|^2 \xi(\tilde X( s) ) + \xi^2(\tilde X( s)) - C_E \xi (  \tilde X(s) ) \right]^{\frac{\beta - 1}{2} } }  ds.
\end{equation}

Let's assume $x \in \partial \Omega$ and $ v \cdot \nabla \xi (x) > 0$. Then by the velocity lemma (Lemma \ref{velocitylemma}) we have $\vb \cdot \nabla \xi(\xb) < 0$.

Claim: for any $0< \delta \ll 1$ small enough, if we let 
\begin{equation} \label{sigma12}
\sigma_1 = \delta \frac{ \vb \cdot \nabla \xi (\xb) } { |v | ^2 + \| E \|_\infty^2 +  \| E \|_\infty+1 },  \text{ and } \sigma_2 = \delta \frac{ v \cdot \nabla \xi (x) } { |v|^2 + \| E \|_\infty^2 +  \| E \|_\infty+1 },
\end{equation}
then $ |\xi(\tilde X(s) |$ is monotonically increasing on $[0, \sigma_1]$, and monotonically decreasing on $[\tb - \sigma_2 , \tb]$. Moreover, we have the following bounds:
\begin{equation} \label{sigma12lowerbd}
|\xi ( \tilde X(\sigma_1) ) | \ge  \frac{ \delta ( \vb \cdot \nabla \xi (\xb) ) ^2 }{ 2 ( |v| ^2 +  \| E \|_\infty^2 +  \| E \|_\infty +1) }, \, |\xi ( \tilde X(\sigma_2 ) ) | \ge \frac{\delta ( v \cdot \nabla \xi (x) ) ^2 }{ 2 ( |v| ^2 +  \| E \|_\infty^2 +  \| E \|_\infty+1 ) },
\end{equation}
\begin{equation}
\begin{split} \label{xiupperbd}
& |\xi (\tilde X(s) |  \le  \frac{ 3 \delta ( \vb \cdot \nabla \xi (\xb) ) ^2 }{ 2 ( |v| ^2 +  \| E \|_\infty^2 +  \| E \|_\infty+1 ) }, \, s \in [ 0, \sigma_1],
\\ &  |\xi (\tilde X(s) | \le \frac{ 3 \delta ( v \cdot \nabla \xi (x) ) ^2 }{ 2 ( |v| ^2 +  \| E \|_\infty^2 +  \| E \|_\infty+1 ) }, \, s \in [ \tb - \sigma_2, \tb],
\end{split}
\end{equation}
and
\begin{equation} \label{vdotnlowerbd}
\begin{split}
& | \tilde V(s) \cdot \nabla \xi (\tilde X(s))|  \ge \frac{ |\vb \cdot \nabla \xi (\xb ) | }{2}, \, s \in [0, \sigma_1],
\\ & | \tilde V(s) \cdot \nabla \xi (\tilde X(s))|  \ge \frac{ |v \cdot \nabla \xi (x ) | }{2}, \, s \in [\tb - \sigma_2, \tb].
\end{split}
\end{equation}

To prove the claim we first note that $\frac{d}{ds } \xi ( \tilde X(s) ) |_{s= 0 } = \vb \cdot \nabla \xi (\xb) < 0$, and
\begin{equation} \label{accelerationbd}
 \begin{split}
\frac{ d^2}{d^2 s } \xi (\tilde X(s) ) ) = & \frac{d}{ds} (\tilde V(s) \cdot \nabla \xi (\tilde X(s))) = \tilde V(s) \cdot \nabla^2 \xi (\tilde X(s)) \cdot \tilde V(s) + E(s ,\tilde X(s)) \cdot \nabla \xi (\tilde X(s))
\\  \le & C(|\tilde V(s) | ^2 + \| E\|_\infty ) \le C(2 |v|^2 + 2 (\tb \| E \|_\infty)^2 +  \| E \|_\infty ) \le C_1 ( |v|^2 +  \| E \|_\infty^2 +  \| E \|_\infty+1),
\end{split} 
\end{equation}
for some $C_1 >0$. Thus if $\delta $ small enough, we have $\frac{d}{ds } \xi ( \tilde X(s) ) < 0 $ for all $s \in [0, \delta \frac{| \vb \cdot \nabla \xi (\xb) |}{ |v| ^2 + \| E \|_\infty^2 +  \| E \|_\infty+1  }]$. Therefore $\xi ( \tilde X(s) ) $ is decreasing on $[0, \sigma_1]$.

Similarly $ \frac{d}{ds } \xi ( \tilde X(s) ) |_{s= \tb } = v \cdot \nabla \xi(x) > 0$, and since $ | \frac{ d^2}{d^2 s } \xi (\tilde X(s) ) ) | \lesssim (|v|^2 +\| E \|_\infty^2 +  \| E \|_\infty +1) $ we have that $ \frac{d}{ds } \xi ( \tilde X(s) )  > 0 $ for all $s \in [ \tb - \delta \frac{ | v \cdot \nabla \xi (v) | }{ |v| ^2 +\| E \|_\infty^2 +  \| E \|_\infty  +1}, \tb ] $ if $\delta $ small enough. Therefore $\xi ( \tilde X(s) ) $ is increasing on $[ \tb - \sigma_2, \tb]$.

Next we establish the bounds (\ref{sigma12lowerbd}), (\ref{xiupperbd}), and (\ref{vdotnlowerbd}). By (\ref{accelerationbd}), we have
\[ \begin{split}
|\xi ( \tilde X(\sigma_1 ) ) | = & \int_0 ^{\sigma_1} -  \tilde V(s) \cdot \nabla \xi ( \tilde X(s) ) ds  
\\ = & \int_0^{\sigma_1}  \left( \int_0 ^s   - \frac{d}{d\tau} (\tilde V(\tau) \cdot \nabla \xi (\tilde X(\tau ) ) )  d\tau -  \vb \cdot \nabla \xi (\xb )  \right)   ds 
\\ \ge & \int_0^{\sigma_1}  \left(|\vb \cdot \nabla \xi (\xb) | - C_1(|v|^2 + \| E \|_\infty^2 +  \| E \|_\infty +1 ) s  \right) ds 
\\ = & \sigma_1 |\vb \cdot \nabla \xi (\xb) | - \frac{\sigma_1^2 }{2 } C_1 (|v|^2 + \| E \|_\infty^2 +  \| E \|_\infty  +1 )  
\\ = & \sigma_1 \left( |\vb \cdot \nabla \xi (\xb) |- \frac{ \delta C_1}{2} |\vb \cdot \nabla \xi (\xb ) | \right) 
\\ \ge & \frac{ \sigma_1}{2}  |\vb \cdot \nabla \xi (\xb) | = \frac{ \delta ( \vb \cdot \nabla \xi (\xb) ) ^2 }{ 2 ( |v| ^2 + \| E \|_\infty^2 +  \| E \|_\infty +1 ) }.
\end{split} \]

And by the same argument we have $|\xi ( \tilde X(\sigma_2 ) ) | \ge \frac{\delta ( v \cdot \nabla \xi (x) ) ^2 }{ 2 ( |v| ^2 + \| E \|_\infty^2 +  \| E \|_\infty +1  ) } $ for $\delta \ll 1$. This proves (\ref{sigma12lowerbd}).

To prove (\ref{xiupperbd}), we have from (\ref{accelerationbd}), for $s \in [0,\sigma_1]$,
\[ \begin{split}
|\xi (\tilde X(s) | \le & s \left( |\vb \cdot \nabla \xi (\xb) |+ \frac{ \delta C_1}{2} |\vb \cdot \nabla \xi (\xb ) | \right)  
\\ \le & \frac{3 s}{2} | \vb \cdot \nabla \xi(\xb) | \le  \frac{ 3 \delta ( \vb \cdot \nabla \xi (\xb) ) ^2 }{ 2 ( |v| ^2 +  \| E \|_\infty^2 +  \| E \|_\infty +1 ) },
\end{split} \]
and $ |\xi (\tilde X(s) | \le \frac{ 3 \delta ( v \cdot \nabla \xi (x) ) ^2 }{ 2 ( |v| ^2 +  \| E \|_\infty^2 +  \| E \|_\infty+1 ) }$ for $s \in [ \tb -\sigma_2, \tb] $. This proves (\ref{xiupperbd}).

Finally for (\ref{vdotnlowerbd}), again from (\ref{accelerationbd}),
\[ 
\begin{split}
| \tilde V(s) \cdot \nabla \xi (\tilde X(s))| \ge & |\vb \cdot \nabla \xi (\xb)  | - \int_0 ^{\sigma_1} C_1 (|v|^2 +  \| E \|_\infty^2 +  \| E \|_\infty +1 ) ds  
\\ \ge & |\vb \cdot \nabla \xi (\xb)  |  - C_1 \delta |\vb \cdot \nabla \xi (\xb ) |  \ge \frac{ |\vb \cdot \nabla \xi (\xb ) | }{2}.
\end{split}
\]
And similarly $| \tilde V(s) \cdot \nabla \xi (\tilde X(s))| \ge \frac{| v \cdot \nabla \xi ( x ) |}{2} $ for $s \in [ \tb - \delta_2, \tb] $. This proves the claim.

\bigskip

\textit{Step 2} \qquad
Recall the definition of $\sigma_1, \sigma_2$ in (\ref{sigma12}), and $C_E$ in (\ref{signEonbdry}).
In this step we establish the lower bound:
\begin{equation} \label{xilowerbd}
 | \xi( \tilde X(s)  ) | > \frac{C_E}{10} (\sigma_2)^2, \, \text{for all } s \in [\sigma_1, \tb - \sigma_2 ].
\end{equation}

Suppose towards contradiction that $I := \{ s \in [ \sigma_1, \tb - \sigma_2]:  | \xi( \tilde X(s)  ) | \le \frac{C_E}{10}  (\sigma_2)^2 \} \neq \emptyset$.


Then from (\ref{velocitylemmaintform}) and (\ref{sigma12lowerbd}) we have
\[ \begin{split}
 \frac{C_E}{10} (\sigma_2)^2 \le &  \delta^2  \frac{C_E}{10} \frac{ (v \cdot \nabla \xi (x))^2 } { |v|^2 +   \| E \|_\infty^2 +  \| E \|_\infty +1 } 
 \\  \le & \delta^2 \frac{C_E}{10} e^{ C_\xi \frac{\| \nabla E\|_\infty  + \| E \|_\infty^2 + \| E \|_\infty}{C_E}}  \frac{  (\vb \cdot \nabla \xi (\xb))^2 } { |v|^2 +  \| E \|_\infty^2 +  \| E \|_\infty  +1 } 
 \\ \le & 2 \delta \frac{C_E}{10} e^{ C_\xi \frac{\| \nabla E\|_\infty  + \| E \|_\infty^2 + \| E \|_\infty}{C_E}}  |\xi (\tilde X(\sigma_1 ) ) |
 \\  < &  |\xi (\tilde X(\sigma_1 ) ) |,
 \end{split} \]
  if $\delta \ll 1$. So $\sigma_1 \notin I$. Let $s^*:= \min\{s \in  I \}$ be the minimum of such $s$. Then clearly 
  \[
  \frac{d}{ds} \xi (\tilde X(s) ) |_{s = s^*} = \tilde V(s^* ) \cdot \nabla \xi ( \tilde X(s^*) )  \ge 0.
  \]
  
Now recall (\ref{Exiboundary}) and (\ref{cXbdd}) from the proof the velocity lemma, we have
\begin{equation} 
\begin{split}
E(s ,\tilde X(s)) \cdot \nabla \xi (\tilde X(s)) = E(s , \overline {\tilde X(s)} ) \cdot \nabla \xi ( \overline {\tilde X(s) }) + c (\tilde X(s) ) \cdot \xi (\tilde X(s) ),
\end{split} 
\end{equation}
with $|c(\tilde X(s) ) | < \frac{C_{\xi}(\| E \|_\infty + \| \nabla E \|_\infty )}{C_E}$.
Thus
\Be \label{accnearbdy}
\begin{split}
   \frac{d}{ds} (\tilde V(s) \cdot \nabla \xi (\tilde X(s))) = & \tilde V(s) \cdot \nabla^2 \xi (\tilde X(s)) \cdot \tilde V(s) + E(s ,\tilde X(s)) \cdot \nabla \xi (\tilde X(s))
   \\ = & \tilde V(s) \cdot \nabla^2 \xi (\tilde X(s)) \cdot \tilde V(s) + E(s , \overline {\tilde X(s)} ) \cdot \nabla \xi ( \overline {\tilde X(s) }) + c (\tilde X(s) ) \cdot \xi (\tilde X(s) ) 
     \\ \ge & C_E - \frac{C_{\xi}(\| E \|_\infty + \| \nabla E \|_\infty )}{C_E} | \xi (\tilde X(s) ) |,
   \end{split}
\Ee
so 
\[
\frac{d}{ds} (\tilde V(s) \cdot \nabla \xi (\tilde X(s))) |_{s = s^*}   \ge  C_E - \delta^2  \frac{C_{\xi}(\| E \|_\infty + \| \nabla E \|_\infty )}{C_E}  \frac{C_E}{10} \frac{ (v \cdot \nabla \xi (x))^2 } { |v|^2 +   \| E \|_\infty^2 +  \| E \|_\infty  + 1}    \ge \frac{C_E}{2},
\]
for $\delta \ll 1 $ small enough. Then we have $ \frac{d}{ds} (\tilde V(s) \cdot \nabla \xi (\tilde X(s)))$ is increasing on the interval $[s^*,\tb]$ as $|\xi (\tilde X(s)) |$ is decreasing. So
\[
\frac{d}{ds} (\tilde V(s) \cdot \nabla \xi (\tilde X(s))) \ge \frac{C_E}{2}, \quad s \in [s^* , \tb].
\]
%
%
 And therefore
 \[ \begin{split}
 |\xi(\tilde X(s^* ) ) | = & \int_{s^*}^ {\tb } \tilde V(s) \cdot \nabla \xi ( \tilde X(s) ) ds  
 \\ = & \int_{s^*}^ {\tb } \left( \int_{s^*}^{s} \frac{d}{d\tau }  ( \tilde V(\tau) \cdot \nabla \xi ( \tilde X(\tau) ) ) d\tau + \tilde V(s^* ) \cdot \nabla \xi( \tilde X(s^*) ) \right) ds
 \\ \ge & \int_{s^*}^ {\tb } (s - s^* ) \frac{C_E}{2} ds  = \frac{C_E }{4} ( \tb - s^* ) ^2 \ge \frac{C_E}{4} (\sigma_2)^2,
\end{split}  \]
which is a contradiction. Therefore we conclude \eqref{xilowerbd}.

\bigskip

\textit{Step 3} \qquad
Let's split the time integration (\ref{trajint}) as
\begin{equation} 
\begin{split}
\int_0^{\tb}  & e^{ -\int_{t -\tb +s}^t\frac{\varpi}{2} \langle V(\tau;t,x,v) \rangle d\tau }  \frac{ 1}{\left[ |\tilde V(s)|^2 \xi(\tilde X( s) ) + \xi^2(\tilde X( s)  - C_E \xi ( \tilde X(s) ) \right]^{\frac{\beta - 1}{2} } }  ds
\\ = & \int_0^{\sigma_1} + \int_{\sigma_1}^{\tb - \sigma_2} + \int_{\tb -\sigma_2}^{\tb }  = (\textbf{I})+ (\textbf{II})+ (\textbf{III}).
\end{split} 
\end{equation}

Let's first estimate $ (\textbf{I}),  (\textbf{III})$:

From \textit{Step 2} we have that $| \xi (\tilde X(s) | $ is monotonically increasing on $[0, \sigma_1] $ and $[\tb - \sigma_2, \tb] $, so we have the change of variables:
\[
ds = \frac{d |\xi | }{ | \tilde V(s) \cdot \nabla \xi (\tilde X(s))|  }.  
\]
Using this change of variable and the bounds (\ref{xiupperbd}), (\ref{vdotnlowerbd}), $ (\textbf{I})$ is bounded by 
\begin{equation} \label{Ibd}
\begin{split}
 (\textbf{I}) \le & \int_0^{\sigma_1}   \frac{ 1 }{\left[ |\tilde V(s)|^2 \xi(\tilde X( s) ) + \xi^2(\tilde X( s)- C_E \xi (\tilde X(s) ) \right]^{\frac{\beta - 1}{2} } }  ds
\\ \le & \int_0^{\frac{ 3 \delta ( \vb \cdot \nabla \xi (\xb) ) ^2 }{ 2 ( |v| ^2 +   \| E \|_\infty^2 +  \| E \|_\infty +1 ) }} \frac{ 1 } {  | \tilde V(s) \cdot \nabla \xi (\tilde X(s))|  ( C_E  |\xi | ) ^{\frac{\beta-1}{2}}} d |\xi |
\\ \le &   \int_0^{\frac{ 3 \delta ( \vb \cdot \nabla \xi (\xb) ) ^2 }{ 2 ( |v| ^2 +   \| E \|_\infty^2 +  \| E \|_\infty +1 ) }} \frac{ 2  } {  | \vb \cdot \nabla \xi (\xb)|  ( C_E  |\xi | )^ {\frac{\beta-1}{2}}}d |\xi | 
\\ = &  \frac{2}{| \vb \cdot \nabla \xi (\xb)| C_E^ {\frac{\beta-1}{2}}} \left[ |\xi |^{\frac{3 - \beta} {2}} \right] _0^{\frac{ 3 \delta ( \vb \cdot \nabla \xi (\xb) ) ^2}{ 2 ( |v| ^2 +  \| E \|_\infty^2 +  \| E \|_\infty +1 ) }} 
\\ = &  \frac{2^{\frac{\beta-1}{2}} \delta ^{\frac{3 -\beta}{2} }} { C_E^ {\frac{\beta-1}{2}} | \vb \cdot \nabla \xi (\xb) | ^{\beta -2 } (|v|^2 +  \| E \|_\infty^2 +  \| E \|_\infty +1)^{\frac{3 -\beta}{2}}}
\\ \lesssim & e^{ 2C_\xi \frac{\| \nabla E\|_\infty  + \| E \|_\infty^2 + \| E \|_\infty}{C_E}}  \frac{  \delta ^{\frac{3 -\beta}{2} }}{C_E^ {\frac{\beta-1}{2}} (\alpha(t,x,v))^{\beta -2 }  (|v|^2 +  \| E \|_\infty^2 +  \| E \|_\infty +1 )^{\frac{3 -\beta}{2}}}.
\end{split} 
\end{equation}
And by the same computation we get 
\begin{equation} \label{IIIbd}
 (\textbf{III}) \lesssim e^{ 2C_\xi \frac{\| \nabla E\|_\infty  + \| E \|_\infty^2 + \| E \|_\infty}{C_E}}  \frac{  \delta ^{\frac{3 -\beta}{2} }}{C_E^ {\frac{\beta-1}{2}} (\alpha(t,x,v))^{\beta -2 }  (|v|^2 +  \| E \|_\infty^2 +  \| E \|_\infty +1 )^{\frac{3 -\beta}{2}}}  .
 \end{equation}


Finally for (\textbf{II}), using the lower bound for $|\xi (\tilde X(s )) |$ in (\ref{xilowerbd}), we have
\begin{equation} \label{IIbd}
\begin{split}
(\textbf{II}) = & \int_{\sigma_1}^{\sigma_2}  e^{ -\int_{t -\tb +s}^t\frac{\varpi}{2} \langle V(\tau;t,x,v) \rangle d\tau }   \frac{ 1 }{\left[ |\tilde V(s)|^2 \xi(\tilde X( s) ) + \xi^2(\tilde X( s) -C_E \xi (\tilde X(s) ) \right]^{\frac{\beta - 1}{2} } }  ds
\\ \le & \int_0^{\tb} e^{ -\int_{t -\tb +s}^t\frac{\varpi}{2} \langle V(\tau;t,x,v) \rangle d\tau } \frac{1}{ | C_E \xi (\tilde X(s)) |^{\frac{\beta -1}{2}} } ds
\\ \lesssim &  \frac{1}{C_E^{\beta-1}( \sigma_2)^{\beta - 1 } }\int_0^{\tb} e^{ \int_{t -\tb +s}^t \frac{\varpi}{2}    d\tau } ds 
\\ \lesssim & \frac{ (|v| +\|E \|_\infty + \| E\|_\infty^2 +1)^{\beta -1} } { C_E^{\beta-1} \delta ^{\beta -1 } ( \alpha ( t,x,v) )^{\beta - 1 } }\int_0^{\tb} e^{ ( s - \tb )  \frac{\varpi}{2}   } ds 
  \lesssim   \frac{ (|v| +\|E \|_\infty + \| E\|_\infty^2 +1)^{\beta -1} } { C_E^{\beta-1} \delta ^{\beta -1 } ( \alpha ( t,x,v) )^{\beta - 1 } } \frac{2}{ \varpi }.
\end{split} 
\end{equation}

This proves (\ref{genlemma1}) for the case $ x \in \partial \Omega$ and $t \le \tb$. 

%
%

\textit{Step 4} \qquad
Now suppose $x \in \partial \Omega$ and  $ \tb >  t$. It suffices to bound the integral:
\begin{equation} \label{trajintt}
\int_0^{t} e^{ -\int_{s}^t\frac{\varpi}{2} \langle V(\tau;t,x,v) \rangle d\tau }  \frac{ 1}{\left[ | V(s)|^2 \xi( X( s) ) + \xi^2( X( s) - C_E \xi (  X(s) ) \right]^{\frac{\beta - 1}{2} } }  ds.
\end{equation}
Denote
\[
X(0;t,x,v ) = x_0, V(0;t,x,v) = v_0.
\]
Let  
\[
\sigma_2 = \delta \frac{ v \cdot \nabla \xi (x) } { |v|^2 + \| E \|_\infty + \| E \|_\infty^2 +1 }
\]
as defined in (\ref{sigma12}). If 
\[
\sigma_2 \ge t,
\]
then from \textit{Step 2} $|\xi (X(s) ) | $ is decreasing on $[0,t]$, and by
(\ref{xiupperbd}), (\ref{vdotnlowerbd}), and the bound for (\textbf{III}) (\ref{IIIbd}), we get the desired estimate.
Now we assume 
\[
\sigma_2 < t.
\]
So from (\ref{sigma12lowerbd}) we have
\begin{equation} \label{sigma2lowerbd}
|\xi (  X(\sigma_2 ) ) | \ge \frac{\delta ( v \cdot \nabla \xi (x) ) ^2 }{ 2 ( |v| ^2 +  \| E \|_\infty + \| E \|_\infty^2+1 ) }.
\end{equation}

Now if $|\xi ( x_0 ) | \le \delta \frac{ \alpha^2 ( t,x,v) }{10(|v|^2 + \| E \|_\infty + \| E \|_\infty^2+1)} $, 
\begin{equation} 
\begin{split}
\alpha^2(t,x,v) \lesssim & e^{ C_\xi \frac{\| \nabla E\|_\infty  + \| E \|_\infty^2 + \| E \|_\infty}{C_E}}   \alpha^2 (0,x_0, v_0 )    \\ \lesssim & e^{ C_\xi \frac{\| \nabla E\|_\infty  + \|E\|_\infty^2 + \| E \|_\infty}{C_E}}  ( (\nabla \xi (x_0 ) \cdot v_0 ) ^2 + (|v_0|^2 + |\xi(x_0)| + \| E\|_\infty ) |\xi (x_0 ) |)
\\ \lesssim & e^{ C_\xi \frac{\| \nabla E\|_\infty  + \|E\|_\infty^2 + \| E \|_\infty}{C_E}} (\nabla \xi (x_0 ) \cdot v_0 ) ^2 + \delta \alpha^2 (t,x,v),
\end{split}
\end{equation}
 So 
\begin{equation} \label{alphaxi0v0}
\frac{1}{2} \alpha(t,x,v) \lesssim  e^{ C_\xi \frac{\| \nabla E\|_\infty  + \|E\|_\infty^2 + \| E \|_\infty}{C_E}} |\nabla \xi(x_0) \cdot v_0 | ,
\end{equation}
if $ \delta \ll 1 $ is small enough.

Claim: 
\[
 \nabla \xi (x_0 ) \cdot v_0 <0 .
 \]
Since otherwise by (\ref{accnearbdy}) we have
\[
\frac{d}{ds} | \xi (X(s) ) | < 0,
\]
for all $s \in [0,t]$, so $|\xi(X(s))| $ is always decreasing, which contradicts (\ref{sigma2lowerbd}).

Therefore $\nabla \xi (x_0 ) \cdot v_0 <0$, and we can run the same argument from \textit{Step 1}, \textit{Step 2}, \textit{Step 3} with $\nabla \xi (\xb) \cdot \vb$ replaced by $\nabla \xi (x_0) \cdot v_0$, and by (\ref{alphaxi0v0}) we get the same estimate.
 
 If $|\xi ( x_0 ) | > \delta \frac{  \alpha^2 ( t,x,v) }{10(|v|^2 +\| E \|_\infty^2 + \| E \|_\infty +1)} $, then  we have
 \begin{equation}
  \frac{C_E \sigma_2 ^2}{10} =   \delta^2  \frac{C_E}{10} \frac{ (v \cdot \nabla \xi (x))^2 } { |v|^2 +\|E\|_\infty^2 + \| E \|_\infty + 1 } < C_E \delta | \xi (x_0)|< |\xi (x_0 )|,
   \end{equation}
for $\delta \ll 1 $ small enough. Therefore by (\ref{sigma2lowerbd}) and the same argument in \textit{Step 3} we get the same lower bound
\begin{equation} \label{xilowerbdx0}
 | \xi( s  ) | > \frac{C_E}{10} (\sigma_2)^2, \, \text{for all } s \in [0, t - \sigma_2 ].
\end{equation}
And therefore we get the desired estimate.

\bigskip

\textit{Step 5} \qquad
We now consider the case when $x \in \Omega$ and $ t \ge \tb$. We need to bound the integral (\ref{trajint}).
Let
\[
\sigma_1 = \delta \frac{ \vb \cdot \nabla (\xb) }{ |v|^2 + \| E\|_\infty^2 + \| E \|_\infty +1 },
\]
as defined in (\ref{sigma12lowerbd}). If
\[
\sigma_1 \ge t,
\]
then from \textit{Step 2} $|\xi (\tilde X(s) ) | $ is increasing on $[0, \tb]$,and by
(\ref{xiupperbd}), (\ref{vdotnlowerbd}), and the bound for (\textbf{I}) in (\ref{Ibd}), we get the desired estimate.

Now we assume 
\[
\sigma_1 < t.
\]
So from (\ref{sigma12lowerbd}) we have
\begin{equation} \label{sigma1lowerbd}
|\xi (  \tilde X(\sigma_1 ) ) | \ge \frac{\delta ( \vb \cdot \nabla \xi (\xb) ) ^2 }{ 2 ( |v| ^2 +\| E \|_\infty^2 + \| E \|_\infty+ 1 ) }.
\end{equation}

Now if 
\begin{equation} \label{xixlowerbd1}
|\xi ( x ) | \le \delta \frac{ \alpha^2 ( t,x,v) }{10(|v|^2 +\| E \|_\infty^2 + \| E \|_\infty +1)} ,
\end{equation}
we have
\begin{equation} 
\begin{split}
\alpha^2(t,x,v)    \le &  (\nabla \xi (x ) \cdot v ) ^2 + C(|v|^2  + \| E \|_\infty+1 ) |\xi (x )| 
\\ \le & (\nabla \xi (x ) \cdot v ) ^2 + \delta \alpha^2 (t,x,v) \le (\nabla \xi (x ) \cdot v ) ^2 + \frac{1}{10} \alpha^2(t,x,v),
\end{split}
\end{equation}
if $ \delta \ll 1 $ is small enough. So 
\begin{equation} \label{alphaxiv}
\frac{1}{2} \alpha(t,x,v) \le |\nabla \xi(x) \cdot v | .
\end{equation}

Claim: 
\[
 \nabla \xi (x ) \cdot v >0 .
 \]
Since otherwise by (\ref{accnearbdy}) we have
\[
\frac{d}{ds} | \xi (\tilde X(s) ) | > 0,
\]
for all $s \in [0,\tb]$, so $|\xi(\tilde X(s))| $ is always increasing, thus 
\[
|\xi(\tilde X(s))| \le \delta \frac{ \alpha^2 ( t,x,v) }{10(|v|^2 +\| E \|_\infty^2 + \| E \|_\infty +1)},
\]
for all $s \in [0,\tb]$, which contradicts (\ref{sigma1lowerbd}).

Therefore $\nabla \xi (x ) \cdot v > 0$, and we can run the same argument from \textit{Step 2}, \textit{Step 3}, \textit{Step 4} , and by (\ref{alphaxiv}) we get the same estimate.
 
 If 
 \Be \label{xixupperbd1}
 |\xi ( x ) | > \delta \frac{  \alpha^2 ( t,x,v) }{10(|v|^2 +\| E \|_\infty^2 + \| E \|_\infty +1)},
 \Ee
 we claim:
\begin{equation} \label{xilowerbd2}
|\xi ( \tilde X(s)  ) | \ge \delta^2 \frac{ \alpha^2(t,x,v) }{ |v|^2 +\| E \|_\infty^2 + \| E \|_\infty + 1 },
\end{equation}
for all $s \in [ \sigma_1, \tb ] $. Since otherwise let 
\[
s^* : = \min \{ s  \in [ \sigma_1, t ] :  |\xi ( \tilde X(s)  ) | < \delta^2 \frac{ \alpha^2(t,x,v) }{ |v|^2 + \| E \|_\infty^2 + \| E \|_\infty +1 } \}.
\]
From (\ref{sigma1lowerbd}) we have $s^* > \sigma_1$, and 
\[
\frac{d}{ds} | \xi (\tilde X(s^* ) ) | < 0.
\]
And from (\ref{accnearbdy}) we have
\[
\frac{d^2}{ds^2} | \xi (\tilde X(s ) ) | < 0,
\]
for all $s \in [s^*,t] $. So $| \xi (\tilde X(s ) ) | $ is always decreasing on $[s^*,\tb]$. Therefore
\[
|\xi(x) | = | \xi (\tilde X(\tb ) ) | <| \xi (\tilde X(s^* ) ) |< \delta^2 \frac{ \alpha^2(t,x,v) }{ |v|^2 + \| E \|_\infty^2 + \| E \|_\infty +1 },
\]
which contradicts (\ref{xixupperbd1}). Therefore the lower bound (\ref{xilowerbd2}) and the estimates (\ref{IIbd}), (\ref{Ibd}) gives the desired bound.

\bigskip

\textit{Step 6} \qquad
Finally we consider the case $x \in \Omega$ and $t < \tb$.
First suppose 
\[
| \xi(x) | \le \delta \frac{ \alpha^2(t,x,v) }{ 10 ( |v|^2 + \| E \|_\infty^2 + \| E \|_\infty +1 ) }.
\]
From (\ref{alphaxiv}) we have
\[
\frac{\alpha(t,x,v) }{2} \le | v \cdot \nabla \xi (x) |.
\]

If $v \cdot \nabla \xi (x) > 0 $, then by (\ref{accnearbdy}) we have $\xi ( X(t+ t')  ) = 0 $ for some $t' \lesssim \frac{\delta}{C_E ^2 } < 1$. Therefore we can extend the trajectory until it hits the boundary and conclude the desired bound from \textit{Step 3}.

If $v \cdot \nabla \xi (x) < 0 $, again by (\ref{accnearbdy}) we have $ | \xi (X(s) ) | $ is increasing on $[0,t ] $ and $|V(s) \cdot \nabla \xi (X(s) ) | $ is decreasing on $[0,t]$. Therefore using the change of variable $s \mapsto |\xi |$:
\begin{equation} \begin{split}
\int_0^{t} & e^{ -\int_{s}^t\frac{\varpi}{2} \langle V(\tau;t,x,v) \rangle d\tau }  \frac{ 1}{\left[ | V(s)|^2 \xi( X( s) ) + \xi^2( X( s) -C_E \xi (  X(s) ) \right]^{\frac{\beta - 1}{2} } }  ds
\\ \lesssim & \int_0^{\delta \frac{ \alpha^2(t,x,v) }{ 10 ( |v|^2 +\| E \|_\infty^2 + \| E \|_\infty + 1 ) }} \frac{1}{ | V(s) \cdot \nabla \xi (X(s)) |  (C_E |\xi |)^{\frac{\beta -1}{2}} }  d|\xi | 
\lesssim  \int_0^{\delta \frac{ \alpha^2(t,x,v) }{ 10 ( |v|^2 + \| E \|_\infty^2 + \| E \|_\infty +1 ) }} \frac{1}{ | v \cdot \nabla \xi(x) | (C_E |\xi |)^{\frac{\beta -1}{2}} }  d|\xi | 
\\ \lesssim  &\int_0^{\delta \frac{ \alpha^2(t,x,v) }{ 10 ( |v|^2 + \| E \|_\infty^2 + \| E \|_\infty +1 ) }} \frac{1}{ | \alpha(t,x,v)  (C_E|\xi |)^{\frac{\beta -1}{2}} }  d|\xi | 
\lesssim \frac{  \delta ^{\frac{3 -\beta}{2} }}{ C_E^{\frac{\beta -1}{2} } (\alpha(t,x,v))^{\beta -2 }  (|v|^2 + \| E \|_\infty^2 + \| E \|_\infty +1 )^{\frac{3 -\beta}{2}}},
\end{split}
\end{equation}
which is the desired estimate.

Now suppose 
\begin{equation} \label{xilowerbd3}
| \xi(x) | > \delta \frac{ \alpha^2(t,x,v) }{ 10 ( |v|^2 + \| E \|_\infty^2 + \| E \|_\infty +1 ) },
\end{equation}
and
\[
| \xi(x_0) | \le \delta \frac{ \alpha^2(t,x,v) }{ 10 ( |v|^2 + \| E \|_\infty^2 + \| E \|_\infty +1 ) }.
\]
Then by (\ref{alphaxi0v0}) we have
\begin{equation} \label{alphaxi0v02}
\frac{\alpha(t,x,v) }{2} \lesssim e^{ C_\xi \frac{\| \nabla E\|_\infty  + \|E\|_\infty^2 + \| E \|_\infty}{C_E}} |\nabla \xi(x_0) \cdot v_0 |.
\end{equation}
Now if $v_0 \cdot \nabla \xi (x_0) > 0 $, then from (\ref{accnearbdy}) we have $|\xi (X(s) ) |$ is decreasing for all $s \in [0, t] $. And this contradicts with (\ref{xilowerbd3}).
So we must have
\[
v_0 \cdot \nabla \xi (x_0) < 0.
\]
Then we can define $\sigma_1 = \delta \frac{ |v_0 \cdot \nabla \xi (x_0 ) | }{ |v|^2 + \|E\|_\infty^2 + \| E \|_\infty + 1 }$ as before. 
Now if $\sigma_1 \ge t$ then $|\xi (X(s) ) | $ is increasing on $[0,t]$, using the change of variable $x \mapsto |\xi | $ and the estimate (\ref{Ibd}) and (\ref{alphaxi0v02}) we get the desired bound.

If $\sigma_1 < t$, then from (\ref{sigma12lowerbd}) we have
\[
|\xi (X(\sigma_1) ) | \ge \delta \frac{ (v_0 \cdot \nabla \xi (x_0 ))^2 }{ 2 (|v|^2 + \|E\|_\infty^2 + \| E \|_\infty + 1 )}.
\]
And then from the argument for (\ref{xilowerbd2}) we get
\[
|\xi (  X(s)  ) | \ge \delta^2 \frac{ \alpha^2(t,x,v) }{ |v|^2 +  \|E\|_\infty^2 + \| E \|_\infty +1 },
\]
for all $s \in [\sigma_1, t]$. This lower bound combined with the estimate (\ref{IIbd}), (\ref{Ibd}) gives the desired bound.

Finally we left with the case
\[
| \xi(x_0) | > \delta \frac{ \alpha^2(t,x,v) }{ 10 ( |v|^2 + \|E\|_\infty^2 + \| E \|_\infty + 1 ) }.
\]
Then again, from the argument for (\ref{xilowerbd2}) we get
\[
|\xi (  X(s)  ) | \ge \delta^2 \frac{ \alpha^2(t,x,v) }{ |v|^2 +  \|E\|_\infty^2 + \| E \|_\infty+ 1 },
\]
for all $s \in [0, t ] $. This lower bound combined with the estimate (\ref{IIbd}) gives the desired bound.

\end{proof}
Let $\beta = 1$ in (\ref{pge2seqf}), and denote
\[
\nu^m_{\varpi} = 
\nu(\sqrt \mu f^m )+ \frac{v}{2} \cdot E  + \varpi \langle v \rangle  + t \varpi \frac{v}{\langle v \rangle }\cdot E    -  \alpha^{-1} ( \partial_t \alpha + v \cdot \nabla_x \alpha + E \cdot \nabla_v \alpha )  \ge \frac{\varpi }{2} \langle v \rangle.
\]
Then \eqref{pge2seqf} becomes
 \begin{equation} \label{seqforc1fixedpotential} \begin{split}
 \bigg \{ & \partial_t + v\cdot \nabla_x  + E \cdot \nabla_v + \nu^m_{\varpi} \bigg\} ( e^{-\varpi \langle v \rangle t } \alpha \partial f^{m+1})
\\ = &e^{-\varpi \langle v \rangle t }  \alpha \mathcal G^m:= \mathcal{N}^m(t,x,v)
\\  \lesssim &  e^{-\varpi \langle v \rangle t }  \alpha \left\{  |\partial f^{m+1} | + e^{-\frac{\theta}{2} |v|^2 } \| e^{\theta |v|^2 } f_0 \|_{\infty}^2 + P(\| e^{\theta |v|^2} f_0 \|_\infty) \times \int_{\mathbb R^3 } \frac{e^{-C_\theta |v-u|^2 }}{|v -u |^{2 - \kappa}  } | \partial f^m (u) | du  \right\}.
\end{split} \end{equation}
And for $(x,v) \in \gamma_-$, we have:
\begin{equation} \begin{split} \label{seqbdry}
 e^{-\varpi \langle v \rangle t } & \alpha | \partial  f^{m+1} (t,x,v) |  
\\  \lesssim & \sqrt {\mu (v) } \langle v \rangle ^2 \int_{n(x) \cdot u > 0} | \partial f^m(t,x,u) |\mu ^{1/4}  \langle u \rangle (n(x) \cdot u ) du +  e^{-\frac{\theta}{2}  |v| ^2} P(\| e^{\theta |v|^2 } f_0 \|_\infty).
\end{split} \end{equation}
Let $(x,v) \notin \gamma_0$ and $(t^0,x^0,v^0) = (t,x,v)$. Define the stochastic (diffuse) cycles as 
\begin{equation} \label{diffusecycles} \begin{split}
& t^1 = t - \tb(t,x,v), \, x^1 = \xb(t,x,v) = X(t - \tb(t,x,v);t,x,v), 
\\ & v_b^0 = V(t - \tb(t,x,v);t,x,v) = \vb(t,x,v),
\end{split} \end{equation}
 and $v^1 \in \mathbb R^3$ with $n(x^1) \cdot v^1 > 0$. 
For $l\ge1$, define
\[ \begin{split}
& t^{l+1} = t^l - \tb(t^l,x^l,v^l), x^{l + 1 } = \xb(t^l,x^l,v^l), 
\\ & v_b^l = \vb(t^l,x^l,v^l), 
\end{split} \]
and $v^{l+1} \in \mathbb R^3 \text{ with } n(x^{l+1}) \cdot v^{l+1} > 0$.
Also, define 
\[
X^l(s) = X(s;t^l,x^l,v^l), \, V^l(s) = V(s;t^l,x^l,v^l),
\]
 so $X(s) = X^0(s), V(s) = V^0(s)$. We have the following lemma.

\begin{lemma} \label{expandtrajl}
If $t^1 < 0$, then
\Be \label{C1trajectoryt1less0}
 e^{-\varpi \langle v \rangle t } \alpha | \partial  f^{m+1} (t,x,v) |
 \lesssim \alpha(0,X^0(0), V^0(0) ) \partial f^{m+1} (0, X^0(0) , V^0(0) ) + \int_0^t \mathcal N^m(s,X^0(s), V^0(s) ) ds.
\Ee

If $t^1 > 0$, then
\begin{equation} \label{C1trajectory}
\begin{split}
& e^{-\varpi \langle v \rangle t } \alpha | \partial  f^{m+1} (t,x,v) |
\\ \lesssim & e^{-\frac{\theta}{2}  |\vb^0| ^2} P(\| e^{\theta |v|^2 } f_0 \|_\infty) +  \int_{t^1}^t \mathcal N^m(s, X^0(s), V^0(s) ) ds
\\ & +  \sqrt {\mu (\vb^0) } \langle \vb^0 \rangle^2 \int_{\prod_{j=1}^{l-1} \mathcal V_j} \sum_{i=1}^{l-1} \textbf{1}_{\{t^{i+1} < 0 < t^i \}}  |\alpha  \partial f^{m+1-i}  (0,X^{i}(0), V^{i}(0)) |    \, d \Sigma_{i}^{l-1}
\\ & +  \sqrt {\mu (\vb^0) } \langle \vb^0 \rangle^2 \int_{\prod_{j=1}^{l-1} \mathcal V_j} \sum_{i=1}^{l-1} \textbf{1}_{\{t^{i+1} < 0 < t^i \}}  \int_0^{t^i} \mathcal N^{m-i}(s, X^i(s), V^i(s) ) ds  \, d \Sigma_{i}^{l-1}
\\ & +  \sqrt {\mu (\vb^0) } \langle \vb^0 \rangle^2 \int_{\prod_{j=1}^{l-1} \mathcal V_j} \sum_{i=1}^{l-1} \textbf{1}_{\{t^{i+1} > 0 \}}  \int_{t^{i+1}}^{t^i} \mathcal N^{m-i}(s, X^i(s), V^i(s) ) ds  \, d \Sigma_{i}^{l-1}
\\ & +  \sqrt {\mu (\vb^0) } \langle \vb^0 \rangle^2 \int_{\prod_{j=1}^{l-1} \mathcal V_j} \sum_{i=2}^{l-1} \textbf{1}_{\{t^{i} > 0 \}}   e^{-\frac{\theta}{2}  |\vb^{i-1}| ^2} P(\| e^{\theta |v|^2 } f_0 \|_\infty)  \, d \Sigma_{i-1}^{l-1}
\\ & +  \sqrt {\mu (\vb^0) } \langle \vb^0 \rangle^2  \int_{\prod_{j=1}^{l-1} \mathcal V_j}  \textbf{1}_{\{t^{l} > 0 \}}  e^{-\varpi \langle \vb^{l-1} \rangle t^l } \alpha (t^l,x^l, \vb^{l-1}) | \partial  f^{m+1-(l-1)} (t^l,x^l,\vb^{l-1}) | d \Sigma_{l-1}^{l-1},
\end{split} \end{equation}
where $\mathcal V_j = \{ v^j \in \mathbb R^3: n(x^j ) \cdot v^j > 0 \}$,
and 
\[ \begin{split}
d \Sigma_i^{l-1} = & \{\prod_{j=i+1}^{l-1}  \mu(v^j) c_\mu |n(x^j ) \cdot v^j | dv^j \} 
\{ e^{\varpi \langle v^i \rangle t^i } \mu^{1/4}(v^i) \langle v^i \rangle d v^i \}
\\ & \{\prod_{j=1}^{i-1} \sqrt{\mu(\vb^j ) } \langle \vb^j \rangle \mu^{1/4}(v^j ) \langle v^j \rangle e^{\varpi \langle v^j \rangle t^j } d v^j\},
\end{split} \]
where $c_\mu$ is the constant that $ \int_{\mathbb R^3 }  \mu(v^j) c_\mu |n(x^j ) \cdot v^j | dv^j = 1$.

\end{lemma}

\begin{proof}
For $t^1 < 0$, we use (\ref{seqforc1fixedpotential}) to obtain
\begin{equation}  \label{C1t1<0}
\begin{split}
 e & ^{-\varpi  \langle v \rangle t }  \alpha | \partial  f^{m+1} (t,x,v) |
\\ &  \le e^{ -\int_s^t \nu^m_{\varpi} (\tau, X^0(\tau), V^0(\tau) d\tau} \alpha \partial f^{m+1} (0, X^0(0) , V^0(0) ) + \int_0^t e^{ -\int_s^t \nu^m_{\varpi} (\tau, X^0(\tau), V^0(\tau) d\tau } \mathcal N^m(s,X^0(s), V^0(s) ) ds
\\ &  \le  \alpha \partial f^{m+1} (0, X^0(0) , V^0(0) ) + \int_0^t \mathcal N^m(s,X^0(s), V^0(s) ) ds.
\end{split} 
\end{equation}

Consider the case of $t^1 > 0$. We prove by induction on $l$, the number of iterations. First for $l =1$, along the characteristics, for $t^1>0 $, we have
\[ \begin{split}
 e & ^{-\varpi \langle v \rangle t }  \alpha | \partial  f^{m+1} (t,x,v) |
\\ &  \le  e^{-\varpi \langle \vb^0 \rangle t^1 } \alpha (t^1,x^1, \vb^0) | \partial  f^{m+1} (t^1,x^1,\vb^0) |  + \int_{t^1}^t \mathcal N^m(s, X^0(s), V^0(s) ) ds.
\end{split} \]

Now using diffuse boundary condition, apply (\ref{seqbdry}) to the first term above to further estimate
\[ \begin{split}
 e & ^{-\varpi  \langle v \rangle t }  \alpha | \partial  f^{m+1} (t,x,v) |
\\  \lesssim & \sqrt {\mu (\vb^0) } \langle \vb^0 \rangle^2 \int_{\mathcal V_1} | \partial f^m(t,x^1,v^1) |\mu ^{1/4}(v^1) \langle v^1 \rangle   (n(x^1) \cdot v^1) dv^1 
\\ &  +  e^{-\frac{\theta}{2}  |\vb^0| ^2} P(\| e^{\theta |v|^2 } f_0 \|_\infty) +  \int_{t^1}^t \mathcal N^m(s, X^0(s), V^0(s) ) ds
\\  = & \sqrt {\mu (\vb^0) } \langle \vb^0 \rangle^2 \int_{\mathcal V_1}   e ^{-\varpi \langle v^1 \rangle t^1 } \alpha(t^1,x^1,v^1) | \partial f^m(t,x^1,v^1) | e ^{\varpi  \langle v^1 \rangle t^1 }\mu ^{1/4}(v^1) \langle v^1 \rangle   dv^1 
\\ &  +  e^{-\frac{\theta}{2}  |\vb^0| ^2} P(\| e^{\theta |v|^2 } f_0 \|_\infty) +  \int_{t^1}^t \mathcal N^m(s, X^0(s), V^0(s) ) ds.
\end{split} \]

Now we continue to express $e ^{-\varpi  \langle v^1 \rangle t^1 } \alpha(t^1,x^1,v^1) | \partial f^m(t,x^1,v^1) | $ via backward trajectory to get
\[ \begin{split}
e & ^{-\varpi  \langle v^1 \rangle t^1 } \alpha(t^1,x^1,v^1) | \partial f^m(t,x^1,v^1) |
\\ \le & \textbf{1}_{ \{ t^2 < 0 < t^1 \}} \left\{ \alpha(0, X^1(0), V^1(0) ) | \partial f^m(0,X^1(0), V^1(0)) | + \int_0^{t^1} \mathcal N^{m-1}(s, X^1(s), V^1(s) ) ds   \right\}
\\ & + \textbf{1}_{\{t^2 > 0 \}} \left\{e^{-\varpi \langle \vb^1 \rangle t^2 } \alpha (t^2,x^2, \vb^1) | \partial  f^{m} (t^2,x^2,\vb^1) |  + \int_{t^2}^{t^1} \mathcal N^{m-1}(s, X^1(s), V^1(s) ) ds    \right\}.
\end{split} \]

Plug into the previous inequality we conclude that
\[ \begin{split}
e & ^{-\varpi \langle v \rangle t }  \alpha(t,x,v) | \partial  f^{m+1} (t,x,v) |
\\ \lesssim & e^{-\frac{\theta}{2}  |\vb^0| ^2} P(\| e^{\theta |v|^2 } f_0 \|_\infty) +  \int_{t^1}^t \mathcal N^m(s, X^0(s), V^0(s) ) ds
\\ & +  \sqrt {\mu (\vb^0) } \langle \vb^0 \rangle^2 \int_{ \mathcal V_1}  \textbf{1}_{\{t^2 < 0 < t^1\}}  \alpha(0, X^1(0), V^1(0) ) | \partial f^m(0,X^1(0), V^1(0)) |   \times e ^{\varpi \langle v^1 \rangle t^1 }\mu ^{1/4}(v^1) \langle v^1 \rangle   dv^1 
\\ & +  \sqrt {\mu (\vb^0) } \langle \vb^0 \rangle^2 \int_{ \mathcal V_1}  \textbf{1}_{\{t^2 < 0 < t^1\}}  \int_0^{t^1} \mathcal N^{m-1}(s, X^1(s), V^1(s) ) ds  \times e ^{\varpi  \langle v^1 \rangle t^1 }\mu ^{1/4}(v^1) \langle v^1 \rangle   dv^1 
\\ & +  \sqrt {\mu (\vb^0) } \langle \vb^0 \rangle^2 \int_{ \mathcal V_1}  \textbf{1}_{\{t^2 > 0\}} \int_{t^2}^{t^1} \mathcal N^{m-1}(s, X^1(s), V^1(s) ) ds  \times e ^{\varpi  \langle v^1 \rangle t^1 }\mu ^{1/4}(v^1) \langle v^1 \rangle   dv^1 
\\ & +  \sqrt {\mu (\vb^0) } \langle \vb^0 \rangle^2 \int_{ \mathcal V_1}  \textbf{1}_{\{t^2 > 0\}}e^{-\varpi \langle \vb^1 \rangle t^2 } \alpha (t^2,x^2, \vb^1) | \partial  f^{m} (t^2,x^2,\vb^1) |  \times e ^{\varpi  \langle v^1 \rangle t^1}\mu ^{1/4}(v^1) \langle v^1 \rangle   dv^1,
\end{split} \]
and it equals (\ref{C1trajectory}) for $l =2 $.

Assume (\ref{C1trajectory}) is valid for $l \ge 2$. We use diffuse boundary condition (\ref{seqbdry}) to express the integrand of the last term of (\ref{C1trajectory}) as
\[ \begin{split}
 \textbf{1}_{\{t^l > 0\}} & e^{-\varpi\langle \vb^{l-1} \rangle t^l } \alpha (t^l,x^l, \vb^{l-1}) | \partial  f^{m+1 - (l-1)} (t^l,x^l,\vb^{l-1}) | 
 \\ \lesssim &  \sqrt {\mu (\vb^{l-1}) } \langle \vb^{l-1} \rangle^2 \int_{\mathcal V_l}  \textbf{1}_{\{t^l > 0\}}   e ^{-\varpi \langle v^l \rangle t^l } \alpha(t^l,x^l,v^l) | \partial f^{m+1-l}(t^l,x^l,v^l) | 
 \\ & \times e ^{\varpi \langle v^l \rangle t^l }\mu ^{1/4}(v^l) \langle v^l \rangle   dv^l 
 +  e^{-\frac{\theta}{2}  |\vb^{l-1}| ^2} P(\| e^{\theta |v|^2 } f_0 \|_\infty).
\end{split} \]
Then we decompose $ \textbf{1}_{\{t^l > 0\}} =  \textbf{1}_{\{ t^{l+1} < 0 < t^l\}} +  \textbf{1}_{\{t^{l+1} > 0\}} $,and estimate via backward trajectory to get:
\[ \begin{split}
 &  \textbf{1}_{\{t^l > 0\}} e ^{-\varpi \langle v^l \rangle t^l } \alpha(t^l,x^l,v^l) | \partial f^{m+1-l}(t^l,x^l,v^l) | 
\\ \le & \textbf{1}_{\{ t^{l+1} < 0 < t^l\}} \left\{\alpha (0 , X^l(0), V^l(0) ) | \partial f^{m+1- l}(0, X^l(0), V^l(0) ) | + \int_0^{t^l} \mathcal N^{m+1 - (l +1 ) } ( s, X^l(s) , V^l(s) ) ds \right \}
\\  + & \textbf{1}_{\{t^{l+1} > 0\}} \Bigg\{   e ^{-\varpi \langle \vb^l \rangle t^{l+1} } \alpha(t^{l+1},x^{l+1},\vb^l) | \partial f^{m+1- l}(t^{l+1},x^{l+1},\vb^l) |   
 + \int_{t^{l+1}}^{t^l} \mathcal N^{m+1 - (l +1 ) } ( s, X^l(s) , V^l(s) ) ds \Bigg\}.
\end{split} \]

Plug this into the previous inequality and integrate over $\prod_{j=1}^{l-1} \mathcal V_j $, we obtain a bound for the last term of (\ref{C1trajectory}) as
\[ \begin{split}
 \sqrt {\mu (\vb^0) } & \langle \vb^0 \rangle^2  \int_{\prod_{j=1}^{l-1} \mathcal V_j}  \textbf{1}_{\{t^{l} > 0 \}}  e^{-\varpi \langle \vb^{l-1} \rangle t^l } \alpha (t^l,x^l, \vb^{l-1}) | \partial  f^{m+1-(l-1)} (t^l,x^l,\vb^{l-1}) |  \, d \Sigma_{l-1}^{l-1}
 \\ \lesssim & \sqrt {\mu (\vb^0) }  \langle \vb^0 \rangle^2 \int_{\prod_{j=1}^{l-1} \mathcal V_j}  \textbf{1}_{\{t^{l} > 0 \}} e^{-\frac{\theta}{2}  |\vb^{l-1}| ^2} P(\| e^{\theta |v|^2 } f_0 \|_\infty) \, d \Sigma_{l-1}^{l-1}
 \\ & + \sqrt {\mu (\vb^0) }  \langle \vb^0 \rangle^2 \int_{\prod_{j=1}^{l} \mathcal V_j}  \textbf{1}_{\{t^{l+1} < 0 < t^l \}} \alpha (0 , X^l(0), V^l(0) ) | \partial f^{m+1-l}(0, X^l(0), V^l(0) ) |
 \\ & \times e ^{\varpi \langle v^l \rangle t^l }\mu ^{1/4}(v^l) \langle v^l \rangle   dv^l  \sqrt {\mu (\vb^{l-1}) } \langle \vb^{l-1} \rangle^2  \, d \Sigma_{l-1}^{l-1}
  \\ & + \sqrt {\mu (\vb^0) }  \langle \vb^0 \rangle^2 \int_{\prod_{j=1}^{l} \mathcal V_j}  \textbf{1}_{\{t^{l+1} < 0 < t^l \}} \int_0^{t^l} \mathcal N^{m+1 - (l +1 ) } ( s, X^l(s) , V^l(s) ) ds 
  \\ & \times e ^{\varpi \langle v^l \rangle t^l }\mu ^{1/4}(v^l) \langle v^l \rangle   dv^l  \sqrt {\mu (\vb^{l-1}) } \langle \vb^{l-1} \rangle^2  \, d \Sigma_{l-1}^{l-1}
  \\ & + \sqrt {\mu (\vb^0) }  \langle \vb^0 \rangle^2 \int_{\prod_{j=1}^{l} \mathcal V_j}  \textbf{1}_{\{t^{l+1} >0 \}} e ^{-\varpi \langle \vb^l \rangle t^{l+1} } \alpha(t^{l+1},x^{l+1},\vb^l) | \partial f^{m+1- l}(t^{l+1},x^{l+1},\vb^l) |  
  \\ & \times e ^{\varpi \langle v^l \rangle t^l }\mu ^{1/4}(v^l) \langle v^l \rangle   dv^l  \sqrt {\mu (\vb^{l-1}) } \langle \vb^{l-1} \rangle^2  \, d \Sigma_{l-1}^{l-1}
    \\ & + \sqrt {\mu (\vb^0) }  \langle \vb^0 \rangle^2 \int_{\prod_{j=1}^{l} \mathcal V_j}  \textbf{1}_{\{t^{l+1} >0 \}} \int_{t^{l+1}}^{t^l} \mathcal N^{m+1 - (l +1 ) } ( s, X^l(s) , V^l(s) ) ds  
  \\ & \times e ^{\varpi \langle v^l \rangle t^l }\mu ^{1/4}(v^l) \langle v^l \rangle   dv^l  \sqrt {\mu (\vb^{l-1}) } \langle \vb^{l-1} \rangle^2  \, d \Sigma_{l-1}^{l-1}
   \\ = &  \sqrt {\mu (\vb^0) }  \langle \vb^0 \rangle^2 \int_{\prod_{j=1}^{l} \mathcal V_j}  \textbf{1}_{\{t^{l} > 0 \}} e^{-\frac{\theta}{2}  |\vb^{l-1}| ^2} P(\| e^{\theta |v|^2 } f_0 \|_\infty) \, d \Sigma_{l-1}^{l}
    \\ & + \sqrt {\mu (\vb^0) }  \langle \vb^0 \rangle^2 \int_{\prod_{j=1}^{l} \mathcal V_j}  \textbf{1}_{\{t^{l+1} < 0 < t^l \}} \alpha (0 , X^l(0), V^l(0) ) | \partial f^{m+1-l}(0, X^l(0), V^l(0) ) | \, d \Sigma_{l}^{l}
    \\ & + \sqrt {\mu (\vb^0) }  \langle \vb^0 \rangle^2 \int_{\prod_{j=1}^{l} \mathcal V_j}  \textbf{1}_{\{t^{l+1} < 0 < t^l \}} \int_0^{t^l} \mathcal N^{m+1 - (l +1 ) } ( s, X^l(s) , V^l(s) ) ds \, d \Sigma_{l}^{l}
\\ &  + \sqrt {\mu (\vb^0) }  \langle \vb^0 \rangle^2 \int_{\prod_{j=1}^{l} \mathcal V_j}  \textbf{1}_{\{t^{l+1} >0 \}} \int_{t^{l+1}}^{t^l} \mathcal N^{m+1 - (l +1 ) } ( s, X^l(s) , V^l(s) ) ds \, d \Sigma_{l}^{l}
   \\ & + \sqrt {\mu (\vb^0) }  \langle \vb^0 \rangle^2 \int_{\prod_{j=1}^{l} \mathcal V_j}  \textbf{1}_{\{t^{l+1} >0 \}} e ^{-\varpi \langle \vb^l \rangle t^{l+1} } \alpha(t^{l+1},x^{l+1},\vb^l) | \partial f^{m+1- l}(t^{l+1},x^{l+1},\vb^l) |   \, d \Sigma_{l}^{l}.
\end{split} \]
Adding this to (\ref{C1trajectory}) we conclude the lemma.

\end{proof}

\begin{lemma} \label{ltailbd}
Let $0< T <1$, then there exists $l_0 \gg 1$ such that for $l \ge l_0$ and for all $(t,x,v) \in [0,T] \times \bar \Omega \times \mathbb R^3$,  we have
\Be \label{trajexpansionendterm}
\int_{\prod_{j=1}^{l-1} \mathcal V_j} \textbf{1}_{\{t^{l}(t,x,v,v^1,...,v^{l-1} ) >0 \}} \, d \, \Sigma_{l-1}^{l-1} \lesssim_{\Omega, \|E\|_\infty} \left( \frac{1}{2} \right)^{l}.
\Ee

\end{lemma}

\begin{proof}
First, since
\[
|\vb^j|^2 \lesssim |v^j |^2 + t^2 \| E \|_\infty^2, \, \langle \vb^j \rangle \lesssim \langle v^j \rangle + t \| E \|_\infty,
\]

for some fixed constant $C_{0 } > 0$, 
\[ \begin{split}
 \, d \, \Sigma_{l-1}^{l-1} = & e^{\varpi \langle v^{l-1} \rangle t^{l-1} } \mu^{1/4}(v^{l-1}) \langle v^{l-1} \rangle d v^{l-1}  \prod_{j=1}^{l-2} \sqrt{\mu(\vb^j ) } \langle \vb^j \rangle \mu^{1/4}(v^j ) \langle v^j \rangle e^{\varpi \langle v^j \rangle t^j } d v^j
 \le (C_{0 })^l \prod_{j=1}^{l-1} \mu^{1/8}(v^j) dv^j.
\end{split} \]
Choose a sufficiently small $\delta = \delta (C_0) > 0$. Define
\[
\mathcal V_j^\delta = \{ v^j \in \mathcal V_j: v^j \cdot n(x^j) \ge \delta, |v^j | \le \delta^{-1}\},
\]
where we have $\int_{\mathcal V_j \setminus \mathcal V_j^\delta} C_0 \mu^{1/8} (v^j ) dv^j \lesssim \delta$.

On the other hand if $v^j \in \mathcal V_j^\delta$, we claim that $ (t^j - t^{j+1}) \gtrsim \delta^3$.

Since $\Omega$ is $C^2$ and convex, we have $|x - y|^2 \gtrsim_\Omega |(x -y ) \cdot n(x) | $ for all $x,y \in \partial \Omega$. Thus
\[ \begin{split}
|\int_{t^{j+1}}^{t^j} V^j(s) ds | ^2 & = | x^{j+1} - x^j |^2 \gtrsim |(x^{j+1} - x^j) \cdot n(x^j)|   = | \int_{t^{j+1}}^{t^j} V^j(s)  \cdot n(x^j)  | 
\\ & \ge | v^j \cdot n(x^j) | (t ^j - t^{j+1} ) - | \int_{t^{j+1}}^{t^j} \int_{t^{j+1}}^s E^j ( \tau) \cdot n(x^j) d\tau ds |.
\end{split} \]
Therefore
\[ \begin{split}
\frac{1}{ t^j - t^{j+1}} \left( |\int_{t^{j+1}}^{t^j} V^j(s) ds | ^2 + | \int_{t^{j+1}}^{t^j} \int_{t^{j+1}}^s E^j ( \tau) \cdot n(x^j) d\tau ds | \right) \gtrsim  | v^j \cdot n(x^j) | > \delta.
\end{split} \]
But
\[ \begin{split}
\frac{1}{ t^j - t^{j+1}} & \left( |\int_{t^{j+1}}^{t^j} V^j(s) ds | ^2 + | \int_{t^{j+1}}^{t^j} \int_{t^{j+1}}^s E^j ( \tau) \cdot n(x^j) d\tau ds | \right)
\\ \le & \frac{1}{ t^j - t^{j+1}} \left[ (t^j - t^{j+1}) ^2  |v^j | ^2 +(t^j - t^{j+1} ) ^4 \| E\|_\infty ^2 + (t^j - t^{j+1} ) ^2 \| E \|_\infty^2 \right]
\\ \le   & (t^j - t^{j+1}) ( \delta^{-2} + \| E \|_\infty ^2 ) +  (t^j - t^{j+1}) ^3 \| E\|_\infty^2 
\\ \le &  (t^j - t^{j+1})( \delta^{-2} + \| E \|_\infty ^2 + t^2 \| E \|_\infty ^2)
\\ \le & (t^j - t^{j+1})( 2 \delta^{-2} ).
\end{split} \]
Therefore 
\Be \label{lowerbddfortjtjplus1}
(t^j - t^{j+1} ) \ge \frac{ \delta ^3} { C_\Omega ( 1 + \delta^2 \| E \|_\infty^2 ) },
\Ee
so $(t^j - t^{j+1} ) \ge \delta ^3 / C_\Omega$ if we choose $\delta < \frac{1}{\| E \|_\infty} $.

Now if $t^l \ge 0$ then there are at most $ \left[ \frac{ C_\Omega }{ \delta ^3 }  \right] + 1$ numbers of $v^m \in \mathcal V^\delta_m$ for $1 \le m \le l-1$. Equivalently there are at least $l -2 - \left[ \frac{ C_\Omega }{ \delta ^3 }  \right]$ numbers of $v^{m_i} \in \mathcal V_{m_i} \setminus \mathcal V_{m_i }^\delta$. Therefore we have:
\Be \label{tailtermsmall} \begin{split}
\int_{\prod_{j=1}^{l-1} \mathcal V_j} & \textbf{1}_{\{t^{l}(t,x,v,v^1,...,v^{l-1} ) >0 \}} \, d \, \Sigma_{l-1}^{l-1}
\\ \le & \sum_{m = 1}^{  \left[ \frac{ C_\Omega }{ \delta ^3 }  \right] +1} \int_{ \left\{ \parbox{15em}{there are exactly $m$ of $v^{m_i} \in \mathcal V_{m_i}^\delta $ and $l -1 -m $ of $v^{m_i} \in \mathcal V_{m_i } \setminus \mathcal V^\delta_{m_i}$ } \right\} } \prod_{j =1}^{l-1} C_0 \mu^{1/8}(v^j) d v^j
\\ \le & \sum_{m = 1}^{  \left[ \frac{ C_\Omega }{ \delta ^3 }  \right]+1} {l-1  \choose m} \left\{ \int_{\mathcal V} C_0 \mu^{1/8} (v) dv \right\}^m \left\{ \int_{\mathcal V \setminus \mathcal V^\delta} C_0 \mu^{1/8} (v) dv \right\}^{l-1-m}
\\ \le & \left(  \left[ \frac{ C_\Omega }{ \delta ^3 }  \right] + 1 \right) ( l -1)^{ \left[ \frac{ C_\Omega }{ \delta ^3 }  \right] + 1} ( \delta)^{l - 2 -  \left[ \frac{ C_\Omega }{ \delta ^3 }  \right] } \left\{ \int_{\mathcal V} C_0 \mu^{1/8} (v) dv \right\}^{ \left[ \frac{ C_\Omega }{ \delta ^3 }  \right]+1}
\\ \le & C \delta^{l/2} \le C (\frac{1}{2})^{l},
\end{split} \Ee

if $l \gg 1$, say $ l = 2 \left(  \left[ \frac{ C_\Omega }{ \delta ^3 }  \right] +1 \right) ^2$.

\end{proof}

\begin{proof} [Proof of Theorem \ref{C1linearthm}]
By the Duhamel's formulation, we use \ref{seqforc1fixedpotential} to estimate $|e^{-\varpi \langle v \rangle t  } \alpha \partial f^{m+1}|$ along the characteristic in a bulk,
then from (\ref{genlemma1}), (\ref{C1trajectoryt1less0}), (\ref{C1trajectory}), and (\ref{trajexpansionendterm}) we can carry the same argument as in the proof of (\ref{c1bddsequence}) to get
\Be \label{W1inftyestimateexternalp}
\sup_{m} \sup_{0 \le t \le T}  \|  e^{-\varpi  \langle v \rangle t } \alpha  \partial  f^{m} (t,x,v) \|_\infty \lesssim  P(\| e^{\theta |v|^2} f_0\|_\infty ) + \|  \alpha \partial f_0 \|_\infty  < \infty.
\Ee
Then by passing the limit and the weak-$*$ lower-semi continuity of $L^\infty$, we conclude (\ref{weightedW1inftyforexternalp}).

Now we consider the continuity of $e^{-\varpi \langle v \rangle t  } \alpha \partial f$. From the explicit formulas of $\partial f^m$ from (\ref{gradientfexpalongtraj}) and the assumption that $\alpha \nabla f_0 \in C^0$, we have $e^{-\varpi \langle v \rangle t  } \alpha \partial f^m \in C^0([0,T] \times (\bar \Omega \times \mathbb R^3 ) \setminus \gamma_0 ) $. Now since $e^{-\varpi \langle v \rangle t  } \alpha[ \partial f^{m+1} - \partial f^m ]$ satisfies
\begin{equation} \label{} \begin{split}
 \bigg \{ & \partial_t + v\cdot \nabla_x  + E \cdot \nabla_v + \nu(\sqrt \mu ( f^m - f^{m-1} ))
 \\& - \frac{v}{2} \cdot \nabla E   + \varpi \langle v \rangle + \varpi \frac{v}{\langle v \rangle} \cdot E t  -\alpha ^{-1}( \partial_t \alpha + v \cdot \nabla_x \alpha + E \cdot \nabla_v \alpha ) \bigg\} (  e^{-\varpi \langle v \rangle t  } \alpha ( \partial f^{m+1} - \alpha f^m))
\\ = &   e^{-\varpi \langle v \rangle t}  \alpha ( \mathcal G^m -\mathcal G^{m-1}).
\end{split} \end{equation}
We can follow the $W^{1,\infty}$ estimate from (\ref{W1inftyestimateexternalp}) for $e^{-\varpi \langle v \rangle t  } \alpha[ \partial f^{m+1} - \partial f^m ]$ to show that $e^{-\varpi  \langle v \rangle t } \alpha  \partial  f^{m}$ is a Cauchy sequence in $L^\infty$. Thus $e^{-\varpi  \langle v \rangle t } \alpha  \partial  f^{m} \rightarrow e^{-\varpi  \langle v \rangle t } \alpha  \partial  f^{} $ strongly in $L^\infty$ so that $e^{-\varpi  \langle v \rangle t } \alpha  \partial  f^{} \in C^0([0,T] \times (\bar \Omega \times \mathbb R^3 ) \setminus \gamma_0 )  $.

\end{proof}

\bigskip

\section{Weighted $W^{1,\infty}$ estimate for the Vlasov-Poisson-Boltzmann equation}

In this chapter we construct the local-in-time weighted $W^{1,\infty}$ solution of the system (\ref{Bextfield1}), (\ref{VPB2}), (\ref{VPB3}).



Let $f^0 = \sqrt \mu$. We start with the sequence for $m \ge 0$
\begin{equation} \label{VPBsq1}
(\partial_t + v \cdot \nabla_x - \nabla \phi^m \cdot \nabla_v  + \frac{v}{2} \cdot \nabla \phi^m + \nu( \sqrt \mu f^m) ) f^{m+1} = \Gamma_{\text{gain}} (f^m,f^m),
\end{equation}
\begin{equation} \label{VPBsq2}
\phi^m(t,x) = \phi_{F^m}(t,x) + \phi_E(t,x), \text{  } \frac{\p \phi_E}{\p n } > C_E > 0 \text{ on } \p \Omega,
\end{equation}
\begin{equation} \label{VPBsq3}
-\Delta_x \phi_{F^m}(t,x) = \int_{\mathbb R^3 } \sqrt \mu f^m dv - \rho_0, \,\,  \frac{\partial \phi_{F^m}}{\partial n } = 0 \text{  on } \partial \Omega,
\end{equation}
with the initial data $f^m(0,x,v) = f_0(x,v)$, and boundary condition for all $(x,v) \in \gamma_-$ be
\[
\begin{split}
f^1(t,x,v) & =  c_\mu \sqrt{\mu(v)} \int_{n\cdot u >0 } f_0(x,v) \sqrt{\mu(u)} (n(x) \cdot u ) du,
\\ f^{m+1}(t,x,v) & = c_\mu \sqrt{\mu(v)} \int_{n\cdot u >0 } f^m(t,x,v) \sqrt{\mu(u)} (n(x) \cdot u ) du, \, m \ge 1.
\end{split}
\]
Now let $\partial \in \{ \nabla_x, \nabla_v \}$. Taking $ \partial [ (\ref{VPBsq1}) ] $ we have
\begin{equation} \begin{split} \label{traneqforpartialf}
& (\partial_t + v \cdot \nabla_x -\nabla \phi^m \cdot \nabla_v + \frac{v}{2} \cdot  \nabla \phi^m + \nu(\sqrt \mu f^m) ) \partial f^{m+1} 
\\ = & \partial \Gamma_{gain} (f^m,f^m) - \partial v \cdot \nabla_x f^{m+1} + \partial \nabla \phi^m \cdot \nabla_v f^{m+1} - \partial ( \frac{v}{2} \cdot \nabla \phi^m ) f^{m+1} - \partial (\nu (\sqrt \mu f^m ) ) f^{m+1}  \\ : = & \mathcal G^m.
\end{split} \end{equation}
Let $X^m(s;t,x,v) ,V^m(s;t,x,v)$ be the position and velocity at time $s$ of the trajectory starting from $(t,x,v)$ corresponding to the potential $-\nabla \phi^m$. So it satisfies
\[
\frac{d X^m(s;t,x,v)}{ds} = V^m(s;t,x,v), \, \frac{d V^m(s;t,x,v)}{ds} = -\nabla \phi^m(s, X^m(s;t,x,v)).
\]
Also denote:
\[ \begin{split}
 t^1 = t - \tb(t,x,v), x^1 =  X^m(t^1;t,x,v), \vb^0 = V^m( t^1;t,x,v), \text{ and }  v^1 \in \mathbb R^3 \text{ with } n(x^1 ) \cdot v^1 > 0,
\end{split} \]
and inductively for $k \ge 1$,
\[ \begin{split}
&  t^{k + 1 } = t^k - \tb(t^k,x^k,v^k), x^{k+1} = X^{m-(k-1)}(t^{k+1};t^k,x^k,v^k), 
\\ & \vb^k =V^{m-k}(t^{k+1};t^k,x^k,v^k),  \text{ and } v^{k+1} \in \mathbb R^3 \text{ with } n(x^{k+1}) \cdot v^{k+1} > 0.
\end{split} \]
Before the local existence let's first prove the following lemma:

\begin{lemma}[]
If $(f,\phi_F)$ solves (\ref{VPB3}), then
\begin{equation} \label{linfinitybdofpotential}
\| \phi_F(t) \|_{C^{1,1-\delta}} \lesssim_{\delta,\Omega} \| e^{\theta |v|^2 } f(t) \|_\infty, \text{ for any } 0< \delta <1,
\end{equation}
and
\begin{equation} \label{c2bdofpotential}
\| \nabla ^2 \phi_F(t) \|_\infty \lesssim \| e^{\theta |v|^2 } f(t) \|_\infty + \| e^{-\varpi \langle v \rangle t } \alpha \nabla_{x} f (t) \|_\infty.
\end{equation}
\end{lemma}

\begin{proof}
For any $p>3$, from Morrey inequality and elliptic estimate we have
\[
\| \phi_F(t) \|_{C^{1,1-3/p}} \lesssim_{p,\Omega} \| \phi_F(t) \|_{W^{2,p} (\Omega) } \lesssim \| \int_{\mathbb R^3 } f (t,x,v) \sqrt{\mu(v) } dv - \rho_0 \|_{L^p (\Omega) } \lesssim \| e^{\theta |v|^2} f(t) \|_\infty.
\]
Let $p = 3/\delta$ we conclude (\ref{linfinitybdofpotential}). 

Next we show (\ref{c2bdofpotential}).
By Schauder estimate, we have, for $p >3$ and $\Omega \subset \mathbb R^3$,
\[
\| \nabla^2 \phi_F(t) \|_\infty \le \| \phi_F \|_{C^{2, 1 - \frac{3}{p}}} \lesssim_{p,\Omega} \| \int_{\mathbb R^3} f (t) \sqrt \mu dv \|_{C^{0, 1 - \frac{3}{p}}}.
\]
Then by Morrey inequality, $W^{1,p} \subset C^{0, 1 -\frac{3}{p}} $ with $p >3$ for a domain $\Omega \subset \mathbb R^3$ with a smooth boundary $\partial \Omega$, we derive
\[ \begin{split}
 \| & \int_{\mathbb R^3} f (t) \sqrt \mu dv \|_{C^{0, 1 - \frac{3}{p}}} 
 \\ & \lesssim  \| \int_{\mathbb R^3} f (t) \sqrt \mu dv \|_{W^{1,p} }
 \\ & \lesssim \| e^{\theta |v|^2 } f(t) \|_\infty (\int_{\mathbb R^3 } \sqrt \mu e^{-\theta |v|^2 } dv) + \| \int_{\mathbb R^3} \nabla_x f(t) \sqrt \mu dv \|_{L^p(\Omega ) }
 \\ & \lesssim  \| e^{\theta |v|^2 } f(t) \|_\infty +  \| e^{-\varpi \langle v \rangle t } \alpha \nabla_{x} f (t) \|_\infty \| \int_{\mathbb R^3 } e^{ \varpi \langle v \rangle t } \sqrt \mu \frac{1}{\alpha} dv \|_{L^p(\Omega ) }.
\end{split} \]
Note that $e^{ \varpi \langle v \rangle t } \sqrt \mu  \le e^{-\frac{1}{8} |v|^2 } $ for $|v| \gg 1 $. So we only need to show that
\begin{equation} \label{int1alphaLpbdd}
\| \int_{\mathbb R^3 } e^{-\frac{1}{8} |v|^2 }  \frac{1}{\alpha} dv \|_{L^p(\Omega ) } < \infty.
\end{equation}
Since $\frac{1}{\alpha} \lesssim \frac{1}{\alpha^\beta} +1 $ for $\beta > 1 $. It suffices to show that $\| \int_{\mathbb R^3 } e^{-\frac{1}{8} |v|^2 }  \frac{1}{\alpha^\beta} dv \|_{L^p(\Omega ) } < \infty$ for some $\beta > 1$.

Since $\alpha$ is bounded from below when $x$ is away from the boundary of $\Omega$, it suffices to only consider the case when $x$ is close enough to $\partial \Omega$. From the computation in (\ref{frac1alphaint}), we get
\begin{equation} \label{int1overalphadv}
 \int_{\mathbb R^3 } e^{-\frac{1}{8} |v|^2 }  \frac{1}{\alpha^\beta} dv \lesssim \frac{1}{(\xi(x)^2 - 2 E(t,\bar x ) \cdot \nabla \xi(\bar x) \xi(x) )^{\frac{\beta -1 }{2}}} \lesssim \frac{1}{ |\xi(x) | ^{\frac{\beta -1 }{2}}}.
\end{equation}
So it suffices to show
\begin{equation} \label{int1overxi}
\int_{d(x,\partial \Omega) \ll 1 } \frac{1}{ |\xi(x) | ^{\frac{(\beta -1 )p}{2}}} dx < \infty.
\end{equation}
Since $\xi(x) = \xi(\bar x ) + \nabla \xi(x')(x - \bar x ) =  \nabla \xi(x')(x - \bar x )$ for some $x'$ in between $x$ and $\bar x$. And $|\nabla \xi(x) | > c $ for $d(x,\partial \Omega) \ll 1 $ by our assumption on $\xi$, we have
\[
|\xi(x)| = | \nabla \xi(x') | | x- \bar x | \cos (\theta) > c | x- \bar x | \cos (\theta),
\]
where $\theta$ is the angle between the vectors $\nabla \xi (x') $ and $x - \bar x$. And since $\bar x $ satisfies $(x - \bar x)^2 = \min_{ \{ y \in \mathbb R^3 : \xi(y) = 0 \}} (x-y)^2 $. From lagrange multiplier we have the vectors $x -\bar x $ and $\nabla \xi(\bar x )$ are parallel to each other. Therefore $\theta$ is the angle in between $\nabla \xi (x') $ and $\nabla \xi (\bar x)$. And since $\xi$ is $C^2$, we have $\cos(\theta) > \frac{1}{2}$ once $d(x, \partial \Omega) \ll 1$.
Thus
\[
\int_{d(x,\partial \Omega) \ll 1 } \frac{1}{ |\xi(x) | ^{\frac{(\beta -1 )p}{2}}} dx \lesssim \int_{d(x,\partial \Omega) \ll 1 } \frac{1}{ |x - \bar x | ^{\frac{(\beta -1 )p}{2}}} dx.
\]
Now from (\ref{eta}), for any $p \in \partial \Omega$ we can locally define the parametrization:
\[ \begin{split}
\eta_p: \{( x_{\|,1}, x_{\|,2} , x_n ) \in \mathbb R^3 : x_n > 0 \} \cap B(0; \delta_1 ) \to \Omega \cap B(p;\delta_2 );
\\ ( x_{\|,1}, x_{\|,2} , x_n ) \mapsto \eta_p( x_{\|,1}, x_{\|,2} , x_n ),
\\ \eta_p ( x_{\|,1}, x_{\|,2} , x_n ) = \eta_p( x_{\|,1}, x_{\|,2} , 0 ) + x_n [-n (\eta_p( x_{\|,1}, x_{\|,2} , 0))],
\end{split} \]
with $\eta_p( x_{\|,1}, x_{\|,2} , 0 ) \in \partial \Omega$, for sufficiently small $\delta_1,\delta_2 \ll 1$. Then
\[
\int_{\Omega \cap B(p;\delta_2 ) }  \frac{1}{|x - \bar x  | ^{\frac{(\beta -1 )p}{2}}} dx \lesssim \int_{ |x_n | < \delta_1 } \frac{1}{|x_n|^{\frac{(\beta -1 )p}{2}}} d_{x_n} < \infty,
\]
if we pick $\beta < \frac{2}{p}+1 $. And since $\partial \Omega$ is compact, we can get (\ref{int1overxi}) by covering $\partial \Omega$ with finitely many such balls. And therefore we get (\ref{int1alphaLpbdd}).

\end{proof}

\begin{proof}[proof of Theorem \ref{WlinftyVPBthm}]


\textit{Step 1.} For the sequence (\ref{VPBsq1}), we claim that there exists a $C_1 \gg 1$ large enough and $0< T \ll 1$ small enough such that if we let $\theta' = \theta - T$,

\begin{equation} \label{uniformmlinftybdd}
\sup_m \sup_{ 0 \le t \le T}  \|  e^{\theta' |v|^2} f^m(t,x,v) \|_\infty \le \sup_m \sup_{ 0 \le t \le T}  \|  e^{(\theta-t) |v|^2} f^m(t,x,v) \|_\infty <C_1 \| e^{\theta |v|^2 } f_0 \|_\infty .
\end{equation}
Suppose (\ref{uniformmlinftybdd}) is true for all $0 \le i \le m$. 
Then from (\ref{linfinitybdofpotential}) we have
\begin{equation}
\sup_m \sup_{0 \le t \le T} \| \nabla \phi^m (t) \|_\infty < C_\Omega C_1 \| e^{\theta |v|^2 } f_0 \|_\infty< M.
\end{equation}
Then if we choose 
\Be \label{choiceofT}
T< \frac{1}{ 2(M^2 + M + 1)},
\Ee
we have $| V^i(s;t,x,v) | \le |v| + t \| \nabla \phi^i  \|_\infty < |v| + 1 $, and 
\begin{equation} \label{expintvdotphi}
 \int_0^t |\frac{V(s)}{2} \cdot \nabla \phi ^m (s)  | ds  < M \int_0^t( |v| + t M) ds  < t M |v| + t^2 M^2 < \langle v \rangle,
\end{equation}
and from (\ref{intVbdd})
\begin{equation} \label{expintvdotphi2}
 \int_0^t |\frac{V(s)}{2} \cdot \nabla \phi ^m (s)  | ds <  M  \int_0^t |V(s) | ds < 5Mt(M + D ) + 4MD < 5tM^2 + 9MD< C_\Omega M ,
\end{equation}
for $0< t < T$. Now from (\ref{expintvdotphi}), (\ref{expintvdotphi2}) and following the argument in estimating along the backward trajectories from Lemma \ref{expandtrajl} we have for $1 \le l \le m$, if $t^1 < 0$, then
\begin{equation} \label{inftytraj1}
\begin{split}
e^{(\theta - t)  |v|^2 } | & f^{m+1} (t,x,v) |
\\  \le & e^{(\theta - t) |v|^2 }e^{ t M |v| + t^2 M^2  } |f^{m+1} (0,X^m(0), V^m(0)) | + e^{(\theta - t) |v|^2 } e^{ C_\Omega M } \int_0^t \Gamma_{\text{gain}}(f^m, f^m) (s, X^m(s) , V^m(s)) ds.
 \\ \le &e^{\theta |v|^2 } e^{t(M|v| - |v|^2) + 1}  |f^{m+1} (0,X^m(0), V^m(0)) | + e^{(\theta - t) |v|^2 } e^{ C_\Omega M } \int_0^t \Gamma_{\text{gain}}(f^m, f^m) (s, X^m(s) , V^m(s)) ds
  \\ \le &e^{\theta |v|^2 } e^{t\frac{M^2}{4} + 1}  |f^{m+1} (0,X^m(0), V^m(0)) | + e^{(\theta - t) |v|^2 } e^{ C_\Omega M } \int_0^t \Gamma_{\text{gain}}(f^m, f^m) (s, X^m(s) , V^m(s)) ds
  \\ \lesssim & e^{\theta |v|^2 }   |f^{m+1} (0,X^m(0), V^m(0)) | + e^{(\theta - t) |v|^2 } e^{ C_\Omega M } \int_0^t \Gamma_{\text{gain}}(f^m, f^m) (s, X^m(s) , V^m(s)) ds.
 \end{split}
\end{equation}
If $t^1 > 0$, then
\begin{equation} \label{inftytraj2}
\begin{split}
&  e^{(\theta - t) |v|^2 }  | f^{m+1} (t,x,v) |
\\ \le &e^{(\theta - t) |v|^2 }  e^{C_\Omega M } \int_{t^1}^t \Gamma_{\text{gain}}(f^m, f^m) (s, X^m(s), V^m(s) ) ds
\\ & + e^{(\theta - t) |v|^2 } e^{\langle v \rangle}  c_\mu \sqrt {\mu (\vb^0) } \int_{\prod_{j=1}^{l-1} \mathcal V_j} \sum_{i=1}^{l-1} \textbf{1}_{\{t^{i+1} < 0 < t^i \}} e^{(\theta - t^i) |v^i|^2 } | f^{m+1-i}  (0,X^{m-i}(0), V^{m-i}(0)) |    \, d \Sigma_{i}^{l-1}
\\ & +   e^{(\theta - t) |v|^2 } e^{\langle v \rangle}  c_\mu \sqrt {\mu (\vb^0) }  \int_{\prod_{j=1}^{l-1} \mathcal V_j} \sum_{i=1}^{l-1} \textbf{1}_{\{t^{i+1} < 0 < t^i \}}  e^{(\theta - t^i) |v^i|^2 } \int_0^{t^i} \Gamma_{\text{gain}}(f^{m-i}, f^{m-i} ) (s, X^{m-i}(s), V^{m-i}(s) ) ds  \, d \Sigma_{i}^{l-1}
\\ & +  e^{(\theta - t) |v|^2 } e^{\langle v \rangle}  c_\mu \sqrt {\mu (\vb^0) }  \int_{\prod_{j=1}^{l-1} \mathcal V_j} \sum_{i=1}^{l-1} \textbf{1}_{\{t^{i+1} > 0 \}}  e^{(\theta - t^i) |v^i|^2 } \int_{t^{i+1}}^{t^i} \Gamma_{\text{gain}}(f^{m-i}, f^{m-i}) (s, X^{m-i}(s), V^{m-i}(s) ) ds  \, d \Sigma_{i}^{l-1}
\\ & + e^{(\theta - t) |v|^2 }  e^{\langle v \rangle} c_\mu  \sqrt {\mu (\vb^0) }   \int_{\prod_{j=1}^{l-1} \mathcal V_j}  \textbf{1}_{\{t^{l} > 0 \}} e^{(\theta - t^{l-1} ) |v^{l-1} |^2}   f^{m+1-(l-1)} (t^l,x^l,\vb^{l-1}) | d \Sigma_{l-1}^{l-1},
\end{split} \end{equation}
where  $\mathcal V_j = \{ v^j \in \mathbb R^3: n(x^j ) \cdot v^j > 0 \}$,
and 
\[ \begin{split}
d \Sigma_i^{l-1} = & \{\prod_{j=i+1}^{l-1}  \mu(v^j) c_\mu |n(x^j ) \cdot v^j | dv^j \} 
\{  e^{ \langle v^i \rangle } \sqrt{ \mu (v^i)} \langle v^i \rangle e^{-(\theta - t^i ) |v^i|^2 }  d v^i \}
\\ & \{\prod_{j=1}^{i-1}  e^{ \langle v^j \rangle  }  \sqrt {\mu (v^j )} \langle v^j  \rangle d v^j \},
\end{split} \]
where $c_\mu$ is the constant that $ \int_{\mathbb R^3 }  \mu(v^j) c_\mu |n(x^j ) \cdot v^j | dv^j = 1$.

Now we have for all $0 \le i \le l-1$,
\[ \begin{split}
& e  ^{(\theta - t^i)|v^i|^2}  \int_0^{t^i}  \Gamma_{\text{gain}}(f^{m-i}, f^{m-i} ) (s, X^{m-i}(s), V^{m-i}(s) ) ds 
\\ =&  e^{(\theta - t^i)|v^i|^2} \int_0^{t^i}  \int_{\mathbb R^3 } \int_{\mathbb S^2} |V^{m-i}(s) - u|^{\kappa} q_0 (\frac{ V^{m-i}(s) -u}{ |V^{m-i}(s) -u| } \cdot w ) \sqrt{\mu(u) } f^{m-i} (u')f^{m-i} (V^{m-i}(s)') d\omega du ds
\\ \le & \left(\sup_{0 \le s \le t } \|  e^{ (\theta - s ) |v|^2} f^{m-i}  (s) \|_\infty \right)^2 \times  \int_0^{t^i}  \int_{\mathbb R^3 }   |V^{m-i}(s) - u|^{\kappa}   \sqrt{\mu(u)} e^{(\theta - t^i)|v^i|^2}  e^{- ( \theta - s )  |u'|^2 } e^{-(\theta - s ) |V^{m-i}(s)' | ^2 } du ds
\\ = &  \left(\sup_{0 \le s \le t } \|  e^{ (\theta - s ) |v|^2} f^{m-i}  (s) \|_\infty \right)^2 \times  \int_0^{t^i}  \int_{\mathbb R^3 }   |V^{m-i}(s) - u|^{\kappa}  \sqrt{\mu(u)} e^{(\theta - t^i)|v^i|^2}  e^{- ( \theta - s )  |u|^2 } e^{-(\theta - s ) |V^{m-i}(s) | ^2 } du ds
\\ \lesssim & \left(\sup_{0 \le s \le t } \|  e^{ (\theta - s ) |v|^2} f^{m-i}  (s) \|_\infty \right)^2 \times  \int_0^{t^i}  \int_{\mathbb R^3 }   |V^{m-i}(s) - u|^{\kappa}  \sqrt{\mu(u)} e^{(s - t^i)|V^{m-i}(s)|^2}  e^{- ( \theta - s )  |u|^2 }du ds
\\ \lesssim & \left(\sup_{0 \le s \le t } \|  e^{ (\theta - s ) |v|^2} f^{m-i}  (s) \|_\infty \right)^2 \times \int_0^{t^i} e^{-(t^i -s ) |V^{m-i}(s) | ^2 } \langle V^{m-i}(s) \rangle \{ \textbf{1}_{|v^i| > N } + \textbf{1}_{|v^i| \le N } \} ds
\\ \lesssim &  \left(\sup_{0 \le s \le t } \|  e^{ (\theta - s ) |v|^2} f^{m-i}  (s) \|_\infty \right)^2 \times \left( \frac{1}{N} + 2N t\right) < \epsilon  \left(\sup_{0 \le s \le t } \|  e^{ (\theta - s ) |v|^2} f^{m-i}  (s) \|_\infty \right)^2.
\end{split} \]
If we choose sufficiently large $N \gg 1$ an then small $0 < T \ll \theta$.
Where we have used $\frac{|v^i|}{2} \le |V^{m-i}(s) |  \le 2 |v^i| $, for $ |v^i| > N \gg 1 $, and $|V^{m-i}(s) | \le 2N$ if $|v^i| \le N$. And that $e^{(\theta -  t^i )|v^i|^2}  \le e^{\theta t ^2 M^2 } e^{(\theta - t^i ) |V^{m-i}(s) |^2}  < e^\theta e^{(\theta - t^i ) |V^{m-i}(s) |^2} $. And that $|u'|^2 + |V^{m-i}(s) '| ^2 = |u|^2 + |V^{m-i}(s) | ^2$.

And by the same argument we have
\[ \begin{split}
& e  ^{(\theta - t^i)|v^i|^2}  \int_{t^{i+1}}^{t^i}  \Gamma_{\text{gain}}(f^{m-i}, f^{m-i} ) (s, X^{m-i}(s), V^{m-i}(s) ) ds 
\lesssim   \left(\sup_{0 \le s \le t } \|  e^{ (\theta - s ) |v|^2} f^{m-i}  (s) \|_\infty \right)^2 \times \left( \frac{1}{N} + 2N t\right).
\end{split} \]

Now from (\ref{lowerbddfortjtjplus1}), we have $t^j - t^{j+1} \ge \frac{\delta ^3 }{C_\Omega (1 + \delta^2 \| E \|_\infty^2 )}$ for $v^j \in \mathcal V_j^\delta$. But from (\ref{choiceofT}), if $t^l \ge 0$ then there are at most $ \left[ \frac{ C_\Omega }{ \delta ^3 }  \right] + 1$ numbers of $v^m \in \mathcal V^\delta_m$ for $1 \le m \le l-1$. Thus
for $ l > 2 \left(  \left[ \frac{ C_\Omega }{ \delta ^3 }  \right] +1 \right) ^2$, we have from (\ref{tailtermsmall}) that
\[
 \int_{\prod_{j=1}^{l-1} \mathcal V_j}  \textbf{1}_{\{t^{l} > 0 \}} e^{(\theta - t^{l-1} ) |v^{l-1} |^2}  d \Sigma_{l-1}^{l-1} \lesssim_{\Omega, M} \left(\frac{1}{2}\right)^l.
\]

Therefore from the above estimates we have for (\ref{inftytraj1}) and (\ref{inftytraj2}) the following estimate:
\[
\begin{split}
&  e^{(\theta - t) |v|^2 }  | f^{m+1} (t,x,v) |
\\ \le &  l C  ^l \| e^{\theta |v|^2 } f_0 \| _\infty + e^{C_\Omega M } l C ^l \left( \max_{1 \le i \le l-1} \sup_{0 \le s \le t } \|  e^{ (\theta - s ) |v|^2} f^{m+1-i}  (s) \|_\infty \right)^2   \left( \frac{8N}{N^2} + 2N t\right) 
\\ & + C \max_{1 \le i \le l-1} \sup_{0 \le s \le t } \|  e^{ (\theta - s ) |v|^2} f^{m+1-i}  (s) \|_\infty \left( \frac{1}{2} \right) ^l.
\end{split}
\]
We can now choose a large $l$ then large $C_1$ then large $N$ and finally small $T$ to conclude the uniform-in-m estimate
\begin{equation}
 \sup_{ 0 \le t \le T} \|  e^{\theta' |v|^2 }   f^{m+1} (t ) \|_\infty \le \sup_{ 0 \le t \le T} \|  e^{(\theta - t) |v|^2 }   f^{m+1} (t ) \|_\infty \le C_1 \| e^{\theta |v|^2} f_0 \|_\infty,
\end{equation}
with $\theta' = \theta - T$. This proves (\ref{uniformmlinftybdd}).

		\vspace{4pt}

\textit{Step 2.} \quad
We claim that there exists $0 < \theta' \ll 1$,  $\varpi \gg 1$, $T = T(\|e^{\theta |v|^2 } f_0 \|_\infty, \varpi) \ll 1 $, and a $C_1 > 0 $ such that 
\begin{equation} \begin{split} \label{c1bddsequence}
\sup_{m} \sup_{0 \le t \le T}  \| e^{\theta' |v| ^2 } e^{-\varpi  \langle v \rangle t } \alpha  \partial  f^{m} (t,x,v) \|_\infty \le C_1\left( P(\| e^{\theta |v|^2} f_0\|_\infty ) + \| e^{\theta |v|^2 } \alpha \partial f_0 \|_\infty \right) < \infty.
\end{split} \end{equation}

%
%
%

From (\ref{traneqforpartialf}) and direct computation we have
\begin{equation} \label{} \begin{split}
 \bigg \{ & \partial_t + v\cdot \nabla_x  -\nabla \phi^m \cdot \nabla_v + \nu(\sqrt \mu f^m)
 \\ & + \frac{v}{2} \cdot \nabla \phi^m  + 2 \theta' v \cdot \nabla \phi^m + \varpi \langle v \rangle - \varpi \frac{v}{\langle v \rangle} \cdot  \nabla \phi^m t  -\alpha ^{-1}( \partial_t \alpha + v \cdot \nabla_x \alpha -\nabla \phi^m \cdot \nabla_v \alpha ) \bigg\} (e^{\theta' |v|^2 }  e^{-\varpi \langle v \rangle t  } \alpha \partial f^{m+1})
\\ = &  e^{\theta' |v|^2 } e^{-\varpi \langle v \rangle t}  \alpha ( \partial_t + v \cdot \nabla_x -\nabla \phi^m \cdot \nabla_v  + \frac{v}{2} \cdot \nabla \phi^m + \nu(\sqrt \mu f ^m)) \partial f^{m+1}
\\ = & e^{\theta' |v|^2 } e^{-\varpi \langle v \rangle t }  \alpha \mathcal G^m .
\end{split} \end{equation}
From \eqref{VPBsq2}, \eqref{VPBsq3}, $\phi^m$ satisfies 
\Be \label{phimboundarycondition}
- \frac{\partial \phi^m }{ \partial n }  = - \frac{\partial \phi_{F^m} }{ \partial n } - \frac{\partial \phi_E }{\partial n } = - \frac{\partial \phi_E }{\partial n }> C_E > 0,
\Ee
 on $\partial \Omega$ for every $m$. Note that if we let $E(t,x) = -\nabla \phi^m (t,x)$ in the definition of $\alpha(t,x,v)$ in (\ref{alphadef}), we have the same $\alpha(t,x,v)$ for all $m$, as $\nabla \phi_{F^m} (t,x) \cdot \nabla \xi(x) = 0 $ for all $x \in \partial \Omega$.
Therefore by (\ref{velocitylemmaintform}) we have that
\begin{equation} \label{velocitylemmam}
\alpha^{-1} ( \partial_t \alpha + v \cdot \nabla_x \alpha -\nabla \phi^m \cdot \nabla_v \alpha )  \le \frac{C_\xi}{C_E} ( \| \nabla \phi ^m \|_\infty + \| \nabla^2 \phi^m \|_\infty)  \langle v \rangle.
\end{equation}
By our choice of $f^0$ we have $\phi ^0 = \phi_E$, thus if we choose $\varpi $ large enough, we have
\[
  \frac{v}{2} \cdot \nabla \phi^0 +  2 \theta' v \cdot \nabla \phi^0  + \varpi \langle v \rangle + \varpi \frac{v}{\langle v \rangle}  \cdot \nabla \phi^0 t     -  \alpha^{-1} ( \partial_t \alpha + v \cdot \nabla_x \alpha -\nabla \phi^0 \cdot \nabla_v \alpha )  \ge \frac{\bar \omega}{2} \langle v \rangle.
\]
Now if we let
\[
\bar \nu_{\varpi}^m : = \nu(\sqrt \mu f ) + \frac{v}{2} \cdot \nabla \phi^m  + 2 \theta' v \cdot \nabla \phi^m + \varpi \langle v \rangle + \varpi \frac{v}{\langle v \rangle}  \cdot \nabla \phi^m t     -  \alpha^{-1} ( \partial_t \alpha + v \cdot \nabla_x \alpha -\nabla \phi^m \cdot \nabla_v \alpha ),
\]
and
\[
\mathcal N^m =  e^{\theta' |v|^2 } e^{-\varpi \langle v \rangle t }  \alpha \mathcal G^m,
\]
we have
\begin{equation} \label{VPBeqpartilf}
(\partial_t + v\cdot \nabla_x  -\nabla \phi^m \cdot \nabla_v + \bar \nu_{\varpi}^m)( e^{\theta' |v|^2 } e^{-\varpi \langle v \rangle t} \alpha \partial f^{m+1}) = \mathcal N^m.
\end{equation}
Now since
\[ \begin{split}
e^{\theta ' |v|^2 } \Gamma_{gain}(\partial f^m,f^m )  = &  e^{\theta ' |v|^2 }  \int \int |v -u |^\kappa q_0 \sqrt{\mu(u)} \partial f^m(u') f^m(v') d \omega du
\\ \le & e^{\theta ' |v|^2 } \| e^{2\theta' |v|^2 } f^m \|_\infty \int \int |v -u |^\kappa q_0 \sqrt{\mu(u)} e^{\theta' |u'|^2 }\partial f^m(u')  e^{-\theta' |u'|^2 }e^{-2 \theta' |v'|^2 } d \omega du
\\ \le & \| e^{2\theta' |v|^2 } f^m \|_\infty \int \int |v -u |^\kappa q_0 \sqrt{\mu(u)} e^{\theta' |u'|^2 }\partial f^m(u')  e^{- \theta' |v'|^2 }  e^{-\theta ' |u|^2 } d \omega du
\\ \lesssim & \| e^{2\theta' |v|^2 } f^m \|_\infty \Gamma_{gain} ( e^{\theta' |v|^2} \partial f^m,   e^{- \theta ' |v|^2 })
\\ \lesssim & \| e^{2\theta' |v|^2 } f^m \|_\infty \int_{\mathbb R^3 } \frac{ e^{-C_{\theta'} |u -v | ^2 } } {|u -v |^{2 -\kappa} }  |e^{\theta ' |u|^2 } \partial f^m(t,x,u ) | du,
\end{split} \]
where we've used $|v'|^2 + |u'|^2 = |v|^2 + |u|^2 $.
Also
\[ \begin{split}
e^{\theta ' |v|^2 } \nu(\sqrt{\mu} \partial f^m ) f^{m+1} \le & \| e^{2 \theta' |v|^2 } f^{m+1} \|_\infty e^{-\theta ' |v|^2 } \nu(\sqrt{\mu} \partial f^m ) 
\\ \lesssim & \| e^{2 \theta' |v|^2 } f^{m+1} \|_\infty  \int_{\mathbb R^3 } \frac{ e^{-C_{\theta'} |u -v | ^2 } } {|u -v |^{2 -\kappa} } | \partial f^m(t,x,u ) | du.
\end{split} \]

Thus from (\ref{uniformmlinftybdd}) we have the following bound for $\mathcal N^m$:

\begin{equation} \label{Nbdd}\begin{split}
| & \mathcal   N^m(t,x,v) | 
\\ = &    e^{\theta' |v|^2 } e^{-\varpi  \langle v \rangle t}  \alpha (t,x,v)
\\ & \times  \bigg( \partial \Gamma_{gain} (f^m,f^m) - \partial v \cdot \nabla_x f^{m+1} + \partial \nabla \phi \cdot \nabla_v f^{m+1} 
+ \partial ( \frac{v}{2} \cdot \nabla \phi^m ) f^{m+1} - \partial (\nu (\sqrt \mu f^m ) ) f^{m+1}
 \bigg)
\\  \lesssim &   (1 + \| \nabla^2 \phi^m \|_\infty)  [ P(\| e^{\theta |v|^2 } f_0 \|_\infty )  + |e^{\theta' |v|^2 }  e^{-\varpi \langle v \rangle t  }  \alpha  \partial f^{m+1}(t,x,v) | ]
\\ & +  \| e^{\theta |v|^2 } f_0 \|_\infty  e^{-\varpi \langle v \rangle t }  \alpha(t,x,v) \int_{\mathbb R^3 }  \frac{ e^{-C_\theta |u -v | ^2 } } {|u -v |^{2 -\kappa} } | e^{\theta ' |u|^2 }\partial f^m(t,x,u ) | du.
\end{split} \end{equation}

We claim that there exists, $C_1 > 0 $, $\varpi \gg 1 $, and $T \ll 1$ such that if
\begin{equation} \begin{split} \label{c1bdinduction}
\bar \nu_\varpi^i & = \nu(\sqrt \mu f ) + \frac{v}{2} \cdot \nabla \phi^{i}  + 2 \theta' v \cdot \nabla \phi^{i} + \varpi \langle v \rangle  + \varpi \frac{v}{\langle v \rangle }  \cdot \nabla \phi^i t    -  \alpha^{-1} ( \partial_t \alpha + v \cdot \nabla_x \alpha -\nabla \phi^{i} \cdot \nabla_v \alpha ) \ge \frac{\varpi}{2} \langle v \rangle,
\end{split} \end{equation}
for all $ 1 \le i \le m-1$, and
\begin{equation} \label{inductionassumptionc1}
\begin{split}
\max_{0 \le i \le m} \sup_{0 \le t \le T}  \| e^{\theta' |v|^2 } e^{-\varpi\langle v \rangle t } \alpha  \partial  f^{i} (t,x,v) \|_\infty \le C_1\left( P(\| e^{\theta |v|^2} f_0\|_\infty ) + \| e^{\theta' |v|^2 } \alpha \partial f_0 \|_\infty \right) < \infty,
\end{split} \end{equation}
then
\begin{equation} \label{largeomegabar}
\begin{split}  \bar \nu_\varpi^{m} =  \nu(\sqrt \mu f ) + \frac{v}{2} \cdot \nabla \phi^m  + 2 \theta' v \cdot \nabla \phi^m + \varpi \langle v \rangle + \varpi \frac{v}{ \langle v \rangle }  \cdot \nabla \phi^m t  -  \alpha^{-1} ( \partial_t \alpha + v \cdot \nabla_x \alpha -\nabla \phi^m \cdot \nabla_v \alpha ) \ge \frac{\varpi}{2} \langle v \rangle ,
\end{split} \end{equation}
and
\begin{equation} \label{C1bddinduction}
 \sup_{0 \le t \le T}  \| e^{\theta' |v|^2 } e^{-\varpi \langle v \rangle t  } \alpha  \partial  f^{m+1} (t,x,v) \|_\infty \le C_1\left( P(\| e^{\theta |v|^2} f_0\|_\infty ) + \| e^{\theta' |v|^2 } \alpha \partial f_0 \|_\infty \right).
\end{equation}
To prove (\ref{largeomegabar}), note that from (\ref{velocitylemmam}), (\ref{c2bdofpotential}), and (\ref{inductionassumptionc1}) we have
\[ \begin{split}
& \alpha^{-1}  ( \partial_t \alpha + v \cdot \nabla_x \alpha -\nabla \phi^m \cdot \nabla_v \alpha ) 
 \\ \lesssim & ( \| \nabla \phi^m \|_\infty + \| \nabla^2 \phi^m \|_\infty ) \langle v \rangle
 \\ \lesssim  & (\| e^{\theta' |v|^2 } f^m(t) \|_\infty + \| e^{-\varpi \langle v \rangle t } \alpha \nabla_{x} f^m (t) \|_\infty)\langle v \rangle
  \lesssim (P(\| e^{\theta |v|^2} f_0\|_\infty ) + \| \alpha \partial f_0 \|_\infty)\langle v \rangle.
   \end{split} \]
   Therefore (\ref{largeomegabar}) can be achieved once we choose $\varpi \gg 1$ large enough.
   
First for $t^1 < 0$, using the Duhamel's formulation we have from (\ref{VPBeqpartilf})

\begin{equation}  \label{}
\begin{split}
 e^{\theta' |v|^2 } & e ^{-\varpi \langle v \rangle t  }  \alpha | \partial  f^{m+1} (t,x,v) |
\\   \le &  e^{ -\int_s^t \nu^m_{\varpi} (\tau, X^m(\tau), V^m(\tau) d\tau}  e^{\theta' |V^m(0)|^2 } \alpha \partial f^{m+1} (0, X^m(0) , V^m(0) ) 
\\ & + \int_0^t e^{ -\int_s^t \nu^m_{\varpi} (\tau, X^m(\tau), V^m(\tau) d\tau } \mathcal N^m(s,X^m(s), V^m(s) ) ds.
\end{split} 
\end{equation}
Thus by (\ref{Nbdd}) we have
\[ \begin{split}
 \sup_{0 \le t \le T} & \| \textbf{1}_{ \{ t^1 < 0 \}} e^{-\varpi \langle v \rangle t } e^{\theta' |v|^2 } \alpha  \partial  f^{m+1} (t,x,v) \|_\infty
\\ \le & \sup_{0 \le t \le T}  \| e^{ -\int_0^t \nu^m_{\varpi} (\tau, X^m(\tau), V^m(\tau) d\tau} e^{\theta' |V^m(0)|^2 } \alpha \partial f^{m+1} (0, X^m(0) , V^m(0) ) 
\\ & + \int_0^t e^{ -\int_s^t \nu^m_{\varpi} (\tau, X^m(\tau), V^m(\tau) d\tau } \mathcal N^m(s,X^m(s), V^m(s) ) ds \|_\infty
 \\ \le & \| e^{\theta' |v|^2 } \alpha \partial f_0 \|_\infty + T   (1 + \| \nabla^2 \phi^m \|_\infty) [ P(\| e^{\theta |v|^2} f_0 \|_\infty) +   \sup_{0 \le t \le T}  \| e^{\theta' |v|^2 }  e^{-\varpi \langle v \rangle t  } \alpha  \partial  f^{m+1} (t,x,v) \|_\infty]
 \\ & + P(\| e^{\theta |v|^2} f_0 \|_\infty)   \sup_{0 \le t \le T}  \| e^{\theta' |v|^2 } e^{-\varpi \langle v \rangle t  } \alpha  \partial  f^{m} (t,x,v) \|_\infty
 \\ & \times \int_0^t \int_{\mathbb R^3} e^{ -\int_s^t\frac{\varpi}{2} \langle V^m(\tau;t,x,v) \rangle d\tau }  \frac{e^{-\varpi  \langle V^{m}(s;t,x,v) \rangle  s }}{e^{-\varpi  \langle u \rangle s}} \frac{e^{-C_\theta |V^m(s)-u|^2 }}{|V^m(s) -u |^{2 - \kappa}  }  \frac{\alpha(s,X^m(s) ,V^m(s))}{\alpha(s,X^m(s) ,u)} du  ds.
\end{split} \]
Now since
\[
 \langle u \rangle -  \langle V^m(s;t,x,v) \rangle \le 2 \langle u - V^m(s;t,x,v) \rangle,
\]
we have
\[
 \frac{e^{-\varpi  \langle V^{m}(s;t,x,v) \rangle  s }}{e^{-\varpi  \langle u \rangle s}} e^{-C_\theta |V^m(s)-u|^2 }
\lesssim   e^{ -\frac{ C_\theta |V^m(s) - u | ^2}{2}}.
\]
Thus
\begin{equation} \label{genlemma1toget}
\begin{split}
\int_0^t & \int_{\mathbb R^3} e^{ -\int_s^t\frac{\varpi}{2} \langle V^m(\tau;t,x,v) \rangle d\tau }  \frac{e^{-\varpi  \langle V^{m}(s;t,x,v) \rangle  s }}{e^{-\varpi  \langle u \rangle s}}  \frac{e^{-C_\theta |V^m(s)-u|^2 }}{|V^m(s) -u |^{2 - \kappa}  }  \frac{\alpha(s,X^m(s) ,V^m(s))}{\alpha(s,X^m(s) ,u)} du  ds
\\ \lesssim &    \int_0^t \int_{\mathbb R^3} e^{ -\int_s^t\frac{\varpi}{2} \langle V^m(\tau;t,x,v) \rangle d\tau }  \frac{e^{-\frac{C_\theta}{2} |V^m(s)-u|^2 }}{|V^m(s) -u |^{2 - \kappa}  }  \frac{\alpha(s,X^m(s) ,V(s))}{\alpha(s,X^m(s) ,u)} du  ds.
\end{split} 
\end{equation}
Note that, for any $\beta > 1$,
\[
\frac{1}{\alpha(x, X^m(s) ,u ) } \lesssim \frac{1}{ (\alpha (x, X^m(s) , u ) )^\beta } + 1.
\]
So from (\ref{phimboundarycondition}) we can let $1 < \beta \le 2$, and apply (\ref{genlemma1}) to (\ref{genlemma1toget}) to have
\begin{equation}
\begin{split}
 \int_0^t  & \int_{\mathbb R^3} e^{ -\int_s^t\frac{\varpi}{2} \langle V^m(\tau;t,x,v) \rangle d\tau }  \frac{e^{-\varpi  \langle V^{m}(s;t,x,v) \rangle  s }}{e^{-\varpi  \langle u \rangle s}} \frac{e^{-C_\theta |V^m(s)-u|^2 }}{|V^m(s) -u |^{2 - \kappa}  }  \frac{\alpha(s,X^m(s) ,V^m(s))}{\alpha(s,X^m(s) ,u)} du  ds
\\ \lesssim &  e^{ C ( \| \nabla \phi^m \|_\infty^2 + \| \nabla^2 \phi^m \|_\infty )}   \left( \frac{ \delta^{\frac{3 - \beta}{2}} (\alpha(t,x,v) )^{3 - \beta } }{ (|v|^2 + 1 )^{\frac{3 -\beta}{2}} } + \frac{ (|v| + 1 )^{\beta - 1} (\alpha(t,x,v) )^{2 - \beta }}{ \delta ^{\beta - 1} \varpi \langle v \rangle  }   \right)
\\ \lesssim &  e^{ C ( \| \nabla \phi^m \|_\infty^2 + \| \nabla^2 \phi^m \|_\infty  ) }  \left( \delta^{\frac{3 - \beta}{2}} + \frac{1}{\delta ^{\beta - 1} \varpi} \right),
\end{split}
\end{equation}
where we used $\alpha(s,X^m(s) ,V^m(s)) \lesssim e^{ C ( \| \nabla \phi^m \|_\infty^2 + \| \nabla^2 \phi^m \|_\infty ) }  \alpha(t,x,v)$.

If $t^1(t,x,v) \ge 0 $, the backward trajectory first hits the boundary, then from (\ref{C1trajectory}) we have the following line-by-line estimate
\[
 \begin{split}
| & \textbf{1}_{ \{ t^1 > 0 \}} e^{\theta' |v|^2 }   e^{-\varpi \langle v \rangle t } \alpha  \partial  f^{m+1} (t,x,v)|
 \\ \lesssim & P(\| e^{\theta |v|^2} f_0 \|_\infty) + T(1 + \|\nabla^2 \phi^m \|_\infty) \sup_{0 \le t \le T}  \| e^{\theta' |v|^2 } e^{-\varpi \langle v \rangle t  } \alpha  \partial  f^{m+1} (t,x,v) \|_\infty 
 \\ + & l (Ce^{Ct^2 } )^l \max_{1 \le i \le l-1} \| e^{\theta' |v|^2 }  \alpha \partial f_0 ^{m + 1 - i } \|_\infty
 \\  + & P(\| e^{\theta |v|^2} f_0 \|_\infty) \|  \sup_{0 \le t \le T}  \| e^{\theta' |v|^2 }  e^{-\varpi \langle v\rangle t } \alpha  \partial  f^{m+1} (t,x,v) \|_\infty
 \\ & \times (e^{Ct})^2 \int_{t_1}^t \int_{\mathbb R^3} e^{ -\int_s^t\frac{\varpi}{2} \langle V^m(\tau;t,x,v) \rangle d\tau }  \frac{e^{-\frac{C_\theta}{2} |V^m(s)-u|^2 }}{|V^m(s) -u |^{2 - \kappa}  }  \frac{\alpha(s,X^m(s) ,V^m(s))}{\alpha(s,X^m(s) ,u)} du  ds.
\\  + &  T  l (Ce^{Ct^2 } )^l \max_{1 \le i \le l-1} (1 + \|\nabla^2 \phi^{m-i} \|_\infty) \max_{1 \le i \le l-1} \sup_{0 \le t \le T} \| e^{\theta' |v|^2 }  e^{-\varpi \langle v \rangle t } \alpha  \partial  f^{m +1 - i} (t,x,v) \|_\infty  
\\ + & Tl (Ce^{Ct^2 } )^l   \max_{1 \le i \le l-1} (1 + \|\nabla^2 \phi^{m-i}  \|_\infty )P(\| e^{\theta |v|^2} f_0 \|_\infty) 
\\ + & P(\| e^{\theta |v|^2} f_0 \|_\infty)   \max_{1 \le i \le l-1} \sup_{0 \le t \le T}  \| e^{\theta' |v|^2 } e^{-\varpi \langle v \rangle t } \alpha  \partial  f^{m+1 - i} (t,x,v) \|_\infty
\\ & \times l (Ce^{Ct^2 } )^l \max_{ 1 \le i \le l-1} \int_{0}^{t^i} \int_{\mathbb R^3} e^{ -\int_s^{t^i} \frac{\varpi}{2} \langle V^{m-i}(\tau;t,x,v) \rangle d\tau }  \frac{e^{-\frac{C_\theta}{2} |V^{m-i}(s)-u|^2 }}{|V^{m-i}(s) -u |^{2 - \kappa}  }  \frac{\alpha(s,X^{m-i}(s) ,V^{m-i}(s))}{\alpha(s,X^{m-i}(s) ,u)} du  ds
\\ + & P(\| e^{\theta |v|^2} f_0 \|_\infty)   \max_{1 \le i \le l-1} \sup_{0 \le t \le T}  \| e^{\theta' |v|^2 } e^{-\varpi\langle v \rangle t } \alpha  \partial  f^{m+1 - i} (t,x,v) \|_\infty
\\ & \times l (Ce^{Ct^2 } )^l \max_{ 1 \le i \le l-1} \int_{t^{i+1}}^{t^i} \int_{\mathbb R^3} e^{ -\int_s^{t^i}\frac{\varpi}{2} \langle V^{m-i}(\tau;t,x,v) \rangle d\tau }  \frac{e^{-\frac{C_\theta}{2} |V^{m-i}(s)-u|^2 }}{|V^{m-i}(s) -u |^{2 - \kappa}  }  \frac{\alpha(s,X^{m-i}(s) ,V^{m-i}(s))}{\alpha(s,X^{m-i}(s) ,u)} du  ds
\\ + & C \left( \frac{1}{2} \right) ^l \sup_{0 \le t \le T} \|e^{\theta' |v|^2 }  e^{-\varpi \langle v \rangle t } \alpha  \partial  f^{m+1 - (l-1) } (t,x,v) \|_\infty.
\end{split}
\]
We again apply (\ref{genlemma1}) to get
\[
\begin{split}
| &  \textbf{1}_{ \{ t^1 > 0 \}}  e^{\theta' |v|^2 } e^{-\varpi \langle v \rangle t } \alpha  \partial  f^{m+1} (t,x,v) |
\\ \lesssim & C_l e^{Cl t^2 } \left( \delta^{\frac{3 - \beta}{2}} + \frac{1}{\delta ^{\beta - 1} \varpi}  \right)  P(\| e^{\theta |v|^2} f_0 \|_\infty) \max_{ 0 \le i \le l-1 } e^{ C ( \| \nabla \phi^{m-i} \|_\infty^2 + \| \nabla^2 \phi^{m-i} \|_\infty + \| \nabla \phi^{m-i} \|_\infty )}   
\\ & \times \max_{m -(l-2) \le i \le m} \sup_{0 \le t \le T}  \| e^{\theta' |v|^2 } e^{-\varpi \langle v \rangle t } \alpha  \partial  f^{i} (t,x,v) \|_\infty
\\ & +  T(1 + \|\nabla^2 \phi^m \|_\infty) \sup_{0 \le t \le T}  \| e^{\theta' |v|^2 } e^{-\varpi  \langle v \rangle t } \alpha  \partial  f^{m+1} (t,x,v) \|_\infty 
\\  + &  T  l (Ce^{Ct^2 } )^l \max_{1 \le i \le l-1} (1 + \|\nabla^2 \phi^{m-i} \|_\infty) \max_{1 \le i \le l-1} \sup_{0 \le t \le T} \| e^{\theta' |v|^2 } e^{-\varpi \langle v \rangle t } \alpha  \partial  f^{m +1 - i} (t,x,v) \|_\infty  
\\ + & Tl (Ce^{Ct^2 } )^l   \max_{1 \le i \le l-1} (1 + \|\nabla^2 \phi^{m-i}  \|_\infty )P(\| e^{\theta |v|^2} f_0 \|_\infty) 
+  l (Ce^{Ct^2 } )^l  \| \alpha \partial f_0  \|_\infty +   P(\| e^{\theta |v|^2} f_0 \|_\infty) 
\\ & + C \left( \frac{1}{2} \right) ^l \max_{m-(l-2) \le i \le m}  \sup_{0 \le t \le T} \|e^{\theta' |v|^2 } e^{-\varpi \langle v \rangle t} \alpha  \partial  f^{i } (t,x,v) \|_\infty.
\end{split}
\]

Now if we let $P(\| e^{\theta |v|^2} f_0\|_\infty ) + \| e^{\theta ' |v|^2 } \alpha \partial f_0 \|_\infty = M_1 <\infty $. From (\ref{c2bdofpotential}) and the induction hypothesis, from (\ref{c1bdinduction}) we have
\[ \begin{split}
\max_{0 \le i \le {l-1} } ( 1 + \| \nabla^2 \phi^{m-i} \|_\infty) \lesssim &  \max_{0 \le i \le l-1 } \left( \| e^{\theta |v|^2 } f^{m-i} (t) \|_\infty + \|  e^{-\varpi \langle v \rangle t} \alpha \nabla_{x} f^{m-i} (t) \|_\infty \right)
 \lesssim C_1M_1.
\end{split} \]
Therefore we have
\[
\begin{split}
|  \textbf{1}_{ \{ t^1 > 0 \}}  & e^{-\varpi \int_{0}^t \langle V^m(\tau) \rangle  d\tau } \alpha  \partial  f^{m+1} (t,x,v) |
\\ \lesssim & C_l e^{Cl t^2 } \left( \delta^{\frac{3 - \beta}{2}} + \frac{1}{\delta ^{\beta - 1} \varpi}  \right)  P(\| e^{\theta |v|^2} f_0 \|_\infty)e^{C_1 M + CM^2 }  
\\ & \times \max_{m -(l-2) \le i \le m} \sup_{0 \le t \le T}  \| e^{\theta ' |v|^2 } e^{-\varpi \langle v \rangle t } \alpha  \partial  f^{i} (t,x,v) \|_\infty
\\ & +  TC_1M \sup_{0 \le t \le T}  \| e^{\theta ' |v|^2 }  e^{-\varpi \langle v \rangle t } \alpha  \partial  f^{m+1} (t,x,v) \|_\infty 
\\  + &  T  l (Ce^{Ct^2 } )^l C_1M  \max_{1 \le i \le l-1} \sup_{0 \le t \le T} \| e^{\theta ' |v|^2 } e^{-\varpi \langle v \rangle t } \alpha  \partial  f^{m +1 - i} (t,x,v) \|_\infty  
\\ + & Tl (Ce^{Ct^2 } )^l  C_1MP(\| e^{\theta |v|^2} f_0 \|_\infty) 
\\ & +  l (Ce^{Ct^2 } )^l  \| e^{\theta ' |v|^2 } \alpha \partial f_0  \|_\infty +   P(\| e^{\theta |v|^2} f_0 \|_\infty) 
\\ & + C \left( \frac{1}{2} \right) ^l \max_{m-(l-2) \le i \le m}  \sup_{0 \le t \le T} \| e^{\theta ' |v|^2 } e^{-\varpi \langle v \rangle t } \alpha  \partial  f^{i } (t,x,v) \|_\infty.
\end{split}
\]

Finally we choose a large $l$ then large $C_1$ then small $\delta$ then large $\varpi$ and finally small $T$ to conclude the claim (\ref{C1bddinduction}):
\[ \begin{split}
 \sup_{0 \le t \le T} &  \| e^{-\varpi \langle v \rangle t } \alpha  \partial  f^{m+1} (t,x,v) \|_\infty
 \\ \le & \frac{1}{8}  \max_{m -(l-2) \le i \le m} \sup_{0 \le t \le T}  \| e^{-\varpi \langle v \rangle t } \alpha  \partial  f^{i} (t,x,v) \|_\infty
 + \frac{C_1}{2} \left(  \| e^{\theta ' |v|^2 } \alpha \partial f_0 \|_\infty + P ( \| e^{\theta |v|^2 } f_0 \|_\infty ) \right)
 \\ \le & \frac{1}{8} C_1 M + \frac{1}{2} C_1M < C_1 M.
\end{split} \]
This proves (\ref{c1bddsequence}).

		\vspace{4pt}

\textit{Step 3.} \quad
Now consider taking $\nabla_v$ derivative of the sequence (\ref{VPBsq1}) and adding the weight function $e^{-\varpi \langle v \rangle t}$, we get
\Be\begin{split}\label{eqtn_g_v}
			&[\p_t + v\cdot \nabla_x - \nabla_x \phi^m \cdot \nabla_v + \frac{v}{2} \cdot \nabla_x \phi^m + \varpi \langle v \rangle - \frac{v}{ \langle v \rangle } \varpi t \cdot \nabla_x \phi^m + \nu(\sqrt \mu f^m ) ]  (e^{-\varpi \langle v \rangle t} \nabla_v f )
			\\ =& e^{-\varpi \langle v \rangle t} \left( - \nabla_v \nu(\sqrt \mu f^m ) f^{m+1} - \nabla_x f^{m+1} - \frac{1}{2} \nabla_x\phi^m f^{m+1} + \nabla_v \Gamma_{\text{gain}}(f^m,f^m) \right),
		\end{split}\Ee
with the boundary bound for $(x,v) \in\gamma_-$
		\Be\label{bdry_g_v}
		\big|\nabla_v f^{m+1}  \big| \lesssim   |v| \sqrt{\mu} \int_{n \cdot u>0} |f^m| \sqrt{\mu} \{n \cdot u \} \dd u \ \ \text{on } \ \gamma_-.
		\Ee
And 
\[
 \frac{v}{2} \cdot \nabla_x \phi^m + \varpi \langle v \rangle - \frac{v}{ \langle v \rangle } \varpi t \cdot \nabla_x \phi^m + \nu(\sqrt \mu f^m ) > \frac{\varpi }{2} \langle v \rangle,
 \]
 for $\varpi \gg 1$.

We claim:
\begin{equation} \label{l3l1stabilityseq}
\sup_{m} \sup_{0 \le t \le T} \| e^{-\varpi \langle v \rangle t } \nabla_v f^m(t) \|_{L^3_x(\Omega)L_v^{1+\delta}(\mathbb R^3 ) } < \infty.
\end{equation}		
		Using the Duhamel's formulation, from (\ref{eqtn_g_v}) we obtain the following bound along the characteristics
		\begin{eqnarray}
		& \,\,\, & |e^{-\varpi \langle v \rangle t }  \nabla_v  f^{m+1}(t,x,v)|  \nonumber
		\\
		& \le &   \mathbf{1}_{ \{ \tb^m(t,x,v)> t \}}  
		 e^{ -\int_0^t  - \frac{C}{2}\langle V^m(\tau) \rangle d\tau } |\nabla_v f^{m+1}(0,X^m(0;t,x,v), V^m(0;t,x,v))|\label{g_initial}\\
		& + &   \ \mathbf{1}_{ \{ \tb^m(t,x,v)<t \} }
		 e^{-\varpi \langle \vb^m \rangle \tb} \mu(\vb^m)^{\frac{1}{4}}  \int_{n(\xb^m) \cdot u>0} 
		| f^m(t-\tb^m, \xb^m, u) |\sqrt{\mu} \{n(\xb^m) \cdot u\} \dd u\label{g_bdry}\\
		&  +&   \int^t_{\max\{t-\tb, 0\}} 
		 e^{ -\int_s^t  - \frac{\varpi}{2}\langle V^m(\tau) \rangle d\tau } e^{-\varpi \langle V^m(s) \rangle s }  
		  |\nabla_x f^{m+1}(s, X^m(s),V^m(s))|
		\dd s\label{g_x}\\
		&   + &\int^t_{\max\{t-\tb, 0\}} \label{g_Gamma}
		(1+ \| e^{\theta' |v|^2} f^m \|_\infty + \| e^{\theta' |v|^2 }f^{m+1}  \|_\infty )
		e^{ -\int_s^t  - \frac{\varpi}{2}\langle V^m(\tau) \rangle d\tau }  e^{-\varpi \langle V^m(s) \rangle s } 
		\\ \notag && \,\,\,  \times \int_{\R^3} \frac{e^{-C_{\theta'} |V^m(s) - u |^2 }}{ |V^m(s) - u | ^{2 -\kappa } } \nabla_v f^m(s,X^m(s),u)| \dd u 
		\dd s\label{g_K}\\
		& +  & \label{nablaphigint}
		\| e^{\theta' |v|^2} f^{m+1}\|_\infty \int^t_{\max\{t-\tb, 0\}} 
		  e^{ -\int_s^t  - \frac{\varpi}{2}\langle V^m(\tau) \rangle d\tau } e^{-\varpi \langle V^m(s) \rangle s } e^{-\theta' |V^m(s) |^2 }
		  \\&& \notag \times |\nabla_x \phi^m (s, X^m(s;t,x,v))| 
		\dd s. \label{g_phi}
		\end{eqnarray}
%
We first have
		\Be\begin{split}\label{est_g_initial}
			&\| (\ref{g_initial})\|_{L^3_x L^{1+ \delta}_v}\\
			\lesssim & \left(
			\int_{\O}
			\left(\int_{\R^3} |e^{\theta' |V^m(0) | ^2 } \nabla_v f^{m+1}(0,X^m(0 ), V^m(0 ))|^3 
			\right)
			\left(
			\int_{\R^3} e^{-(1+ \delta) \frac{3}{2-\delta} \theta' |V^m(0) | ^2 } \dd v \right)^{\frac{2-\delta}{1+ \delta}}
			\right)^{1/3} \\
			\lesssim & \
			\left(\iint_{\O \times \R^3} |e^{\theta' |V^m(0) | ^2 } \nabla_v f^{m+1}(0,X^m(0;t,x,v), V^m(0;t,x,v))|^3 \dd v \dd x\right)^{1/3}\\
			\lesssim & \ \| e^{\theta' |v| ^2 } \nabla_v f (0) \|_{L^3_{x,v}},
		\end{split}
		\Ee
		where we have used a change of variables $(x,v) \mapsto (X^m(0;t,x,v), V^m(0;t,x,v))$.

		\hide
		Also we use $|V(0;t,x,v)| \gtrsim |v|$ for $|v|\gg 1$, from (\ref{decay_phi}), and hence $\tilde{w}(V(0;t,x,v))^{- (1+ \delta) \frac{3}{2-\delta}} \in L_v^1 (\R^3)$.\unhide
		
		Clearly 
		\Be\label{est_g_bdry}
		\|(\ref{g_bdry})\|_{L^3_x L^{1+ \delta}_v}  \lesssim \sup_{0 \leq s \leq t} \| e^{\theta' |v|^2} f^m (s) \|_\infty.
		\Ee
		
		From $W^{1,2}(\O)\subset L^6(\O)\subset L^2(\O)$ for a bounded $\O \subset \R^3$, and the change of variables $(x,v) \mapsto (X(s;t,x,v), V(s;t,x,v))$ for fixed $s\in(\max\{t-\tb,0\},t)$,\Be
		\begin{split}\label{est_g_phi}
			\|(\ref{nablaphigint})\|_{L^3_x L^{1+ \delta}_v} 
			\lesssim & \  \| e^{\theta' |v|^2} f^{m+1}\|_\infty \int^t_{\max\{t-\tb,0\}} \| e^{- \frac{\theta' }{2}|v|^2 }  \nabla_x \phi^m(s,X(s;t,x,v))  \|_{L^3_{x,v} }\| 
			e^{- \frac{\theta' }{2}|v|^2}
			\|_{L^{
					\frac{3(1+ \delta)}{2- \delta}
				}_v}\\
			\lesssim & \ \| e^{\theta' |v|^2} f^{m+1}\|_\infty \int^t_{\max\{t-\tb,0\}} \| \nabla_x \phi^m (s)   \|_{L^3_{x} }
			\lesssim   \| e^{\theta' |v|^2} f^{m+1}\|_\infty \int^t_{\max\{t-\tb,0\}} \| \phi^m (s) \|_{W^{2,2}_{x} } 
			\\
			\lesssim & \ \| e^{\theta' |v|^2} f^{m+1}\|_\infty \int^t_{\max\{t-\tb,0\}} \| \int_{\mathbb R^3} \sqrt \mu f^m(s) dv - \rho_0 \|_{2}.
			\\ \lesssim &  t  \| e^{\theta' |v|^2} f^{m+1}\|_\infty \| e^{\theta' |v|^2 } f^m \|_\infty .
		\end{split}
		\Ee

%
%

		
		Next we have from (\ref{int1overalphadv}), (\ref{int1overxi}), for $\frac{3\delta}{ 2 (1+\delta) } < 1$, equivalently $0 < \delta < 2 $,
				\Be\label{init_p_xf}
		\begin{split}
			 \|(\ref{g_x})\|_{L^3_x L^{1+ \delta}_v}  \le &\left\|\left\| \int^t_{\max\{t-\tb, 0\}}
			\nabla_x f^{m+1}(s,X^m(s),V^m(s)) \dd s
			\right\|_{L_{v}^{1+ \delta}(  \R^3)}\right\|_{L^{3}_x}\\
			= & \ \left\|\left\| \int^t_{\max\{t-\tb, 0\}}    \frac{ e^{\theta' |V^m(s) |^2 } e^{-\varpi \langle V^m(s) \rangle s} \alpha \nabla_x f^{m+1}(s,X^m(s),V^m(s))}{e^{\theta' |V^m(s) |^2 } e^{-\varpi \langle V^m(s) \rangle s} \alpha}
			\dd s
			\right\|_{L_{v}^{1+ \delta}(  \R^3)}\right\|_{L^{3}_x}
			\\
			\le & \sup_{0 \le t \le T} \ \left\|  e^{\theta' |v|^2 } e^{-\varpi \langle v \rangle t} \alpha \nabla_x f^{m+1} \right\|_\infty \\
			& \times \left\|
			 \left\| \int^t_{\max\{t-\tb, 0\}}  \frac{e^{- \theta' |V^m(s) |^2 } e^{\varpi \langle V^m(s) \rangle s }}{\alpha(s, X^m(s), V^m(s) )}		\dd s
\right\|_{L_{v}^{1+\delta}( \R^3)}\right\|_{L^{3}_x}\\
			\lesssim &   e^{C ( \| \nabla \phi^m \|_\infty + \| \nabla \phi^m \|_\infty^2 +\| \nabla ^2 \phi^m \|_\infty) }
 \sup_{0 \le t \le T} \ \left\|  e^{\theta' |v|^2 } e^{-\varpi \langle v \rangle t } \alpha \nabla_x f^{m+1} \right\|_\infty  
					\\ & \times  t  \int_{\Omega }  \left( \int_{\mathbb R^3 }   \frac{e^{- \frac{\theta'}{2} |v |^2 }}{(\alpha (t, x, v ))^{1+\delta}}  \dd v \right)^{\frac{3}{1+\delta}}  \dd x			
		\\ \lesssim & t e^{C ( \| \nabla \phi^m \|_\infty + \| \nabla \phi^m \|_\infty^2 +\| \nabla ^2 \phi^m \|_\infty) }  \sup_{0 \le t \le T} \ \left\|  e^{\theta' |v|^2 } e^{-\varpi \langle v \rangle t } \alpha \nabla_x f^{m+1} \right\|_\infty,
		\end{split}\Ee
		where we have used 
		\[
		\alpha(s,X^m(s;t,x,v),V^m(s;t,x,v)) \ge  e^{ - C ( \| \nabla \phi^m \|_\infty + \| \nabla ^2 \phi^m \|_\infty) } \alpha(t,x,v).
		\]
		 		
Next, we consider (\ref{g_Gamma}). From (\ref{genlemma1}) and the computations in (\ref{int1overalphadv}), (\ref{int1overxi}), we have for $1 < \beta <2$,
\begin{equation} \label{nonlocall3l1est} \begin{split}
 & \|(\ref{g_Gamma})\|_{L^3_x L^{1+ \delta}_v}  
 \\ \le & \left\|\left\| \int^t_{\max\{t-\tb, 0\}}e^{ -\int_s^t  - \frac{\varpi}{2}\langle V^m(\tau) \rangle d\tau }  e^{-\varpi \langle V^m(s) \rangle s } 
 \int_{\R^3} \frac{e^{-C_{\theta'} |V^m(s) - u |^2 }}{ |V^m(s) - u | ^{2 -\kappa } } \nabla_v f^m(s,X^m(s),u)| \dd u \dd s
			\right\|_{L_{v}^{1+ \delta}(  \R^3)}\right\|_{L^{3}_x}
\\ \le & \sup_{0 \le t \le T} \ \left\|  e^{\theta' |v|^2 } e^{-\varpi \langle v \rangle t} \alpha \nabla_x f^{m} \right\|_\infty
\\ & \times   \left\|\left\| \int^t_{\max\{t-\tb, 0\}}e^{ -\int_s^t  - \frac{\varpi}{2}\langle V^m(\tau) \rangle d\tau } \int_{\R^3} \frac{e^{-C_{\theta'} |V^m(s) - u |^2 }}{ |V^{m}(s) - u | ^{2 -\kappa } } \frac{e^{- \theta' |u|^2 } e^{\varpi \langle V^{m-1}(s) \rangle s}}{\alpha(s,X(s),u)} \dd u \dd s
			\right\|_{L_{v}^{1+ \delta}(  \R^3)}\right\|_{L^{3}_x}
\\ \lesssim &e^{C \| \nabla \phi^{m-1} \|_\infty }   \sup_{0 \le t \le T} \ \left\|  e^{\theta' |v|^2 } e^{-\varpi \langle v \rangle t } \alpha \nabla_x f^{m} \right\|_\infty 
\\ & \times \left\|\left\| \int^t_{\max\{t-\tb, 0\}}e^{ -\int_s^t  - \frac{\varpi}{2}\langle V^m(\tau) \rangle d\tau } \int_{\R^3} \frac{e^{-C_{\theta'} |V^m(s) - u |^2 }}{ |V^{m}(s) - u | ^{2 -\kappa } } \frac{e^{- \frac{\theta'}{2} |u|^2 }  }{(\alpha(s,X(s),u))^\beta} \dd u \dd s
			\right\|_{L_{v}^{1+ \delta}(  \R^3)}\right\|_{L^{3}_x}
\\ \lesssim &e^{C (\| \nabla \phi^{m-1} \|_\infty  + \| \nabla \phi^m \| +  \| \nabla \phi^m \|^2 +  \| \nabla^2 \phi^m \|)}  \sup_{0 \le t \le T} \ \left\|  e^{\theta' |v|^2 } e^{- \varpi \langle v \rangle t } \alpha \nabla_x f^{m} \right\|_\infty 
\\ & \times \left\|\left\|
\frac{  \delta ^{\frac{3 -\beta}{2} }}{(\alpha(t,x,v))^{\beta -2 }  (|v|^2 + 1 )^{\frac{3 -\beta}{2}}}  + \frac{ (|v| +1)^{\beta -1} } { \delta ^{\beta -1 } \varpi \langle v \rangle  ( \alpha ( t,x,v) )^{\beta - 1 } }			\right\|_{L_{v}^{1+ \delta}(  \R^3)}\right\|_{L^{3}_x}
\\ \lesssim &e^{C (\| \nabla \phi^{m-1} \|_\infty  + \| \nabla \phi^m \| +  \| \nabla \phi^m \|^2 +  \| \nabla^2 \phi^m \|)}  \sup_{0 \le t \le T} \ \left\|  e^{\theta' |v|^2 } e^{-\varpi \langle v \rangle t } \alpha \nabla_x f^{m} \right\|_\infty 
\\ & \times \left( O(\delta^{\frac{3-\beta}{2}} )   + \frac{1}{ \delta ^{\beta -1 } \varpi}  \left\|\left\|
   \frac{ 1} { \langle v \rangle^{2 -\beta}  ( \alpha ( t,x,v) )^{\beta - 1 } }			\right\|_{L_{v}^{1+ \delta}(  \R^3)}\right\|_{L^{3}_x} \right)
\\ \lesssim & C(\delta^{\frac{3-\beta}{2}}  +  \frac{1}{ \delta ^{\beta -1 } \varpi}  )e^{C (\| \nabla \phi^{m-1} \|_\infty  + \| \nabla \phi^m \|_\infty +  \| \nabla \phi^m \|_\infty^2 +  \| \nabla^2 \phi^m \|_\infty)}  \sup_{0 \le t \le T} \ \left\|  e^{\theta' |v|^2 } e^{-\varpi \langle v \rangle t } \alpha \nabla_x f^{m} \right\|_\infty,
\end{split} \end{equation}
for $\beta$ satisfies $\frac{  (\beta-1)(1+\delta) -1}{2} \frac{3}{1+\delta} < 1$, which is equivalent to $\beta< \frac{5}{3} + \frac{1}{1+\delta}$. Therefore any $1 < \beta < \frac{5}{3}$ would work.

		Collecting terms from (\ref{g_initial})-(\ref{nablaphigint}), and (\ref{est_g_initial}), (\ref{est_g_bdry}), (\ref{est_g_phi}), (\ref{init_p_xf}), (\ref{nonlocall3l1est}), we derive
		\Be\begin{split}\label{bound_nabla_v_g}
			&\sup_m \sup_{0 \leq s \leq t}\| e^{-\varpi \langle v \rangle t }  \nabla_vf^m(s) \|_{L^3_xL^{1+ \delta}_v} \\
	\lesssim & \| e^{\theta' |v| ^2 } \nabla_v f (0) \|_{L^3_{x,v}} + \sup_m \| e^{\theta' |v|^2 } f^m \|_\infty )^2 + \sup_m \| e^{\theta' |v|^2 } f^m \|_\infty
	\\ & + \sup_m e^{C (\| \nabla \phi^{m} \|_\infty   +  \| \nabla \phi^m \|_\infty^2 +  \| \nabla^2 \phi^m \|_\infty)} \sup_m \sup_{0 \le t \le T} \ \left\|  e^{\theta' |v|^2 } e^{-\varpi \langle v \rangle t } \alpha \nabla_x f^{m} \right\|_\infty
	\\ < & \infty.
			\end{split} \Ee
This proves (\ref{l3l1stabilityseq}).


		\vspace{4pt}

\textit{Step 4.} \quad
let $ h^m = e^{-\varpi \langle v \rangle t } f^m $ where $f^m$ is constructed in (\ref{VPBsq1}). We claim for $\varpi \gg 1$, and $0<T \ll 1$ small enough, that
\begin{equation} \label{strongcovinL1+}
h^m \to h\text{ strongly in } \, L^{\infty}((0,T); L^{1+}( \Omega \times \mathbb R^3 )),
\end{equation}
for some $h$.
%
By direction computation we get from (\ref{VPBsq1}) that
\begin{equation} \begin{split} \label{VPBsqw}
 (\p_t + & v\cdot \nabla_x - \nabla_x \phi^m \cdot \nabla_v + \frac{v}{2} \cdot \nabla_x \phi^m + \varpi \langle v \rangle - \frac{v}{ \langle v \rangle } \varpi t \cdot \nabla_x \phi^m + \nu(\sqrt \mu f^m ) ) (h^{m+1}) 
 \\ = & e^{-\varpi \langle v \rangle t }  \Gamma_{gain} (f^m,f^m).
\end{split} \end{equation}
Note that $h^{m+1} - h^m$ satisfies $h^{m+1} - h^m)|_{t=0} \equiv0$, so from (\ref{VPBsqw}) we have
		\begin{equation} \label{wf^m+1-wf^m}
		\begin{split}
		& \big[\p_t + v\cdot \nabla_x- \nabla_x  \phi^m  \cdot \nabla_v
		 + \frac{v}{2} \cdot \nabla_x \phi^m + \varpi \langle v \rangle - \frac{v}{ \langle v \rangle } \varpi t \cdot \nabla_x \phi^m + \nu(\sqrt \mu f^m ) \big](h^{m+1}-h^{m}) \\
		&=  (\nabla_x \phi_{F^m}- \nabla_x \phi_{F^{m-1}} ) \cdot \nabla_v (h^{m})   - (\frac{v}{2} -\frac{v}{ \langle v \rangle } \varpi t  ) \cdot (\nabla_x \phi_{F^m} - \nabla_x \phi_{F^{m-1}}) h^{m}  \\
	           & + e^{-\varpi \langle v \rangle t }  \Gamma_{\text{gain}}(f^m,f^m) - e^{-\varpi \langle v \rangle t}\Gamma_{\text{gain}}(f^{m-1},f^{m-1}) - \nu (\sqrt \mu (f^{m} -f^{m-1} ) ) h^m.
		\end{split}\end{equation} 
Now since 
\[
\nu_{\varpi}^m := \frac{v}{2} \cdot \nabla_x \phi^m + \varpi \langle v \rangle - \frac{v}{ \langle v \rangle } \varpi t \cdot \nabla_x \phi^m + \nu(\sqrt \mu f^m ) > \frac{\varpi}{2} \langle v \rangle,
\]
for  $\varpi \gg 1$.
Then by the Green's theorem for $L^{1+ \delta}$-space with $0 < \delta \ll 1$, we obtain from (\ref{wf^m+1-wf^m}) that  
		\Be\begin{split}\label{f-g_energy}
			&\| [h^{m+1}- h^m](t)\|_{1+ \delta}^{1+ \delta} + \int^t_0 \| (\nu_{\varpi}^m )^{ {1}/{1+ \delta}} [h^{m+1}- h^m] \|_{1+ \delta}^{1+ \delta}  + \int_0^t |[h^{m+1}- h^m]|_{1+ \delta, + }^{1 + \delta} \\
			\leq & \ \| [h^{m+1}- h^m](0)\|_{1+ \delta}^{1+ \delta}+
			\int^t_0 \iint_{\O \times \R^3}
			|\text{RHS of } (\ref{wf^m+1-wf^m}) | |h^{m+1}- h^m|^{\delta}
			+\int_0^t |[h^{m+1}- h^m]|_{1+ \delta, - }^{1 + \delta} .
		\end{split}
		\Ee	
For $0 < \delta \ll 1$, by the H\"older inequality with $1=\frac{1}{\frac{3(1+ \delta)}{2- \delta}} + \frac{1}{3} +  \frac{1}{\frac{1+ \delta}{\delta}}$ and the Sobolev embedding $W^{1, 1+ \delta} (\O)\subset L^{\frac{3(1+ \delta)}{2- \delta}}(\O)$ when $\O \subset \R^3$,
		\Be\label{gronwall_f-g}
		\begin{split}
			&\int^t_0 \iint_{\O \times \R^3} |(\nabla_x \phi_{F^m}- \nabla_x \phi_{F^{m-1}}) \cdot \nabla_v h^m|  |h^{m+1}-h^m|^{\delta}  \\
			& \lesssim 
			\int^t_0 \| \nabla_x \phi_{F^m}- \nabla_x \phi_{F^{m-1}}\|_{L^{ \frac{3(1+ \delta)}{2- \delta}}_x} \| \nabla_v h^m \|_{L^{3}_xL^{1+ \delta}_v} \left\| |h^m-h^{m-1}|^\delta\right\|_{L_{x,v}^{\frac{1+ \delta}{\delta}}}\\
			& \lesssim \sup_{0 \leq s \leq t} \| \nabla_v h^m (s) \|_{L^{3}_xL^{1+ \delta}_v}\times \int^t_0 \| [h^m-h^{m-1}](s) \|_{1+ \delta}^{1+ \delta} \dd s.
		\end{split} \Ee
%
%
%
We also have
\Be \label{RHS1} \begin{split}
& \int_0^t  \int_\Omega \int_{\mathbb R^3 } e^{- ( 1+\delta) \varpi \langle v \rangle s }\Gamma_{\text{gain}} (f^m, f^m - f^{m-1} ) | (f^m - f^{m-1} ) (v) | ^\delta dv dx ds
\\ \lesssim & \int_0^t  \int_\Omega \int_{\mathbb R^3 }   e^{- ( 1+\delta) \varpi \langle v \rangle s } \| e^{\theta' |v|^2 } f^m \|_\infty  (\int_{\mathbb R^3 } \frac{e^{-C_{\theta'} |v-u|^2 }}{|v -u |^{2 - \kappa}  } | ( f^m  -f^{m-1} )(u) | du ) | (f^m - f^{m-1} ) (v) | ^\delta dv dx ds
\\   \lesssim &  \int_0^t \int_\Omega \int_{\mathbb R^3 }  e^{- ( 1+\delta) \varpi \langle v \rangle s }   \| e^{\theta' |v|^2 } f^m \|_\infty   (\int_{\mathbb R^3 } (\frac{e^{-C_{\theta'} |v-u|^2 }}{|v -u |^{2 - \kappa}  } )|  f^m (u)- f^{m-1} (u) |^{1+\delta} du)^{1/(1+\delta)}
\\   & \qquad \qquad \qquad  \times  |f^m(v) - f^{m-1}(v) | ^\delta dv dx ds
\\  \lesssim &  \| e^{\theta' |v|^2 } f^m \|_\infty \int_0^t \int_\Omega \int_{\mathbb R^3 }  e^{- ( 1+\delta) \varpi \langle v \rangle s }  | f^{m}(v) - f^{m-1}(v) | ^{1 +\delta} dvdxds 
\\ & + \| e^{\theta' |v|^2 }  f^m\|_\infty \int_0^t \int_\Omega \int_{\mathbb R^3 } \int_{\mathbb R^3 }  e^{- ( 1+\delta) \varpi \langle v \rangle s } (\frac{e^{-C_{\theta'} |v-u|^2 }}{|v -u |^{2 - \kappa}  } )|  f^m (u) - f^{m-1}(u) |^{1+\delta} du dv dx ds
\\  = &  \| e^{\theta' |v|^2}f^m  \|_\infty \int_0^t \int_\Omega \int_{\mathbb R^3 }   e^{- ( 1+\delta) \varpi \langle v \rangle s }| f^{m}(v)-f^{m-1}(v) | ^{1+ \delta} dvdxds 
\\ & + \| e^{\theta' |v|^2 }f^m  \|_\infty \int_0^t \int_\Omega \int_{\mathbb R^3 } (\int_{\mathbb R^3 } \left( \frac{e^{-\varpi \langle v \rangle s }}{e^{-\varpi \langle u \rangle s }} \right)^{1+\delta} \frac{e^{-C_{\theta'} |v-u|^2 }}{|v -u |^{2 - \kappa}  } ) dv) e^{- ( 1+\delta) \varpi \langle u \rangle s } |  f^m (u) - f^{m-1}(u) |^{1 + \delta} dudxds
\\  = &  \| e^{\theta' |v|^2 }f^m  \|_\infty \int_0^t \int_\Omega \int_{\mathbb R^3 }   e^{- ( 1+\delta) \varpi \langle v \rangle s } | f^{m}(v)-f^{m-1}(v) | ^{1+ \delta} dvdxds 
\\ & + \| e^{\theta' |v|^2 }f^m  \|_\infty \int_0^t \int_\Omega \int_{\mathbb R^3 } (\int_{\mathbb R^3 } \frac{e^{2(1+\delta)\varpi \langle v -u \rangle-C_{\theta'} |v-u|^2 }}{|v -u |^{2 - \kappa}  } ) dv)  e^{- ( 1+\delta) \varpi \langle u \rangle s } |  f^m (u) - f^{m-1}(u) |^{1 + \delta} dudxds
\\  \lesssim &  \| e^{\theta' |v|^2 } f^m \|_\infty  \int_0^t \int_\Omega \int_{\mathbb R^3 }  e^{- ( 1+\delta) \varpi \langle v \rangle s } |f^m(v)- f^{m-1}(v) | ^{1+\delta} dvdxds 
\\ & +  \| e^{\theta' |v|^2}f^m  \|_\infty \int_0^t \int_\Omega \int_{\mathbb R^3 }   e^{- ( 1+\delta) \varpi \langle u \rangle s } | f^m(u) - f^{m-1} (u) |^{1 + \delta} dudxds.
\end{split} \Ee
And similarly, we have
\Be \label{RHS2} \begin{split}
\nu(\sqrt \mu(f^m -f^{m-1}) ) & e^{- ( 1+\delta) \varpi \langle v \rangle s} |f^m(v) -f^{m-1}(v)|^\delta 
\\ & \lesssim \| e^{\theta'|v|^2 } f^m\|_\infty  \left(\int_{\mathbb R^3 } \frac{e^{-C_{\theta'} |v-u|^2 }}{|v -u |^{2 - \kappa}  } | ( f^m  -f^{m-1} )(u) | du \right)  e^{- ( 1+\delta) \varpi \langle v \rangle s } | (f^m - f^{m-1} ) (v) | ^\delta.
\end{split} \Ee
Thus we use \eqref{RHS1}, \eqref{RHS2} to conclude that 
\begin{equation} \label{hseqnonlocalbdd}
\begin{split}
\int_0^t \iint_{\Omega \times \mathbb R^3} & |\text{RHS of } (\ref{wf^m+1-wf^m}) | |h^{m+1}- h^m|^{\delta} 
\\ \lesssim & ( \max_{i=m,m-1} \sup_{0 \le s \le t} \| e^{\theta' |v|^2 } f^i(s) \|_\infty + \sup_{0 \leq s \leq t} \| \nabla_v h^m (s) \|_{L^{3}_xL^{1+ \delta}_v} ) \int_0^t \| [h^m - h^{m-1}](s) \|_{1+\delta}^{1+\delta}.
\end{split}
 \end{equation}
		
		Then following the argument of (\ref{gamma-intestimate}) and applying the trace theorem, we can obtain
		\Be\label{gamma_+,1+delta}
		\begin{split}
			&\int_0^t |[h^{m+1}-h^{m}]|_{1+ \delta, - }^{1 + \delta} \\
			&\lesssim   o(1) \int_0^t |[h^{m+1}-h^m]|_{1+ \delta, + }^{1 + \delta}+\| [h^{m+1}- h^m](0)\|_{1+ \delta}^{1+ \delta} \\
			&  +
			\sup_{0 \leq s \leq t}\big\{ 1+ \| \nabla_v h^{m-1} (s) \|_{L^{3}_xL^{1+ \delta}_v} + \|e^{\theta' |v|^2} f^{m-1}(s) \|_\infty + \|e^{\theta' |v|^2}  f^{m-2} (s) \|_\infty
			\big\}  \int^t_0 \| [h^{m-1}-h^{m-2}](s) \|_{1+ \delta}^{1+ \delta}.
		\end{split}
		\Ee
		
Now using $[h^{m+1} - h^m] ( 0) = 0 $, and combining (\ref{f-g_energy}), (\ref{hseqnonlocalbdd}), and (\ref{gamma_+,1+delta}) we conclude that
\[ \begin{split}
\sup_{0 \leq s \leq t} & \| h^{m+1} (s) - h^m(s) \|_{1+\delta}^{1+\delta} 
\\  \lesssim&  t(1 + \sup_{0 \le s \le t }  \sup_i  \| e^{\theta' |v|^2} f^i \|_\infty + \sup_{0 \leq s \leq t} \sup_{i} \| \nabla_v h^i (t) \|_{L^3_x L^{1+ \delta}_v} )
\\ & \times   ( \sup_{0 \le s \le t }  \| h^{m} (s) - h^{m-1}(s) \|_{1+\delta}^{1+\delta} +  \sup_{0 \le s \le t }  \| h^{m-1} (s) - h^{m-2}(s) \|_{1+\delta}^{1+\delta}).
\end{split} \]
Then by (\ref{uniformmlinftybdd}), (\ref{l3l1stabilityseq}), we have for $t \ll 1$ small enough,
\[ \begin{split}
\sup_{0 \leq s \leq t} & \| h^{m+1} (s) - h^m(s) \|_{1+\delta}^{1+\delta} + \sup_{0 \leq s \leq t}  \| h^{m+2} (s) - h^{m+1}(s) \|_{1+\delta}^{1+\delta} 
\\  \le&  O(t)    \left( \sup_{0 \le s \le t }  \| h^{m} (s) - h^{m-1}(s) \|_{1+\delta}^{1+\delta} +  \sup_{0 \le s \le t }  \| h^{m-1} (s) - h^{m-2}(s) \|_{1+\delta}^{1+\delta}\right).
\end{split} \]		
Therefore, inductively we have
\[ \begin{split}
& \sup_{0 \leq s \leq t}   \| h^{m+1} (s) - h^m(s) \|_{1+\delta}^{1+\delta} 
\\ & \le  \sup_{0 \leq s \leq t}  \| h^{m+1} (s) - h^m(s) \|_{1+\delta}^{1+\delta} + \sup_{0 \leq s \leq t}  \| h^{m+2} (s) - h^{m+1}(s) \|_{1+\delta}^{1+\delta} 
\\ & \le O(t)^m.
\end{split} \]
Hence we derive stability
\[
\sup_{0 \leq s \leq t}  \| h^{m} (s) - h^l(s) \|_{1+\delta}^{1+\delta}  \le O(t)^{\min\{m,l\}}.
\]
Therefore we conclude
\[
h^m \to h\text{ strongly in } \, L^{\infty}((0,T); L^{1+}( \Omega \times \mathbb R^3 )),
\]
for some $h$, and this proves (\ref{strongcovinL1+}).

		\vspace{4pt}

\textit{Step 5.} \quad
From (\ref{uniformmlinftybdd}) we have up to a subsequence the weak-$\ast$ convergence: $ e^{\theta' |v|^2} f^m(t,x,v) \overset{\ast}{\rightharpoonup} e^{\theta' |v|^2} f(t,x,v) $ in $L^\infty ([0,T) \times \Omega \times \mathbb R^3) \cap L^\infty ([0,T) \times \gamma) $ for some $f$. By (\ref{strongcovinL1+}) the limit is unique, therefore $ (e^{\theta' |v|^2} f^m(t,x,v), e^{\theta' |v|^2} f^{m+1}(t,x,v)) \overset{\ast}{\rightharpoonup} (e^{\theta' |v|^2} f(t,x,v), e^{\theta' |v|^2} f(t,x,v))$.

		Thus from (\ref{VPBsq1}), we have for any $\varphi \in C^\infty_c (\R \times \bar{\O} \times \R^3)$, 
		\Be\begin{split}\label{weak_form_ell}
			&\int_0^T\iint_{\O \times \R^3} f^{m+1} [-\p_t  - v\cdot \nabla_x  +\nabla_x \phi_E \cdot \nabla_v  + \frac{v}{2} \cdot \nabla_x \phi_E ] \varphi
			+\underbrace{ f^{m+1} \{\nabla_x \phi_{F^m} \cdot \nabla_v\varphi
				+ \frac{v}{2} \cdot \nabla_x \phi_{F^m} \varphi\}}_{(\ref{weak_form_ell})_\phi}
			\\
			=& \int_0^T \iint_{\O \times \R^3}  
			  \underbrace{\Gamma_{\text{gain}} ( {f^m} , {f^m} ) \varphi}_{(\ref{weak_form_ell})_\text{gain}} -   \underbrace{\nu (\sqrt \mu f^m ) f^{m+1}\varphi }_{(\ref{weak_form_ell})_{\text{loss}}}\\
			&+ \int^T_0 \int_{\gamma_+} f^{m+1} \varphi - \int^T_0 \int_{\gamma_-} 
			c_\mu \sqrt{\mu} \int_{n \cdot u>0} f^m \sqrt{\mu}\{n \cdot u\} \dd u 
			\varphi.
		\end{split}\Ee
Except the underbraced terms in (\ref{weak_form_ell}) all terms converges to limits with $f$ instead of $f^{m+1}$ or $f^m$.

		We define, for $(t,x,v) \in \mathbb{R} \times  \bar{\Omega} \times \mathbb{R}^{3}$ and for $0 < \delta \ll 1$, 
		\begin{equation}\label{Z_dyn}
		\begin{split}
		f_{\delta}^m(t,x,v) 
		:=  &  \ \kappa_\delta (x,v) f^m(t,x,v)
		\\
		: =  &  \  \chi\Big(\frac{|n(x) \cdot v|}{\delta}-1\Big) 
		\Big[ 1- \chi(\delta|v|) \Big]  
		f^m(t,x,v )
		.
		\end{split}
		\end{equation}
		Note that $f_\delta (t,x,v)=0$ if either $|n(x) \cdot v| \leq \delta$ or $|v| \geq \frac{1}{\delta}$.
		%
Now		\Be\notag
		\begin{split}
			&\left|\int^T_0 \iint (\ref{weak_form_ell})_\text{loss}
			- \int^T_0 \iint \nu (\sqrt{\mu } f) \varphi\right|
			\\
			\leq &\left| \int^T_0 \iint_{\O \times \R^3}
			\int_{\R^3}   |v-u|^\kappa q_0 \{f^m ( u) - f(u) \}\sqrt{\mu(u)}  \dd u f^{m+1} (v) \varphi(t,x,v
			) \dd v \dd x \dd t\right|\\
			&+\left|\int^T_0 \iint_{\O \times \R^3}
			\int_{\R^3}   |v-u|^\kappa q_0  f(u) \sqrt{\mu(u)}  \dd u \{ f^{m+1} (v)- f(v) \} \varphi(t,x,v
			) \dd v \dd x \dd t\right|.
		\end{split}\Ee
		The second term converges to zero from the weak$-*$ convergence in $L^\infty$ by (\ref{uniformmlinftybdd}). The first term is bounded by, from (\ref{uniformmlinftybdd}),
		\Be \label{difference_f^ell-f}
		\begin{split}
			&\left[\int^T_0 \left\|  \int_{\R^3} \kappa_\delta (x,u) (f^m(t,x,u) - f(t,x,u)) \langle u\rangle^\kappa \sqrt{\mu(u)} \dd u \right\|^2_{L^2(\O \times \R^3)}\right]^{1/2}\\
			&\times \sup_{0 \leq t \leq T} \|w_{\vartheta} f^{m+1}(t)\|_\infty 
			+O(\delta).
		\end{split} \Ee
		
		On the other hand, from Lemma \ref{extension_dyn}, we have an extension $\bar{f}^m(t,x,v)$ of $\kappa_\delta (x,u)  f^m(t,x,u)$. Note that from (\ref{linfinitybdofpotential}) $\sup_m \| \nabla \phi^m \|_\infty < \infty$, and $ \nabla \phi^{m-1} \cdot \nabla_v f^m = \nabla_v \cdot (  \nabla \phi^{m-1} f^m )$ with $\sup_m \| \nabla \phi^{m-1} f^m \|_{L^2} < \infty$. Thus we apply the average lemma (see Theorem 7.2.1 in page 187 of \cite{gl}, for example) to $\bar{f}^m(t,x,v)$. From (\ref{uniformmlinftybdd}),
		
		\Be\label{H_1/4}
		\sup_m\left\| \int_{\R^3}   \bar{f}^m(t,x,u)  
		{\langle u\rangle}^{\kappa} \sqrt{\mu(u)} \dd u
		\right\|_{H^{1/4}_{t,x} (\R \times \R^3)}< \infty.
		\Ee
		Then by $H^{1/4} \subset\subset L^2$, up to subsequence, we conclude that 
		\[
		\int_{\R^3} \kappa_\delta (x,u)  f^m (t,x,u)  \langle u\rangle^\kappa \sqrt{\mu(u)} \dd u 
		\rightarrow \int_{\R^3} \kappa_\delta (x,u)  f(t,x,u) \langle u\rangle^\kappa \sqrt{\mu(u)} \dd u \  \ \text{strongly in } L^2_{t,x}.
		\]
		So we conclude that $(\ref{difference_f^ell-f})\rightarrow 0$ as $m \rightarrow \infty$.
		
		For $(\ref{weak_form_ell})_{\text{gain}}$ let us use a test function $\varphi_1(v) \varphi_2 (t,x)$. From the density argument, it suffices to prove a limit by testing with $\varphi(t,x,v)$.
		
		We use a standard change of variables $(v,u) \mapsto(v^\prime,u^\prime)$ and $(v,u) \mapsto(u^\prime,v^\prime)$ (for example see page 10 of \cite{gl}) to get    
		\begin{eqnarray} 
		&&\int^T_0 \iint  (\ref{weak_form_ell})_{\text{gain}} - \int^T_0 \iint \Gamma_{\text{gain}}(f,f) \varphi \nonumber\\
		& =& \ \int^T_0 \iint\Gamma_{\text{gain}}(f^m - f,f^m) \varphi + \int^T_0 \iint\Gamma_{\text{gain}}( f,f^m-f) \varphi \nonumber \\
		&= &   \int^T_0 \iint_{\O \times \R^3} \left( \int_{\R^3}  \int_{\S^2} 
		(   f^m (t,x,u) -  f(t,x,u)  )   \sqrt{\mu(u^\prime)}     |v-u|^\kappa q_0   \varphi_1(v^\prime)            \dd \o\dd u\right)\nonumber\\
		&& \ \ \ \ \ \ \ \ \ \ \ \ \ \   \times 
		f^m (t,x,v)   
		\varphi_2 (t,x)
		\dd v \dd x\dd t  \label{weak_fell-f1}\\
		&  +&  \int^T_0 \iint_{\O \times \R^3} \left( \int_{\R^3}  \int_{\S^2} 
		(   f^m (t,x,u) -  f(t,x,u)  )   \sqrt{\mu(v^\prime)}    |v-u|^\kappa q_0  \varphi_1(u^\prime)            \dd \o\dd u\right)\nonumber\\
		&& \ \ \ \ \ \ \ \ \ \ \ \ \ \   \times 
		f (t,x,v)   
		\varphi_2 (t,x)
		\dd v \dd x\dd t  . \label{weak_fell-f2}
		\end{eqnarray} 
		For $N\gg 1$ we decompose the integration of (\ref{weak_fell-f1}) and (\ref{weak_fell-f2}) using 
		\Be
		\begin{split}\label{decomposition_N}
			1=& \{1- \chi (|u|-N)\}\{1- \chi (|v|-N)\}\\
			& +  \chi (|u|-N) + \chi (|v|-N) - \chi (|u|-N)   \chi (|v|-N).
		\end{split} \Ee
		Note that $\{1- \chi (|u|-N)\}\{1- \chi (|v|-N)\} \neq 0$ if $|v| \leq N+1$ and $|u| \leq N+1$, and if $ \chi (|u|-N) + \chi (|v|-N) - \chi (|u|-N)   \chi (|v|-N) \neq 0$ then either $|v|\geq N$ or $|u| \geq N$. From (\ref{uniformmlinftybdd}), the second part of (\ref{weak_fell-f1}) and (\ref{weak_fell-f2}) from (\ref{decomposition_N}) are bounded by 
		\Be
		\begin{split}
			&\int^T_0 \iint_{\O \times \R^3} \int_{\R^3} \int_{\S^2} [\cdots]  \times \{\chi (|u|-N) + \chi (|v|-N) - \chi (|u|-N)   \chi (|v|-N)\}\\
			\leq &  \ \sup_\ell \| w_{\vartheta}f^\ell \|_\infty \| w_{\vartheta} f\|_\infty \times 
			\{
			e^{-\frac{\vartheta}{2} |v|^2} e^{-\frac{\vartheta}{2} |u|^2} 
			\}\{ \mathbf{1}_{|v|\geq N} +  \mathbf{1}_{|u|\geq N}  \}\\
			\leq& \  O(\frac{1}{N}). \notag
		\end{split}
		\Ee	
		Now we only need to consider the parts with $ \{1- \chi (|u|-N)\}\{1- \chi (|v|-N)\}$. Then 
		\Be\begin{split}\label{bound_weak_fell-f1}
			&(\ref{weak_fell-f1})\\
			= & \int^T_0 \iint_{\O \times \R^3}    \int_{\R^3}  
			(   f^m (t,x,u) -  f(t,x,u)  ) \\
			& \ \ \ \ \ \ \ \ \ \ \ \ \ \   \ \  \times  \{1- \chi (|u|-N)\}\left(\int_{\S^2}  \sqrt{\mu(u^\prime)}     |v-u|^\kappa q_0   \varphi_1(v^\prime)            \dd \o \right)\dd u \\
			& \ \ \ \ \ \ \ \ \ \ \ \ \ \   \ \  \times 
			\{1- \chi (|v|-N)\}  f^m (t,x,v)   
			\varphi_2 (t,x)
			\dd v \dd x\dd t.
		\end{split}
		\Ee
		Let us define 
		\Be\label{Phi_v}
		\Phi _v(u ) :=  \{1- \chi (|u|-N)\}\int_{\S^2}  \sqrt{\mu(u^\prime)}     |v-u|^\kappa q_0   \varphi_1(v^\prime)            \dd \o \ \ \text{for} \ |v| \leq N+1.
		\Ee          
		
		For $0<\delta\ll1$ we have $O(\frac{N^3}{\delta^3})$ number of $v_i \in \R^3$ such that $\{v \in \R^3: |v| \leq N+1\}\subset\bigcup_{i=1}^{O(\frac{N^3}{\delta^3})} B(v_i, \delta)$. Since (\ref{Phi_v}) is smooth in $u$ and $v$ and compactly supported, for $0<\e\ll1$ we can always choose $\delta>0$ such that
		\Be\label{Phi_v_continuous}
		|\Phi_v(u) - \Phi_{v_i} (u)|< \e  \ \ \text{if} \ v \in B(v_i, \delta).
		\Ee
		
		Now we replace $\Phi_v(u)$ in the second line of (\ref{bound_weak_fell-f1}) by $\Phi_{v_i}(u)$ whenever $v \in B(v_i, \delta)$. Moreover we use $\kappa_\delta$-cut off in (\ref{Z_dyn}). If $v$ is included in several balls then we choose the smallest $i$.  From (\ref{Phi_v_continuous}) and (\ref{uniformmlinftybdd}) the difference of (\ref{bound_weak_fell-f1}) and the one with $\Phi_{v_i}(u)$ can be controlled and we conclude that
		\Be
		\begin{split}\label{bound_weak_fell-f1_ell}
			(\ref{bound_weak_fell-f1}) & =\{O( \e)+ O(\delta)\} \sup_{m} \| w_{\vartheta} f^m \|_\infty  ^2\\
			&+  \int^T_0 \int_{\O  } \sum_{i} \int_{\R^3} \mathbf{1}_{v \in B(v_i,\delta) } \int_{\R^3}  
			\kappa_\delta(x,u)  (   f^m (t,x,u) -  f(t,x,u)  ) \Phi_{v_i} (u)  \dd u \\
			& \ \ \ \ \ \ \ \ \ \ \ \ \ \ \ \ \ \ \ \ \  \ \ \ \ \ \ \ \ \ \  \times 
			\{1- \chi (|v|-N)\}  f^m (t,x,v)   
			\varphi_2 (t,x)
			\dd v \dd x\dd t.
		\end{split}\Ee
		From Lemma \ref{extension_dyn} and the average lemma
		\Be\label{H_1/4_gain}
		\max_{ 1\leq i \leq O(\frac{N^3}{\delta^3})}\sup_m \left\| \int_{\R^3} \kappa_\delta (x,u) f^m(t,x,u)  
		\Phi_{v_i} (u) \dd u
		\right\|_{H^{1/4}_{t,x} (\R \times \R^3)}< \infty.
		\Ee
		For $i=1$ we extract a subsequence $m_1 \subset\mathcal{I}_1$ such that 
		\Be\label{strong_converge_extract}
		\int_{\R^3} \kappa_\delta (x,u) f^{m_1}(t,x,u)  
		\Phi_{v_i} (u) \dd u \rightarrow \int_{\R^3} \kappa_\delta (x,u) f (t,x,u)  
		\Phi_{v_i} (u) \dd u \  \text{ strongly in }  \ L^2_{t,x}.
		\Ee
		Successively we extract subsequences $\mathcal{I}_{O(\frac{N^3}{\delta^3})} \subset \cdots\subset \mathcal{I}_2 \subset \mathcal{I}_1$. Now we use the last subsequence $m \in\mathcal{I}_{O(\frac{N^3}{\delta^3})}$ and redefine $f^m$ with it. Clearly we have (\ref{strong_converge_extract}) for all $i$. Finally we bound the last term of (\ref{bound_weak_fell-f1_ell}) by
		\Be\notag
		\begin{split}
			&C_{\varphi_2,N} \max_i \int^T_0\left\|
			\int_{\R^3} \kappa_\delta (x,u) (f^m(t,x,u)- f(t,x,u)  )
			\Phi_{v_i} (u) \dd u
			\right\|_{L^2_{t,x}} \sup_m \| w_{\vartheta} f^m \|_\infty\\
			&\rightarrow 0 \ \ \text{as} \ \ m \rightarrow \infty.
		\end{split}
		\Ee
		Together with (\ref{bound_weak_fell-f1_ell}) we prove $(\ref{weak_fell-f1})\rightarrow 0$. Similarly we can prove $(\ref{weak_fell-f2})      \rightarrow 0$.  
		
		Now we consider $(\ref{weak_form_ell})_{\phi}$. From 
		\Be\begin{split}\notag
			- ( \Delta \phi_{F^m}-  \Delta \phi) 
			= \int  \kappa_\delta (f^m- f)\sqrt{\mu} +   \int (1-\kappa_\delta) (f^m -f)\sqrt{\mu} ,
		\end{split} \Ee
		we have
		\Be\begin{split}
			\| \nabla_x \phi_{F^m} - \nabla_x \phi \|_{L^2_{t,x}} 
			\leq  \left\| \int  \kappa_\delta (f^m- f)\sqrt{\mu}\right\|_{L^2_{t,x}}
			+ O(\delta) \sup_m \| w_{\vartheta} f ^m \|_\infty.
		\end{split}
		\Ee
		Then following the previous argument, we prove $  \nabla_x \phi_{F^m} \rightarrow  \nabla_x \phi$ strongly in $L^2_{t,x}$ as $m \rightarrow \infty$. Combining with $e^{\theta'|v|^2} f^m \overset{\ast}{\rightharpoonup} e^{\theta ' |v|^2 }  f$ in $L^\infty$, we prove $\int^T_0 \iint_{\O \times \R^3}(\ref{weak_form_ell})_{\phi}$ converges to $\int^T_0 \iint_{\O \times \R^3}f  \{\nabla_x \phi \cdot \nabla_v\varphi
		+ \frac{v}{2} \cdot \nabla_x \phi \varphi\}$. This proves the existence of a (weak) solution $f \in L^\infty$. 
		
				\vspace{4pt}

\textit{Step 6.} \quad
From (\ref{uniformmlinftybdd}) and the weak-$*$ lower-semi continuity of $L^\infty$ we conclude (\ref{linfinitybddsolution}).
To prove (\ref{weightedC1bddsolution}), we have from (\ref{c1bddsequence}) that $e^{\theta' |v|^2}e^{-\varpi \langle v \rangle t } \partial f^{m+1}$ has (up to subsequence) a weak-$*$ limit. So for any test function $\varphi(t,x,v)$ we have
\[ \begin{split}
\lim_{m \to \infty} & \int_0^T \iint_{\Omega \times \mathbb R^3 } e^{\theta' |v|^2}e^{-\varpi \langle v \rangle t } \partial f^{m+1} \varphi 
\\ = & \lim_{m \to \infty}   \bigg( \int_0^T \iint_{\Omega \times \mathbb R^3 } \partial(e^{\theta' |v|^2}e^{-\varpi \langle v \rangle t }   \varphi )f^{m+1}
+ \int^T_0 \int_{\gamma_+} e^{\theta' |v|^2}e^{-\varpi \langle v \rangle t } f^{m+1} \varphi 
\\ &  -  \int^T_0 \int_{\gamma_-} 
			c_\mu \sqrt{\mu} e^{\theta' |v|^2}e^{-\varpi \langle v \rangle t } \int_{n \cdot u>0} f^m \sqrt{\mu}\{n \cdot u\} \dd u \varphi \bigg)
\\  =  & \int_0^T \iint_{\Omega \times \mathbb R^3 } \partial(e^{\theta' |v|^2}e^{-\varpi \langle v \rangle t }   \varphi )f^{}
+ \int^T_0 \int_{\gamma_+} e^{\theta' |v|^2}e^{-\varpi \langle v \rangle t } f^{} \varphi 
  -  \int^T_0 \int_{\gamma_-} 
			c_\mu \sqrt{\mu} e^{\theta' |v|^2}e^{-\varpi \langle v \rangle t } \int_{n \cdot u>0} f \sqrt{\mu}\{n \cdot u\} \dd u \varphi
\\  = &  \int_0^T \iint_{\Omega \times \mathbb R^3 } e^{\theta' |v|^2}e^{-\varpi \langle v \rangle t } \partial f^{} \varphi.
\end{split} \]
Therefore $e^{\theta' |v|^2}e^{-\varpi \langle v \rangle t } \partial f^{m+1} \overset{\ast}{\rightharpoonup} e^{\theta' |v|^2}e^{-\varpi \langle v \rangle t } \partial f \in L^\infty$. And (\ref{weightedC1bddsolution}) is obtained by the weak-$*$ lower-semi continuity.
And similarly, from (\ref{l3l1stabilityseq}) we conclude (\ref{L3L1plusbddsolution}). 

Finally, we prove the uniqueness of the solution. Assume $G_0(x,v) = \sqrt \mu g_0(x,v)$ satisfies (\ref{VPBf0assumption}) and $G(t,x,v) = \sqrt \mu g(t,x,v)$ is a solution to (\ref{Boltzmann_E}), (\ref{Bextfield1}), (\ref{VPB2}) with $g(0,x,v) = g_0(x,v)$. Now replace $h^{m+1} - h^m$ by $e^{-\varpi \langle v \rangle t} f - e^{-\varpi \langle v \rangle t} g  $ in the equation (\ref{wf^m+1-wf^m}) and by the same argument as (\ref{gronwall_f-g}) \--- (\ref{gamma_+,1+delta}) we conclude
\[
\| e^{-\varpi \langle v \rangle t} f(t) - e^{-\varpi \langle v \rangle t} g(t) \|_{L^{1+\delta}(\Omega \times \mathbb R^3 ) } \lesssim_t \| f_0 - g_0 \|_{L^{1+\delta}(\Omega \times \mathbb R^3 ) },
\]
and thus the uniqueness.
		\vspace{4pt}

\end{proof}

\appendix
\section{}
 
Recall $\kappa_\delta(x,v)$ in (\ref{Z_dyn}). Let us denote $f_{\delta}(t,x,v) 
	:=      \kappa_\delta (x,v) f(t,x,v)$. We assume that $f(s,x,v)=e^s f_0(x,v)$ for $s<0$. Then $
	\| f_{\delta} \|_{L^{2} (\mathbb{R} \times \Omega \times \mathbb{R}^{3})} 
	\lesssim   \| f \|_{L^{2} (\mathbb{R}_{+} \times \Omega \times \mathbb{R}^{3})}
	+ \| f_{0}\|_{L^{2} (\Omega \times \mathbb{R}^{3})}$, 
	$\| f_{\delta} \|_{L^{2} ( \mathbb{R} \times \gamma)}  \lesssim  \| f_{\gamma} \|_{L^{2} ( \mathbb{R}_{+} \times \gamma)} + \| f_{0} \|_{L^{2} (\gamma)}$.

\begin{lemma} \label{extension_dyn}Assume $\O$ is convex in (\ref{Bextfield1}) and $\sup_{0 \leq t \leq T}\|E(t)\|_{L^\infty (\O)} < \infty$. Let $\bar{E}(t,x) = \mathbf{1}_{\O}(x) E(t,x)$ for $x \in \R^3$. There exists $\bar{f}(t,x,v) \in L^{2}( \mathbb{R} \times   \mathbb{R}^{3} \times \mathbb{R}^{3})$, an extension of $f_{\delta}$, such that 
		\begin{equation}\notag
		\bar{f}  |_{\Omega \times \mathbb{R}^{3}}\equiv f_{\delta}    \  \text{ and } \  \bar{f}  |_{\gamma}\equiv f_{ \delta} |_{\gamma}   \  \text{ and } \ \bar{f } |_{t=0} \equiv f_{\delta} |_{t=0}.
		\end{equation} Moreover, in the sense of distributions on $\mathbb{R} \times \mathbb{R}^{3} \times \mathbb{R}^{3}$,
		\begin{equation}\label{eq_barf_dyn}
		[\partial_{t} +  v\cdot \nabla_{x} + \bar{E} \cdot \nabla_{v}]\bar{f} =  h  ,
		\end{equation}
		where
		\Be\begin{split}\label{barf_h}
			h_{} (t,x,v)=& \ 
			\kappa_\delta(x,v)    \mathbf{1}_{t \in [0,\infty)} 
			[ \partial_{t} + v\cdot \nabla_{x}   + E
			\cdot \nabla_{v}  ] f\\
			&
			+ \kappa_\delta(x,v)  \mathbf{1}%
			_{t \in ( - \infty, 0 ]} e^t
			[1  + v\cdot \nabla_{x}   + E
			\cdot \nabla_{v}] f_0 \kappa_\delta(x,v) \\
			& 
			+ f (t,x,v) [v\cdot \nabla_{x} + E
			\cdot \nabla_{v}] \kappa_\delta(x,v),\\
		\end{split}\Ee
		where $t_{\mathbf{b}}^{EX}, x_{\mathbf{b}}^{EX}, t_{\mathbf{f}}^{EX}, x_{\mathbf{f}}^{EX}$ are defined in (\ref{def_tb_EX}).
		
		Moreover,
		\Be\begin{split}\label{estimate_h1_h2}
			\| h_{} \|_{  L^{2}( \mathbb{R} \times   \mathbb{R}^{3} \times \mathbb{R}^{3})}  
			\lesssim  & \  \|  [ \partial_{t} + v\cdot \nabla_{x}   + E
			\cdot \nabla_{v}  ] f\|_{L^{2}(
				\mathbb{R}_{+} \times 
				\Omega \times \mathbb{R}^{3})} 
			+  \| f \|_{L^{2} (\R \times \Omega \times \mathbb{R}^{3})} \\
			&  
			+ \| [ v\cdot \nabla_{x} + E\cdot \nabla_{v} ] f_{0} \|_{L^{2 } (\Omega \times \mathbb{R}^{3})}.
		\end{split}\Ee
		
	\end{lemma}
	
	\begin{proof} In the sense of distributions  
		\begin{equation} \label{eqtn_f_delta}
		\partial_{t} f_{\delta}+ v\cdot \nabla_{x} f_{\delta} + E
		\cdot \nabla_{v} f_{\delta} = h  \text{ in }  (\ref{barf_h}).
		\end{equation}
		Clearly $| [v\cdot \nabla_{x} + E
		\cdot \nabla_{v}] \kappa_\delta(x,v)| \lesssim_\delta 1$.\hide
		\begin{eqnarray}
		&&\Big|\{v\cdot \nabla_{x} + \e^{2} \Phi \cdot \nabla_{v}\} [1-\chi(\frac{%
			n(x) \cdot v}{\delta}) \chi \big( \frac{ \xi(x )}{\delta}\big) ]
		\chi(\delta|v|)\Big|  \label{der_chi} \\
		&=&\Big| - \frac{1}{\delta} \{v\cdot \nabla_{x} n(x) \cdot v + \e^{2} \Phi
		\cdot n(x) \} \chi^{\prime} \big(\frac{n(x) \cdot v}{\delta} \big) \chi %
		\big( \frac{ \xi(x )}{\delta}\big) \chi(\delta|v|)  \notag \\
		&& - \ \frac{1}{\delta} v\cdot \nabla_{x} \xi(x) \chi^{\prime} \big( \frac{%
			\xi(x)}{\delta}\big) \chi (\frac{n(x) \cdot v}{\delta}) \chi(\delta|v|) + \e%
		^{2}\delta \Phi \cdot \frac{v}{|v|} \chi^{\prime} (\delta|v|)[1-\chi(\frac{%
			n(x) \cdot v}{\delta}) \chi \big( \frac{\xi(x)}{\delta}\big) ] \Big|  \notag
		\\
		&\leq& \frac{4}{\delta}( |v|^{2}\|\xi\|_{C^2} + \e^{2}\|\Phi\|_\infty )
		\chi(\delta|v|) + \frac{C_{\Omega}}{\delta} |v|\chi(\delta|v|) + \e^{2}
		\delta\|\Phi\|_\infty \mathbf{1}_{|v| \leq {2}{\delta}^{-1}}  \notag \\
		&\lesssim & {\delta^{-3}} \mathbf{1}_{|v| \leq 2 \delta^{-1}}.  \notag
		\end{eqnarray}\unhide

		For $x \in \R^3 \backslash \bar{\O}$ we define
		\Be\begin{split}\label{def_tb_EX}
			\tb^{EX}(x,v) &:= \sup\{s \geq 0: x-\tau v \in \R^3 \backslash \bar{\O}
			\ \text{ for all } \ \tau \in (0,s)
			\},\\
			\tf^{EX}(x,v) &:= \sup\{s \geq 0: x+\tau v \in \R^3 \backslash \bar{\O}
			\ \text{ for all } \ \tau \in (0,s)
			\},
		\end{split}\Ee
		and $\xb^{EX}(x,v) = x- \tb^{EX}(t,x,v))v$, $\xf^{EX}(x,v) = x + \tf^{EX}(t,x,v))v$.
		
		We define, for $x \in \R^3 \backslash \bar{\O}$,
		\Be\label{def_f_E}
		\begin{split}
			f_E (t,x,v) =& \mathbf{1}_{\xb^{EX} (t,x,v) \in \p \O} f_\delta(t-\tb^{EX}(x,v), \xb^{EX} (x,v),v)\\
			+& \mathbf{1}_{\xf^{EX} (t,x,v) \in \p \O}f_\delta(t+\tf^{EX}(x,v), \xf^{EX} (x,v),v).
		\end{split}\Ee
		Recall that, from (\ref{Z_dyn}), $f_\delta\equiv0$ when $n(x) \cdot v = 0$, and hence $f_E\equiv0$ for $n(x) \cdot v = 0$. Since $\O$ is convex if $v\neq 0$ then $\{\xb^{EX} (x,v) \in \p \O\} \cap \{\xf^{EX} (x,v) \in \p \O\}= \emptyset$. Note that 
		\Be\label{no_jump_bdry}
		f_E(t,x,v) = f_\gamma(t,x,v) = f_\delta (t,x,v) \ \ \text{for }   x \in \p\O.
		\Ee
			And since for any $s>0$, 
			\[ \begin{split}
			&(t +s - \tb^{EX}(x+sv,v), \xb^{EX}(x+sv,v),v  ) = (t - \tb^{EX}(x,v),\xb^{EX}(x,v),v) 
			\\ & (t +s + \tf^{EX}(x+sv,v), \xf^{EX}(x+sv,v),v  ) = (t - \tf^{EX}(x,v),\xf^{EX}(x,v),v),
			\end{split} \]
		so in the sense of distribution, in $\R \times [\R^3 \backslash \bar{\O}] \times \R^3$ 
		\Be
		\label{eqtn_f_E}
		\p_t f_E + v\cdot \nabla_x f_E = 0.
		\Ee

		We define 
		\Be\label{def_bar_f}
		\bar{f}(t,x,v) : = \mathbf{1}_{\O} (x)  f_\delta (t,x,v)
		+ \mathbf{1}_{\R^3 \backslash \bar{\O}} (x) f_E (t,x,v).
		\Ee

		From (\ref{eqtn_f_delta}), (\ref{no_jump_bdry}), and (\ref{eqtn_f_E}) we prove (\ref{eq_barf_dyn}). The estimates of (\ref{estimate_h1_h2}) are direct consequence of Lemma \ref{tracebddpotential}.
		\end{proof}

\textbf{Acknowledgements.} This paper is part of the author's thesis. He thanks his advisor Professor Chanwoo Kim for helpful discussions. The research is supported in part by National Science Foundation under Grant No. 1501031.

\end{document}